\documentclass[a4,12pt]{amsart}
\usepackage{amssymb,amstext,amscd,amsthm,amsfonts,mathdots,mathrsfs}
\usepackage{hyperref}
\usepackage{pdflscape}
\usepackage{setspace}
\usepackage[all]{xy}
\usepackage{enumerate}
\usepackage{dashbox}
\newcommand\dboxed[1]{\dbox{\ensuremath{#1}}}
\usepackage{indentfirst}
\usepackage{ytableau}
\usepackage{color}
\usepackage{colortbl}
\usepackage{amsmath}
\usepackage{tikz}
\usetikzlibrary{shapes.geometric, arrows}
\usetikzlibrary{calc}
\def\arraystretch{1.2} 
\numberwithin{equation}{section}
\usepackage{adjustbox}

\def\k{\Bbbk}
\def\Z{\mathbb{Z}}
\def\N{\mathbb{N}}

\def\P{\mathcal{P}}
\def\ssum{\textstyle\sum\limits}
\def\Hom{\operatorname{Hom}}
\def\End{\operatorname{End}}
\def\Ext{\operatorname{Ext}}
\def\max{\operatorname{max}}
\def\ev{\operatorname{ev}}
\def\deg{\operatorname{deg}}

\def\soc{\operatorname{soc}}

\def\res{\operatorname{res}}
\def\ch{\operatorname{char}}
\def\defect{\operatorname{def}}
\def\hub{\operatorname{hub}}
\newcommand{\Ind}{\operatorname{Ind}}
\newtheorem{thm}{Theorem}

\newtheorem{rem}[thm]{Remark}

\newcommand{\Q}{\mathbb Q}

\newcommand{\K}{\mathbb K}

\newcommand{\Rad}{\operatorname{Rad}}
\newcommand{\Soc}{\operatorname{Soc}}

\def\pcl{P^+_{cl,k}}
\def\mod{\operatorname{\mathsf{mod}}}
\def\proj{\operatorname{\mathsf{proj}}}
\def\Db{\mathsf{D}^{\rm b}}
\def\Kb{\mathsf{K}^{\rm b}}
\def\thick{\operatorname{\mathsf{thick}}}
\def\add{\operatorname{\mathsf{add}}}
\def\silt{\operatorname{\mathsf{silt}}}
\def\tilt{\operatorname{\mathsf{tilt}}}
\def\H{\mathcal{H}}
\newcommand{\twosilt}{\mathsf{2\mbox{-}silt}\hspace{.02in}}

\newtheorem{theorem}{Theorem}[section] 
\newtheorem{theorem*}{Theorem}
\newtheorem{lemma}[theorem]{Lemma}
\newtheorem{corollary}[theorem]{Corollary}
\newtheorem*{corollary*}{COROLLARY}
\newtheorem{proposition}[theorem]{Proposition}
\newtheorem{example}[theorem]{Example}
\newtheorem{definition}[theorem]{Definition}

\newtheorem{remark}[theorem]{Remark}
\theoremstyle{definition}

\newtheorem*{maintheorem}{MAIN THEOREM}
\newtheorem*{introtheorem}{THEOREM}

\newcommand{\De}{\Delta}
\newcommand{\La}{\Lambda}

\begin{document}
\setlength{\baselineskip}{17pt}

\title[Representation type of cyclotomic KLR algebras in affine type C]
{Representation type of higher level cyclotomic quiver Hecke algebras in affine type C}

\author{Susumu Ariki}
\address{}
\email{ariki@ist.osaka-u.ac.jp}

\author{Berta Hudak}
\address{Okinawa Institute of Science and Technology, Okinawa, 904-0495, Japan.}
\email{berta.hudak@oist.jp}
\author{Linliang Song}
\address{School of Mathematical Science, Key Laboratory of Intelligent Computing and Applications(Ministry of Education), Tongji University, Shanghai, 200092, China.}
\email{llsong@tongji.edu.cn}
\author{Qi Wang}
\address{School of Mathematical Sciences, Dalian University of Technology, Dalian, 116024, China}
\email{wang2025@dlut.edu.cn}
\date{\today}



\begin{abstract}
We determine the representation type of cyclotomic quiver Hecke algebras of affine type C. In the tame cases, we explicitly describe their basic algebras under the assumption $\ch \k\ne2$, relying on the Morita invariance of cellularity.
\end{abstract}
\maketitle

\tableofcontents
\section{Introduction}
Representation type serves as a fundamental tool in the representation theory of finite-dimensional algebras, especially, over an algebraically closed field $\k$. Here, we consider the category of finitely generated left modules, so that all modules are assumed to be finite-dimensional. Namely, representation type gives us criteria whether we can study the module category in depth or we must be content with either, study of better behaved subcategories, or, study on the Grothendieck group of the module category, such as character formulas for irreducible modules, etc. 

A finite-dimensional $\k$-algebra $A$ is said to be \emph{representation-finite} if it admits only finitely many indecomposable modules up to isomorphism; otherwise, $A$ is said to be \emph{representation-infinite}. 
A representation-infinite $\k$-algebra $A$ is said to be \emph{tame} if all but finitely many $d$-dimensional indecomposable $A$-modules can be organized in finitely many one-parameter families, for each dimension $d$, and it is called \emph{wild} if there is an exact $\k$-linear functor sending modules over the free associative algebra $\k\langle x,y\rangle$ to modules over $A$ which preserves indecomposability and respects isomorphism classes.
It is known as the famous (Finite-)Tame-Wild Trichotomy (\cite{Dr-tame-wild}) that the representation type of any finite-dimensional algebra over $\k$ is exactly one of representation-finite, tame\footnote{Following Erdmann \cite{Er-tame-block}, our tame representation type, tame for short, excludes representation-finite algebras.} and wild.

It is a natural desire to find such criteria for well-known classes of algebras. The class of path algebras is the most famous class of algebras, and Dynkin quivers of finite ADE and affine ADE types appear beautifully in the criteria. Another important class of algebras is the class of group algebras such as those of the symmetric groups. 

The modular representation theory of the symmetric group has a long history. Class of algebras which the group algebras of the symmetric group belong started with the class of the group algebras of finite Coxeter groups. Then, the class was expanded to their $q$-deformation, that is, the class of Iwahori-Hecke algebras, and then to the class of cyclotomic Hecke algebras (\cite{AK-algebra, BM-Hecke-alg}) associated with complex reflection groups, in which the algebras associated with complex reflection groups $G(m,1,n)$, so-called Ariki-Koike algebras, received detailed study (e.g., \cite{bk-graded-decomp-numbers, DJM-cyclotomic-q-schur-alg, F-ariki-koike-alg, LM-cyclotomic-hecke}). Currently, we study algebras in the much wider class of cyclotomic quiver Hecke algebras (\cite{kl-diagrammatic, Ro-2kac}), which are associated with Lie theoretic data: the Lie type determined by a symmetrizable (generalized) Cartan matrix $\mathsf{A}$, an element $\beta$ in the positive cone $Q_+$ of the root lattice, and a dominant integral weight $\Lambda$ in the weight lattice. Those data come from categorification theorems which categorify weight spaces $V(\Lambda)_{\Lambda-\beta}$ of the integrable highest weight module $V(\Lambda)$ over the Kac-Moody Lie algebra $\mathfrak{g}(\mathsf{A})$ of the symmetrizable Cartan matrix. In our setting, the module category over the cyclotomic quiver Hecke algebra $R^\Lambda(\beta)$ categorifies the weight space. For example, the group algebras of the symmetric group in positive characteristics and Hecke algebras of type A at roots of unity are associated with level one dominant integral weights of type $A^{(1)}_\ell$, and Hecke algebras of type B at roots of unity are associated with level two dominant integral weights of type $A^{(1)}_\ell$. The cyclotomic quiver Hecke algebras are also called cyclotomic Khovanov-Lauda-Rouquier algebras, cyclotomic KLR algebras for short. 

Cyclotomic quiver Hecke algebras are graded algebras. In particular, the group algebras of the symmetric group are graded algebras. This finding, due to Brundan and Kleshchev \cite{BK-block}, could not be seen by using Coxeter generators: their deep insight led them to the finding of Khovanov-Lauda-Rouquier generators in the group algebras of the symmetric group. 

Recently, cyclotomic quiver Hecke algebras of affine type other than $A^{(1)}_\ell$ attracts researchers in this field. For example, Park, Speyer and the first author \cite{APS-type-C} introduced Specht modules for type $C^{(1)}_\ell$, Evseev and Mathas \cite{EM-cellular-symmetrictypeA} proved 
and  Mathas and Tubbenhauer \cite{MT-cellularity-type-C} reproved that the cyclotomic quiver Hecke algebras of type $C^{(1)}_\ell$ are graded cellular algebras\footnote{For the recent progress on cyclotomic quiver Hecke algebras of finite type, see \cite{MT-decnumbers-type-C}.}.  Some experimental calculations of the decomposition numbers have been carried out by Chung, Mathas and Speyer \cite{CMS-type-c-decomposition-matrix}. 

In this article, we determine representation type for all cyclotomic quiver Hecke algebras $R^\Lambda(\beta)$ of type $C^{(1)}_\ell$, where $\ell\ge 2$. Since we already know representation type of $R^\Lambda(\beta)$ when $\Lambda$ is a fundamental weight, we assume that the level $k$ of the dominant integral weight $\Lambda$ is greater than or equal to $2$. 
We denote the set of weights of $V(\Lambda)$ by $P(\Lambda)$. Recall that $R^\Lambda(\beta)$ and $R^\Lambda(\Lambda-w\Lambda+w\beta)$, for $w\in W$, where $W$ is the (affine) Weyl group, have the same representation type, so that it suffices to consider those $\beta\in Q_+$ such that $\Lambda-\beta$ are dominant integral weights. 
Furthermore, $\Lambda-\beta$ is not a maximal weight if and only if there exists $w\in W$ such that $w(\Lambda-\beta)$ is dominant but not maximal. 

\begin{maintheorem}\label{Mainth}
Suppose that the level of $\Lambda$ is $k\ge2$ and we write
$$
\Lambda=m_0\Lambda_0+m_1\Lambda_1+\cdots+m_\ell\Lambda_\ell,
$$
where $m_0, m_1, \ldots, m_\ell\in\Z_{\ge0}$ and $m_0+m_1+\cdots+m_\ell=k$.
\begin{enumerate}
\item[(1)] If $\Lambda-\beta$ is not a maximal weight, then $R^\Lambda(\beta)$ is wild.

\item[(2)]
Suppose that $\Lambda-\beta$ is a dominant maximal weight in $P(\Lambda)$. 
\begin{enumerate}
\item[(a)] $R^{\Lambda}(\beta)$ is of finite representation type if one of the following holds.
    \begin{itemize}
         \item[(f1)]
             $\beta=\alpha_a$, for $0\le a\le \ell$, and $m_a\ge2$.
         \item[(f2)]
             $\beta=\alpha_0+\alpha_1$, and $m_0\ge1$, $m_1=0$ or $m_0=m_1=1$.
         \item[(f3)]
             $\beta=\alpha_{\ell-1}+\alpha_\ell$, and $m_{\ell-1}=0$, $m_\ell\ge1$ or $m_{\ell-1}=m_\ell=1$.
         \item[(f4)]
             $\beta=\alpha_a+\cdots+\alpha_b$, for $1\le a<b\le\ell-1$, and 
             $m_i=\delta_{ai}+\delta_{bi}$, for $a\le i\le b$. 
        \item[(f5)]
             $\beta=\alpha_0+2\alpha_1+\cdots+2\alpha_a+\alpha_{a+1}$, for $0\le a\le\ell-2$, and $m_i=\delta_{ai}$, for $0\le i\le a+1$.
        \item[(f6)]
             $\beta=\alpha_{b-1}+2\alpha_b+\cdots+2\alpha_{\ell-1}+\alpha_\ell$, for $2\le b\le\ell$, and $m_i=\delta_{bi}$, for $b-1\le i\le\ell$. 
    \end{itemize}
  
\item[(b)] $R^{\Lambda}(\beta)$ is of tame representation type if one of the following holds.
    \begin{itemize}
       \item[(t1)]
             $\beta=\alpha_0+2\alpha_1$, $m_0=0$ and $m_1=2$.
       \item[(t2)]
             $\beta=2\alpha_{\ell-1}+\alpha_\ell$, $m_{\ell-1}=2$ and $m_\ell=0$.
       \item[(t3)]
             $\beta=\alpha_0+\alpha_1$, $m_0\ge2$ and $m_1=1$.
       \item[(t4)]
             $\beta=\alpha_{\ell-1}+\alpha_\ell$, $m_{\ell-1}=1$ and $m_\ell\ge2$.
       \item[(t5)]
             $\beta=\alpha_0+\cdots+\alpha_a$, for $1\le a\le \ell-1$, $m_0\ge1$ and $m_i=\delta_{ia}$, for $1\le i\le a$, except for the case $a=1$ and $m_0=1$, which is (f2). 
       \item[(t6)]
             $\beta=\alpha_a+\cdots+\alpha_\ell$, for $1\le a\le \ell-1$, $m_\ell\ge1$ and $m_i=\delta_{ai}$, for $a\le i\le \ell-1$, except for $a=\ell-1$ and $m_\ell=1$, which is (f3). 
       \item[(t7)]
             $\beta=\alpha_0+\alpha_1$, $m_0=1$ and $m_1=2$. 
       \item[(t8)]
             $\beta=\alpha_{\ell-1}+\alpha_\ell$, $m_{\ell-1}=2$ and $m_\ell=1$. 
       \item[(t9)]
             $\beta=\alpha_a+\cdots+\alpha_b$, for $1\le a<b\le \ell-1$, either $m_a\ge 2$ and $m_i=\delta_{ib}$, for $a<i\le b$, or $m_b\ge2$ and $m_i=\delta_{ai}$, for $a\le i<b$. 
       \item[(t10)]
             $\beta=\alpha_0+\alpha_i$, for $2\le i\le \ell$, $m_0=m_i=2$.
       \item[(t11)]
             $\beta=\alpha_i+\alpha_\ell$, for $0\le i\le \ell-2$, $m_i=m_\ell=2$.
       \item[(t12)]
             $\beta=\alpha_0+\alpha_1+\alpha_{\ell-1}+\alpha_\ell$ where $\ell\ge4$, $m_0=m_\ell=1$ and $m_1=m_{\ell-1}=0$.
       \item[(t13)]
             $\beta=\alpha_0+\alpha_1+\alpha_i$, for $3\le i\le \ell$, $m_0=1$, $m_1=0$ and $m_i=2$.
       \item[(t14)]
             $\beta=\alpha_i+\alpha_{\ell-1}+\alpha_\ell$, for $0\le i\le \ell-3$, $m_i=2$ and $m_{\ell-1}=0$, $m_\ell=1$.
       \item[(t15)]
             $\beta=\alpha_{a-1}+2\alpha_a+\alpha_{a+1}$, for $2\le a\le \ell-2$, 
             $m_a=2$, $m_{a\pm 1}=0$, and $\ch \k\ne 2$.
       \item[(t16)]
             $\beta=2\alpha_a+\alpha_{a+1}$, for $1\le a\le \ell-2$, $m_a=3$,$m_{a+1}=0$ and $\ch \k\ne 3$.
       \item[(t17)]
             $\beta=\alpha_{a-1}+2\alpha_a$, for $2\le a\le \ell-1$, $m_a=3$, $m_{a-1}=0$ and $\ch \k\ne 3$.
       \item[(t18)]
             $\beta=\alpha_a+\alpha_b$, for $1\le a<b\le \ell-1$ where $a\le b-2$, $m_a=m_b=2$.
       \item[(t19)]
             $\beta=2\alpha_a$, for $1\le a\le \ell-1$, $m_a=4$ and $\ch \k\ne 2$. 
       \item[(t20)]  
        $\beta=2\alpha_0+2\alpha_1$, $m_0=2$, $m_1=0$ and $\ch\k\ne 2$.
       \item[(t21)]
       $\beta=2\alpha_{\ell-1}+2\alpha_\ell$, $m_{\ell-1}=0$, $m_\ell=2$ and $\ch\k\ne 2$.
    \end{itemize}

\item[(c)] $R^{\Lambda}(\beta)$ is of wild representation type otherwise.
\end{enumerate}
\end{enumerate}
\end{maintheorem}

The proof of MAIN THEOREM uses the idea to introduce quiver structure on the set of dominant maximal weights $\max^+(\Lambda)$, which was found and applied to type $A^{(1)}_\ell$ in \cite{ASW-rep-type}. 
However, we choose a different strategy than the [{\it loc. cit.}] after introducing the quiver of dominant maximal weights. 
While we first fixed a certain neighborhood of the weight $\Lambda$, which was found by consideration on the coefficients of $\beta$, and started with showing that those weights outside the neighborhood give us wild cyclotomic KLR algebras in \cite{ASW-rep-type}, we start with investigating dominant maximal weights $\Lambda'$ which can be reached by at most one step, two steps, three steps from $\Lambda$ one by one first, and determine representation type of the associated cyclotomic KLR algebras $R^\Lambda(\beta_{\Lambda'})$. 
Then, we reach the conclusion that algebras which cannot be reached by less than or equal to three steps are wild. See Section 4 for the details. 

In the course of the proof, we obtain explicit presentations of non-wild algebras, see Section 6 and Section 7. 
In type $A^{(1)}_\ell$, all tame $R^\Lambda(\beta_{\Lambda'})$ associated with dominant maximal weights $\Lambda'$ are Brauer graph algebras. It implies that all tame cyclotomic KLR algebras of type $A^{(1)}_\ell$ are Brauer graph algebras, and this fact allowed us to determine the Morita equivalence classes\footnote{Precisely speaking, we need either $\ch \k\ne2$ or the cyclotomic KLR algebra being a basic algebra.} of tame cyclotomic KLR algebras of type $A^{(1)}_\ell$. In type $C^{(1)}_\ell$, there are tame cyclotomic KLR algebras $R^\Lambda(\beta)$ which are not Brauer graph algebras. One already appeared in \cite[Lemma 3.1]{CH-type-c-level-1} as a level one cyclotomic KLR algebra, which is the algebra (5) in \cite[Theorem 1]{AKMW-cellular-tamepolygrowth}. The other tame algebras appear as level three cyclotomic KLR algebras in this paper, i.e., (t7) and (t8). 
For the former case, we need to recall Skowro\'{n}ski's classification of standard domestic symmetric algebras (\cite{Sk06}). However, since $R^\La(\beta)$ is cellular (see \cite{EM-cellular-symmetrictypeA}), it is natural to assume that $\ch \k \ne 2$ and utilize Morita invariance of the cellularity. Then, the cyclotomic KLR algebras that are derived equivalent to the algebra from \cite{CH-type-c-level-1} must appear in the list \cite[Theorem 1]{AKMW-cellular-tamepolygrowth}, and one can check that other algebras in the list do not appear as cyclotomic KLR algebras of type $C^{(1)}_\ell$ by excluding Brauer graph algebras and those with a different number of simple modules in the list. 
For the latter case, we may use silting theory to find Morita equivalence classes in the derived equivalence class of the algebra (t7) (or equivalently, (t8)). See Theorem \ref{enumeration of Morita classes} for the method, and see Proposition \ref{Morita classes of (t7)} for the Morita equivalence classes which are in the derived equivalence class of (t7). Otherwise, tame cyclotomic KLR algebras of type $C^{(1)}_\ell$ are Brauer graph algebras. As was shown in \cite{ASW-rep-type}, their Brauer graphs are straight lines except for one Brauer graph (i.e., the cases (t1) and (t2)), and we may read off the set of multiplicities of vertices. Then, we assign the multiplicities to vertices. In the following, we give Morita equivalence classes of finite and tame algebras $R^\La(\beta)$ in explicit forms\footnote{We do not know whether all the possible assignment of the given multiset of multiplicities to vertices actually appear.}.

\begin{introtheorem}[finite cases]\label{finite algebras}
Let $R^\Lambda(\beta)$ be a cyclotomic KLR algebra of type $C^{(1)}_\ell$ and suppose that $R^\Lambda(\beta)$ is of finite representation type.
If $\ch \k\ne2$, then $R^\Lambda(\beta)$ is Morita equivalent to one of the following algebras\footnote{These algebras already appeared in \cite[8.1]{ASW-rep-type}.}.
\begin{itemize}
\item[(a)] Symmetric local algebra $\k[X]/(X^m)$, for $m\ge 2$.
\item[(b)] Brauer tree algebra whose Brauer tree is a straight line.
\end{itemize}
\end{introtheorem}

\begin{introtheorem}[tame cases]\label{tame algebras}
Let $R^\Lambda(\beta)$ be a cyclotomic KLR algebra of type $C^{(1)}_\ell$ and suppose that $R^\Lambda(\beta)$ is of tame representation type. If $\ch \k\ne2$, then $R^\Lambda(\beta)$ is Morita equivalent to one of the following algebras.
\begin{itemize}
\item[(a)]
Symmetric local algebras (2), (3), (4) in \cite[8.2]{ASW-rep-type}.

\item[(b)]
Brauer graph algebra whose Brauer graph is a straight line and the multiset of the multiplicities of vertices is $\{1,t,2t,\dots,2t\}$, for $t\ge1$, $\{4,2,2\}$ or Brauer graph algebras (5), (7) in \cite[8.2]{ASW-rep-type}, or the Brauer graph algebra without an exceptional vertex whose Brauer graph is as follows.
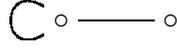
\[
\xymatrix@C=1cm{\circ\ar@(dl,ul)@{-}^{\ }&\circ \ar@{-}[l]^{\ }}
\]

\item[(c)] The algebra $\k Q/J$, where the quiver $Q$ is
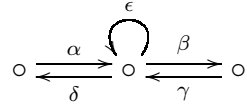
\begin{center}
$\xymatrix@C=1cm{\circ \ar@<0.5ex>[r]^{\alpha} &\circ \ar@<0.5ex>[l]^{\delta}\ar@(ur,ul)_{\epsilon}\ar@<0.5ex>[r]^{\beta}&\circ\ar@<0.5ex>[l]^{\gamma}}$
\end{center}
and the relations given by the admissible ideal $J$ are
\begin{center}
$\alpha\beta=\gamma\delta=0$, $\alpha\epsilon=\epsilon\beta=\gamma\epsilon=\epsilon\delta=0$, $\delta\alpha=\epsilon^2=\beta\gamma$.
\end{center}

\item[(d)]
The algebra $\k Q/J$, where the quiver $Q$ is
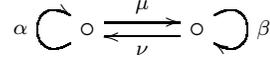
\begin{center}
$\xymatrix@C=1cm{\circ \ar@<0.5ex>[r]^{\mu}\ar@(dl,ul)^{\alpha}&\circ \ar@<0.5ex>[l]^{\nu}\ar@(ur,dr)^{\beta}}$
\end{center}
and the relations given by the admissible ideal $J$ are
\begin{center}
$\alpha^2=0$, $\beta^2=\nu\mu$, $\alpha\mu=\mu\beta$, $\beta\nu=\nu\alpha$.
\end{center}

\item[(e)]
The algebra $\k Q/J$, where the quiver $Q$ is
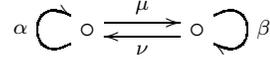
\begin{center}
$\xymatrix@C=1cm{\circ \ar@<0.5ex>[r]^{\mu}\ar@(dl,ul)^{\alpha}&\circ \ar@<0.5ex>[l]^{\nu}\ar@(ur,dr)^{\beta}}$
\end{center}
and the relations given by the admissible ideal $J$ are
\begin{center}
$\alpha^2=\mu\nu$, $\beta^2=\nu\mu$, $\alpha\mu=\mu\beta$, $\beta\nu=\nu\alpha$, $\mu\nu\mu=\nu\mu\nu=0$.
\end{center}
    \end{itemize}
\end{introtheorem}

As we mentioned, in general it is difficult to study the category of all finite-dimensional modules and instead, we try to find nice subcategories.
One such example is the representation theory of quantum affine algebras, in which field researchers found good subcategories to study such as the Hernandez-Leclerc categories: these categories have been actively studied by cluster algebra techniques in recent years. We claim that the subcategories of modules over tame $R^\Lambda(\beta)$'s are also such nice subcategories, for which we have more chance to tackle difficult problems like finding a dimension formula for irreducible modules or decomposition numbers. Besides, in affine type A they are related to the classical subject of affine Hecke algebras in type A: if we consider the Serre subcategory consisting of modules whose composition factors belong to a given finite set of irreducible modules, then one obtains a filtration of the Serre subcategory over the affine Hecke algebra by the Serre subcategories over cyclotomic Hecke algebras which share the same set of irreducible modules. Then one may use grading and results from \cite{ASW-rep-type}. 

Another fascinating aspect of this paper is that we connect the recently emerging theory of Brauer graph algebras, $\tau$-tilting theory and silting theory with the representation theory of cyclotomic quiver Hecke algebras: 
in affine type A, all tame blocks are Brauer graph algebras and we applied results by Opper and Zvonareva which they obtained by using a version of Fukaya category, and, as we have explained in the previous page, we utilize $\tau$-tilting theory to build a complete framework (see Theorem \ref{enumeration of Morita classes}) for finding Morita equivalence classes in the derived equivalence class of a given symmetric algebra.
This will benefit not only the study in other types, but also the research of symmetric algebras in general.

\subsection*{Conventions}
Set $\N:=\{1,2,\ldots\}$ and $\Z_{\ge 0}:=
\{0,1,2,\ldots\}$. For $m,m'\in \Z$, we write $m\equiv_2 m'$ if $m-m'$ is even, and $m\not\equiv_2 m'$ otherwise.

We use left modules throughout the paper. Hence, the basic algebra of an algebra $A$ is $\End_A(P)^{\rm op}$, where $P$ is a progenerator which is basic. Suppose that $\alpha: P_i\rightarrow P_j$ and $\beta: P_j\rightarrow P_k$ are $A$-module homomorphisms between indecomposable projective $A$-modules $P_i$, $P_j$ and $P_j$, $P_k$, respectively. 
Then, the composition $\beta\circ \alpha: P_i\rightarrow P_k$ is denoted by $\alpha\beta$ since we consider the opposite algebra of $\End_A(P)$. 
When $\alpha$ and $\beta$ are irreducible homomorphisms, we view them as arrows of the Gabriel quiver of $A$. Then, our convention is that concatenation of the arrow $\alpha: i\rightarrow j$ and the arrow $\beta: j\rightarrow k$ is $\alpha\beta: i\rightarrow k$. Let $Q$ be the Gabriel quiver and $A=\k Q/J$, for an admissible ideal $J$. Then, an $A$-module is a vector space $M$ equipped with an algebra homomorphism $\rho_M: A\rightarrow \End_{\k}(M)$, and we study $A$-modules as an assignment of matrices to arrows that satisfy the defining relations given by $J$. However, some representation theorists go further to consider 
decomposition of $M$ into $M=\oplus_{i=1}^n e_iM$ where $1=\sum_{i=1}^n e_i$ is the sum of pairwise orthogonal primitive idempotents, and interpret $M$ into a representation of the quiver $Q$. 
Then, they prefer to think that arrows $i\rightarrow j$ are elements of $e_jAe_i$, not $e_iAe_j$ 
which we have just seen in the description of $\End_A(P)^{\rm op}$. 
This is because they prefer to assign a linear map $e_iM\rightarrow e_jM$ to an arrow $i\rightarrow j$.
If we use that interpretation, the standard recipe is that we put $e_iM$ on the vertex $i$ and we consider irreducible homomorphisms $\alpha\in e_jAe_i$ and $\beta\in e_kAe_j$ as arrows $i\rightarrow j$ and $j\rightarrow k$. Then the composition needs to be denoted by $\beta\alpha$ since we must have $\rho_M(\beta)\rho_M(\alpha)=\rho_M(\beta\alpha)$. We do not adopt that convention and do not use representations of quivers. Since $\alpha\in e_iAe_j$ and $\beta\in e_jAe_k$, we have $\alpha: e_jM\rightarrow e_iM$, $\beta: e_kM\rightarrow e_jM$ and $\alpha\beta:e_kM\rightarrow e_iM$ in our convention\footnote{Namely, it is a representation of the opposite quiver of $Q$.}.  

\section{Preliminaries}
We review some background materials which we need in this paper, including the definition of cyclotomic KLR algebras, and the fundamentals of silting/tilting theory. Additionally, we provide several lemmas in this section for later use. 

\subsection{Cartan datum in affine type C}
Set $I=\{0, 1, 2, \ldots, \ell\}$ with $\ell\ge 2$.
The \emph{affine Cartan matrix} $\mathsf{A}$ of type $C^{(1)}_\ell$ is defined by 
$$
\mathsf{A}=(a_{ij})_{i,j\in I}:=
\left(
\begin{array}{ccccccc}
2 &-1& 0 &~\ldots~&0&0~&0\\
-2 & 2&-1&\ldots&0&0&0 \\
0& -1 &2&\ldots&0&0&0\\
\vdots&\vdots&\vdots&\ddots &\vdots &\vdots&\vdots\\
0 &0&0&\ldots&2 &-1& 0 \\
0 &0&0&\ldots&-1& 2& -2 \\
0 &0&0&\ldots&0& -1&2 
\end{array}\right),
$$
where the rows and the columns are labeled by $0, 1, \ldots, \ell$ in this order. 
If we drop the first row and the first column of $\mathsf{A}$, we obtain the Cartan matrix $\mathsf{A}'$ of type $C_\ell$; in this case, the simple roots are realized in the lattice $\Z\epsilon_1\oplus \Z\epsilon_2\oplus \cdots \oplus \Z\epsilon_\ell$ as
$$
\alpha_1=\epsilon_1-\epsilon_2, 
\quad
\alpha_2=\epsilon_2-\epsilon_3,
\quad \ldots, \quad 
\alpha_{\ell-1}=\epsilon_{\ell-1}-\epsilon_{\ell},
\quad
\alpha_\ell=2\epsilon_\ell, 
$$
and the root system is given by
$$
\{\pm 2\epsilon_i \mid 1\le i\le \ell\}
\sqcup 
\{\pm\epsilon_i\pm \epsilon_j \mid 1\le i<j \le \ell\}.
$$
We denote by $\Delta_{\text{fin}}^\pm$ the set of positive or negative roots of the finite root system of type $C_\ell$. 
Note that $\Delta_{\text{fin}}^-=-\Delta_{\text{fin}}^+$. 
Since the highest root $\theta=2\alpha_1+2\alpha_2+\cdots+2\alpha_{\ell-1}+\alpha_\ell$ (of type $C_\ell$) and $\alpha_0=\delta-\theta$, the null root in type $C^{(1)}_\ell$ is 
$$
\delta=\alpha_0+2\alpha_1+2\alpha_2+\cdots+2\alpha_{\ell-1}+\alpha_\ell.
$$
Then, the positive real root system $\Delta_{\text{re}}^+$ of type $C_\ell^{(1)}$ is given by
$$
\Delta_{\text{re}}^+=\{\beta+m\delta \mid m\ge 0, \beta\in \Delta_{\text{fin}}^+ \text{\ or \ } \Delta_{\text{fin}}^-+\delta\}.
$$
We denote by $\Pi:=\{\alpha_i \mid i\in I \}$ the set of \emph{simple roots} of type $C^{(1)}_\ell$. 

Let $\Pi^\vee:=\{\alpha_i^\vee \mid i\in I \}$ be the set of simple coroots such that $\langle \alpha_i^\vee, \alpha_j\rangle=a_{ij}$, for $i, j\in I$. We may set a scaling element $d$ by $\langle d, \alpha_0\rangle=1$ and $\langle d, \alpha_i\rangle=0$ for $i\in I/\{0\}$.
Then, $\{ \alpha_0^\vee, \alpha_1^\vee, \ldots, \alpha_\ell^\vee, d\}$ forms a basis of the Cartan subalgebra of the Kac-Moody Lie algebra $\mathfrak g$ (associated with the  Cartan datum of type $C_\ell^{(1)}$). 
The canonical central element of $\mathfrak g$ is $c=\alpha_0^\vee+\alpha_1^\vee+\cdots+\alpha_\ell^\vee$. 
Moreover, we have $\langle d, \delta\rangle=1$, and $\langle \alpha_i^\vee, \delta\rangle=0$, for $i\in I$. 

The \emph{fundamental weight} $\Lambda_j$ ($j\in I$) is defined by $\langle \alpha_i^\vee, \Lambda_j\rangle=\delta_{ij}$ and $\langle d, \Lambda_j\rangle=0$. 
Then, the weight lattice is $P:=\Z \Lambda_0\oplus \Z \Lambda_1\oplus \cdots \oplus \Z \Lambda_{\ell}\oplus \Z \delta$. 
A weight $\lambda\in P$ is said to be \emph{dominant} if $\langle \alpha_i^\vee, \lambda\rangle\ge 0$, for $i\in I$. 
Then, the set of dominant (integral) weights is given by $P^+:=\Z_{\ge 0} \Lambda_0\oplus \Z_{\ge 0} \Lambda_1\oplus \cdots \oplus \Z_{\ge 0} \Lambda_{\ell}\oplus \Z \delta$.
Note that $P$ contains the root lattice $Q$ spanned by all simple roots, i.e., $Q:=\Z \alpha_0\oplus \Z \alpha_1\oplus \cdots \oplus \Z \alpha_{\ell}$. 
We denote the positive cone of the root lattice by $Q_+:=\Z_{\ge 0} \alpha_0\oplus \Z_{\ge 0} \alpha_1\oplus \cdots \oplus \Z_{\ge 0} \alpha_{\ell}$. 
For any $\beta\in Q_+$, the \emph{height} of $\beta=\sum_{i\in I}m_i\alpha_i\in Q_+$ is defined by $|\beta|:=\sum_{i\in I}m_i$.

We define, for a natural number $k\ge 1$,
$$
\pcl:=\left \{\ssum_{i=0}^{\ell}m_i\Lambda_i\mid m_i\ge 0, \ssum_{i=0}^{\ell}m_i=k \right \}\subseteq P^+.
$$
Here, the word \emph{cl} stands for the classical dominant integral weights.
The value $\langle c, \Lambda\rangle=k$, for $\Lambda\in \pcl$, is called the \emph{level} of $\Lambda$.
Set $\varpi_i:=\Lambda_i-\Lambda_0$ ($i\in I \setminus \{0\}$) as (12.4.3) in Kac's book \cite{K-Lie-alg}; these are fundamental weights of $\mathfrak{sp}(2\ell, \mathbb C)$.
Fix $\Lambda=\sum_{i=0}^{\ell}m_i\Lambda_i\in \pcl$. 
Then, Young-Hun Kim, Se-jin Oh and Young-Tak Oh introduced in \cite[Proposition 2.1]{KOO} the set
$$
\mathcal{C}(\Lambda):=\left\{\ssum_{i=1}^{\ell}p_i\varpi_i \mid p_i\ge 0, \ssum_{i=1}^{\ell}p_i\le k, \ssum_{i=1}^{\ell}(p_i-m_i)(\mathsf{A}')^{-1}u_i\in \Z^\ell \right\},
$$
where $u_i$'s are unit vectors.
The inverse $(\mathsf{A}')^{-1}$ is easy to calculate:
$$
(\mathsf{A}')^{-1}=
\left(
\begin{array}{ccccccc}
1 &1& 1&~\ldots~&1&1~&1\\
1 & 2&2&\ldots&2&2&2 \\
1& 2&3&\ldots&3&3&3\\
\vdots&\vdots&\vdots&\ddots &\vdots &\vdots&\vdots\\
1 &2&3&\ldots&\ell-2&\ell-1& \ell-1 \\
1/2 &1&3/2&\ldots&\ell/2-1&(\ell-1)/2&\ell/2 
\end{array}\right).
$$
We say that $\Lambda, \Lambda'\in \pcl$ are \emph{equivalent} if $\mathcal{C}(\Lambda)=\mathcal{C}(\Lambda')$, and we denote $\Lambda\sim \Lambda'$.

\subsection{Dominant maximal weight}
Let $U_v(\mathfrak g)$ be the quantum group of $\mathfrak g$. 
Given a $\Lambda \in P^+$, we denote by $V(\Lambda)$ the integrable highest weight module with the highest weight $\Lambda$ and by $P(\Lambda)$ the set of weights of $V(\Lambda)$. 
A weight $\lambda \in P(\Lambda)$ is said to be \emph{maximal} if $\lambda+\delta\notin P(\Lambda)$.
Let $\max(\Lambda)$ be the set of maximal weights in $P(\Lambda)$. 
It is known that 
\begin{equation}\label{equ::weight-set}
P(\Lambda)= \bigsqcup_{\lambda\in \max(\Lambda)}
\{\lambda -m\delta\mid m\in \Z_{\ge 0}\}.
\end{equation}

The set of all dominant maximal weights of $V(\Lambda)$ is defined as
$$
\max^+(\Lambda):=\max(\Lambda)\cap P^+ .
$$
Let $W$ be the Weyl group generated by $\{r_i\}_{i\in I}$ acting on $P$ by $r_i \mu =\mu -\langle \alpha_i^\vee, \mu\rangle\alpha_i$, for $\mu\in P$ and $i\in I$.
Then, it is known (e.g., \cite[Proposition 11.2(a)]{K-Lie-alg}) that any element in $\max(\Lambda)$ is $W$-conjugate to an element in $\max^+(\Lambda)$.

\subsection{Cyclotomic KLR algebra}
Let $\k$ be an algebraically closed field.
For any $i,j\in I$, we take a family $Q_{i,j}(u,v)\in\k[u,v]$ of polynomials such that $Q_{i,i}(u,v)=0$, $Q_{i,j}(u,v) =Q_{j,i}(v,u)$, and for any $i<j$,
$$
\begin{aligned}
Q_{i,j}(u,v)= \left\{
\begin{array}{ll}
u-v^2 & \text{ if } i=0, j=1, \\
u-v & \text{ if } i\ne 0, j=i+1, j\ne \ell, \\
u^2-v & \text{ if } i=\ell-1, j=\ell, \\
1 & \text{ otherwise. }
\end{array}\right.
\end{aligned} 
$$

We denote by $\mathfrak S_n$ the symmetric group generated by elementary transpositions $\{s_i\mid 1\le i\le n-1\} $. 
Then, the action of $\mathfrak S_n$ on $I^n$ is given by 
\begin{center}
$s_i\cdot (\nu_1, \nu_2, \ldots, \nu_i, \nu_{i+1}, \ldots, \nu_n)=(\nu_1, \nu_2, \ldots, \nu_{i+1}, \nu_i, \ldots, \nu_n)$.
\end{center}

Recall that, a finite-dimensional $\k$-algebra $A$ is said to be $\Z$-graded if it is equipped with a $\k$-vector space decomposition $A=\oplus_{m\in \Z} A_m$ satisfying $A_mA_n\subseteq A_{m+n}$. 
Here, elements in $A_m$ are called \emph{homogeneous} of degree $m\in \Z$. 
Let $q$ be an indeterminate. 
Then, the graded dimension $\dim_q A$ of $A$ is defined by
$$
\dim_q A :=\ssum_{m\in\Z}(\dim A_m)q^m \in \Z_{\ge0}[q,q^{-1}].
$$
\begin{definition}\label{def::cyclotomic-quiver}
Fix $\Lambda\in \pcl$.
Let $R^{\Lambda}(n)$ be the $\Z$-graded $\k$-algebra generated by
$$
\{ e(\nu)\mid \nu=(\nu_1, \nu_2, \ldots, \nu_n)\in I^n \}, \quad 
\{x_i \mid 1\le i \le n \}, \quad 
\{\psi_j \mid 1\le j\le n-1\},
$$
subject to
\begin{enumerate}
\item
$e(\nu)e(\nu')=e(\nu)\delta_{\nu, \nu'},\quad \sum_{\nu\in I^n}e(\nu)=1,\quad x_ix_j=x_jx_i,\quad  x_ie(\nu)=e(\nu)x_i$,
\item $\psi_i e(\nu)=e(s_i(\nu))\psi_i,\quad  \psi_{i}\psi_j=\psi_j\psi_i$ if $|i-j|>1$, $\quad$ $\psi_ix_j=x_j\psi_i$ if $j\ne i,i+1$,

\item $\psi_i^2 e(\nu)=Q_{\nu_i,\nu_{i+1}}(x_i,x_{i+1})e(\nu)$,

\item
$(\psi_ix_{i+1}-x_i\psi_i)e(\nu)=(x_{i+1}\psi_i-\psi_ix_i)e(\nu)=e(\nu)\delta_{\nu_i,\nu_{i+1}}$,

\item
$(\psi_{i+1}\psi_i\psi_{i+1}-\psi_i\psi_{i+1}\psi_{i})e(\nu)$
$$
=\left\{
\begin{array}{ll}
\frac{Q_{\nu_i,\nu_{i+1}}(x_i,x_{i+1})-Q_{\nu_i,\nu_{i+1}}(x_{i+2}, x_{i+1})}{x_i-x_{i+2}}e(\nu) & \text{ if } \nu_i=\nu_{i+2}, \\
0 & \text{otherwise,}
\end{array}
\right.
$$

\item $x_1^{\langle \alpha^\vee_{\nu_1}, \Lambda \rangle}e(\nu)=0$,
\end{enumerate} 
and the $\Z$-grading on $R^\Lambda(n)$ is given by
$$
\deg(e(\nu))=0,
\quad  
\deg(x_ie(\nu))=2\mathsf{d}_{\nu_i},
\quad
\deg(\psi_ie(\nu))=-\mathsf{d}_{\nu_i}a_{\nu_i,\nu_{i+1}}, 
$$
with $(\mathsf{d}_0, \mathsf{d}_1, \ldots, \mathsf{d}_{\ell-1}, \mathsf{d}_\ell)=(2, 1,\ldots,1,2)$. 
We call $R^\Lambda(n)$ the cyclotomic quiver Hecke algebra of type $C^{(1)}_\ell$, and this algebra was introduced by Mikhail Khovanov and Aaron Lauda \cite{kl-diagrammatic}. 
Note that the (affine) quiver Hecke algebra $R(n)$ obtained by omitting the relation (6) was also introduced by Raphael Rouquier \cite{Ro-2kac}, independent of \cite{kl-diagrammatic}. 
Thus, the cyclotomic quiver Hecke algebra is also known as the cyclotomic Khovanov-Lauda-Rouquier algebra. 
\end{definition}

Given a positive root $\beta\in Q_+$ with $|\beta|=n$, we set
$$
e(\beta):=\ssum_{\nu\in I^\beta}e(\nu)
\quad\text{with}\quad
I^\beta:=\left \{\nu=(\nu_1, \nu_2, \ldots, \nu_n)\in I^n\mid \ssum_{i=1}^n \alpha_{\nu_i}=\beta \right \}.
$$
This is a central idempotent of $R^\Lambda(n)$. We may distinguish the component of $R^\Lambda(n)$ associated with $e(\beta)$ as follows.

\begin{definition}
We define $R^{\Lambda}(\beta):=R^\Lambda(n)e(\beta)$.
\end{definition}

We may define $R^{\Lambda}(\beta)$ with the same defining relations of $R^\Lambda(n)$, just by replacing $I$ with $I^\beta$.

\begin{remark}\label{rem::iso-type-A-C}
Fix $\Lambda=\sum_{i\in I} m_i\Lambda_i\in \pcl$.
It is known, e.g., \cite[page 25]{Ro-2kac} or \cite[Lemma 3.2]{AIP-rep-type-A-level-1}, that $R(n)$ or $R^\Lambda(n)$ (of type $C^{(1)}_\ell$) does not depend on the choice of $Q_{i,j}(u,v)$, up to isomorphism. 
Let $R^\Lambda_A(n)$ be the cyclotomic KLR algebra of type $A_{\ell}^{(1)}$ whose definition uses polynomials $Q_{i,i+1}(u,v)=u-v$ for $i\in \Z/(\ell+1)\Z$, and $Q_{i,j}(u,v)=1$ if $j\not\equiv_{\ell+1} i, i\pm 1$. 
Suppose that 
$$
\beta\in \Z_{\ge0}\alpha_1\oplus\Z_{\ge0}\alpha_2\oplus\cdots\oplus\Z_{\ge0}\alpha_{\ell-1}.
$$
Then, $\beta$ may be viewed as an element in the positive cone of the root lattice for the type $A^{(1)}_\ell$. 
Under this circumstance, we have an isomorphism of algebras $R^\Lambda(\beta)\cong R^{\Lambda_A}_A(\beta)$, where $\Lambda_A=\Lambda-m_0\Lambda_0-m_\ell\Lambda_\ell$. In the rest of the paper, we write $R^{\Lambda}_A(\beta)$ instead of $R^{\Lambda_A}_A(\beta)$ by abuse of notation.
\end{remark}

Let $\sigma: I\rightarrow I$ be the involution given by $\sigma(i)=\ell-i$. Given a dominant integral weight $\Lambda=\sum_{i\in I}m_i\Lambda_i\in \pcl$ and a positive root $\beta=\sum_{i\in I}n_i\alpha_i\in Q_+$, we define
\begin{equation}\label{def::sigma}
\sigma\Lambda:=\ssum_{i\in I}m_i\Lambda_{\sigma(i)}
\quad \text{and} \quad
\sigma\beta:=\ssum_{i\in I}n_i\alpha_{\sigma(i)}.
\end{equation}
Using Remark \ref{rem::iso-type-A-C}, we may assume that $R^{\Lambda}(\beta)$ and $R^{\sigma \Lambda}(\sigma \beta)$ share the same family of polynomials $Q_{i,j}(u,v)\in\k[u,v]$.

\begin{proposition}[{\cite[Lemma 3.1]{Ar-rep-type}}]\label{prop::iso-sigma}
There is an algebra isomorphism
$$
R^{\Lambda}(\beta)\cong R^{\sigma \Lambda}(\sigma \beta).
$$
\end{proposition}

There is a symmetric bilinear form $(-,-)$ on the weight lattice $P$ such that 
$$
(\Lambda_i,\alpha_j)=\mathsf{d}_j \delta_{ij}, 
\quad
(\alpha_i,\alpha_j)=\mathsf{d}_i a_{ij}. 
$$
with $(\mathsf{d}_0, \mathsf{d}_1, \ldots, \mathsf{d}_{\ell-1}, \mathsf{d}_\ell)=(2, 1,\ldots,1,2)$.
The \emph{defect} of $R^{\Lambda}(\beta)$ is given by
$$
\defect_\Lambda(\beta):=(\Lambda,\beta)-(\beta,\beta)/2.
$$
We sometimes omit $\Lambda$ from the subscript and write $\defect(\beta)$ instead of $\defect_\Lambda (\beta)$.
In level one, we experienced the validity of Erdmann-Nakano type theorems, see \cite{AP-rep-type-C-level-1} and \cite{CH-type-c-level-1}. 
Hence, it is of interest to list defect values here. 
In the representation-finite cases, the value is $1$ except for the following three cases. 
\begin{itemize}
\item
(f1): $\defect(\beta)=m_a-1$ if $1\le a\le\ell-1$, and $\defect(\beta)=2m_a-2$ if $a=0,\ell$.
\item
(f2) or (f3): $\defect(\beta)=2$ for $m_0=m_1=1$ or $m_{\ell-1}=m_\ell=1$, and $\defect(\beta)=2m_i-1$ for $i=0$ or $\ell$.
\end{itemize}
In the tame cases, the value is $2$ only for $5$ cases, and the other $16$ cases may have different values as listed below.
\begin{itemize}
\item
(t3) or (t4): $\defect(\beta)=2m_i\ge4$ for $i=0$ or $\ell$.
\item
(t5) or (t6): $\defect(\beta)=2m_i\ge2$ for $i=0$ or $\ell$.
\item
(t7) or (t8): $\defect(\beta)=3$.
\item
(t9): $\defect(\beta)=m_i\ge2$ for $i=a$ or $b$.
\item
(t10) or (t11): $\defect(\beta)=3$ if $i\ne\ell$ or $0$, and $\defect(\beta)=4$ if $i=\ell$ or $0$.
\item
(t13) or (t14): $\defect(\beta)=2$ if $i\ne\ell$ or $0$, and $\defect(\beta)=3$ if $i=\ell$ or $0$.
\item
(t16) or (t17): $\defect(\beta)=3$.
\item
(t19): $\defect(\beta)=4$.
\item 
(t20) or (t21): $\defect(\beta)=4$.
\end{itemize}

Let $n\ge 1$ be a natural number and $\lambda=(\lambda_1, \lambda_2, \ldots )$ a sequence of non-negative integers. 
We call $\lambda$ a \emph{partition} of $n$ if $|\lambda|:=\lambda_1+\lambda_2+\cdots =n$ and $\lambda_1\ge \lambda_2\ge \cdots\ge 0$.
A \emph{$k$-multipartition} of $n$ is an ordered $k$-tuple of partitions $\lambda=(\lambda^{(1)},\lambda^{(2)},\ldots,\lambda^{(k)})$ such that $|\lambda^{(1)}|+|\lambda^{(2)}|+\cdots+|\lambda^{(k)}|=n$. 
We denote by $\P_{k,n}$ the set of all $k$-multipartitions of $n$.

A Young diagram is considered as a realization of a partition. Here, the Young diagram of a $k$-multipartition $\lambda=(\lambda^{(1)},\lambda^{(2)},\ldots,\lambda^{(k)})$ can be visualized as a column vector whose entries are $\lambda^{(i)}$'s in increasing order from top to bottom.
We say that a node of $\lambda\in \P_{k,n}$ is \emph{removable} (resp., \emph{addable}) if one obtains a new $k$-multipartition after removing (resp., adding) the node from (resp., to) $\lambda$.

Let $g_\ell: \Z\rightarrow \Z/2\ell\Z$
be the natural projection and we define $f_\ell: \Z/2\ell\Z\rightarrow I$ by 
$$
f_\ell(a+2\ell\Z):=\left\{
\begin{array}{ll}
a & \text{ if } 0\le a\le \ell, \\
2\ell-a& \text{ if } \ell+1\le a\le 2\ell-1.
\end{array}\right. 
$$
For any $m\in \mathbb Z$, we set $\overline{m}:=(f_\ell\circ g_\ell)(m)\in I$.
In other words, the values periodically repeat in the order of $0~1~2~\cdots~\ell-1~\ell~\ell-1~\cdots~2~1$.

Fix $\Lambda=\Lambda_{i_1}+\Lambda_{i_2}+\cdots +\Lambda_{i_k} \in \pcl$ and $\lambda=(\lambda^{(1)},\lambda^{(2)},\ldots,\lambda^{(k)}) \in \P_{k,n}$. 
Let $p$ be a node in the $a$-th row and $b$-th column of $\lambda^{(s)}$. Then, the \emph{residue} of $p$ is defined by
$$
\res p :=\overline{b-a+i_s} \quad \in I,
$$
and $p$ is said to be an \emph{$i$-node} if $\res p=i$. 
As $\lambda=(\lambda^{(1)},\lambda^{(2)},\ldots,\lambda^{(k)})$ can be visualized as a column vector of Young diagrams, we say that $\lambda^{(s)}$ is below $\lambda^{(t)}$ if $s>t$. We set $\#\text{addable}_{\res p}(\lambda)$ as the number of addable ($\res p$)-nodes of $\lambda$ below $p$, and set $\#\text{removable}_{\res p}(\lambda)$ as the number of removable ($\res p$)-nodes of $\lambda$ below $p$. 
If $p$ is a removable $i$-node of $\lambda$, we define
\begin{center}
$d_p(\lambda):=\mathsf{d}_i\cdot (\#\text{addable}_{\res p}(\lambda)-\#\text{removable}_{\res p}(\lambda))$
\end{center}
with $(\mathsf{d}_0, \mathsf{d}_1, \ldots, \mathsf{d}_{\ell-1}, \mathsf{d}_\ell)=(2, 1,\ldots,1,2)$ as mentioned before.

A standard tableau $T=(T^{(1)}, T^{(2)},\ldots, T^{(k)})$ of shape $\lambda\in \P_{k,n}$ is given by bijectively inserting the integers $1,2,\ldots, n$ into the nodes of the Young diagram of $\lambda$, such that each $T^{(i)}$ is a standard tableau of $\lambda^{(i)}$, i.e., the entries in $T^{(i)}$ are strictly increasing along the rows from left to right and down the columns from top to bottom. 
We denote by $\text{Std}(\lambda)$ the set of all standard tableaux of $\lambda$.
The \emph{residue sequence} of $T$ is defined as $\mathbf{i}_T:=(i_1,i_2,\ldots,i_n)\in I^n$, such that $i_r=\res p$ if the integer $r$ is filled in the node $p$ of $\lambda$. We then define the \emph{degree} of $T$ (see \cite[(1.4)]{APS-type-C}) inductively by 
\begin{equation}
\deg(T):=\left\{\begin{array}{ll}
\deg(T\downarrow_n)+d_p(\lambda) & \text{ if } n>0, \\
0 & \text{ if } n=0,
\end{array}\right.
\end{equation}
where $T\downarrow_n$ is the tableau obtained by removing $p$ from $T$ and the integer $n>0$ is filled in the node $p$ of $\lambda$.

Using values $\deg(T)$, we may define the action of Chevalley generators on the $\Q[v,v^{-1}]$-span of all $k$-multipartitions to make it into a module over the quantum group $U_v(\mathfrak{g})$. We call this $U_v(\mathfrak{g})$-module the level $k$ deformed Fock space. We denote the empty $k$-multipartiton by $v_\Lambda$, which generates $V(\Lambda)$ as a $U_v(\mathfrak{g})$-submodule.
For the precise definition of the action when $k=1$, see \cite{AP-rep-type-C-level-1} or \cite{CH-type-c-level-1}. 
The level $k$ deformed Fock space we use here is the $k$-fold tensor product of level one deformed Fock spaces. The next theorem follows from the computation in the level $k$ deformed Fock space.

\begin{theorem}[{\cite[Theorem 2.5]{APS-type-C}}]\label{theo::graded}
For any positive root $\beta\in Q_+$ with $|\beta|=n$ and $\nu, \nu'\in I^\beta$, the graded dimension of $e(\nu)R^\Lambda(\beta)e(\nu')$ is
$$
\dim_q e(\nu)R^\Lambda(\beta)e(\nu')=
\sum_{\substack{\mathbf{i}_S=\nu,\ \mathbf{i}_T=\nu',\\
S,T\in\emph{Std}(\lambda),\ \lambda\in\P_{k,n}}}
q^{\deg(S)+\deg(T)}.
$$
\end{theorem}

\begin{example}
Let $\Lambda=\Lambda_0+\Lambda_1$ and $\ell=2$. We consider  $R^{\Lambda}(\delta)$ with $\delta=\alpha_0+2\alpha_1+\alpha_2$.  
Set $e:=e(0121)$. Then, $\dim_q eR^{\Lambda}(\delta)e=1+2q^2+3q^4+2q^6+q^8$ due to the following pattern:
\begin{center}
\scalebox{0.7}{$\ytableausetup{centertableaux, smalltableaux}
\vcenter{\xymatrix@C=0.1cm@R=0.5cm{
&&&&
\left (\emptyset,\emptyset \right ) \ar[d]
&&&&\\
&&&&
\left( \ytableaushort{0},\emptyset\right )_{0} \ar[drrr] \ar[d] \ar[dlll]
&&&&\\
&\left (\ytableaushort{01},\emptyset\right)_{2} \ar[d]
&&&\left (\ytableaushort{0,1}, \emptyset \right )_{1}\ar[d]&&&
\left ( \ytableaushort{0},\ytableaushort{1}\right )_{0}\ar[d]
&\\
&\left (\ytableaushort{012},\emptyset\right )_{0} \ar[d]\ar[dl]\ar[dr]
&&&\left (\ytableaushort{0,1,2}, \emptyset \right )_{0}\ar[dr]\ar[d]\ar[dl]&&&
\left( \ytableaushort{0},\ytableaushort{12} \right )_{0} \ar[dr]\ar[d]\ar[dl]
&\\
\left (\ytableaushort{0121}, \emptyset\right)_{2}
&
\left (\ytableaushort{012,1}, \emptyset\right)_{1}
&
\left (\ytableaushort{012},\ytableaushort{1}\right)_{0}
&\left (\ytableaushort{01,1,2}, \emptyset \right )_{2}&\left (\ytableaushort{0,1,2,1}, \emptyset \right )_{1}&\left (\ytableaushort{0,1,2}, \ytableaushort{1}\right )_{0}&
\left (\ytableaushort{01}, \ytableaushort{12}\right)_{2}
&
\left (\ytableaushort{0,1}, \ytableaushort{12}\right)_{1}
&
\left (\ytableaushort{0}, \ytableaushort{121}\right )_{0}
}}$}
\end{center}
where the subscript number in each vertex gives the corresponding $d_p(\lambda)$.
\end{example}

In the following, we are going to introduce the divided power induction functor $f_i^{(r)}$ (see \cite[Section 4.6]{bk-graded-decomp-numbers}) from the category of $R^\Lambda(\beta)$-modules to the category of $R^\Lambda(\beta+r\alpha_i)$-modules, for $r\in \Z_{\ge 0}$. Let $R(\beta)$ be the (affine) KLR algebra, namely, the algebra defined by dropping the cyclotomic condition $x_1^{\langle \alpha^\vee_{\nu_1}, \Lambda\rangle}e(\nu)=0$ from the defining relations of $R^\La(\beta)$.
Then, the definition of $f_i^{(r)}$ starts with the result in \cite[Section 2.2]{kl-diagrammatic} that the polynomial representation $P(i^{(r)})=\k[x_1,\dots,x_r]$ over $R(r\alpha_i)$, 
whose degree is given by 
$$
\deg(x_1^{m_1}\cdots x_r^{m_r})=\mathsf{d}_i\left(2m_1+\cdots+2m_r-\frac{r(r-1)}{2}\right),
$$
satisfies 
$$
R(r\alpha_i)\cong 
P(i^{(r)}) \left< \mathsf{d}_i\frac{r(r-1)}{2}
\right>\oplus \cdots\oplus P(i^{(r)})\left < -
\mathsf{d}_i\frac{r(r-1)}{2}\right>,
$$
where $R(r\alpha_i)$ is the regular representation, and $r!$ copies of $P(i^{(r)})$ with shifts appear on the right hand side.

\begin{example}
$R(2\alpha_i)$ is the $\k$-algebra generated by $x_1, x_2, \psi$ of degree 
$$ 
\deg x_1=\deg x_2=2\mathsf{d}_i, \quad \deg \psi=-2\mathsf{d}_i,
$$ 
which are subject to
$$ 
x_1x_2=x_2x_1, \; \psi x_2-x_1\psi=1=x_2\psi-\psi x_1, \; \psi^2=0.
$$
Then, $R(2\alpha_i)=\k[x_1,x_2]\oplus \k[x_1,x_2]\psi$. Define $e_1=x_2\psi$ and $e_2=-\psi x_1$. Then $1=e_1+e_2$, $e_se_t=\delta_{st}e_s$, for $s=1,2$. Since $\psi=\psi e_1\in R(2\alpha_i)e_1$, we have 
    $$ P(i^{(2)})\langle -\mathsf{d}_i\rangle \cong \k[x_1,x_2]\psi= R(2\alpha_i)e_1, \;\;
       P(i^{(2)})\langle \mathsf{d}_i\rangle \cong \k[x_1,x_2]= R(2\alpha_i)e_2. $$
\end{example}

Using the $R(r\alpha_i)$-module $P(i^{(r)})$, we define the divided power induction functor $f_i^{(r)}$ as follows. 

\begin{definition}
Let $\theta_i^{(r)}(M):=\Ind_{R(\beta)\otimes R(r\alpha_i)}^{R(\beta+r\alpha_i)}(M\otimes P(i^{(r)}))$ for an $R(\beta)$-module $M$. 
Based on \cite[Lemma 4.4]{bk-graded-decomp-numbers}, we define
$$
f_i^{(r)}:={\rm pr}\circ \theta_i^{(r)}\circ {\rm Infl}\langle r^2-r(\Lambda-\beta,\alpha_i)\rangle,
$$
where ${\rm pr}$ is the tensor functor defined by the $(R^\Lambda(\beta+r\alpha),R(\beta+r\alpha))$-bimodule $R^\Lambda(\beta+r\alpha)$, and ${\rm Infl}$ is the inflation functor from the category of $R^\Lambda(\beta)$-modules to the category of $R(\beta)$-modules with respect to the quotient algebra homomorphism $R(\beta)\to R^\Lambda(\beta)$. 
\end{definition}

We need the following lemma proved in \cite[Lemma 4.8]{bk-graded-decomp-numbers}. 

\begin{lemma}
The divided power induction functor $f_i^{(r)}$ is an exact functor and it sends projective modules to projective modules.
\end{lemma}

Indeed, if $\beta=\sum_{j=1}^s n_j\alpha_{i_j}$ for some $n_j\in \Z_{\ge 0}$ and $i_j\in I$, the element
$$
f_{i_s}^{(n_s)}\cdots f_{i_2}^{(n_2)}f_{i_1}^{(n_1)}v_\Lambda
$$
in the level $k$ deformed Fock space of type $C^{(1)}_\ell$ uniquely determines the projective module which is one of the direct summands of $R^\Lambda(\beta)e(\nu)$ where $\nu=(i_1^{n_1}, i_2^{n_2}, \ldots, i_s^{n_s})$, and all the other direct summands are shifts of this projective module. 
This fact together with Theorem \ref{theo::graded} allows us to compute the graded dimension of the endomorphism algebra of a certain well-chosen direct sum of indecomposable projective $R^\Lambda(\beta)$-modules, and to apply lemmas on graded dimensions in the next subsection to prove wildness of $R^\Lambda(\beta)$. 

\begin{remark}
The divided restriction functor $e_i^{(r)}$ is also an exact functor and it sends 
projective modules to projective modules.
\end{remark}

\subsection{Some tame and wild algebras}
We review a few tame and wild algebras in this subsection. Besides, it is well-known that $\k[x]/(x^n)$ for any $n\ge 2$ is a representation-finite local algebra. The wild algebras below will give us a reduction method for proving wildness, because, if $e$ is an idempotent of a finite-dimensional algebra $A$ and a factor algebra of $eAe$ is wild, then $A$ is wild.

\begin{proposition}\label{prop::tame-local-alg}
Let $A=\k Q/J$ be a local algebra with
$$
Q:\vcenter{\xymatrix@C=0.8cm{\circ \ar@(dl,ul)^{x} \ar@(ur,dr)^{y}}}.
$$
\begin{enumerate}
\item If $J=\left<x^2, y^2, xy-yx \right>$, then $A$ is tame.
\item If $J=\left<x^2-y^2, xy, yx \right>$, then $A$ is wild.
\item If $J=\left<x^3, y^2, x^2y, xy-yx \right>$, then $A$ is wild.
\item If $J=\left<x^m-y^n, xy, yx \right>$ for some $m,n\ge 2$ and $m+n\ge 5$, then $A$ is tame.
\end{enumerate}
\end{proposition}
\begin{proof}
See \cite{Ringel-local-alg} for (1)-(3) and see \cite[Theorem III.1 (a)]{Er-tame-block} for (4).
\end{proof}

\begin{lemma}\label{lem::wild-three-loops-part1}
If the graded dimension of a graded local algebra $A$ satisfies
$$
\dim_q A-1-mq\in q^2\Z_{\ge 0}[q] 
\quad\text{or}\quad 
\dim_q A-1-mq^2\in q^3\Z_{\ge 0}[q],
$$
for $3\le m\in \Z_{\ge 0}$, then $A$ is wild.
\end{lemma}
\begin{proof}
Let $J$ be the span of elements of degree greater than or equal to $2$ or $3$, respectively. 
Then, $J$ is a two-sided ideal of $A$, and we have 
$$
\dim_q A/J=1+mq \quad\text{or}\quad  \dim_q A/J=1+mq^2,
$$
respectively. In either case, $A/J$ is the radical square zero local algebra whose Gabriel quiver has at least $3$ loops. 
Hence, $A/J$ is wild by \cite[I.10.10(a)]{Er-tame-block} or \cite[(1.1)]{Ringel-local-alg}, and so is $A$.
\end{proof}

\begin{lemma}\label{lem::wild-three-loops-part2}
If the graded dimension of a graded local algebra $A$ satisfies
$$
\dim_q A-1-q-mq^2\in q^3\Z_{\ge 0}[q],
$$
for $3\le m\in \Z_{\ge 0}$, then $A$ is wild.
\end{lemma}
\begin{proof}
There exists an $x\in A$ spanning the degree $1$ part of $A$. If $x^2=0$, then the degree $2$ part of $A$ has a basis $\{y_1,y_2,\dots,y_{m-1},y_m\}$. If $x^2\neq 0$, we have a basis $\{x^2,y_1,y_2,\dots,y_{m-1}\}$ in the degree $2$ part of $A$. In both cases, the Gabriel quiver of $A$ has at least $m\ge 3$ loops. Hence, $A$ is wild.
\end{proof}

\begin{lemma}\label{local algebra 2+3}
If the graded dimension of a symmetric graded local algebra $A$ satisfies
$$
\dim_q A-1-m_1q-m_2q^2\in q^3\Z_{\ge 0}[q],
$$
for $m_1, m_2\in \Z_{\ge 0}$ with $m_1+m_2 \ge 5$, then $A$ is wild.
\end{lemma}
\begin{proof}
Note that $\Rad^3 A$ is contained in the span of elements of degree greater than or equal to $3$. It follows that 
$$
\dim (\Rad A/\Rad^2A)+\dim (\Rad^2A/\Rad^3A)\ge m_1+m_2\ge 5. 
$$
If $\dim (\Rad A/\Rad^2A) \ge 3$, then the Gabriel quiver of $A$ has at least $3$ loops, and $A$ is wild. Otherwise, we have $\dim (\Rad^2A/\Rad^3A) \ge 3$, and $A$ is again wild by \cite[Theorem III.4]{Er-tame-block}.
\end{proof}

\begin{lemma}\label{lem::wild-two-point-alg}
Let $e_1, e_2$ be two different primitive idempotents of $A$. If 
$$
\dim_q e_iAe_j-\delta_{ij}-m_{ij}q^2\in q^3\Z_{\ge 0}[q]
$$
for $m_{ij}\in \Z_{\ge0}$ such that $m_{11}+m_{22}\ge 3$ and $m_{12}+m_{21}\ge 2$, then $A$ is wild.
\end{lemma}
\begin{proof}
By \cite[Lemma 1.3]{Ar-rep-type}, the Gabriel quiver of $(e_1+e_2)A(e_1+e_2)$ has 
$$
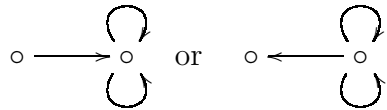
\xymatrix@C=1cm{\circ\ar[r]^{ }& \circ\ar@(ul,ur)^{ }\ar@(dl,dr)_{ }}
\quad\text{or}\quad
\xymatrix@C=1cm{\circ& \circ \ar[l]^{ }\ar@(ul,ur)^{ }\ar@(dl,dr)_{ }}
$$
as a subquiver. Then, $A$ is wild by \cite[Theorem 1]{H-wild-two-point}.
\end{proof}

\begin{lemma}
Let $A=\k[X]/(X^2)$ and $B=\k Q/J$ be the algebra given by 
$$
Q:\vcenter{\xymatrix@C=1cm{\circ \ar@<0.5ex>[r]^{\mu} & \circ\ar@<0.5ex>[l]^{\nu}}}
\quad \text{and}\quad
J:\langle \mu\nu\mu, \nu\mu\nu\rangle.
$$
Then, the tensor product algebra $A\otimes B$ is wild.
\end{lemma}
\begin{proof}
By tensoring $A$ with $B$, each vertex gets one loop. The tensor product $A\otimes B$ has the minimal wild algebra numbered 32 in Table W in \cite{H-wild-two-point} as a factor algebra. 
\end{proof}

The next lemma by Kang and Kashiwara \cite[Lemma~4.2]{KK-categorification} is stated for the cyclotomic affine quiver Hecke algebra $R(n)$, but the proof works for $R^\La(\beta)$ (by applying $M=R^\La(\beta)$ there).

\begin{lemma}\label{KK-lemma}
If $\nu\in I^\beta$ satisfies $\nu_i=\nu_{i+1}$ and $fe(\nu)=0$, for $f\in \k[x_1,\dots,x_n]$, then $(\partial_if)e(\nu)=0$ and $(s_if)e(\nu)=0$, where 
$\partial_i f=\frac{s_if-f}{x_i-x_{i+1}}$. 
\end{lemma}
\begin{proof}
First we recall the following equation from \cite[(3.7)]{KK-categorification}
\begin{align}\label{equ:KK-diff}
(\psi_i f-(s_if)\psi_i)e(\nu)=(\partial_i f)e(\nu).
\end{align}
Then, we have 
\begin{align*}
0&= (x_i-x_{i+1})\psi_ife(\nu)\psi_i \\
&= (x_i-x_{i+1})\psi_if\psi_i e(\nu) \overset{\eqref{equ:KK-diff}}=   (x_i-x_{i+1})((s_if)\psi_i+ \partial_i f  )\psi_i e(\nu)\\
&= (x_i-x_{i+1}) (\partial_i f)  \psi_i e(\nu)  \quad (\text{since } \psi_i^2e(\nu)=0)\\
&= (s_if-f)\psi_ie(\nu)\overset{\eqref{equ:KK-diff}}= (\psi_if-\partial_i f-f\psi_i )e(\nu)\\
&= (\partial_i f)e(\nu) \quad (\text{since } fe(\nu)=0).
\end{align*}
Moreover, we also obtain $(s_if)e(\nu)=fe(\nu)+(x_i-x_{i+1})(\partial_if)e(\nu)=0$.
\end{proof}

The following tensor product lemma is useful. We prove the lemma only for $C^{(1)}_\ell$ here by using the graded dimension formula, but the lemma holds for general Lie type by a different argument \cite{Murata-tensor product lemma}.
 
\begin{lemma}\label{tensor product lemma}
Suppose that we have two intervals $I_1$ and $I_2$ in $I=\{0,1,\dots,\ell\}$ which satisfy $a_{ij}=0$ for $(i,j)\in I_1\times I_2$, and $\beta=\beta_1+\beta_2$ with
$$
\beta_1\in \ssum_{i\in I_1} \Z_{\ge 0}\alpha_i
\quad \text{and} \quad
\beta_2\in \ssum_{i\in I_2} \Z_{\ge 0}\alpha_i.
$$
We denote by $\nu_1\ast\nu_2$ the concatenation of $\nu_1\in I^{\beta_1}$ and $\nu_2\in I^{\beta_2}$, and we define
$$
e:=\ssum_{\nu_1\in I^{\beta_1},\ \nu_2\in I^{\beta_2}}e(\nu_1\ast\nu_2).
$$
Then, there is an isomorphism of graded algebras
$$
eR^\Lambda(\beta)e\cong R^{ \Lambda'}(\beta_1)\otimes R^{\Lambda''}(\beta_2)
$$
such that $\Lambda'=\sum_{i\in I_1} \langle \alpha_i^\vee, \Lambda\rangle\Lambda_i$ and $\Lambda''=\sum_{i\in I_2} \langle \alpha_i^\vee, \Lambda\rangle \Lambda_i$.
Moreover, $R^\Lambda(\beta)$ is graded Morita equivalent to $R^{ \Lambda'}(\beta_1)\otimes R^{\Lambda''}(\beta_2)$.
\end{lemma}
\begin{proof}
We define an algebra homomorphism $\mathscr{F}: R^{ \Lambda'}(\beta_1)\otimes R^{\Lambda''}(\beta_2)\rightarrow eR^\Lambda(\beta)e$ by the following assignment:
\begin{gather*}
1\otimes 1\mapsto e, \quad 
e(\nu_1)\otimes e(\nu_2) \mapsto e(\nu_1\ast\nu_2), \\
\psi_i\otimes 1 \mapsto \psi_i, \quad 
1\otimes \psi_i \mapsto \psi_{|\beta_1|+i}, \\
x_i\otimes 1 \mapsto x_i, \quad
1\otimes x_i \mapsto x_{|\beta_1|+i}.
\end{gather*}
Indeed, it is clear that the images of $e(\nu_1)\otimes 1$, $x_i\otimes 1$ and $\psi_i\otimes 1$ commute with the images of $1\otimes e(\nu_2)$, $1\otimes x_j$ and $1\otimes \psi_j$. Since $e$ is the unit of $eR^\Lambda(\beta)e$, the unit 
maps to the unit and
$$
e=\ssum_{\nu_1\in I^{\beta_1}}\left(\ssum_{\nu_2\in I^{\beta_2}} e(\nu_1\ast\nu_2)\right)
=\ssum_{\nu_1\in I^{\beta_1}}\left(\ssum_{\nu_2\in I^{\beta_2}} \mathscr{F}(e(\nu_1)\otimes e(\nu_2))\right)
$$
such that $\mathscr{F}(1\otimes 1)=\sum_{\nu_1\in I^{\beta_1}} \mathscr{F}(e(\nu_1)\otimes 1)$ is satisfied. Similarly, $\mathscr{F}(1\otimes 1)=\sum_{\nu_2\in I^{\beta_2}} \mathscr{F}(1\otimes e(\nu_2))$ is satisfied. Then, the orthogonality relations among $\mathscr{F}(e(\nu_1)\otimes 1)$ and among $\mathscr{F}(1\otimes e(\nu_2))$ hold by the same rewriting of the unit $1$. 

It is also easy to see that other commutation relations among the generators of $R^{\Lambda'}(\beta_1)$ and the generators of $R^{\Lambda''}(\beta_2)$ hold on their images. 

Now, let $m:=|\beta_1|$, $\nu_1=(i_1,i_2,\ldots,i_m)$ and $\nu_2$ starts with $i\in I^{\beta_2}$. Then,
$$
\begin{aligned}
x_{m+1}^{\langle \alpha_i^\vee,\Lambda''\rangle}\psi_m^2e(i_1,i_2,\ldots,i_m,i, \ldots)
&=\psi_m x_m^{\langle \alpha_i^\vee, \Lambda''\rangle}e(i_1,i_2,\ldots,i_{m-1}, i, i_m,\ldots)\psi_m \\
&=\psi_m x_m^{\langle \alpha_i^\vee, \Lambda''\rangle}\psi_{m-1}^2e(i_1,\dots, i_{m-1},i,i_m,\dots)\psi_m \\
&=\ldots\ldots \\
&=\psi_m\cdots \psi_1 x_1^{\langle \alpha_i^\vee, \Lambda''\rangle}e(i,i_1,\ldots,i_m,\ldots)\psi_1\cdots\psi_m=0.
\end{aligned}
$$
Here, the last equality uses $\langle \alpha_i^\vee, \Lambda''\rangle=\langle \alpha_i^\vee, \Lambda\rangle$. Hence, we have 
$$
\mathscr{F}(1\otimes x_1^{\langle \alpha_i^\vee,\Lambda''\rangle}e(\nu_2))=\ssum_{\nu_1\in I^{\beta_1}} x_{m+1}^{\langle \alpha_i^\vee, \Lambda''\rangle}e(\nu_1\ast\nu_2)=0,
$$
and $\mathscr{F}$ induces an algebra homomorphism $R^{\Lambda'}(\beta_1)\otimes R^{\Lambda''}(\beta_2)\longrightarrow eR^\Lambda(\beta)e$. 
We then observe that $e\psi_we\ne0$ implies $w=w_1w_2$ with $(w_1,w_2)\in \mathfrak S_{|\beta_1|}\times \mathfrak S_{|\beta_2|}$. 
Hence, the algebra homomorphism $\mathscr{F}$ is surjective. 

To show the injectivity of $\mathscr{F}$, we look at the graded dimensions. Let $K(\nu,\lambda)$ be the sum of monomials $q^{\deg(T)}$ over standard tableaux $T$ of $\lambda$ and $\textbf{i}_T=\nu$. Then, we have
$$
\begin{aligned}
\dim_q R^{\Lambda'}(\beta_1)&=\ssum_{\lambda\vdash |\beta_1|}\left(\ssum_{\nu_1,\nu_1'\in I^{\beta_1}} K(\nu_1,\lambda) K(\nu_1',\lambda)\right ), \\
\dim_q R^{\Lambda''}(\beta_2)&=\ssum_{\lambda\vdash |\beta_2|}\left(\ssum_{\nu_2,\nu_2'\in I^{\beta_2}} K(\nu_2,\lambda) K(\nu_2',\lambda)\right ), \\
\dim_q eR^{\Lambda}(\beta)e&=\ssum_{\lambda\vdash |\beta|}\left(\ssum_{\substack{\nu_1,\nu_1'\in I^{\beta_1}\\ \nu_2,\nu_2'\in I^{\beta_2}}} 
K(\nu_1\ast\nu_2,\lambda) K(\nu_1'\ast\nu_2',\lambda)\right ).
\end{aligned}
$$
Since $K(\nu_1\ast\nu_2,\lambda)\ne0$ only if the multipartition $\lambda$ with respect to $\Lambda$ is a union of multipartitions $\lambda_1$ with respect to $\Lambda'$ and $\lambda_2$ with respect to $\Lambda''$, we have
$$
\dim_q eR^{\Lambda}(\beta)e=
\ssum_{\substack{\lambda_1\vdash|\beta_1|\\ \lambda_2 \vdash|\beta_2|}}
\left(\ssum_{\substack{\nu_1,\nu_1'\in I^{\beta_1}\\ \nu_2,\nu_2'\in I^{\beta_2}}} 
K(\nu_1,\lambda_1) K(\nu_2,\lambda_2) K(\nu_1',\lambda_1) K(\nu_2',\lambda_2)\right ),
$$
which shows $\dim_q eR^{\Lambda}(\beta)e=\dim_q R^{\Lambda'}(\beta_1)
\dim_q R^{\Lambda''}(\beta_2)$.

Finally, we prove that $R^{\Lambda}(\beta)$ and $R^{\Lambda'}(\beta_1)\otimes R^{\Lambda''}(\beta_2)$ are graded Morita equivalent. To see this, it suffices to show that the indecomposable projective $R^{\Lambda}(\beta)$-modules that appear as direct summands of $R^\Lambda(\beta)e(\nu)$, for any $\nu\in I^\beta$, appear as direct summands of $R^\Lambda(\beta)e$. 
Let $n_1:=|\beta_1|$, $n_2:=|\beta_2|$ and $n:=n_1+n_2$. 
Each $\nu\in I^\beta$ defines a black-white sequence of length $n$ with $n_1$ black entries and $n_2$ white entries. Let $w\in \mathfrak S_n$ be the distinguished right coset representative of $(\mathfrak S_{n_1}\times \mathfrak S_{n_2})\backslash \mathfrak S_n$ which changes the black-white sequence by place permutation to the black-white sequence whose first $n_1$ entries are black and the remaining $n_2$ entries are white. We choose a reduced expression of $w$ and define $\psi_w$. Then, there exist $\nu_1\in I^{\beta_1}$ and $\nu_2\in I^{\beta_2}$ such that we have an $R^\Lambda(\beta)$-module homomorphism 
$R^\Lambda(\beta)e(\nu)\rightarrow R^\Lambda(\beta)e(\nu_1\ast\nu_2)$ defined by the right multiplication with $\psi_w$. 

Using the same reduced expression but in the reversed order, we have another $R^\Lambda(\beta)$-module homomorphism $R^\Lambda(\beta)e(\nu_1\ast\nu_2)\rightarrow R^\Lambda(\beta)e(\nu)$ by the right multiplication with $\psi_{w^{-1}}$. We compute the composition: they are given by right multiplication with
$$ e(\nu_1\ast\nu_2)\psi_{w^{-1}}\psi_we(\nu_1\ast\nu_2)\quad\text{or}\quad
e(\nu)\psi_w\psi_{w^{-1}}e(\nu). 
$$
Write $\psi_w=\psi_{i_1}\psi_{i_2}\cdots \psi_{i_r}$. Then,
\begin{align*}
e(\nu)\psi_w\psi_{w^{-1}} &=e(\nu)\psi_{i_1}\cdots\psi_{i_r}^2\cdots\psi_{i_1} \\
&=\psi_{i_1}\cdots\psi_{i_{r-1}}e(s_{i_{r-1}}\cdots s_{i_1}\nu)\psi_{i_r}^2\psi_{i_{r-1}}\cdots \psi_{i_1}. 
\end{align*}
By the minimality of the right coset representative $w$, the entries at $i_r$ and $i_r+1$ are neither (white, white) nor 
(black, black). It follows that $e(s_{i_{r-1}}\cdots s_{i_1}\nu)\psi_{i_r}^2=e(s_{i_{r-1}}\cdots s_{i_1}\nu)$. We continue the same argument. Then,
\begin{equation*}
e(\nu)\psi_w\psi_{w^{-1}}=e(\nu)\psi_{i_1}\cdots\psi_{i_{r-1}}^2\cdots \psi_{i_1}
= \dots =e(\nu)\psi_{i_1}^2=e(\nu), 
\end{equation*}
and $e(\nu_1\ast\nu_2)\psi_{w^{-1}}\psi_w=e(\nu_1\ast\nu_2)$. Hence, we have  
$R^\Lambda(\beta)e(\nu)\cong R^\Lambda(\beta)e(\nu_1\ast\nu_2)$, and this suffices to see that $R^\Lambda(\beta)$ is graded Morita equivalent to $R^{\Lambda'}(\beta_1)\otimes R^{\Lambda''}(\beta_2)$.
\end{proof}

\subsection{Brauer graph algebra}
It is well-known in the literature that Brauer tree algebras are representation-finite, and other Brauer graph algebras, i.e., the remaining algebras whose Brauer graph is either not a tree or with multiple exceptional vertices, are tame. There is an in-depth introduction to Brauer graph algebras, see \cite{Schroll-Brauer-graph}. Besides, some of the latest progress on the derived equivalence of Brauer graph algebras can be found in \cite{AZ-brauer-graph} and \cite{OZ-brauer-graph}. We then will not review the definition of the Brauer graph and its associated algebra. We use the same conventions in this paper as we have given in \cite{ASW-rep-type}. Although any tame cyclotomic KLR algebra in type $A_\ell^{(1)}$ can be realized as a Brauer graph algebra up to Morita equivalence, we point out that it is not always the case in type $C_\ell^{(1)}$, as we mentioned in the introduction.

We remark that, \cite[Lemma~3.1]{CH-type-c-level-1} refers to \cite{AKMW-cellular-tamepolygrowth} for the tame algebra $R^{\Lambda_1}(\delta)$ with $\ell=2$, because the assumption that $\ch \k\ne 2$ in \cite{AKMW-cellular-tamepolygrowth} is put only for guaranteeing Morita invariant property of cellularity, and the bound quiver algebra mentioned there is tame in $\ch \k=2$ as well. Hence, as long as we are content with representation type, 
the characteristic of the field $\k$ does not matter, but if we want to determine the Morita equivalent classes of a cellular algebra, we must note that the basic algebra of a cellular algebra is not necessarily cellular unless $\ch \k\ne2$ or the algebra itself is basic.

We give two examples of Brauer graph algebras in the following, which appear as tame cyclotomic KLR algebras in type $C_\ell^{(1)}$. 

\begin{lemma}\label{Brauer graph algebra case 1}
Suppose $\Lambda=m_0\Lambda_0+m_1\Lambda_1+\cdots+m_\ell\Lambda_\ell\in \pcl$. Then, $R^\Lambda(\alpha_0+\alpha_1)$ is tame if $m_0\ge2$ and $m_1=1$, namely (t3) in {\rm MAIN THEOREM}. More precisely, it is Morita equivalent to the Brauer graph algebra whose Brauer graph is displayed as
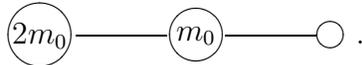
\[
\xymatrix@C=1.2cm{
*[Fo]{2m_0} \ar@{-}[r]
& *+[Fo]{m_0}\ar@{-}[r]
& *++[Fo]{\  }
}
\].
\end{lemma}
\begin{proof}
Let $A:=R^\Lambda(\alpha_0+\alpha_1)$. 
We define $e_1:=e(01)$ and $e_2:=e(10)$. 
Then,
\begin{center}
$\dim_q e_1Ae_1=1+\ssum_{i=1}^{m_0} q^{2(2i-1)}+\ssum_{i=1}^{m_0-1}2q^{4i}+q^{4m_0}$, 
\end{center}
\begin{center}
$\dim_q e_2Ae_2=1+\ssum_{i=1}^{m_0} q^{2i},
\quad
\dim_q e_1Ae_2=\dim_q e_2Ae_1=\ssum_{i=1}^{m_0} q^{2(2i-1)}$.
\end{center}
We show that $e_iAe_j$ has a basis as follows.
$$
\begin{aligned}
e_1Ae_1&=\k\text{-span}\{x_1^ax_2^be_1\mid 0\le a\le m_0-1, 0\le b\le 2 \},\\
e_2Ae_2&=\k\text{-span}\{x_2^ae_2\mid 0\le a\le m_0\},\\
e_1Ae_2&=\k\text{-span}\{\psi_1x_2^ae_2\mid 0\le a\le m_0-1\},\\
e_2Ae_1&=\k\text{-span}\{\psi_1 x_1^ae_1 \mid 0\le a\le m_0-1\}.
\end{aligned}
$$
The required basis for $e_2Ae_2$ follows from  $x_1e_2=0$ and the graded dimension above.
Moreover, $\psi_1^2e_1=(x_1-x_2^2)e_1$ implies that $0=\psi_1x_1e_2\psi_1=x_2\psi_1e_2\psi_1=x_2\psi_1^2e_1=x_2(x_1-x_2^2)e_1$, and hence $x_2^3e_1=x_1x_2e_1$. 
This together with $x_1^{m_0}e_1=0$ and the graded dimensions imply the required bases for $e_1Ae_1$, 
$e_1Ae_2$ and $e_2Ae_1$. 
For $e_2Ae_1$, apply the anti-involution which fixes generators $e_1,e_2, x_1,x_2, \psi_1$ elementwise. 

Set $\alpha:=x_2e_1$, $\mu:=\psi_1e_2$ and $\nu:=\psi_1e_1$. We have 
$$
\alpha\mu=x_2\psi_1e_2=\psi_1x_1e_2=0, 
\quad 
\nu\alpha= \psi_1x_2e_1=x_1\psi_1e_1=0.
$$
Moreover, $\mu\nu=\psi_1^2e_1=(x_1-x_2^2)e_1=x_1e_1-\alpha^2$ such that
$(\mu\nu)^{m_0}=-\alpha^{2m_0}$. 
By comparing dimensions, $A$ is isomorphic to the Brauer graph algebra whose Brauer graph is 
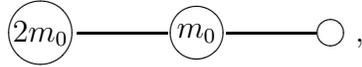
\[
\xymatrix@C=1.2cm{
*[Fo]{2m_0} \ar@{-}[r]
& *+[Fo]{m_0}\ar@{-}[r]
& *++[Fo]{\  }
}
\],
proving the assertion. 
\end{proof}

\begin{lemma}\label{Brauer graph algebra case 2}
Suppose $\Lambda=\Lambda_a+t\Lambda_\ell$ with $t\ge 1$ and $\beta=\alpha_a+\alpha_{a+1}+\cdots+\alpha_\ell$, for some $1\le a\le \ell-2$. 
This is (t6) in {\rm MAIN THEOREM} and the basic algebra of $R^\Lambda(\beta)$  is isomorphic to the Brauer graph algebra whose Brauer graph is displayed as
$$
\entrymodifiers={+[Fo]}
\xymatrix@C=1.2cm{
\ \  \ar@{-}[r]
& t \ar@{-}[r]
& 2t \ar@{-}[r]
& 2t \ar@{.}[r]
& 2t \ar@{-}[r]
& 2t
}\ ,
$$
where the number of vertices is $\ell-a+2$.
\end{lemma}
\begin{proof}
Let $b:=\ell-a+1$ and $e:=e_1+e_2+\cdots+e_b$, where $e_i=e(\nu_i)$ for $1\le i\le b$, and 
$$
\begin{aligned}
\nu_1&=(a~a+1~a+2~\ldots~\ell-3~\ell-2~\ell-1~\ell),\\
\nu_2&=s_{b-1}\nu_1=(a~a+1~a+2~\ldots~\ell-3~\ell-2~\ell~\ell-1),\\
\nu_3&=s_{b-1}s_{b-2}\nu_2=(a~a+1~a+2~\ldots~\ell-3~\ell~\ell-1~\ell-2),\\
    &\ldots \\
\nu_{b-1}&=s_{b-1}s_{b-2}\cdots s_3s_2\nu_{b-2}=(a~\ell ~\ell-1~\ell-2~\ldots~a+2~a+1),\\
\nu_b&=s_{b-1} s_{b-2}\cdots s_2s_1 \nu_{b-1}=(\ell~\ell-1~\ell-2~\ldots~a+2~a+1~a).
\end{aligned}
$$
Write $A:=eR^\Lambda(\beta)e$. We may compute the graded dimensions as follows.
$$
\begin{aligned}
\dim_q e_1Ae_1&=1+\ssum_{i=1}^t q^{4i}, 
\quad 
\dim_q e_2Ae_2=1+\ssum_{i=1}^{2t} q^{2i}+\ssum_{i=t}^{t-1} q^{4i},\\
\dim_q e_iAe_i&=1+\ssum_{i=1}^{2t-1} 2q^{2i}+q^{4t},\quad  \text{ for } 3\le i\le b,\\
\dim_q e_iAe_j&=\left\{\begin{array}{ll}
\sum_{1\le i\le t}q^{4i-2}  & \text{if } (i,j)=(1,2),(2,1),\\
\sum_{1\le i\le 2t}q^{2i-1} & \text{if } |i-j|=1, i,j\ge 2,\\
0& \text{otherwise.}
\end{array}\right.
\end{aligned} 
$$
We then find that the basis of $e_iAe_j$ is given as
$$
\begin{aligned}
e_1Ae_1&=\k\text{-span}\{ x_b^me_1\mid 0\le m\le t\}, \\
e_2Ae_2&=\k\text{-span}\{x_{b-1}^sx_b^me_2\mid 0\le s\le t-1, 0\le m\le 2\},\\
e_1Ae_2&=\k\text{-span}\{x_b^a\psi_{b-1}e_2\mid 0\le a\le t-1\},\\
e_2Ae_1&=\k\text{-span}\{\psi_{b-1}x_b^ae_1\mid 0\le a\le t-1\},
\end{aligned}
$$
and for any $i\ge 2$,
$$
\begin{aligned}
e_{i+1}Ae_{i+1}&=\k\text{-span}\{x_{b-1}^me_{i+1}, x_{b-1}^mx_be_{i+1} \mid 0\le m\le 2t-1\},\\
e_iAe_{i+1}&=\k\text{-span}\{x_{b-1}^a x_{b}^m \psi_{b-i}\psi_{b-i+1}\ldots \psi_{b-1}e_{i+1} \mid 0\le a\le t-1, 0\le m \le 1\},\\
e_{i+1}Ae_{i}&=\k\text{-span}\{ \psi_{b-1}\psi_{b-2}\ldots \psi_{b-i+1}\psi_{b-i} x_{b-1}^ax_{b}^me_{i} \mid 0\le a\le t-1, 0\le m \le 1\}. 
\end{aligned}
$$
\begin{itemize}
\item 
$x_1e_1=0$ and $\psi_i e_1=0$ for $1\le i \le b-2$ imply that $x_j e_1=0$ for $2\le j\le b-1$. 
Then, we have the required basis for $e_1Ae_1$ by the graded dimension. Similarly, we have 
\begin{equation}\label{equ:xact1}
x_i e_j=0 \text{ for } 1\le i\le b-j, \text{ and } x_1^t e_b=0.
\end{equation} 
Moreover, for any $1\le j\le b$, we have 
\begin{equation}\label{equ:xact2}
\begin{aligned}
x_{b-j+1}^te_{j}&=x_{b-j+1}^t\psi_{b-j}^2 e_j=\psi_{b-j}x^t_{b-j}e(s_{b-j}\nu_j)\psi_{b-j}\\
&=\ldots\\
&= \psi_{b-j}\cdots \psi_2\psi_1 x^t_1e(s_1s_2\cdots s_{b-j}\nu_j) \psi_1\psi_2\cdots \psi_{b-j}=0.
\end{aligned}
\end{equation}
In particular, $x^t_{b-1}e_2=0$. On the other hand, $x_b \psi_{b-1}^2e_2=\psi_{b-1}x_{b-1}e_1 \psi_{b-1}=0$. This implies 
\begin{equation}\label{x3ss}
x_b^3e_2=x_{b-1}x_be_2
\end{equation}
and hence, the required basis for $e_2Ae_2$ is obtained by the graded dimension. 

\item 
For $j\ge 3$, $\psi_h e_j=0$ with $b-j+1\le h\le b-2$ implies $(x_{b-j+2}^2-x_{b-j+1})e_j=\psi_{b-j+1}^2e_j=0$ and $(x_{h+1}-x_{h})e_j=\psi_{h}^2e_j=0$ for $b-j+2\le h\le b-2$. Therefore, 
\begin{equation}\label{qeu010}
x_{b-j+2}^{2t}e_j=0
\end{equation} 
by \eqref{equ:xact2}, and  
\begin{equation}\label{equ:xact3}
x_h e_j=x_{b-j+2}e_j \text{ for } b-j+3\le h\le b-1.  
\end{equation}
 
\item For $j\ge 3$, we have 
$$
\begin{aligned}
     x_b\psi_{b-1}^2e_j&=\psi_{b-1}x_{b-1}e(s_{b-1}\nu_j)\psi_{b-1}\\
     &= \psi_{b-1}x_{b-1}\psi^2_{b-2 }e(s_{b-1}\nu_j)\psi_{b-1}\\
     &=\ldots\\
     &= \psi_{b-1}\psi_{b-2}\cdots \psi_{b-j+1} x_{b-j+1}e_{j-1} \psi_{b-j+1} \cdots \psi_{b-2}\psi_{b-1}\\
     &\overset{\eqref{equ:xact1}}=0.
\end{aligned}
$$
This implies that 
\begin{equation}\label{equ:xact4}
x_b^2e_j=x_bx_{b-1}e_j \text{ for }3\le j\le b,
\end{equation}
and it gives the required basis of $e_jAe_j$ for $3\le j\le b$. Furthermore, the required basis of $e_iAe_j$ with $|i-j|=1$ follows from \eqref{equ:xact1}--\eqref{equ:xact3} and the graded dimensions. 
\end{itemize}

We now are able to find the basic algebra of $R^\Lambda(\beta)$. For any $1\le i\le b-1$, we set 
$$
\mu_i:=\psi_{b-i}\psi_{b-i+1}\cdots \psi_{b-1}e_{i+1}\in e_iAe_{i+1},\quad
\nu_i:=\psi_{b-1}\psi_{b-2}\cdots \psi_{b-i+1}\psi_{b-i} e_{i}\in e_{i+1}Ae_i,
$$
and $\alpha:=x_be_b\in e_b Ae_b$. Then, $\mu_i\mu_{i+1}=0=\nu_{i+1}\nu_i$ for $1\le i\le b-2$, and 
$$ 
\begin{aligned}
\mu_{b-1}\alpha&=\psi_1\psi_2\cdots \psi_{b-1} x_b e_b=x_1\psi_1\psi_2\cdots \psi_{b-1} e_b\overset{\eqref{equ:xact1}}=0, \\
\alpha \nu_{b-1}&= x_b \psi_{b-1}\psi_{b-2}\cdots \psi_1 e_{b-1}=\psi_{b-1}\psi_{b-2}\cdots \psi_1 x_1e_{b-1}\overset{\eqref{equ:xact1}}=0.
\end{aligned}
$$ 
We compute $\mu_i\nu_i$ and $\nu_i\mu_i$ as follows.
\begin{itemize}
\item 
$\nu_1\mu_1=\psi_{b-1}^2e_2=(x_{b}^2-x_{b-1})e_2$ and 
$$
\mu_2\nu_2= \psi_{b-2}\psi_{b-1}^2\psi_{b-2}e_2= \psi_{b-2}(x_{b-1}-x_b)\psi_{b-2}e_2\overset{\eqref{equ:xact1}}=-x_b\psi_{b-2}^2e_2= -x_b e_2. 
$$
This together with \eqref{x3ss} and \eqref{qeu010} imply $(\nu_1\mu_1)^t=-(\mu_2\nu_2)^{2t}$.

\item 
Similar computation shows that  $\mu_i\nu_i=-x_b e_i$ for $3\le i\le b-1$, and $\nu_j\mu_j= (x_{b-1}-x_b)e_{j+1}$ for $2\le j\le b-1$. This together with \eqref{qeu010} and  \eqref{equ:xact4} imply that 
$$
(\nu_{i}\mu_i)^{2t}=-(\mu_{i+1}\nu_{i+1})^{2t} \text{ for } 2\le i\le b-2, \text{ and } (\nu_{b-1}\mu_{b-1})^{2t}=-\alpha^{2t}. 
$$
\end{itemize}
We conclude that $A$ is isomorphic to the Brauer graph algebra whose Brauer graph is 
$$
\entrymodifiers={+[Fo]}
\xymatrix@C=1.2cm{
\ \   \ar@{-}[r]
& t \ar@{-}[r]
& 2t \ar@{-}[r]
& 2t \ar@{.}[r]
& 2t \ar@{-}[r]
& 2t
}\ ,
$$
where the number of vertices is $b+1$.
By the crystal computation, we see that the number of simple modules of $R^\Lambda(\beta)$ is exactly $b$. Therefore, $A$ is the basic algebra of $R^\Lambda(\beta)$. 
\end{proof}

\subsection{Tilting mutation and derived equivalence}\label{tilting mutation}
In this subsection only, we denote by $\mod A$ the category of finitely generated right $A$-modules and by $\proj A$ the full subcategory of $\mod A$ consisting of projective $A$-modules. 
This is harmless when we apply the silting theory to a cyclotomic quiver Hecke algebra, because the algebra admits an anti-involution which fixes generators and relations, and the anti-involution swaps left modules and right modules.
Let $\Kb(\proj A)$ be the homotopy category of bounded complexes of finitely generated projective $A$-modules. 
We denote by $\Db(\mod A)$ the derived category of $\mod A$, which is the localization of $\Kb(\proj A)$ with respect to quasi-isomorphisms. 
Both $\Kb(\proj A)$ and $\Db(\mod A)$ are triangulated categories.

Two algebras $A$ and $B$ are said to be \emph{Morita equivalent} if there is a category equivalence $\mod A\cong \mod B$, while $A$ and $B$ are said to be \emph{derived equivalent} if there is a triangle equivalence between the derived categories $\Db(\mod A)$ and $\Db(\mod B)$. If $A$ is a local algebra, then the derived equivalence implies Morita equivalence \cite[Theorem 2.3]{Y-derived-local}. The remarkable derived equivalences of algebras are induced by classical tilting modules, and this area of study has developed into a very extensive research direction now. We refer readers to the Handbook of Tilting Theory \cite{HHK-handbook} to find more details. In particular, it is proven in \cite[Theorem 6.4]{Rickard-tilting-complex} by Rickard that $A$ is derived equivalent to $B$ if and only if there exists a \emph{tilting complex} $T$ in $\Kb(\proj A)$ satisfying $B\cong \End_{\Kb(\proj A)}(T)$. Further, $\Kb(\proj A)$ is triangle equivalent to $\Kb(\proj B)$ if and only if $A$ and $B$ are derived equivalent. Thus, it suffices to study tilting complexes in $\Kb(\proj A)$ in order to understand the derived equivalence of $A$.

Let us review the silting theory, a generalization of tilting theory. Silting is also known as \emph{half-tilting}. 
A core concept in silting theory is \emph{silting mutation} introduced by  Aihara and Iyama in \cite{AI-silting}. 
In ideal cases, we can classify Morita equivalence classes of algebras in the derived equivalence class of $A$ by computing a finite number of tilting complexes by mutation and their endomorphism algebras, as we will see below.

Given a complex $T\in \Kb(\proj A)$, we denote by $\thick T$ the smallest thick subcategory of $\Kb(\proj A)$ containing $T$, and by $\add(T)$ the full subcategory of $\Kb(\proj A)$ whose objects are direct summands of finite direct sums of copies of $T$.

\begin{definition}[{\cite[Definition 2.1]{AI-silting}}]
A complex $T\in \Kb(\proj A)$ is said to be 
\begin{enumerate}
\item presilting (pretilting) if $\Hom_{\Kb(\proj A)}(T,T[i])=0$, for any $i>0$ ($i\ne 0$).
\item silting (tilting) if $T$ is presilting (pretilting) and $\thick T=\Kb(\proj A)$.
\end{enumerate}
\end{definition}

Suppose $T:=X\oplus Y$ is a basic silting complex in $\Kb(\proj A)$.
We take a triangle in $\Kb(\proj A)$ with a minimal left $\add(Y)$-approximation $\pi$:
$$
X\overset{\pi}{\longrightarrow}Z\longrightarrow X'\longrightarrow X[1],
$$
where $X'$ is the mapping cone of $\pi$.
Then, $\mu_X^-(T):=X'\oplus Y$ is again a basic silting complex in $\Kb(\proj A)$, see \cite[Theorem 2.31]{AI-silting}.
We call $\mu_X^-(T)$ the left silting mutation of $T$ with respect to $X$. Dually, we obtain the right silting mutation $\mu_X^+(T)$ of $T$ with respect to $X$. If $X$ is an indecomposable direct summand of $T$, then $\mu_X^{\pm}(T)$ is said to be \emph{irreducible}. If $T$ is a tilting complex, then $\mu_X^{\pm}(T)$ is called a \emph{left/right tilting mutation}.

\begin{example}
Let $A$ be the path algebra of the bipartite quiver: $\xymatrix@C=0.8cm{1 \ar[r]^{\alpha_1} & 2  & 3 \ar[l]_{\alpha_2} \ar[r]^{\alpha_3}  & 4}$. 
We denote by $P_i$ the indecomposable projective $A$-module at vertex $i\in \{1,2,3,4\}$. Then, by direct calculation, we have
$$
\mu_{P_2\oplus P_4}^-(A)=
\left [\begin{smallmatrix}
\xymatrix@C=1cm{P_2\ar[r]^-{(\alpha_1, \alpha_2)^t}& P_1\oplus P_3}\\
\oplus \\
\xymatrix@C=1cm{P_4\ar[r]^-{\alpha_3}& P_3}\\
\oplus \\
\xymatrix@C=1.1cm{0\ar[r] & P_1\oplus P_3}
\end{smallmatrix}  \right ] \quad \text{and} \quad
\mu_{(P_1\oplus P_3)[1]}^+(A[1])=
\left [\begin{smallmatrix}
\xymatrix@C=1cm{P_2\oplus P_4\ar[r]^-{(\alpha_2, \alpha_3)}& P_3}\\
\oplus \\
\xymatrix@C=1cm{P_2\ar[r]^-{\alpha_1}& P_1}\\
\oplus \\
\xymatrix@C=1.1cm{P_2\oplus P_4\ar[r] &0}
\end{smallmatrix}  \right ].
$$
\end{example}

Let $\silt A$ be the set of isomorphism classes of basic silting complexes in $\Kb(\proj A)$. We construct a directed graph $\H(\silt A)$ by drawing an arrow from $T$ to $S$ if $S$ is an irreducible left silting mutation of $T$. On the other hand, we may regard $\silt A$ as a poset concerning a partial order: $T\geqslant S$ if $\Hom_{\Kb(\proj A)}(T,S[i])=0$ for any $i>0$. Then, the directed graph $\H(\silt A)$ is exactly the Hasse quiver of the poset $\silt A$. In other words, the Hasse quiver of $\silt A$ realizes the left/right silting mutations of silting complexes.
\begin{proposition}
For any $S, T \in \silt A$, the following conditions are equivalent.
\begin{enumerate}
\item $S$ is an irreducible left silting mutation of $T$.
\item $T$ is an irreducible right silting mutation of $S$.
\item $T>S$ and there is no $U\in \silt A$ such that $S<U<T$.
\end{enumerate}
\end{proposition}

Since mutation produces strictly decreasing silting complexes with respect to the partial order, $\H(\silt A)$ is an infinite quiver in general. However, the set of endomorphism algebras of silting complexes in $\silt A$ may not be infinite, due to the existence of a certain cyclic phenomenon. Such a cyclic phenomenon has already appeared in the literature, e.g., \cite{Ar-tame-block}, \cite{Au-silting} and \cite{W-two-point-II}. To explain this, we start with the following proposition.

\begin{proposition}[{\cite[Lemma 2.8]{Au-silting}}]\label{lemma-original}
Let $A$ and $B$ be two algebras with a triangle equivalence $\mathscr{T}: \Db(\mod A)\longrightarrow \Db(\mod B)$. Then, the following statements hold.
\begin{enumerate}
\item $\mathscr{T}$ sends silting/tilting complexes in $\Kb(\proj A)$ to that in $\Kb(\proj B)$.
\item $\mathscr{T}$ preserves the partial order on the set of silting complexes.
\item If $T$ is a silting complex in $\Kb(\proj A)$, then $\mathscr{T}(\mu_X^-(T))\cong \mu_{\mathscr{T}(X)}^-(\mathscr{T}(T))$ for some direct summand $X$ of $T$.
\end{enumerate}
\end{proposition}

Let $T=X_1\oplus X_2\oplus\cdots \oplus X_n$ be a tilting complex in $\Kb(\proj A)$ and let $B$ be the endomorphism algebra of $T$. We denote by $Q_1, Q_2, \ldots, Q_n$ the indecomposable projective $B$-modules. Then, the triangle equivalence $\mathscr{T}: \Kb(\proj A)\longrightarrow \Kb(\proj B)$ is induced by mapping $X_i$ to $Q_i$ for $i=1,2,\ldots, n$. We consider the following irreducible left silting mutation:
\begin{center}
$\vcenter{\xymatrix@C=0.4cm@R=1cm{
T\ar[rr]\ar[d]^{\mathscr{T}}&&\mu_{X_i}^-(T)\ar[d]^{\mathscr{T}}&\in \Kb(\proj A)\\
B \ar[rr]&&\mu_{Q_i}^-(B)& \in \Kb(\proj B)}}$.
\end{center}
Note that $\mu_{X_i}^-(T)$ and $\mu_{Q_i}^-(B)$ are again silting but they are not necessarily tilting. 

As $\mathscr{T}$ sends $\add(T/X_i)$-approximation to $\add(B/Q_i)$-approximation, we have the following statement.

\begin{corollary}\label{cor::endomorphism-algebra}
We have $\End_{\Kb(\proj A)} \mu_{X_i}^-(T)\cong \End_{\Kb(\proj B)} \mu_{Q_i}^-(B)$.
\end{corollary}

Suppose that the above $\mu_{X_i}^-(T)$ and $\mu_{Q_i}^-(B)$ are tilting and we are in the situation where there is a mutation chain of tilting complexes $T_1, T_2, \cdots$. Then, we may repeatedly apply Corollary \ref{cor::endomorphism-algebra} to calculate the endomorphism algebra $B_i:=\End_{\Kb(\proj A)} T_i$, as follows. 
\begin{equation}\label{equ::end-alg-tilt-chain}
\vcenter{\xymatrix@C=1cm@R=0.5cm{
T_1\ar[r]\ar[d]^{\mathscr{T}}&T_2\ar[d]^{\mathscr{T}}\ar[r]&T_3\ar[r]\ar[dd]^{\mathscr{T}}&T_4\ar[ddd]^{\mathscr{T}}\ar[r]&T_5\ar[r]\ar[dddd]^{\mathscr{T}}&\cdots\ar@{.>}[ddddd]^{\mathscr{T}}\\
B_1 \ar[r]&\mu_{i_1}^-(B_1)\ar[d]^{\mathscr{T}}& & &\\
&B_2 \ar[r]&\mu_{i_2}^-(B_2)\ar[d]^{\mathscr{T}}&&\\
&&B_3 \ar[r]&\mu_{i_3}^-(B_3)\ar[d]^{\mathscr{T}}&\\
&&&B_4 \ar[r]&\mu_{i_4}^-(B_4)\ar[d]^{\mathscr{T}}\\
&&&&\cdots\ar[r]&\cdots
}}
\end{equation}
Here, $\mu_{i}^-(T)$ stands for the irreducible left tilting mutation of $T$ with respect to the $i$-th indecomposable direct summand of $T$. This gives an efficient method to find derived equivalence classes of $A$. In \eqref{equ::end-alg-tilt-chain}, the cyclic phenomenon we mentioned before is that $B_1, B_2, \ldots, B_s$, for some $s\in \N$, appear in this order alternately in the corresponding chain of endomorphism algebras.

We define $\twosilt A:=\{T\mid A \ge T\ge A[1]\}\subset \silt A$, and elements in $\twosilt A$ are called \emph{2-term silting complexes}. Then, $\twosilt A$ is again a poset, so that its Hasse quiver $\H(\twosilt A)$ is a subquiver of $\H(\silt A)$. It is also worth mentioning that there is a poset isomorphism between $\twosilt A$ and the set of support $\tau$-tilting $A$-modules in the sense of $\tau$-tilting theory, see \cite{AIR} for more details.

Symmetric algebras admit a nice feature in silting theory. Let $A$ be a symmetric algebra. It is proved in \cite{Aihara-symmetric-alg} that any silting complex in $\Kb(\proj A)$ is a tilting complex. Therefore, $\silt A$ coincides with $\tilt A$, the set of isomorphism classes of tilting complexes, and the assumption of \eqref{equ::end-alg-tilt-chain} is automatically satisfied in $\H(\tilt A)$. We obtain the following theorem for symmetric algebras.

\begin{theorem}\label{enumeration of Morita classes}
Let $A_1, A_2, \ldots, A_s$ be finite-dimensional symmetric algebras which are derived equivalent to each other and identify $\mathcal{T}=\Kb(\proj A_i)$ for all $1\le i\le s$. 
Suppose the following conditions hold. 
\begin{enumerate}
\item
The set $\twosilt A_i$ is finite\footnote{This condition is equivalent to that the algebras $A_i$ are $\tau$-tilting finite or brick-finite, see \cite{AIR} and \cite{DIJ-tau-tilting-finite}.}, for $1\le i\le s$.
\item
For each indecomposable projective direct summand $X$ of the left regular module $A_i$, for $1\le i\le s$, we have
$\End_{\mathcal{T}}(\mu_X^-(A_i))\cong A_j$, for some $1\le j\le s$.
\end{enumerate}
Then, any finite-dimensional algebra $B$ which has derived equivalence
$$
\Db(\mod B)\cong \Db(\mod A_1)\left(\cong \Db(\mod A_2)\cong\cdots\cong \Db(\mod A_s)\right)
$$
is Morita equivalent to $A_i$, for some $1\le i\le s$.
\end{theorem}
\begin{proof}
We need the concept of silting-discreteness in silting theory: an algebra $A$ is said to be \emph{silting-discrete} if there is a silting object $T$ such that $\{S\mid T\ge S\ge T[k]\}\subset \silt A$ is a finite set, for any $k\in \N$. A nice property (see \cite{Aihara-symmetric-alg}) of a silting-discrete algebra $A$ is that each silting complex in $\silt A$ can be obtained by iterated irreducible left silting mutation from a shift of the stalk complex $A$. It is then shown in \cite[Theorem 16]{AI-silting-discrete} that $A$ is silting-discrete if and only if there is a silting object $T\in \silt A$ such that $\{S\mid U\ge S\ge U[1]\}$ is finite, for any iterated irreducible left silting mutation $U$ of $T$.

Note that silting-discreteness is equivalent to tilting-discreteness since $A_1$ is a symmetric algebra. Let $X$ be an indecomposable projective summand of $A$. We set 
$$
\mu_Y^-\circ \mu_X^-(A):=\mu_Y^-(\End_{\mathcal{T}} \mu_X^-(A)),
$$
where $Y$ is an indecomposable projective summand of $\End_{\mathcal{T}} \mu_X^-(A)$.

Suppose that $U$ is an iterated irreducible left silting mutation of $A_1 \in \silt A_1$. Using Corollary \ref{cor::endomorphism-algebra} repeatedly, we obtain
$$
U\cong \mu_{X_k}^-\circ\cdots\circ\mu_{X_2}^-\circ\mu_{X_1}^-(A_1),
$$
for some $k\in \N$ and some indecomposable projective summands $X_i$'s of $\End_{\mathcal{T}}(U_{i-1})$, where $U_i:=\mu_{X_i}^-\circ\cdots\circ\mu_{X_1}^-(A_1)$ for $2\le i\le k$. Then, assumption (2) says that $\End_{\mathcal{T}} (U_1)\cong A_j$ for some $1\le j\le s$. We assume that $\End_{\mathcal{T}}(U_{i-1})\cong A_h$, for some $1\le h\le s$, holds. Then, Rickard's Morita theorem implies that there is an auto-equivalence $\mathscr{T}:\mathcal{T}\cong \mathcal{T}$ providing $\mathscr{T}(U_{i-1})=A_h$. 
See \cite[Chapter 3]{KZ-der-Equiv-book}. Hence, we have
$$
\End_{\mathcal{T}}(U_i)=\End_{\mathcal{T}}(\mu_{X_i}(U_{i-1}))\cong \End_{\mathcal{T}}(\mu_{\mathscr{T}(X_i)}(A_h)).
$$
In particular, $\mathscr{T}(X_i)$ is an indecomposable projective direct summand of $A_h$. We deduce by assumption (2) that $\End_{\mathcal{T}} (U_i)\cong A_j$ for some $1\le j\le s$. It finally gives that $\End_{\mathcal{T}} (U)\cong A_j$ for some $1\le j\le s$. On the other hand, using Rickard's Morita theorem again, the set $\{S\mid U\ge S\ge U[1]\}$ is in bijection with the set $\{S\mid A_j\ge S\ge A_j[1]\}$. By assumption (1), we conclude that $A_1$ is tilting-discrete.

Let $B$ be the algebra which is derived equivalent to $A_1$. By Rickard's Morita theorem, there is a tilting complex $T\in \Kb(\proj A_1)$ such that $B\cong \End_{\mathcal{T}}(T)$. Since $A_1$ is tilting-discrete, $T$ is obtained by iterated irreducible left silting mutation from a shift of the stalk complex $A_1$. Then, by the above argument, $\End_{\mathcal{T}}(T)\cong A_j$, for some $1\le j\le s$. 
\end{proof}

\subsection{The derived equivalence class of (t7)}
There is a tame case (t7) of cyclotomic KLR algebras in affine type C, which cannot be realized as a Brauer graph algebra. Then, we may use Theorem \ref{enumeration of Morita classes} to find all Morita equivalence classes of algebras that are derived equivalent to (t7). We consider the following quiver:
\begin{center}
$Q: \xymatrix@C=1cm{\circ \ar@<0.5ex>[r]^{\mu}\ar@(dl,ul)^{\alpha}&\circ \ar@<0.5ex>[l]^{\nu}\ar@(ur,dr)^{\beta}}$,
\end{center}
and define 
\begin{itemize}
\item $A:=\k Q/\{\alpha^2=0, \beta^2=\nu\mu, \alpha\mu=\mu\beta, \beta\nu=\nu\alpha\}$.
\item $B:=\k Q/\{\alpha^2=\mu\nu, \beta^2=\nu\mu, \alpha\mu=\mu\beta, \beta\nu=\nu\alpha, \mu\nu\mu=\nu\mu\nu=0\}$.
\end{itemize}
Here, $A$ is the tame algebra (t7) (See Lemma~\ref{t=2:alpha_0+alpha_1}) and $B$\footnote{We have not checked whether $B$ appears as the basic algebra of some $R^\Lambda(\beta)$.} is a factor algebra of the tame algebra numbered (21) in \cite[Table T]{H-wild-two-point}.

\begin{lemma}
The algebras $A$ and $B$ are cellular. 
\end{lemma}
\begin{proof}
Recall $A=R^{2\La_{\ell-1}+\La_\ell}(\alpha_{\ell-1}+\alpha_\ell)$ and Lemma \ref{t=2:alpha_0+alpha_1} in Section 7 below implies that $A$ is the algebra (t7) in the main theorem. Hence, the cellularity of $A$ follows from \cite[Theorem 4C.3, Corollary 4C.7]{EM-cellular-symmetrictypeA}. \footnote{Note that the failure of the cellularity of the basic algebra $eAe$ of a cellular algebra $A$ comes from the failure of choosing an idempotent $e$ which is fixed by the anti-involution used in defining the cellular algebra. This issue does not appear when 
$R^\La(\beta)$ is basic, or the idempotent $e$ is a sum of $e(\nu)$'s such that $eR^\La(\beta)e$ is basic, since the anti-involution which fixes each generator is the anti-involution in the data defining the cellular algebra structure of $R^\La(\beta)$, which is recalled in the proof of \cite[Theorem 4C.3]{EM-cellular-symmetrictypeA}.}

Let $Q_i$ be the indecomposable projective $B$-module at vertex $i\in \{1,2\}$. Then, $Q_1$ has the $\k$-basis $\{e_1, \alpha, \mu, \alpha\mu=\mu\beta, \alpha^2=\mu\nu, \alpha^3=\alpha\mu\nu=\mu\nu\alpha=\mu\beta\nu \}$ and $Q_2$ has the $\k$-basis $\{e_2, \beta, \nu, \beta\nu=\nu\alpha, \beta^2=\nu\mu, \beta^3=\beta\nu\mu=\nu\mu\beta=\nu\alpha\mu \}$. We take a totally ordered set $ \Phi=\{\phi_1<\phi_2<\phi_3<\phi_4<\phi_5<\phi_6\}$ and define 
$$
\mathcal{M}(\phi_1)=\mathcal{M}(\phi_2)=\mathcal{M}(\phi_5)=\mathcal{M}(\phi_6)=\{1\}, \quad \mathcal{M}(\phi_3)=\mathcal{M}(\phi_4)=\{1,2\}.
$$
We construct $(c_{st}^{\phi_i})_{s,t\in \mathcal{M}(\phi_i)}$ as follows,
$$
c_{11}^{\phi_1}=(e_1), \quad c_{11}^{\phi_2}=(e_2), \quad c_{11}^{\phi_5}=(\alpha^3), \quad c_{11}^{\phi_6}=(\beta^3), 
$$
$$
(c_{st}^{\phi_3})_{s,t\in \mathcal{M}(\phi_3)}=
\begin{pmatrix}
\alpha&\nu\\
\mu&  \beta
\end{pmatrix}, 
\quad
(c_{st}^{\phi_4})_{s,t\in \mathcal{M}(\phi_4)}=
\begin{pmatrix}
\alpha^2&\beta\nu\\
\alpha\mu&  \beta^2
\end{pmatrix}.
$$
Let $\imath$ be the anti-involution of $B$ given by $\imath(e_1)=e_1$, $\imath(e_2)=e_2$, $\imath(\alpha)=\alpha$, $\imath(\beta)=\beta$ and $\imath(\mu)=\nu$, $\imath(\nu)=\mu$. Then, $(\Phi, \mathcal{M}, c_{st}^{\phi_i}, \imath)$ provides a cell datum and $(c_{st}^{\phi_i})_{s,t\in \mathcal{M}(\phi_i)}$ gives a cellular basis of $B$.   
\end{proof}

Since the cyclotomic quiver Hecke algebra has an anti-involution which fixes generators and relations, the category of left $A$-modules and the category of right $A$-modules are equivalent. Thus, it is harmless to work with right $A$-modules instead of left $A$-modules as we mentioned in subsection \ref{tilting mutation}, and we compute with right modules in this subsection. Let $P_i$ be the indecomposable projective $A$-module at vertex $i\in \{1,2\}$. 
We may read the non-zero paths starting from $e_i$ and connect them using an undirected line. It gives the structure of $P_i$ as follows. 
\begin{center}
$P_1=\vcenter{\xymatrix@C=0.1cm@R=0.1cm{
&e_1\ar@{-}[dl]\ar@{-}[dr]&&\\
\alpha\ar@{-}[dr]&&\mu\ar@{-}[dl]\ar@{-}[dr]&\\
&\alpha\mu\ar@{-}[dr]&&\mu\nu\ar@{-}[dl]\\
&&\alpha\mu\nu& }}$ 
$\cong$ $\vcenter{\xymatrix@C=0.1cm@R=0.1cm{
&1\ar@{-}[dl]\ar@{-}[dr]&&\\
1\ar@{-}[dr]&&2\ar@{-}[dl]\ar@{-}[dr]&\\
&2\ar@{-}[dr]&&1\ar@{-}[dl]\\
&&1& }}$, 
$P_2=\vcenter{\xymatrix@C=0.1cm@R=0.1cm{
&e_2\ar@{-}[dl]\ar@{-}[dr]&&\\
\beta\ar@{-}[dr]\ar@{-}[drrr]&&\nu\ar@{-}[dl]\ar@{-}[dr]&\\
&\beta\nu\ar@{-}[dr]&&\nu\mu\ar@{-}[dl]\\
&&\beta\nu\mu& }}$ 
$\cong$ $\vcenter{\xymatrix@C=0.1cm@R=0.1cm{
&2\ar@{-}[dl]\ar@{-}[dr]&&\\
2\ar@{-}[dr]\ar@{-}[drrr]&&1\ar@{-}[dl]\ar@{-}[dr]&\\
&1\ar@{-}[dr]&&2\ar@{-}[dl]\\
&&2& }}$.
\end{center}
It gives 
\begin{center}
\renewcommand\arraystretch{1.2}
\begin{tabular}{c|ccccc}
$\Hom$ & 1& 2 \\ \hline
1 & $e_1, \alpha, \mu\nu, \alpha\mu\nu$  & $\nu, \beta\nu$ \\
2 & $\mu, \alpha\mu$ & $e_2, \beta, \nu\mu, \beta\nu\mu$ 
\end{tabular}
\end{center}
By direct calculation, the Hasse quiver $\H(\twosilt A)$ is given as 
\begin{center}
$\vcenter{\xymatrix@C=1cm@R=0.5cm{
\mu_2^-(A)\ar[dd]&&\mu_1^-(A)\ar[dd]\\
&A\ar[ur]\ar[ul]&\\
\mu_1^-(\mu_2^-(A))\ar[ur]|-{[1]}&&\mu_2^-(\mu_1^-(A))\ar[ul]|-{[1]}}}$,
\end{center}
where $\mu_i^-(-):=\mu_{P_i}^-(-)$, $\xymatrix@C=1cm{X\ar[r]|-{[1]}&Y}$ means $\xymatrix@C=0.7cm{X\ar[r]&Y[1]}$.

\begin{proposition}
We have $\End_{\Kb(\proj A)} \mu_1^-(A)\cong B$.
\end{proposition}
\begin{proof}
By direct calculation, it is easy to find 
\begin{center}
$\mu_1^-(A)=
\left [\begin{smallmatrix}
\xymatrix@C=1.2cm{P_1\ar[r]^-{\nu}& P_2}\\
\oplus \\
\xymatrix@C=1.4cm{0\ar[r]& P_2}
\end{smallmatrix}  \right ]$.
\end{center}

Recall that $A$-module homomorphisms between projective $A$-modules are given by left multiplication with elements from $A$. 
In the diagram below, the top square means that if we set $f^{-1}$ and $f^0$ to be linear combination of $\{e_1, \alpha, \mu\nu, \alpha\mu\nu\}$ and $\{e_2, \beta, \nu\mu, \beta\nu\mu\}$, respectively, and force $\nu f^{-1}=f^0\nu$, then we obtain that $(f^{-1},f^0)$ is linear combination of $(e_1,e_2), (\alpha,\beta), (\mu\nu,0), (0,\nu\mu), (\alpha\mu\nu,0), (0,\beta\nu\mu)$, and 
$\End_{\Kb(\proj A)}(P_1\stackrel{\nu}{\rightarrow} P_2)$ has basis $\{(e_1,e_2), (\alpha,\beta), (0,\nu\mu), (0,\beta\nu\mu)\}$. The meaning of the other three squares is similar. 
\begin{center}
$\xymatrix@C=1.5cm@R=0.7cm{
P_1\ar[r]^-{\nu}\ar[d]& P_2\ar[d]^{(e_1, e_2),\  (\alpha, \beta),\  (0, \nu\mu)\sim_h(-\mu\nu, 0), \  (0, \beta\nu\mu)\sim_h(-\alpha\mu\nu, 0)}\\
P_1\ar[r]^-{\nu}\ar[d]& P_2\ar[d]^{(0,\nu\mu),\ (0,\beta\nu\mu)}\\
0\ar[r]\ar[d]& P_2\ar[d]^{(0, e_2), \ (0, \beta),\ (0,\nu\mu),\ (0,\beta\nu\mu)}\\
0\ar[r]\ar[d]& P_2\ar[d]^{(0, e_2), \ (0, \beta),\ (0,\nu\mu)\sim_h 0,\ (0,\beta\nu\mu)\sim_h 0}\\
P_1\ar[r]^-{\nu}& P_2
}$
\end{center}
Set $x:=(0, e_2), y:=(0,\nu\mu), z:=(\alpha, \beta),t:=(0,\beta)$. We have 
\begin{center}
$Q:\vcenter{\xymatrix@C=1cm{1 \ar@<0.5ex>[r]^{x}\ar@(dl,ul)^{z}&2 \ar@<0.5ex>[l]^{y}\ar@(ur,dr)^{t}}}$,
\end{center}
and 
\begin{itemize}
\item $z^2=(0, \beta^2)=(0, \nu\mu)=xy$, $zx=(0, \beta)=xt$, $t^2=yx$, $ty=(0, \beta\nu\mu)=yz$,
\item $z^3=(0,\beta^3)=(0,\beta\nu\mu)=zxy=xyz=xty$, $z^2x=xyx=zxt=xt^2=(0,\nu\mu)=0$, $t^3=tyx=yxt=yzx=(0,\beta\nu\mu)$, $t^2y=yxy=tyz=yz^2=(0, \nu\mu\nu\mu)=0$, 
\item all paths of length 4 are zero.
\end{itemize}
It gives that $\End_{\Kb(\proj A)} \mu_1^-(A)\cong \k Q/J$ with 
$J$ generated by 
$$ 
\{z^2-xy, t^2-yx, zx-xt, ty-yz, xyx, yxy\}. 
$$ 
Therefore, $\End_{\Kb(\proj A)} \mu_1^-(A)$ is isomorphic to $B$.
\end{proof}

Let $Q_i$ be the indecomposable projective $B$-module at vertex $i\in \{1,2\}$. Then, 
\begin{center}
$Q_1=\vcenter{\xymatrix@C=0.1cm@R=0.1cm{
&e_1\ar@{-}[dl]\ar@{-}[dr]&&\\
\alpha\ar@{-}[dr]\ar@{-}[drrr]&&\mu\ar@{-}[dl]\ar@{-}[dr]&\\
&\alpha\mu\ar@{-}[dr]&&\mu\nu\ar@{-}[dl]\\
&&\alpha\mu\nu& }}$ 
$\cong$ $\vcenter{\xymatrix@C=0.1cm@R=0.1cm{
&1\ar@{-}[dl]\ar@{-}[dr]&&\\
1\ar@{-}[dr]\ar@{-}[drrr]&&2\ar@{-}[dl]\ar@{-}[dr]&\\
&2\ar@{-}[dr]&&1\ar@{-}[dl]\\
&&1& }}$,
$Q_2=\vcenter{\xymatrix@C=0.1cm@R=0.1cm{
&e_2\ar@{-}[dl]\ar@{-}[dr]&&\\
\beta\ar@{-}[dr]\ar@{-}[drrr]&&\nu\ar@{-}[dl]\ar@{-}[dr]&\\
&\beta\nu\ar@{-}[dr]&&\nu\mu\ar@{-}[dl]\\
&&\beta\nu\mu& }}$ 
$\cong$ $\vcenter{\xymatrix@C=0.1cm@R=0.1cm{
&2\ar@{-}[dl]\ar@{-}[dr]&&\\
2\ar@{-}[dr]\ar@{-}[drrr]&&1\ar@{-}[dl]\ar@{-}[dr]&\\
&1\ar@{-}[dr]&&2\ar@{-}[dl]\\
&&2& }}$.
\end{center}
The Hasse quiver $\H(\twosilt B)$ is displayed as
\begin{center}
$\vcenter{\xymatrix@C=1cm@R=0.5cm{
\mu_2^-(B)\ar[dd]&&\mu_1^-(B)\ar[dd]\\
&B\ar[ur]\ar[ul]&\\
\mu_1^-(\mu_2^-(B))\ar[ur]|-{[1]}&&\mu_2^-(\mu_1^-(B))\ar[ul]|-{[1]}}}$.
\end{center}
In particular, we have 
\begin{center}
$\mu_1^-(B)=
\left [\begin{smallmatrix}
\xymatrix@C=1.2cm{Q_1\ar[r]^-{\nu}& Q_2}\\
\oplus \\
\xymatrix@C=1.4cm{0\ar[r]& Q_2}
\end{smallmatrix}  \right ]$
and $\mu_2^-(B)=
\left [\begin{smallmatrix}
\xymatrix@C=1.4cm{0\ar[r]& Q_1}\\
\oplus \\
\xymatrix@C=1.2cm{Q_2\ar[r]^-{\mu}& Q_1}
\end{smallmatrix}  \right ]$.
\end{center}
\begin{itemize}
\item One may find $\End_{\Kb(\proj B)} \mu_1^-(B)\cong B$, using
\begin{center}
$\xymatrix@C=1.5cm@R=0.7cm{
Q_1\ar[r]^-{\nu}\ar[d]& Q_2\ar[d]^{(e_1, e_2),\  (\alpha, \beta),\  (0, \nu\mu)\sim_h(-\mu\nu, 0), \  (0, \beta\nu\mu)\sim_h(-\alpha\mu\nu, 0)}\\
Q_1\ar[r]^-{\nu}\ar[d]& Q_2\ar[d]^{(0,\nu\mu),\ (0,\beta\nu\mu)}\\
0\ar[r]\ar[d]& Q_2\ar[d]^{(0, e_2), \ (0, \beta),\ (0,\nu\mu),\ (0,\beta\nu\mu)}\\
0\ar[r]\ar[d]& Q_2\ar[d]^{(0, e_2), \ (0, \beta),\ (0,\nu\mu)\sim_h 0,\ (0,\beta\nu\mu)\sim_h 0}\\
Q_1\ar[r]^-{\nu}& Q_2
}$
\end{center}

\item One may find $\End_{\Kb(\proj B)} \mu_2^-(B)\cong A$, using
\begin{center}
$\xymatrix@C=1.5cm@R=0.7cm{
Q_2\ar[r]^-{\mu}\ar[d]& Q_1\ar[d]^{(e_2, e_1),\  (\beta, \alpha),\ (0,\mu\nu)\sim_h(-\nu\mu,0), \ (0,\alpha\mu\mu)\sim_h(-\beta\nu\mu,0) }\\
Q_2\ar[r]^-{\mu}\ar[d]& Q_1\ar[d]^{(0,\mu\nu),\ (0,\alpha\mu\nu)}\\
0\ar[r]\ar[d]& Q_1\ar[d]^{(0, e_1), \ (0, \alpha),\ (0,\mu\nu),\ (0,\alpha\mu\nu)}\\
0\ar[r]\ar[d]& Q_1\ar[d]^{(0, e_1), \ (0, \alpha),\ (0,\mu\nu)\sim_h 0,\ (0,\alpha\mu\nu)\sim_h 0}\\
Q_2\ar[r]^-{\mu}& Q_1
}$
\end{center}
Here, if one replaces $Q_i$ with $P_i$, then one obtains $\End_{\Kb(\proj A)} \mu_2^-(A)\cong A$.
\end{itemize}

\begin{proposition}\label{Morita classes of (t7)}
If an algebra $C$ is derived equivalent to $A$, then $C$ is isomorphic to $A$ or $B$.
\end{proposition}
\begin{proof}
By direct calculation, we have found that both $\twosilt A$ and $\twosilt B$ are finite. We also obtained in the above that 
\begin{itemize}
\item $\End_{\Kb(\proj A)} \mu_1^-(A)\cong B$ and $\End_{\Kb(\proj A)} \mu_2^-(A)\cong A$.
\item $\End_{\Kb(\proj B)} \mu_1^-(B)\cong B$ and $\End_{\Kb(\proj B)} \mu_2^-(B)\cong A$.
\end{itemize}
Then, the algebra $C$ is Morita equivalent to $A$ or $B$ by  Theorem \ref{enumeration of Morita classes}. 
\end{proof}

\begin{example}
 If we consider the tilting complex $\mu_2^-(\mu_1^-(\mu_1^-(A)))\in \tilt A$, for example, then
the endomorphism algebra is
\begin{center}
$\begin{aligned}
\End_{\Kb(\proj A)} \mu_2^-(\mu_2^-(\mu_1^-(\mu_1^-(A))))
&\cong 
\End_{\Kb(\proj B)} \mu_2^-(\mu_2^-(\mu_1^-(B)))\\
\cong 
\End_{\Kb(\proj B)} \mu_2^-(\mu_2^-(B))
&\cong 
\End_{\Kb(\proj A)} \mu_2^-(A)\cong A
\end{aligned}$.
\end{center}
\end{example}
We can construct the silting quiver $\H(\silt A)$ as in the next page.

\newpage 
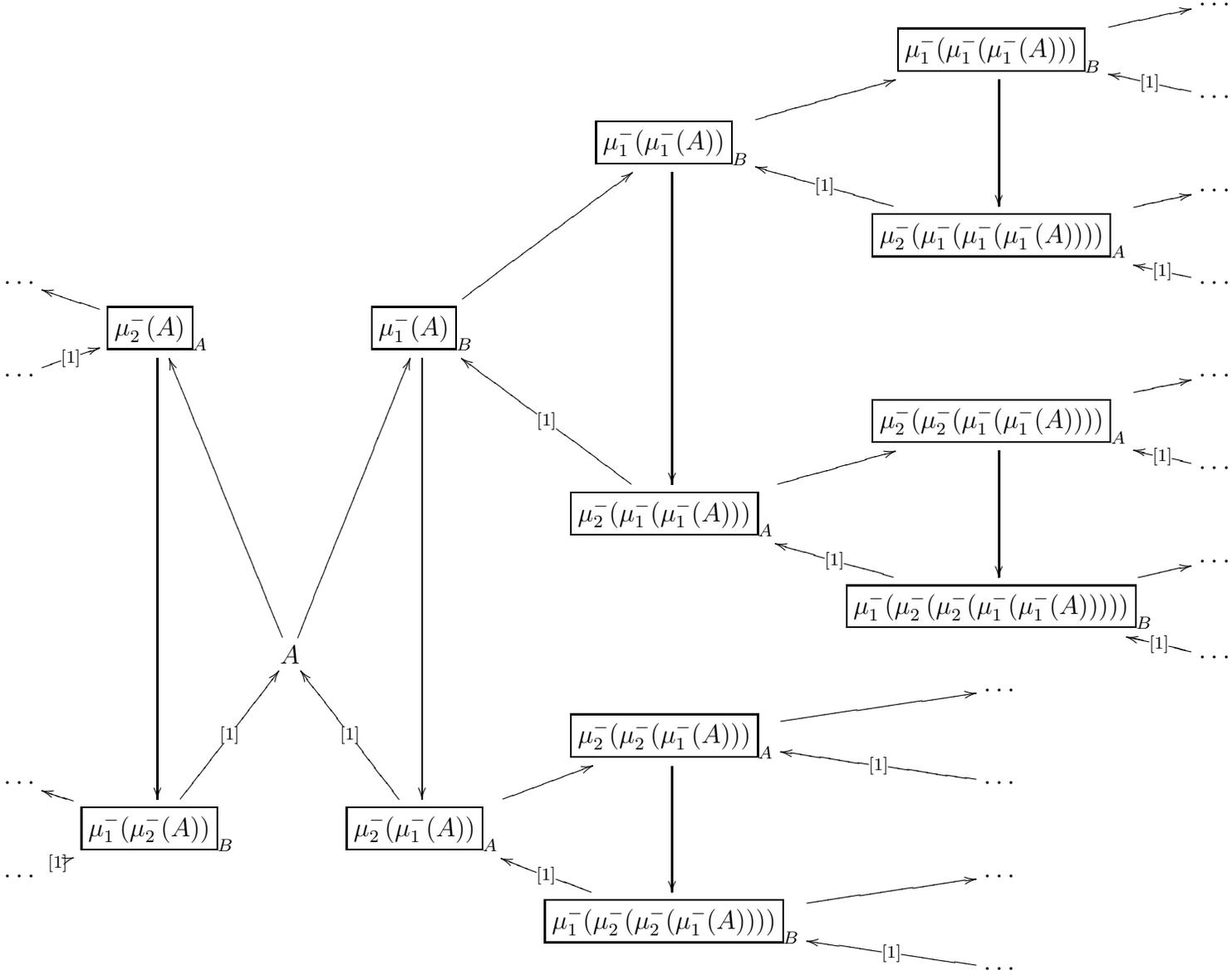
\begin{center}
\begin{adjustbox}{angle=-90}
$\xymatrix@C=0.5cm@R=0.05cm{
&&&&&&\cdots\\
&&&&&\boxed{\mu_1^-(\mu_1^-(\mu_1^-(A)))}_B\ar[ru]\ar[dddd]&\\
&&&&&&\cdots\ar[ul]|-{[1]}\\
&&&&\boxed{\mu_1^-(\mu_1^-(A))}_B\ar[dddddddd]\ar[uur]&&\\
&&&&&&\cdots \\
&&&&&\boxed{\mu_2^-(\mu_1^-(\mu_1^-(\mu_1^-(A))))}_A\ar[uul]|-{[1]}\ar[ur]&\\
\cdots&&&&&&\cdots\ar[ul]|-{[1]}\\
&\boxed{\mu_2^-(A)}_A\ar[ddddddddddd]\ar[ul]&&\boxed{\mu_1^-(A)}_B\ar[ddddddddddd]\ar[uuuur]&&&\\
\cdots \ar[ur]|-{[1]}&&&&&&\cdots\\
&&&&&\boxed{\mu_2^-(\mu_2^-(\mu_1^-(\mu_1^-(A))))}_A\ar[dddd]\ar[ur]&\\
&&&&&&\cdots \ar[ul]|-{[1]}\\
&&&&\boxed{\mu_2^-(\mu_1^-(\mu_1^-(A)))}_A\ar[uuuul]|-{[1]}\ar[uur]&&\\
&&&&&&\cdots\\
&&&&&\boxed{\mu_1^-(\mu_2^-(\mu_2^-(\mu_1^-(\mu_1^-(A)))))}_B\ar[uul]|-{[1]}\ar[ur]&\\
&&A\ar[uuuuuuur]\ar[uuuuuuul]&&&&\cdots \ar[ul]|-{[1]}\\
&&&&&\cdots&\\
&&&&\boxed{\mu_2^-(\mu_2^-(\mu_1^-(A)))}_A\ar[dddd]\ar[ur]&&\\
\cdots&&&&&\cdots \ar[ul]|-{[1]}&\\
&\boxed{\mu_1^-(\mu_2^-(A))}_B\ar[uuuur]|-{[1]}\ar[ul]&&\boxed{\mu_2^-(\mu_1^-(A))}_A\ar[uuuul]|-{[1]}\ar[uur]&&&\\
\cdots \ar[ur]|-{[1]}&&&&&\cdots&\\
&&&&\boxed{\mu_1^-(\mu_2^-(\mu_2^-(\mu_1^-(A))))}_B\ar[uul]|-{[1]}\ar[ur]&&\\
&&&&&\cdots \ar[ul]|-{[1]}&\\
&&&&&&}$
\end{adjustbox}
\end{center}
Here, the subscript $A$ or $B$ in the lower right corner of each vertex $T$ indicates the endomorphism algebra of $T$. 

\section{A connected quiver in affine type C}
Similar to the construction in \cite{ASW-rep-type}, we may construct a connected quiver whose vertex set is $\max^+(\Lambda)$. Let us start with the description of $\max^+(\Lambda)$, which was introduced in \cite{KOO}. Given a dominant weight $\Lambda\in \pcl$, we define 
$$
\pcl(\Lambda):=\{\Lambda'\in \pcl \mid \Lambda\sim \Lambda'\}.
$$
where the equivalence $\Lambda\sim \Lambda'$ was defined in subsection 2.1. In Proposition \ref{theorem-inverse-of-phi} below, we recall the bijection between $\pcl(\Lambda)$ and $\max^+(\Lambda)$.

\begin{definition}
For any $\Lambda=\sum_{i=0}^\ell m_i\Lambda_i\in \pcl$, we set 
$$
\ev(\Lambda):=m_1+m_3+\cdots+m_{2 \left \lfloor (\ell-1)/2 \right \rfloor+1}.
$$
\end{definition}

\begin{proposition}[{\cite[Theorem 2.14]{KOO}}]
$\pcl(\Lambda)=\{\Lambda'\in \pcl\mid \ev(\Lambda)-\ev(\Lambda')\in 2\Z\}$.
\end{proposition}

The distinguished representatives $\text{DR}(\pcl)=\pcl/\sim$ of the equivalence classes of $\pcl$ under $\sim$ are given in \cite[Table 2.2]{KOO}. 
It follows that we have either $\pcl(\Lambda)=\pcl(k\Lambda_0)$ or $\pcl(\Lambda)=\pcl((k-1)\Lambda_0+\Lambda_1)$, for any $\Lambda\in \pcl$.

\begin{example}
Set $k=2$, $\ell=4$. Then, 
\begin{center}
$P_{cl,2}^+(2\Lambda_0)=\{2\Lambda_0, 2\Lambda_1, 2\Lambda_2 , 2\Lambda_3, 2\Lambda_4, \Lambda_0+\Lambda_2, \Lambda_1+\Lambda_3, \Lambda_2+\Lambda_4, \Lambda_0+\Lambda_4\}$
\end{center} 
and 
\begin{center}
$P_{cl,2}^+(\Lambda_0+\Lambda_1)=\{\Lambda_0+\Lambda_1, \Lambda_1+\Lambda_2, \Lambda_2+\Lambda_3, \Lambda_3+\Lambda_4, \Lambda_0+\Lambda_3, \Lambda_1+\Lambda_4\}$.
\end{center}
\end{example}

For any $X=(x_0,x_1,\ldots,x_\ell)\in\Z_{\ge 0}^{\ell+1}$, we define 
$$
\min X:=\min\{x_i\mid 0\le i\le \ell\} \quad\text{and}\quad 
\max X:=\max\{x_i\mid 0\le i\le \ell\}.
$$

\begin{lemma}\label{system of linear equations}
Suppose that $Y=(y_0, y_1, \ldots, y_\ell)\in \Z^{\ell+1}$ satisfies
$$
y_0+y_1+\cdots+y_\ell=0 
\quad \text{and} \quad  
y_1+2y_2+\cdots+\ell y_\ell\in2\Z.
$$
There exists a unique solution $X=(x_0, x_1, \ldots,x_\ell)\in \Z^{\ell+1}$ of $\mathsf{A} X^t=Y^t$, such that $\min\{x_0, x_1, \ldots, x_\ell \} \ge 0$ and $\min\{x_0-1,x_1-2,\ldots,x_{\ell-1}-2,x_\ell-1\}<0$.
\end{lemma}
\begin{proof}
We define $\widehat{X}=(\widehat{x}_0,\widehat{x}_1,\ldots,\widehat{x}_\ell)$ by 
\begin{gather*}
\widehat{x}_0=0,\quad 
\widehat{x}_1=-y_0, \quad 
\widehat{x}_2=-2y_0-y_1, \quad \ldots ,\\
\widehat{x}_{\ell-1}=-(\ell-1)y_0-(\ell-2)y_1-\cdots-2y_{\ell-3}-y_{\ell-2}, \\
2\widehat{x}_\ell=-\ell y_0-(\ell-1)y_1-\cdots-2y_{\ell-2}-y_{\ell-1}=y_1+2y_2+\cdots+\ell y_\ell.
\end{gather*}
It is obvious that $\widehat{X}\in\Z^{\ell+1}$. By our assumption, one may easily check that $\mathsf{A}\widehat{X}^t=Y^t$.
Thus, the set of integral solutions of $\mathsf{A}X^t=Y^t$ is $\widehat{X}+\Z(1,2,\ldots,2,1)$.
We may adjust $m\in\Z$ in $\widehat{X}+m(1,2,\ldots,2,1)$ to obtain the desired solution.
It is also clear that such a solution is unique.
\end{proof}

\begin{definition}
For any $\Lambda\in P$, the hub of $\Lambda$ is defined to be
$$
\hub(\Lambda):=\left (\langle \alpha_0^\vee,\Lambda\rangle,\langle \alpha_1^\vee,\Lambda\rangle,\ldots,\langle \alpha_\ell^\vee, \Lambda\rangle\right ).
$$
In particular, if $\Lambda=\sum_{i=0}^\ell m_i\Lambda_i\in \pcl$, then $\hub(\Lambda)=(m_0, m_1, \ldots, m_\ell)$.
\end{definition}

Fix $\Lambda=\sum_{i=0}^\ell m_i\Lambda_i\in \pcl$ and $\Lambda'=\sum_{i=0}^\ell n_i\Lambda_i\in \pcl(\Lambda)$. 
We define
$$
Y^\Lambda_{\Lambda'}=(y_0,y_1,\ldots,y_\ell):=\hub(\Lambda)-\hub(\Lambda').
$$
Then, 
$$
y_0+y_1+\cdots+y_\ell=\ssum_{i=0}^\ell m_i -\ssum_{i=0}^\ell n_i=k-k=0,
$$
and $\ev(\Lambda)-\ev(\Lambda')\in 2\Z$ implies
$$
y_1+2y_2+\cdots+\ell y_\ell\in \ev(\Lambda)-\ev(\Lambda')+2\Z \subseteq 2\Z.
$$
Hence, we may apply Lemma \ref{system of linear equations}. Using the unique solution $X^\Lambda_{\Lambda'}:=(x_0, x_1, \ldots, x_\ell)$ in Lemma \ref{system of linear equations}, we define
$$
\beta^\Lambda_{\Lambda'}:=\ssum_{i=0}^\ell x_i\alpha_i\in Q_+.
$$
If there is no confusion of $\Lambda$, we will simply write $X_{\Lambda'}$, $Y_{\Lambda'}$ and $\beta_{\Lambda'}$ for $X^\Lambda_{\Lambda'}$, $Y_{\Lambda'}^\Lambda$ and $\beta_{\Lambda'}^\Lambda$, respectively. Now, we are able to explain the bijection between $\pcl(\Lambda)$ and $\max^+(\Lambda)$. 

\begin{proposition}\label{theorem-inverse-of-phi}
Let $\Lambda\in \pcl$. Then, the correspondence $\Lambda'\in \pcl(\Lambda)\mapsto \Lambda-\beta^\Lambda_{\Lambda'}\in \Lambda-Q_+$ gives a bijection between $\pcl(\Lambda)$ and $\max^+(\Lambda)$.
\end{proposition}
\begin{proof}
Since $P=\Z\Lambda_0\oplus\Z\Lambda_1\oplus\cdots\oplus\Z\Lambda_\ell\oplus\Z\delta$, we may write
$$
\Lambda-\beta^\Lambda_{\Lambda'}=\ssum_{i=0}^\ell n_i\Lambda_i+n\delta,
$$
for some $n_0,n_1,\cdots,n_\ell,n\in\Z$. We have $\langle \alpha_i^\vee,\Lambda\rangle-n_i=\langle \alpha_i^\vee, \beta^\Lambda_{\Lambda'}\rangle$.
On the other hand,
\begin{equation}\label{equ:lambda-beta}
\langle \alpha_i^\vee,\Lambda\rangle-\langle \alpha_i^\vee,\Lambda'\rangle=\ssum_{j=0}^\ell \langle \alpha_i^\vee,\alpha_j\rangle x_j=\langle \alpha_i^\vee, \beta^\Lambda_{\Lambda'}\rangle
\end{equation}
by the definition of $\beta^{\Lambda}_{\Lambda'}$. 
Hence, $n_i=\langle \alpha_i^\vee,\Lambda'\rangle$ for $0\le i\le \ell$, and they are nonnegative integers due to $\Lambda'\in \pcl(\Lambda)$. 
Therefore,
$\langle \alpha_i^\vee, \Lambda-\beta^\Lambda_{\Lambda'}\rangle \ge 0$ for $0\le i\le \ell$, and
$$
\Lambda-\beta_{\Lambda'}^\Lambda\in P^+\cap(\Lambda-Q_+)\subseteq P(\Lambda).
$$
By the minimality of the solution $X^\Lambda_{\Lambda'}\in\Z^{\ell+1}$, we also have $\Lambda-\beta_{\Lambda'}^\Lambda+\delta\not\in\Lambda-Q_+$. We have proved that the correspondence defines a map from $\pcl(\Lambda)$ to $\max^+(\Lambda)$.

Suppose $\Lambda-\sum_{j=0}^\ell x_j\alpha_j\in\max^+(\Lambda)$. In particular, $x_j$'s are
nonnegative integers for $0\le j\le \ell$. We may write
$$
\Lambda-\ssum_{j=0}^\ell x_j\alpha_j=\ssum_{i=0}^\ell m_i\Lambda_i+n\delta,
$$
for some $m_0, m_1, \ldots, m_\ell, n\in\Z$ as before. We set $\Lambda'=\sum_{i=0}^\ell m_i\Lambda_i$. Then,
$$
m_i=\langle \alpha_i^\vee, \Lambda'\rangle=\langle \alpha_i^\vee, \Lambda'+n\delta\rangle=\langle \alpha_i^\vee, \Lambda\rangle-\ssum_{j=0}^\ell \langle \alpha_i^\vee, \alpha_j\rangle x_j.
$$
This implies that $X=(x_0, x_1, \ldots, x_\ell)\in \Z_{\ge 0}^{\ell+1}$ is a solution of $\mathsf{A}X^t=Y^t$ for
$Y=\hub(\Lambda)-\hub(\Lambda')$. Since $\Lambda'+n\delta\in \max^+(\Lambda)$ is a dominant integral weight, we have $m_i\ge0$ for $0\le i\le \ell$. Moreover, $(1,1, \ldots,1)\mathsf{A}=(0,0,\ldots,0)$ implies
$$
\langle c,\Lambda'\rangle=\ssum_{i=0}^\ell m_i=\langle c,\Lambda\rangle-\ssum_{i,j=0}^\ell \langle \alpha_i^\vee, \alpha_j\rangle x_j=\langle c,\Lambda\rangle-(1,1,\ldots,1)\mathsf{A}X^t=k.
$$
Hence, $\Lambda'$ belongs to $\pcl$.
By the maximality of $\Lambda-\sum_{j=0}^\ell x_j\alpha_j$, $X$ is the unique solution of $\mathsf{A}X^t=Y^t$ in the sense of Lemma \ref{system of linear equations}. We conclude that
$\sum_{j=0}^\ell x_j\alpha_j=\beta^\Lambda_{\Lambda'}$.
Therefore, the map $\pcl(\Lambda)\rightarrow \max^+(\Lambda)$ is surjective.

If we have the same solution $X\in\Z_{\ge0}^{\ell+1}$ for
$$
Y'=\hub(\Lambda)-\hub(\Lambda') \quad \text{and} \quad Y''=\hub(\Lambda)-\hub(\Lambda''),
$$
then $Y'=X\mathsf{A}^t=Y''$.
Thus, the map $\pcl(\Lambda)\rightarrow \max^+(\Lambda)$ is injective.
\end{proof}

We have the following corollary immediately, and we leave the proof to readers.
\begin{corollary}\label{lem::embedding-Lambda}
Suppose $\Lambda=\bar\Lambda+\tilde \Lambda$ with $\Lambda\in \pcl$, $\bar\Lambda\in P_{cl,k'}^+$ and $\tilde \Lambda\in P_{cl,k-k'}^+$. Then,
$$
P_{cl,k'}^+(\bar\Lambda)+\tilde \Lambda\subset \pcl(\Lambda)
\quad\text{and}\quad
\beta^{\bar\Lambda}_{\Lambda'}=\beta^\Lambda_{\Lambda'+\tilde \Lambda}
$$ 
for any $\Lambda'\in P_{cl,k'}^+(\bar\Lambda)$.
\end{corollary}

Our task is to make $\max^+(\Lambda)$ into a connected quiver in such a way that if
there is an arrow $\Lambda'\rightarrow \Lambda''$ which corresponds to $\Lambda-\sum_{i=0}^\ell x_{i}'\alpha_i$ and $\Lambda-\sum_{i=0}^\ell x_{i}''\alpha_i$, there is a sequence of simple coroots $\alpha_{i_1}^\vee, \alpha_{i_1}^\vee, \ldots, \alpha_{i_s}^\vee$ such that 
$$
\left < \alpha_{i_t}^\vee, \Lambda-\ssum_{i=0}^\ell x_{i}'\alpha_i - \alpha_{i_1}-\alpha_{i_2}-\cdots-\alpha_{i_{t-1}} \right >\ge 1,
$$
and $\sum_{i=0}^\ell x_{i}'\alpha_i +\alpha_{i_1}+\alpha_{i_2}+\cdots+\alpha_{i_s}=\sum_{i=0}^\ell x_{i}''\alpha_i$, for $1\le t\le s$.

\subsection{A connected graph of $\max^+(\Lambda)$}
Fix $\Lambda\in \pcl$.
Suppose $\Lambda'=\Lambda_i+\tilde{\Lambda}\in \pcl(\Lambda)$ for some $i\in I$ and $\tilde{\Lambda}\in P_{cl, k-1}^+$, we define 
$$
\begin{aligned}
\Lambda'_{i^+}&:= \Lambda_{i+2}+\tilde\Lambda \quad\text{if}\quad  0\le i\le \ell-2,\\
\Lambda'_{i^-}&:=\Lambda_{i-2}+\tilde\Lambda \quad\text{if}\quad 2\le i\le \ell.
\end{aligned}
$$

Suppose $\Lambda'=\Lambda_i+\Lambda_j+\tilde \Lambda\in \pcl(\Lambda)$ for some $i,j\in I$ and $\tilde{\Lambda}\in P_{cl,k-2}^+$, we define
$$
\Lambda'_{i^+,j^+}=\Lambda'_{j^+,i^+}:=\Lambda_{i+1}+\Lambda_{j+1}+\tilde{\Lambda}
$$
if $0\le i\le j\le\ell-1$, and
$$
\Lambda'_{i^-,j^-}=\Lambda'_{j^-,i^-}:=\Lambda_{i-1}+\Lambda_{j-1}+\tilde{\Lambda}
$$
if $1\le i\le j\le\ell$.

Suppose $\Lambda'=\Lambda_i+\Lambda_j+\tilde \Lambda\in \pcl(\Lambda)$ for some $i, j\in I$ and $\tilde\Lambda\in P_{cl,k-2}^+$, we define
$$
\Lambda'_{i^-,j^+}=\Lambda'_{j^+,i^-}:=\Lambda_{i-1}+\Lambda_{j+1}+\tilde \Lambda
$$
if $i\ne 0$, $j\ne \ell$, $i-1\ne j$. 

Note that $\Lambda'_{i^+,(i+1)^+}=\Lambda'_{i^+}$ for $0\le i\le \ell-2$ and $\Lambda'_{i^-,(i+1)^-}=\Lambda'_{(i+1)^-}$ for $1\le i\le \ell-1$. 
It is obvious that $\Lambda'_{i^\pm}$, $\Lambda'_{i^{\pm},j^\pm}$, $\Lambda'_{i^\pm,j^\mp}\in \pcl(\Lambda)$.

\begin{definition}\label{def-undireted-graph}
Fix $\Lambda\in \pcl$.
Let $C(\Lambda)$ be the undirected graph with vertex set $\pcl(\Lambda)$, such that an edge between $\Lambda'$ and $\Lambda''$ exists if $\Lambda''=\Lambda'_{i^\pm}$ or $\Lambda'_{i^\pm,j^\pm}$ or $\Lambda'_{i^-,j^+}$.
\end{definition}

\begin{example}\label{example-quiver-k=1-2}
Set $k=2$, $\ell=4$. The graphs $C(2\Lambda_2)$ and $C(\Lambda_1+\Lambda_2)$ are displayed as 
\begin{center}
\scalebox{0.8}{$\vcenter{\xymatrix@C=1.5cm@R=1.4cm{
*+[F]{2\Lambda_{0}}
& 
& 
\\
*+[F]{2\Lambda_{1}}\ar@{-}[r]\ar@{-}[u]
&*+[F]{\Lambda_{0}+\Lambda_{2}}\ar@{-}[lu]
& 
\\
*+[F]{2\Lambda_2}\ar@{-}[r]\ar@{-}[dr]\ar@{-}[ur]\ar@{-}[d]\ar@{-}[u]
&*+[F]{\Lambda_{1}+\Lambda_{3}}
\ar@{-}[r]\ar@{-}[d]\ar@{-}[u]\ar@{-}[lu]\ar@{-}[ld]
&*+[F]{\Lambda_{0}+\Lambda_{4}}\ar@{-}[lu]\ar@{-}[ld]
\\
*+[F]{2\Lambda_{3}}\ar@{-}[r]\ar@{-}[d]
&*+[F]{\Lambda_{2}+\Lambda_{4}}\ar@{-}[ld]
&  
\\
*+[F]{2\Lambda_{4}}
& 
& 
}}$}$\qquad \text{and} \qquad$
\scalebox{0.8}{$\vcenter{\xymatrix@C=2.5cm@R=2cm{
*+[F]{\Lambda_{0}+\Lambda_{1}}
&
\\
*+[F]{\Lambda_1+\Lambda_2}\ar@{-}[rd]
\ar@{-}[r]\ar@{-}[d]\ar@{-}[u]
&*+[F]{\Lambda_{0}+\Lambda_{3}}\ar@{-}[lu]\ar@{-}[ld]\ar@{-}[d]
\\
*+[F]{\Lambda_{2}+\Lambda_{3}}\ar@{-}[r]\ar@{-}[d]
&*+[F]{\Lambda_{1}+\Lambda_{4}}\ar@{-}[ld]
\\
*+[F]{\Lambda_{3}+\Lambda_{4}}
&
}}$}
\end{center}
respectively.
\end{example}

\begin{lemma}
For any $\Lambda',\Lambda''\in \pcl(\Lambda)$, there exists an undirected path from $\Lambda'$ to $\Lambda''$ in $C(\Lambda)$.
In particular, $C(\Lambda)$ is a finite connected graph.
\end{lemma}
\begin{proof}
It suffices to consider $\Lambda\in \text{DR}(\pcl)=\{k\Lambda_0, (k-1)\Lambda_0+\Lambda_1\}$.
If $k=1$, then the assertion is obviously true by level one case, as we will mention in 
Subsection 3.3.
Suppose $k\ge 2$. We show that there is an undirected path from $\Lambda$ to $\Lambda'$, for any $\Lambda'\in \pcl(\Lambda)$. 

Set $\Lambda'=\sum_{i\in I}m_i\Lambda_i\in \pcl(\Lambda)$. If $m_0=k$, then $\Lambda'=\Lambda$ and the assertion is trivial. 

If $m_0=k-1$, then $\Lambda'=(k-1)\Lambda_0+\Lambda_i$ for some $i\ne 0$.  
For $i\equiv_20$ (i.e., $\Lambda=k\Lambda_0$), we have an undirected path 
$$
\xymatrix@C=1.5cm@R=1cm{
*+[F]{k\Lambda_0}\ar@{-}[r]&*+[F]{(k-1)\Lambda_0+\Lambda_{2}}\ar@{-}[r]&\cdots
\ar@{-}[r]&*+[F]{(k-1)\Lambda_0+\Lambda_{i}}}.
$$
For $i\equiv_2 1$ (i.e., $\Lambda=(k-1)\Lambda_0+\Lambda_1$), we have an undirected path
$$
\xymatrix@C=1.5cm@R=1cm{
*+[F]{(k-1)\Lambda_0+\Lambda_1}\ar@{-}[r]&*+[F]{(k-1)\Lambda_0+\Lambda_{3}}\ar@{-}[r]&\cdots
\ar@{-}[r]&*+[F]{(k-1)\Lambda_0+\Lambda_{i}}}.
$$

Suppose $m_0\le k-2$. Then, $\Lambda'=\Lambda_i+\Lambda_j+\tilde\Lambda$ for some $i\le j\in I$.
If $i\equiv_2 0$ or $j\equiv_2 0$, then there is an undirected path from $\Lambda_0$ to $\Lambda_i$ or $\Lambda_j$; this yields an undirected path from $\Lambda_0+\Lambda_j+\tilde \Lambda$ or $\Lambda_0+\Lambda_i+\tilde\Lambda$ to $\Lambda'$. By the induction hypothesis on $k-m_0$, we have
an undirected path from $\Lambda$ to $\Lambda_0+\Lambda_j+\tilde \Lambda$ and  $\Lambda_0+\Lambda_i+\tilde\Lambda$, so that there is an undirected path from $\Lambda$ to $\Lambda'$. If $i\equiv_2 j\equiv_2 1$, then $j-i\equiv_2 0$ and there is an undirected path
$$
\xymatrix@C=1.5cm@R=1cm{
*+[F]{2\Lambda_i}\ar@{-}[r]&*+[F]{\Lambda_i+\Lambda_{i+2}}\ar@{-}[r]&\cdots
\ar@{-}[r]&*+[F]{\Lambda_i+\Lambda_j}}.
$$
Hence, we have an undirected path from $2\Lambda_0$ to $\Lambda_i+\Lambda_j$; this yields an undirected path from $2\Lambda_0+\tilde \Lambda$ to $\Lambda'$. By the induction hypothesis on $k-m_0$, we have an undirected path from $\Lambda$ to $\Lambda'$.
\end{proof}

In order to attach a direction to each edge in $C(\Lambda)$, we compare $X_{\Lambda'}$ and $X_{\Lambda''}$ if there is an edge between $\Lambda'$ and $\Lambda''$, i.e., $\Lambda''=\Lambda'_{i^\pm}$ or $\Lambda'_{i^-,j^+}$
or $\Lambda'_{i^\pm,j^\pm}$. To simplify the notation, we will also denote $\delta=(1,2,2,\ldots,2,1)\in \Z^{\ell+1}$ if there is no confusion in the context. 

For $0\le i\le \ell-2$ and $2\le j\le \ell$, we define 
$$
\Delta_{i^+}=(1,2^{i},1,0^{\ell-i-1}) \in \Z^{\ell+1}, \quad 
\Delta_{j^-}=(0^{j-1},1,2^{\ell-j},1) \in \Z^{\ell+1}. 
$$
Then, we have 
\begin{equation}\label{equ-deltaipm}
\delta-\Delta_{i^+}=\Delta_{(i+2)^-}.
\end{equation}

\begin{lemma}\label{lemma:arrow-i+}
Suppose $\Lambda'=\Lambda_i+\tilde\Lambda\in \pcl(\Lambda)$ for some $0\le i\le \ell-2$ and $\tilde\Lambda\in P_{cl,k-1}^+$. Set $\Lambda'':=\Lambda'_{i^+}$. Then, $\Lambda''_{(i+2)^-}=\Lambda'$ and one of the following holds.
\begin{enumerate}
\item If $\min(X_{\Lambda'}+\Delta_{i^+}-\delta)<0$, then $X_{\Lambda''}=X_{\Lambda'}+\Delta_{i^+}$ and $\min(X_{\Lambda''}+\Delta_{(i+2)^-}-\delta)\ge 0$,
\item If $\min(X_{\Lambda'}+\Delta_{i^+}-\delta)\ge 0$, then $X_{\Lambda''}=X_{\Lambda'}-\Delta_{(i+2)^-}$ and $\min(X_{\Lambda''}+\Delta_{(i+2)^-}-\delta)< 0$.
\end{enumerate}
\end{lemma}
\begin{proof}
We have proved in Lemma \ref{system of linear equations} that $X_{\Lambda'}$ is the unique solution of $\mathsf{A}X^t=Y_{\Lambda'}^t$, satisfying $X_{\Lambda'}\in \Z_{\ge 0}^{\ell+1}$ and $\min(X_{\Lambda'}-\delta)<0$. 
We then find
$$
\mathsf{A}X^t_{\Lambda''}-\mathsf{A}X_{\Lambda'}^t
=Y^t_{\Lambda''}-Y^t_{\Lambda'}
=(0^{i},1,0,-1,0^{\ell-i-2})^t
=\mathsf{A}\Delta_{i^+}^t.
$$
This gives $\mathsf{A}X^t_{\Lambda''}=\mathsf{A}(X^t_{\Lambda'}+\Delta^t_{i^+})$. It is obvious that $X_{\Lambda'}+\Delta_{i^+}\in \Z_{\ge 0}^{\ell+1}$. If $\min(X_{\Lambda'}+\Delta_{i^+}-\delta)<0$, then $X_{\Lambda''}=X_{\Lambda'}+\Delta_{i^+}$ by the uniqueness of the solution, and $\min(X_{\Lambda''}+\Delta_{(i+2)^-}-\delta)=\min(X_{\Lambda'})\ge 0$ by \eqref{equ-deltaipm}.

Suppose $\min(X_{\Lambda'}+\Delta_{i^+}-\delta)\ge0$. 
Due to $\min(X_{\Lambda'}-\delta)<0$ and $\Delta_{i^+}-\delta\notin \Z_{\ge 0}^{\ell+1}$, we have 
$\min(X_{\Lambda'}+\Delta_{i^+}-2\delta)\le 
\min(X_{\Lambda'}-\delta)+\max(\Delta_{i^+}-\delta)<0$.
This implies 
$$
X_{\Lambda''}=X_{\Lambda'}+\Delta_{i^+}-\delta=X_{\Lambda'}-\Delta_{(i+2)^-}
$$
by the uniqueness of the solution, and $\min( X_{\Lambda''}+\Delta_{(i+2)^-}-\delta)=\min(X_{\Lambda'}-\delta)<0$.
\end{proof}

For any $0\le i\le j\le\ell-1$ and $1\le s\le t\le \ell$, we define two vectors in $\Z^{\ell+1}$ as 
$$
\Delta_{i^+,j^+}=\Delta_{j^+,i^+}=(1,2^{i},1^{j-i},0^{\ell-j}), \quad 
\Delta_{s^-,t^-}=\Delta_{t^-,s^-}=(0^s,1^{t-s},2^{\ell-t},1).
$$
It turns out that $\delta-\Delta_{i^+,j^+}=\Delta_{(i+1)^-,(j+1)^-}$.

\begin{lemma}\label{lemma：arrow-i+j+}
Suppose $\Lambda'=\Lambda_i+\Lambda_j+\tilde\Lambda\in \pcl(\Lambda)$ for some $0\le i\le j\le \ell-1$ and $\tilde\Lambda\in P_{cl,k-2}^+$. Set $\Lambda'':= \Lambda'_{i^+,j^+}$. 
Then, $\Lambda''_{(i+1)^-,(j+1)^-}=\Lambda'$ and one of the following holds.
\begin{enumerate}
\item If $\min(X_{\Lambda'}+\Delta_{i^+,j^+}-\delta)<0$, then $X_{\Lambda''}=X_{\Lambda'}+\Delta_{i^+,j^+}$ and $\min(X_{\Lambda''}+\Delta_{(i+1)^-,(j+1)^-}-\delta)\ge 0$.

\item If $\min(X_{\Lambda'}+\Delta_{i^+,j^+}-\delta)\ge 0$, then $X_{\Lambda''}=X_{\Lambda'}-\Delta_{(i+1)^-,(j+1)^-}$ and $\min(X_{\Lambda''}+\Delta_{(i+1)^-,(j+1)^-}-\delta)< 0$.
\end{enumerate}
\end{lemma}
\begin{proof}
Since $Y_{\Lambda''}-Y_{\Lambda'}=(0^i,1,-1,0^{\ell-i-1})+(0^j,1,-1,0^{\ell-j-1})$ and 
$$
\mathsf{A}(0^{i+1},1^{\ell-i-1},1/2)^t=(0^i,-1,1,0^{\ell-i-1})^t,
$$
we obtain 
$$
\begin{aligned}
X_{\Lambda''}-X_{\Lambda'}&\in -(0^{i+1},1^{\ell-i-1},1/2)-(0^{j+1},1^{\ell-j-1},1/2)+\Z\delta\\
&=-\Delta_{(i+1)^-,(j+1)^-}+\Z\delta=\Delta_{i^+,j^+}+\Z\delta.
\end{aligned}
$$
Then, the proof is similar to that of Lemma~\ref{lemma:arrow-i+}.
\end{proof}

For any $0\le i, j\le\ell$ with $i\ne 0, j\ne \ell$ with $i-1\ne j$, we define two vectors in $\Z^{\ell+1}$ as
$$
\Delta_{i^-,j^+}=\Delta_{j^+,i^-}:=
\left\{\begin{array}{ll}
(0^i,1^{j-i+1},0^{\ell-j}) & \text{if } i\le j, \\
(1,2^j,1^{i-j-1},2^{\ell-i},1)& \text{if } i\ge j+2.
\end{array}\right.
$$
It gives that $\delta -\Delta_{i^-,j^+}=\Delta_{(j+1)^-,(i-1)^+}$.

\begin{lemma}\label{lemma：arrow-i-j+}
Suppose $\Lambda'=\Lambda_i+\Lambda_j+\tilde \Lambda\in \pcl(\Lambda)$ for some $0\le i, j\le\ell$ satisfying $i\ne 0$, $j\ne \ell$, $i-1\ne j$ and $\tilde\Lambda\in P_{cl,k-2}^+$. Set $\Lambda''=\Lambda'_{i^-,j^+}$. Then, $\Lambda''_{(j+1)^-,(i-1)^+}=\Lambda'$ and one of the following holds.
\begin{enumerate}
\item If $\min(X_{\Lambda'}+\Delta_{i^-,j^+}-\delta)<0$, then $X_{\Lambda''}=X_{\Lambda'}+\Delta_{i^-,j^+}$ and $\min(X_{\Lambda''}+\Delta_{(j+1)^-,(i-1)^+}-\delta)\ge 0$.

\item If $\min(X_{\Lambda'}+\Delta_{i^-,j^+}-\delta)\ge 0$, then $X_{\Lambda''}=X_{\Lambda'}-\Delta_{(j+1)^-,(i-1)^+}$ and $\min(X_{\Lambda''}+\Delta_{(j+1)^-,(i-1)^+}-\delta)< 0$.
\end{enumerate}
\end{lemma}
\begin{proof}
Similar to the proof of Lemma \ref{lemma：arrow-i+j+}, we obtain 
$$
X_{\Lambda''}-X_{\Lambda'}\in (0^i,1^{\ell-i},1/2)-(0^{j+1},1^{\ell-j-1},1/2)+\Z\delta=\Delta_{i^-,j^+}+\Z\delta.
$$
We omit the details.
\end{proof}

One may also find the relation between $X_{\Lambda'}$ and $X_{\Lambda''}$ if $\Lambda''=\Lambda'_{i^-}$ or $\Lambda'_{i^+,j^-}$
or $\Lambda'_{i^-,j^-}$. We list the corresponding lemmas below and leave the proofs to readers.

\begin{lemma}
Suppose $\Lambda'=\Lambda_i+\tilde \Lambda\in \pcl(\Lambda)$ for some $2\le i\le \ell$ and $\tilde\Lambda\in P_{cl,k-1}^+$. Then,
\begin{enumerate}
\item $X_{\Lambda'_{i^-}}=X_{\Lambda'}+\Delta_{i^-}$, if $\min(X_{\Lambda'}+\Delta_{i^-}-\delta)<0$.

\item $X_{\Lambda'_{i^-}}=X_{\Lambda'}-\Delta_{i-2}$, if $\min(X_{\Lambda'}+\Delta_{i^-}-\delta)\ge 0$.
\end{enumerate}
\end{lemma}

\begin{lemma}
Suppose $\Lambda'=\Lambda_i+\Lambda_j+\tilde \Lambda\in \pcl(\Lambda)$ with $1\le i\le j\le \ell$, $\tilde\Lambda\in P_{cl,k-2}^+$. Then,
\begin{enumerate}
\item $X_{\Lambda'_{i^-,j^-}}=X_{\Lambda'}+\Delta_{i^-,j^-}$, if $\min(X_{\Lambda'}+\Delta_{i^-,j^-}-\delta)<0$.

\item $X_{\Lambda'_{i^-,j^-}}=X_{\Lambda'}-\Delta_{(i-1)^+,(j-1)^+}$, if $\min(X_{\Lambda'}+\Delta_{i^-,j^-}-\delta)\ge 0$.
\end{enumerate}
\end{lemma}

The following lemma is a restatement of  Lemma~\ref{lemma：arrow-i-j+}, if we observe that $\Lambda_{i^+,j^-}=\Lambda_{j^-,i^+}$ and $\Delta_{i^-,j^+}=\Delta_{j^+,i^-}$. 
\begin{lemma}
Suppose $\Lambda'=\Lambda_i+\Lambda_j+\tilde \Lambda\in \pcl(\Lambda)$ for $0\le i,j\le \ell$ satisfying $i\ne \ell$, $j\ne 0$, $j-1\ne i$ and $\tilde\Lambda\in P_{cl,k-2}^+$.
Then,
\begin{enumerate}
\item $X_{\Lambda'_{i^+,j^-}}=X_{\Lambda'}+\Delta_{i^+,j^-}$, if $\min(X_{\Lambda'}+\Delta_{i^+,j^-}-\delta)<0$.

\item $X_{\Lambda'_{i^+,j^-}}=X_{\Lambda'}-\Delta_{(i+1)^-,(j-1)^+}$, if $\min(X_{\Lambda'}+\Delta_{i^+,j^-}-\delta)\ge 0$.
\end{enumerate}
\end{lemma}

For any $\Lambda'\in \pcl(\Lambda)$, we set $|X_{\Lambda'}|:=|\beta_{\Lambda'}|$, i.e., $|X_{\Lambda'}|=\sum_{i\in I}x_i$
if $X_{\Lambda'}=(x_0,x_1,\ldots,x_\ell)$.
According to the above lemmas, we have either $|X_{\Lambda'}|>|X_{\Lambda''}|$ or $|X_{\Lambda'}|<|X_{\Lambda''}|$ if there is an edge between $\Lambda'$ and $\Lambda''$. This leads to the following definition.

\subsection{A connected quiver of $\max^+(\Lambda)$}
Fix $\Lambda\in \pcl$.

\begin{definition}\label{Def:quiver-of-maximal-dominant}
We define $\vec C(\Lambda)$ to be the quiver having $C(\Lambda)$ as its underlying graph, and the orientation of an edge $\xymatrix@C=0.7cm{\Lambda'\ar@{-}[r] &\Lambda''}\in C(\Lambda)$ is given as $\xymatrix@C=0.7cm{\Lambda'\ar[r] &\Lambda''}$ if $|X_{\Lambda''}|>|X_{\Lambda'}|$, or equivalently, $\beta_{\Lambda''}-\beta_{\Lambda'}\in Q_+$.  
\end{definition}

It is clear that the choice of the orientation of $\xymatrix@C=0.5cm{\Lambda'\ar@{-}[r] &\Lambda''}$ is always possible and unique. 
We may explain the details of drawing arrows in $\vec C(\Lambda)$ as follows.

Fix $\Lambda'\in \pcl(\Lambda)$. 
We draw an arrow $\xymatrix@C=0.7cm{\Lambda'\ar[r]^{\Delta} &\Lambda''}$ if $\min(X_{\Lambda'}+\Delta-\delta)<0$, and then $X_{\Lambda''}=X_{\Lambda'}+\Delta$. According to the lemmas we have given in the previous subsection, there are only 5 choices for $\Delta$, as listed below.
\begin{enumerate}
\item For $0\le i\le \ell-2$ with $\langle \alpha_i^\vee, \Lambda' \rangle\ge 1$, we set $\Lambda'':=\Lambda'_{i^+}$ and 
$$
\Delta:=\Delta_{i^+}=(1,2^{i},1,0^{\ell-i-1}).
$$

\item For $2\le i\le \ell$ with $\langle \alpha_i^\vee, \Lambda' \rangle\ge 1$, we set $\Lambda'':=\Lambda'_{i^-}$ and 
$$
\Delta:=\Delta_{i^-}=(0^{i-1}, 1,2^{\ell-i},1).
$$

\item For $0\le i\le j\le \ell-1$ with $i+1\ne j$, $\langle \alpha_i^\vee, \Lambda' \rangle\ge 1$, $\langle \alpha_j^\vee, \Lambda' \rangle\ge 1$, we set $\Lambda'':=\Lambda'_{i^+, j^+}=\Lambda'_{j^+, i^+}$ and 
$$
\Delta:=\Delta_{i^+, j^+}=\Delta_{j^+,i^+}=(1,2^{i},1^{j-i},0^{\ell-j}).
$$
If $i+1=j$, then $\Lambda'_{i^+, (i+1)^+}=\Lambda'_{i^+}$ and this coincides with case (1).

\item For $1\le i\le j\le \ell$ with $i+1\ne j$, $\langle \alpha_i^\vee, \Lambda' \rangle\ge 1$, $\langle \alpha_j^\vee, \Lambda' \rangle\ge 1$, we set $\Lambda'':=\Lambda'_{i^-, j^-}=\Lambda'_{j^-, i^-}$ and 
$$
\Delta:=\Delta_{i^-, j^-}=\Delta_{j^-,i^-}=(0^i,1^{j-i},2^{\ell-j},1).
$$
If $i+1=j$, then $\Lambda'_{(j-1)^-, j^-}=\Lambda'_{j^-}$ and this coincides with case (2).

\item For $0\le i,j\le \ell$ with $i\ne 0, j\ne \ell, i-1\ne j$, $\langle \alpha_i^\vee, \Lambda' \rangle\ge 1$, $\langle \alpha_j^\vee, \Lambda' \rangle\ge 1$, we set $\Lambda'':=\Lambda'_{i^-,j^+}=\Lambda'_{j^+, i^-}$ and 
$$
\Delta:=\Delta_{i^-,j^+}=\Delta_{j^+,i^-}=
\left\{\begin{array}{ll}
(0^i,1^{j-i+1},0^{\ell-j}) & \text{if } i\le j, \\
(1,2^j,1^{i-j-1},2^{\ell-i},1)& \text{if } i\ge  j+2.
\end{array}\right.
$$
\end{enumerate}
We remind the reader that it is still needed to check $\min(X_{\Lambda'}+\Delta-\delta)$ in each case.

\begin{example}\label{example-quiver-k=1-3}
Set $k=2$, $\ell=4$. The quiver $\vec C(2\Lambda_2)$ associated with $X_{\Lambda'}$ is displayed as
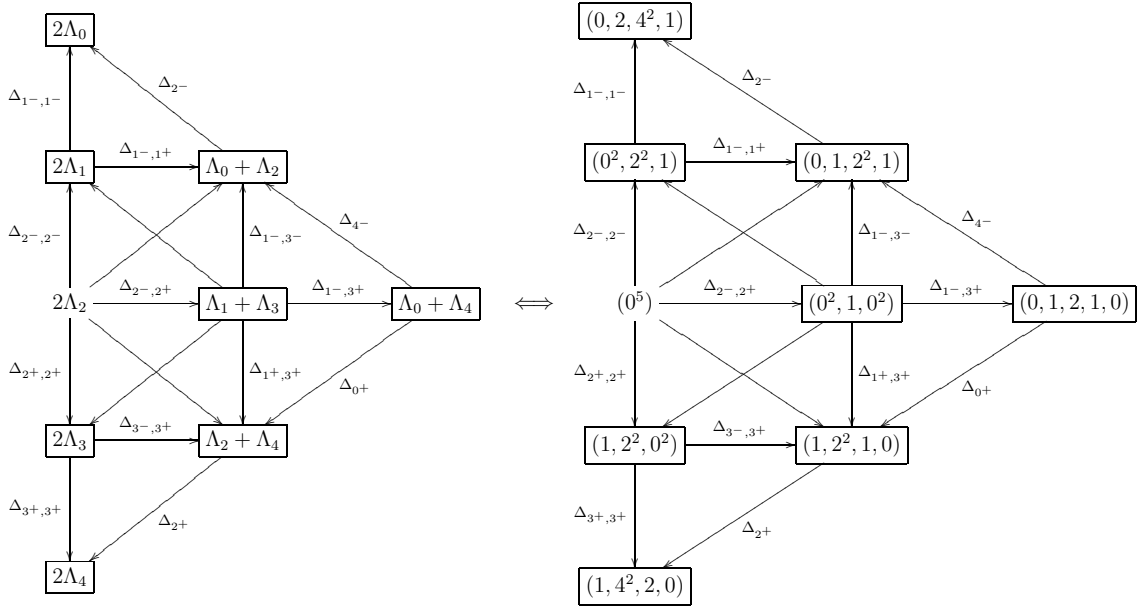
\begin{center}
\scalebox{0.7}{$\vcenter{\xymatrix@C=2cm@R=2cm{
*+[F]{2\Lambda_{0}}
& 
& 
\\
*+[F]{2\Lambda_{1}}\ar^-{\Delta_{1^-,1^+}}[r]\ar^-{\Delta_{1^-,1^-}}[u]
&*+[F]{\Lambda_{0}+\Lambda_{2}}\ar_-{\Delta_{2^-}}[lu]
& 
\\
2\Lambda_2\ar^-{\Delta_{2^-,2^+}}[r]\ar[dr]\ar[ur]\ar_-{\Delta_{2^+,2^+}}[d]\ar^-{\Delta_{2^-,2^-}}[u]
&*+[F]{\Lambda_{1}+\Lambda_{3}}
\ar^-{\Delta_{1^-,3^+}}[r]\ar^-{\Delta_{1^+,3^+}}[d]\ar_-{\Delta_{1^-,3^-}}[u]\ar[lu]\ar[ld]
&*+[F]{\Lambda_{0}+\Lambda_{4}}\ar_-{\Delta_{4^-}}[lu]\ar^-{\Delta_{0^+}}[ld]
\\
*+[F]{2\Lambda_{3}}\ar^-{\Delta_{3^-,3^+}}[r]\ar_-{\Delta_{3^+,3^+}}[d]
&*+[F]{\Lambda_{2}+\Lambda_{4}}\ar^-{\Delta_{2^+}}[ld]
&  
\\
*+[F]{2\Lambda_{4}}
& 
& 
}}\quad \iff 
\vcenter{\xymatrix@C=2cm@R=2cm{
*+[F]{(0,2,4^2,1)}
& 
& 
\\
*+[F]{(0^2,2^2,1)}\ar^-{\Delta_{1^-,1^+}}[r]\ar^-{\Delta_{1^-,1^-}}[u]
&*+[F]{(0, 1, 2^2, 1)}\ar_-{\Delta_{2^-}}[lu]
& 
\\
(0^5)\ar^-{\Delta_{2^-,2^+}}[r]\ar[dr]\ar[ur]\ar_-{\Delta_{2^+,2^+}}[d]\ar^-{\Delta_{2^-,2^-}}[u]
&*+[F]{(0^2, 1, 0^2)}
\ar^-{\Delta_{1^-,3^+}}[r]\ar^-{\Delta_{1^+,3^+}}[d]\ar_-{\Delta_{1^-,3^-}}[u]\ar[lu]\ar[ld]
&*+[F]{(0,1,2,1,0)}\ar_-{\Delta_{4^-}}[lu]\ar^-{\Delta_{0^+}}[ld]
\\
*+[F]{(1,2^2,0^2)}\ar^-{\Delta_{3^-,3^+}}[r]\ar_-{\Delta_{3^+,3^+}}[d]
&*+[F]{(1,2^2, 1,0)}\ar^-{\Delta_{2^+}}[ld]
&  
\\
*+[F]{(1,4^2,2,0)}
& 
& 
}}$}
\end{center}
Besides, the quiver $\vec C(\Lambda_1+\Lambda_2)$ associated with $X_{\Lambda'}$ is displayed as 
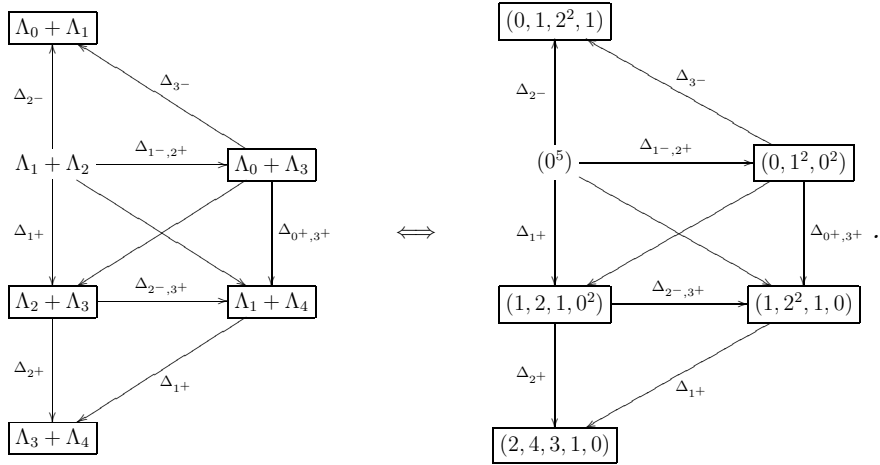
\begin{center}
\scalebox{0.7}{$\vcenter{\xymatrix@C=2.5cm@R=2cm{
*+[F]{\Lambda_{0}+\Lambda_{1}}
&
\\
\Lambda_1+\Lambda_2\ar[rd]
\ar^-{\Delta_{{1}^-,{2}^+}}[r]\ar_-{\Delta_{1^+}}[d]\ar^-{\Delta_{2^-}}[u]
&*+[F]{\Lambda_{0}+\Lambda_{3}}\ar_-{\Delta_{3^-}}[lu]\ar[ld]\ar^-{\Delta_{0^+,3^+}}[d]
\\
*+[F]{\Lambda_{2}+\Lambda_{3}}\ar^-{\Delta_{2^-,3^+}}[r]\ar_-{\Delta_{2^+}}[d]
&*+[F]{\Lambda_{1}+\Lambda_{4}}\ar^-{\Delta_{1^+}}[ld]
\\
*+[F]{\Lambda_{3}+\Lambda_{4}}
&
}}
\qquad \iff \qquad
\vcenter{\xymatrix@C=2.5cm@R=2cm{
*+[F]{(0, 1, 2^2, 1)}
&
\\
(0^5)\ar[rd]
\ar^-{\Delta_{{1}^-,{2}^+}}[r]\ar_-{\Delta_{1^+}}[d]\ar^-{\Delta_{2^-}}[u]
&*+[F]{(0, 1^2, 0^2)}\ar_-{\Delta_{3^-}}[lu]\ar[ld]\ar^-{\Delta_{0^+,3^+}}[d]
\\
*+[F]{(1, 2, 1, 0^2)}\ar^-{\Delta_{2^-,3^+}}[r]\ar_-{\Delta_{2^+}}[d]
&*+[F]{(1,2^2,1,0)}\ar^-{\Delta_{1^+}}[ld]
\\
*+[F]{(2,4,3,1,0)}
&
}}$}.
\end{center}
\end{example}

Recall that $\Delta_{\rm{re}}^+=\{\beta+m\delta \mid m\ge 0, \beta\in \Delta_{\rm{fin}}^+ \text{\ or \ } \delta-\Delta_{\rm{fin}}^+\}$ with
$$
\Delta_{\rm{fin}}^+=\{2\epsilon_i \mid 1\le i\le \ell\}
\sqcup 
\{\epsilon_i\pm \epsilon_j \mid 1\le i<j \le \ell\}.
$$
We call $\overline{\Delta_{\rm{re}}^+}:=\{\beta \in \Delta_{\rm{re}}^+ \mid \beta\in \Delta_{\rm{fin}}^+ \text{\ or \ } \delta-\Delta_{\rm{fin}}^+\}$ the first layer of $\Delta_{\rm{re}}^+$. If an arrow $\xymatrix@C=0.7cm{\Lambda'\ar[r]^{\Delta} &\Lambda''}$ defined in the above (1)-(5) exists (i.e., $\min(X_{\Lambda'}+\Delta-\delta)<0$), then $\Delta$ corresponds to a certain element in $\overline{\Delta_{\rm{re}}^+}$. We then observe that all arrows in $\vec C(\Lambda)$ are labeled by elements in $\overline{\Delta_{\rm{re}}^+}$.  
Let us check it case by case.
\begin{enumerate}
\item $\Delta=\Delta_{i^+}=(1,2^{i},1,0^{\ell-i-1})=\delta-(\epsilon_{i+1}+\epsilon_{i+2})$ for $0\le i\le \ell-2$. This gives
$$
\delta-\{\epsilon_{i}+\epsilon_{i+1}\mid 1\le i\le \ell-1\}\subseteq \overline{\Delta_{\rm{re}}^+}.
$$

\item $\Delta=\Delta_{i^-}=(0^{i-1}, 1,2^{\ell-i},1)=\epsilon_{i-1}+\epsilon_i$ for $2\le i\le \ell$. This gives
$$
\{\epsilon_{i}+\epsilon_{i+1}\mid 1\le i\le \ell-1\}\subseteq \overline{\Delta_{\rm{re}}^+}.
$$

\item $\Delta=\Delta_{i^+, j^+}=(1,2^{i},1^{j-i},0^{\ell-j})$ for $0\le i\le j\le \ell-1$ with $i+1\ne j$. This gives
$$
\delta-\{\epsilon_{i}+\epsilon_{j}\mid 1\le i\le j\le \ell-1, i+1\ne j\}\subseteq \overline{\Delta_{\rm{re}}^+}.
$$

\item $\Delta=\Delta_{i^-, j^-}=(0^i,1^{j-i},2^{\ell-j},1)$ for $1\le i\le j\le \ell$ with $i+1\ne j$. This gives
$$
\{\epsilon_{i}+\epsilon_{j}\mid 1\le i\le j\le \ell-1, i+1\ne j\}\subseteq \overline{\Delta_{\rm{re}}^+}.
$$

\item For $0\le i,j\le \ell$ with $i\ne 0, j\ne \ell, i-1\ne j$,
$$
\Delta=\Delta_{i^-,j^+}=
\left\{\begin{array}{ll}
(0^i,1^{j-i+1},0^{\ell-j}) =\epsilon_i-\epsilon_{j+1}& \text{if } i\le j, \\
(1,2^j,1^{i-j-1},2^{\ell-i},1)=\delta-(\epsilon_{j+1}-\epsilon_i)& \text{if } i\ge  j+2.
\end{array}\right.
$$
This gives 
$$
\{\epsilon_{i}-\epsilon_{j}, \delta-(\epsilon_{i}-\epsilon_{j})\mid 1\le i<j\le \ell-1\}\subseteq \overline{\Delta_{\rm{re}}^+}.
$$
\end{enumerate}

\begin{remark}
In  type $A^{(1)}_\ell$, we have $\Delta_{\rm{fin}}^+=\{\epsilon_i-\epsilon_j \mid 1\le i<j \le \ell+1\}$ and
$$
\overline{\Delta_{\rm{re}}^+}=\{\epsilon_i-\epsilon_j, \delta-(\epsilon_i-\epsilon_j) \mid 1\le i<j \le \ell+1\}.
$$
Elements in $\overline{\Delta_{\rm{re}}^+}$ label all arrows in $\vec C(\Lambda)$ of type $A_\ell^{(1)}$. More precisely, in \cite[Section 3]{ASW-rep-type}, we draw an arrow
$$
\xymatrix@C=1.2cm{\Lambda'=\Lambda_i+\Lambda_j+\tilde\Lambda\ar[r]^-{\Delta_{i,j}} &\Lambda''=\Lambda_{i-1}+\Lambda_{j+1}+\tilde\Lambda}\in \vec C(\Lambda)
$$
if $i-1\not\equiv_{\ell+1}j$ and $\min(X_{\Lambda'}+\Delta_{i,j}-\delta)<0$. Under this setting, $\delta=\alpha_0+\alpha_1+\cdots+\alpha_{\ell}=(1,1,\ldots, 1)$ and $X_{\Lambda''}=X_{\Lambda'}+\Delta_{i,j}$ with
$$
\Delta_{i,j}:=\left\{\begin{array}{ll}
(0^i,1^{j-i+1},0^{\ell-j})= \epsilon_i-\epsilon_{j+1}      & \text{ if } 0< i\le j\le \ell, \\
(1^{j+1},0^{\ell-j})= \delta-(\epsilon_{j+1}-\epsilon_{\ell+1})   & \text{ if } 0=i\le j\le \ell-1, \\
 (1^{j+1},0^{i-j-1},1^{\ell-i+1})=\delta-(\epsilon_{j+1}-\epsilon_{i})& \text{ if } 0\le j<i\le \ell.
\end{array}\right.
$$
\end{remark}

\begin{lemma}\label{lemma-directed-path}
Suppose $\Lambda\in \pcl$ and $\Lambda\ne \Lambda'\in \pcl(\Lambda)$. Then, there is a directed path from $\Lambda$ to $\Lambda'$ in $\vec C(\Lambda)$. 
\end{lemma}
\begin{proof} 
We prove the assertion by induction on $|X_{\Lambda'}|$. More precisely, we may construct a certain $\Lambda''$ such that $|X_{\Lambda''}|<|X_{\Lambda'}|$. Using a suitable lemma given in the previous subsection, we obtain a directed path displayed as $\xymatrix@C=0.7cm{\Lambda\ar[r]&\cdots\ar[r]&\Lambda''\ar[r] &\Lambda'}$.

Write $\Lambda'=\sum_{i=0}^{\ell}m_i\Lambda_i$ and $X_{\Lambda'}=(x_0,x_1,\ldots,x_\ell)$. 
Since $\Lambda'\ne \Lambda$, we have $|X_{\Lambda'}|
>0$. Since $\min (X_{\Lambda'}-\delta)<0$, we have $\min X_{\Lambda'}\in\{0, 1\}$. If moreover, $\min X_{\Lambda'}=1$, we have $x_i=1$ for some $1\le i\le \ell-1$. We divide the proof into the following 4 cases.
\begin{enumerate}
\item Suppose that there are some $0\le i,j\le \ell$ satisfying $i+1<j$, $x_i=x_j=0$, $x_{i+1}=x_{i+2}=\cdots=x_{j-1}\ge 1$. Then, by \eqref{equ:lambda-beta}, we have 
\begin{center}
$\langle \alpha_i^\vee,\Lambda-\Lambda'\rangle=\langle \alpha_i^\vee,\beta_{\Lambda'}\rangle<0,\quad
\langle \alpha_j^\vee,\Lambda-\Lambda'\rangle=\langle \alpha_j^\vee,\beta_{\Lambda'}\rangle<0$.
\end{center}
This implies that $m_i,m_j\ge 1$ and $\Lambda'=\Lambda_i+\Lambda_j+\tilde\Lambda\in \pcl(\Lambda)$ for some $\tilde\Lambda\in P_{cl,k-2}^+$.   
Since $i<j-1$, $\Lambda'_{j^-,i^+}$ is well-defined and $\Delta_{j^-,i^+}=(1,2^i, 1^{j-i-1}, 2^{\ell-j}, 1)$. Since $x_{i+1}=x_{i+2}=\cdots=x_{j-1}\ge 1$, we have $\min (X_{\Lambda'}+\Delta_{j^-,i^+}-\delta)\ge 0$. By Lemma~\ref{lemma：arrow-i-j+}, we have $\xymatrix@C=0.7cm{\Lambda'_{j^-,i^+}\ar[r] &\Lambda'}$ with 
\begin{center}
$X_{\Lambda'}=X_{\Lambda'_{j^-,i^+}}-\Delta_{j^-,i^+}+\delta=X_{\Lambda'_{j^-,i^+}}+\Delta_{(i+1)^-,(j-1)^+}$.
\end{center} 
In this case, we have $\Lambda'':=\Lambda'_{j^-,i^+}$.

\item Suppose $x_i=0$ for some $0\le i\le \ell-1$ and $x_{t}\ge 1$ for all $i+1\le t\le \ell$. 
\begin{itemize}
\item $i=\ell-1$. Then, $\langle \alpha_{\ell-1}^\vee,\beta_{\Lambda'}\rangle\le -2x_\ell\le -2$ and hence, $m_{\ell-1}\ge 2$.
We may write $\Lambda'=2\Lambda_{\ell-1}+\tilde \Lambda$ for some $\tilde\Lambda\in P_{cl,k-2}^+$. Using $\min(X_{\Lambda'}+\Delta_{i^+,i^+}-\delta)\ge0$, we obtain an arrow from $\Lambda'':=\Lambda'_{(\ell-1)^+,(\ell-1)^+}$ to $\Lambda'$ by Lemma~\ref{lemma：arrow-i+j+}. 

\item $i=\ell-2$. Then, $\langle \alpha_{\ell-2}^\vee, \beta_{\Lambda'}\rangle\le -1$ and $m_{\ell-2}\ge 1$ such that $\Lambda'_{i^+}$ is well-defined. Using $\min(X_{\Lambda'}+\Delta_{i^+}-\delta)\ge 0$, we obtain an arrow from $\Lambda'':=\Lambda'_{i^+}$ to $\Lambda'$ by Lemma~\ref{lemma:arrow-i+}.

\item $i\le \ell-3$ and $x_{\ell-1}>2x_\ell$. 
Then, $\langle \alpha_{i}^\vee,\beta_{\Lambda'}\rangle<0$ and 
$\langle \alpha_{\ell}^\vee,\beta_{\Lambda'}\rangle=2x_\ell-x_{\ell-1}<0$.
It gives $m_i,m_{\ell}>0$ and $\Lambda_{\ell^-,i^+}$ is well-defined. We have $\Lambda'':=\Lambda'_{\ell^-,i^+}$ similar to case (1).

\item $i\le \ell-3$, $x_j\le x_{j+1}\le \ldots \le x_{\ell-1}\le 2x_\ell$ and $x_{j-1}>x_j$ for some $i+2\le j\le \ell-1$. Then, $\langle \alpha_{\ell-1}^\vee, \beta_{\Lambda'}\rangle=(x_{\ell-1}-x_{\ell-2})-(2x_\ell-x_{\ell-1})< 0$ if $j=\ell-1$, and $\langle \alpha_{j}^\vee,\beta_{\Lambda'}\rangle=(x_j-x_{j-1})-(x_{j+1}-x_{j})<0$ if $j<\ell-1$; in both cases, we have $m_j>0$. We also have $m_i>0$ due to $\langle \alpha_{i}^\vee,\beta_{\Lambda'}\rangle <0$. Thus, $\Lambda'_{j^-,i^+}$ is well-defined and we may choose $\Lambda'':=\Lambda'_{j^-,i^+}$.

\item $i\le \ell-3$ and $x_{i+1}\le x_{i+2}\le \cdots\le x_{\ell-1}\le 2x_\ell$.
\begin{itemize}
\item If $x_{i+1}\ge 2$, then $\langle \alpha_{i}^\vee,\beta_{\Lambda'}\rangle\le -2$ and $\Lambda'_{i^+,i^+}$ is well-defined. We set $\Lambda'':=\Lambda'_{i^+,i^+}$ due to $\min(X_{\Lambda'}+\Delta_{i^+,i^+}-\delta)\ge 0$.

\item If $x_{i+1}=x_{i+2}=\cdots=x_j=1$ and $x_{j+1}\ge 2$ for some $i+2\le j\le \ell-1$, then $\langle \alpha_{i}^\vee,\beta_{\Lambda'}\rangle<0$ and $\langle \alpha_{j}^\vee,\beta_{\Lambda'}\rangle < 0$. It gives $m_i, m_j>0$, such that $\Lambda'':=\Lambda'_{j^-,i^+}$ is well-defined.

\item If $x_{i+1}=x_{i+2}=\cdots=x_\ell=1$, then $\langle \alpha_{\ell-1}^\vee,\beta_{\Lambda'}\rangle=-1$ and $m_{\ell-1}\ge 1$. It turns out that $\Lambda'':=\Lambda'_{(\ell-1)^-,i^+}$.
\end{itemize}
\end{itemize}

\item Suppose $x_i=0$ for some $1\le i\le \ell$ and $x_t\ge 1$ for all $0\le t\le i-1$. One may check this case using a similar method as in case (2). 

\item Suppose $\min X_{\Lambda'}=1$ (i.e., $x_i\ne 0$ for all $0\le i\le \ell$). Since $\min(X_{\Lambda'}-\delta)<0$, there must exist $x_i=1$ for some $1\le i\le \ell-1$. We denote by $i$ (resp., $j$) the minimal (resp., maximal) number in $\{1,2,\ldots, \ell-1\}$ satisfying $x_i=1$ (resp., $x_j=1$). If $i=j$, then $\langle \alpha_{i}^\vee, \beta_{\Lambda'}\rangle \le -2$ and $m_i\ge 2$.
If $i<j$, then $\langle \alpha_{i}^\vee, \beta_{\Lambda'}\rangle\le -1$ and $\langle \alpha_{j}^\vee, \beta_{\Lambda'} \rangle \le -1$, such that $m_i,m_j\ge 1$. In both cases, $\Lambda'':=\Lambda'_{i^-,j^+}$ is well-defined and $\min (X_{\Lambda'}+\Delta_{i^-,j^+}-\delta)\ge 0$. 
\end{enumerate}
We have completed the proof.
\end{proof}

We have a natural embedding of quivers from lower level to higher level as follows. We omit the proof because it is easy to verify the assertion by the definition of arrows.

\begin{corollary}\label{cor::embedding-path}
Suppose $\Lambda=\bar\Lambda+\tilde \Lambda$ with $\Lambda\in \pcl, \bar\Lambda\in P_{cl,k'}^+$ and $\tilde \Lambda\in P_{cl,k-k'}^+$. There is a directed path
$$
\xymatrix@C=1.5cm@R=1cm{
\Lambda^{(1)}\ar[r]^-{\Delta^{(1)}} &\Lambda^{(2)}\ar[r]^-{\Delta^{(2)}} & \cdots \ar[r]^-{\Delta^{(m-1)}} &\Lambda^{(m)}} \in \vec C(\bar\Lambda)
$$
if and only if there is a directed path
$$
\xymatrix@C=1.5cm@R=1cm{
\Lambda^{(1)}+\tilde\Lambda\ar[r]^-{\Delta^{(1)}} &\Lambda^{(2)}+\tilde\Lambda\ar[r]^-{\Delta^{(2)}} & \cdots \ar[r]^-{\Delta^{(m-1)}} &\Lambda^{(m)}+\tilde\Lambda} \in \vec C(\Lambda).
$$
\end{corollary}

We are able to show that our quiver $\vec C(\Lambda)$ serves the same role as that for type $A^{(1)}_\ell$ in \cite{ASW-rep-type}.
\begin{theorem}\label{theo::arrow-weights}
Suppose $\Lambda'\rightarrow \Lambda''\in \vec{C}(\Lambda)$ and $s:=|X_{\Lambda''}|-|X_{\Lambda'}|$. There is an element   $\mathbf{i}=(i_1,i_2,\ldots, i_s)\in I^s$ and a sequence $\beta_{\Lambda'}=\beta_0, \beta_1,\ldots,\beta_s=\beta_{\Lambda''}\in Q_+$ such that $\beta_t=\beta_{t-1}+\alpha_{i_t}$ and $\langle \alpha_{i_{t}}^\vee, \Lambda-\beta_{t-1}\rangle \ge 1$, for $1\le t\le s$.
\end{theorem}
\begin{proof}
We divide the proof into the following 5 cases. 
\begin{enumerate}
\item $\Lambda''=\Lambda'_{i^+}$. By Definition \ref{Def:quiver-of-maximal-dominant}, $X_{\Lambda''}=X_{\Lambda'}+\Delta_{i^+}$ for some $0\le i\le \ell-2$. This gives $s=2(i+1)$ and $\beta_{\Lambda''}=\beta_{\Lambda'}+\alpha_0+2\alpha_1+\cdots+2\alpha_{i}+\alpha_{i+1}$. 
We set 
\begin{center}
$\mathbf{i}=\left\{\begin{array}{ll}
(0,1)&  \text{if } i=0,\\
(i,i-1,\ldots,2,1,0,1,2,\ldots,i-1,i+1,i)& \text{if } i\ne 0.
\end{array}\right.$
\end{center}
We obviously obtain $\beta_t=\beta_{t-1}+\alpha_{i_t}$ for $1\le t\le s$. 
By \eqref{equ:lambda-beta}, we have $\langle \alpha_{i_{t}}^\vee,\Lambda-\beta_{\Lambda'}\rangle=\langle \alpha_{i_{t}}^\vee, \Lambda'\rangle$. We have $\langle \alpha_{i_{1}}^\vee,\Lambda-\beta_{\Lambda'}\rangle=\langle \alpha_{i}^\vee,\Lambda'\rangle \ge 1$ since $\Lambda'$ is of the form $\Lambda_i+\tilde \Lambda'$ in this case. For $2\le t\le s$, we have 
\begin{center}
$\begin{array}{ll}
\langle \alpha_{i_{t}}^\vee, \Lambda-\beta_{t-1}\rangle &=\langle \alpha_{i_{t}}^\vee, \Lambda-(\beta_0+\sum_{j=1}^{t-1}\alpha_{i_j})\rangle\\ &=\langle \alpha_{i_{t}}^\vee, \Lambda'-\sum_{j=1}^{t-1}\alpha_{i_j}\rangle \\
&\ge -\langle \alpha_{i_{t}}^\vee, \sum_{j=1}^{t-1}\alpha_{i_j}\rangle ,
\end{array}$
\end{center}
which implies $\langle \alpha_{i_{t}}^\vee, \Lambda-\beta_{t-1}\rangle \ge 2$ if $i=0$, and $\langle \alpha_{i_{t}}^\vee, \Lambda-\beta_{t-1}\rangle \ge 1$ if $i\ne 0$.

\item $\Lambda''=\Lambda'_{i^-}$. In this case, $X_{\Lambda''}=X_{\Lambda'}+\Delta_{i^-}$ for some $2\le i\le \ell$. We have $s=2(\ell-i)$ and $\beta_{\Lambda''}=\beta_{\Lambda'}+\alpha_{i-1}+2(\alpha_i+\cdots+\alpha_{\ell-1})+\alpha_\ell$.
Set
\begin{center}
$\mathbf{i}=\left\{
\begin{array}{ll}
(i,i+1, \ldots, \ell-1,\ell, \ell-1,\ldots,i+3,i+2,i-1,i )& \text{if } i\ne \ell, \\
(\ell,\ell-1)&  \text{if } i=\ell.
\end{array}\right.$
\end{center} 
We then omit the details since they are quite similar to the case (1).

\item $\Lambda''=\Lambda'_{i^-,j^+}$. Then, $X_{\Lambda''}=X_{\Lambda'}+\Delta_{i^-,j^+}$ for some $0\le i,j \le \ell$ with $i\ne 0$, $j\ne \ell$, $i-1\ne j$. 
If $i\le j$, then $s=j-i+1$ and $\beta_{\Lambda''}=\beta_{\Lambda'}+\alpha_i+\cdots+\alpha_j$, we set $\mathbf{i}=(i,i+1,\ldots, j)$. If $i\ge j+2$, then $s=2\ell+j-i+1$ and $\beta_{\Lambda''}=\beta_{\Lambda'}+\alpha_0+2(\alpha_1+\cdots+\alpha_j)+(\alpha_{j+1}+\cdots+\alpha_{i-1})+2(\alpha_i+\cdots+\alpha_{\ell-1})+\alpha_\ell$, we set 
\begin{center}
$\mathbf{i}=\left\{\begin{array}{ll}
(0,1,\ldots,i-1,i,\ldots,\ell-1,\ell,\ell-1,\ldots,i+1,i) & \text{if } i\ne \ell,\\
(0,1,\ldots,\ell) & \text{if } i=\ell.      
\end{array}\right.$
\end{center}
for $j=0$, and $\mathbf{i}=(j,j-1,\ldots,1,0,1,\ldots,j-1,j+1,j,j+2,\ldots,i-1,i,\ldots,\ell-1,\ell,\ell-1,\ldots,i+1,i)$ for $j\ge 1$. In both cases, we have $\beta_t=\beta_{t-1}+\alpha_{i_t}$ for $1\le t\le s$. Similar to case (1), we have $\langle \alpha_{i_1}^\vee,\Lambda-\beta_{\Lambda'}\rangle=\langle \alpha_{i}^\vee, \Lambda'\rangle$ or $\langle \alpha_{j}^\vee, \Lambda'\rangle$ $\ge 1$. For $2\le t\le s$, 
we have 
\begin{center}
$\langle \alpha_{i_t}^\vee, \Lambda-\beta_{t-1}\rangle
=\langle \alpha_{i_t}^\vee, \Lambda'-\sum_{r=1}^{t-1}\alpha_{i_r}\rangle\ge -\langle \alpha_{i_t}^\vee, \sum_{r=1}^{t-1}\alpha_{i_r}\rangle$,
\end{center}
it gives $\langle \alpha_{i_{t}}^\vee, \Lambda-\beta_{t-1}\rangle \ge 2$ if $i=\ell, j=0$, and $\langle \alpha_{i_{t}}^\vee, \Lambda-\beta_{t-1}\rangle \ge 1$ otherwise.

\item $\Lambda''=\Lambda'_{i^+,j^+}$. Then, $X_{\Lambda''}=X_{\Lambda'}+\Delta_{i^+,j^+}$ for some $0\le i\le j\le \ell-1$. The case of $j=i+1$ has been proven in case (1) since $\Delta_{i^+,(i+1)^+}=\Delta_{i^+}$.
\begin{itemize}
\item Suppose $i=j$. We have $s=2i+1$ and $\beta_{\Lambda''}=\beta_{\Lambda'}+\alpha_0+2(\alpha_1+\cdots+\alpha_i)$, and we set
\begin{center}
$\mathbf{i}=\left\{\begin{array}{ll}
(0) & \text{if } i=0,\\
(i,i-1,\ldots,1,0,1,\ldots,i) & \text{if } i\ne 0.      
\end{array}\right.
$
\end{center}
It gives $\langle \alpha_{i_{1}}^\vee, \Lambda-\beta_{\Lambda'}\rangle=\langle \alpha_{i}^\vee, \Lambda'\rangle\ge 2$ by our assumption. For $2\le t\le s$, we obtain $\langle \alpha_{i_{t}}^\vee, \Lambda-\beta_{t-1}\rangle
\ge -\langle \alpha_{i_{t}}^\vee, \sum_{r=1}^{t-1}\alpha_{i_r}\rangle=1$ if $t\ne s$, and $\langle \alpha_{i_{t}}^\vee,\Lambda -\beta_{t-1}\rangle=\langle \alpha_{i}^\vee,\Lambda' \rangle\ge 2$ if $t=s$. In fact, set $t=s\ge 2$, we have
\begin{center}
$\langle \alpha_{i_{s}}^\vee, \Lambda-\beta_{s-1}\rangle=\langle \alpha_{i}^\vee, \Lambda-(\beta_0+\sum_{r=1}^{s-1}\alpha_{i_r})\rangle =\langle \alpha_{i}^\vee, \Lambda'-\sum_{r=1}^{s-1}\alpha_{i_r}\rangle$,
\end{center}
combining this with $\langle \alpha_{1}^\vee, \sum_{r=1}^{s-1}\alpha_{i_r}\rangle=a_{11}+a_{10}=0$ if $i=1$ and $\langle \alpha_{i}^\vee, \sum_{r=1}^{s-1}\alpha_{i_r}\rangle=a_{ii}+2a_{i(i-1)}=0$ if $2\le i\le \ell-1$, we obtain the result.

\item Suppose $i+2\ge j$. We have a path 
\begin{center}
$\xymatrix@C=1.2cm@R=1cm{
\Lambda'\ar[r]&\Lambda'_{i^+}\ar[r]&(\Lambda'_{i^+})_{(i+2)^-,j^+}}=\Lambda''$.
\end{center}
Then, the statement holds by composing the results in case (1) and case (3). 
\end{itemize}

\item $\Lambda''=\Lambda'_{i^-,j^-}$. Then, $X_{\Lambda''}+\Delta_{i^-,j^-}$ for some $1\le i\le j\le \ell$, and the case of $i=j-1$ has been proven in case (2) due to $\Delta_{(j-1)^-,j^-}=\Delta_{j^-}$. If $i=j$, then $s=2(\ell-j)+1$ and $\beta_{\Lambda''}=\beta_{\Lambda'}+2(\alpha_j+\cdots+\alpha_{\ell-1})+\alpha_\ell$, we set 
\begin{center}
$\mathbf{i}=\left\{\begin{array}{ll}
(\ell) & \text{if } j=\ell,\\
(j,j+1,\ldots, \ell-1,\ell,\ell-1,\ldots, j) & \text{if } j<\ell.      
\end{array}\right.$
\end{center}
One may show the statement using a similar analysis with case (4). If $i\le j-2$, there is a path 
\begin{center}
$\xymatrix@C=1.2cm@R=1cm{
\Lambda'\ar[r]&\Lambda'_{j^-}\ar[r]&(\Lambda'_{j^-})_{(i)^-,(j-2)^+}}=\Lambda''$.
\end{center}
Then, the statement follows from the results in case (2) and (3).
\end{enumerate}
We have completed the proof.
\end{proof}
    
\subsection{Comparison with previous level one results}
We may understand the construction of \cite{AP-rep-type-C-level-1} and \cite{CH-type-c-level-1} in our broader setting as follows.

In \cite[Proposition 5.1]{AP-rep-type-C-level-1}, it was shown that 
$$
\max^+(\Lambda_0)=\left\{ \Lambda_0+\varpi_i-\frac{i}{2}\delta \mid 0\le i\le \ell, i \in 2\Z_{\ge 0}\right\},
$$
where $\varpi_0:=0$ and 
$$
\varpi_i:=\alpha_1+2\alpha_2+\cdots+(i-1)\alpha_{i-1}+i(\alpha_i+\alpha_{i+1}+\cdots+\alpha_{\ell-1}+\frac{i}{2}\alpha_{\ell}).
$$
We remark that this is the solution of $\mathsf{A} X^t=Y^t$ for $Y= \hub(\Lambda_i)-\hub(\Lambda_0)$ in the sense of Lemma~\ref{system of linear equations}.
Substituting this into our setting, we have
$$
\beta^{\Lambda_0}_{\Lambda_i}=\frac{i}{2}\delta-\varpi_i.
$$
This gives an arrow $\xymatrix@C=0.7cm{\Lambda_i\ar[r]&\Lambda_{i+2}}$ in $\vec C(\Lambda_0)$ because
$$
\left(\frac{i+2}{2}\delta-\varpi_{i+2}\right)-\left(\frac{i}{2}\delta-\varpi_{i}\right)=\alpha_{i+1}+2\alpha_{i+2}+\cdots+2\alpha_{\ell-1}+\alpha_\ell \in Q_+.
$$
Thus, the quiver $\vec C(\Lambda_0)$ is displayed as
\begin{equation}\label{equ::quiver-level-one-0}
\xymatrix@C=1.5cm@R=1cm{
*+[F]{\Lambda_0}\ar[r]&*+[F]{\Lambda_{2}}\ar[r]&*+[F]{\Lambda_{4}}\ar[r]&\cdots
\ar[r]&*+[F]{\Lambda_{2\left \lfloor \ell/2\right \rfloor}}}.
\end{equation}

In \cite[Proposition 2.8]{CH-type-c-level-1}, the authors showed that, for $0\le s\le \ell$,
$$
{\max}^+(\Lambda_s)=\left\{ \Lambda_s+\xi_{s,\pm i}-\frac{i}{2}\delta \mid 0\le i\le \ell, i \in 2\Z_{\ge 0} \right\},
$$
where $\xi_{0,i}=\varpi_i$, and
\begin{align*}
\frac{i}{2}\delta-\xi_{s,i}&=\frac{i}{2}\alpha_0+i\ssum_{j=1}^s \alpha_j+(i-1)\alpha_{s+1}+(i-2)\alpha_{s+2}+\cdots+\alpha_{s+i-1}, \\
\frac{i}{2}\delta-\xi_{s,-i}&=\alpha_{s-i+1}+2\alpha_{s-i+2}+\cdots+(i-1)\alpha_{s-1}+i\ssum_{j=s}^{\ell-1} \alpha_j+\frac{i}{2}\alpha_\ell.
\end{align*}
This leads to the identities
$$
\beta^{\Lambda_s}_{\Lambda_{s+i}}=\frac{i}{2}\delta-\xi_{s,i} \quad \text{and} \quad \beta^{\Lambda_s}_{\Lambda_{s-i}}=\frac{i}{2}\delta-\xi_{s,-i}.
$$
Moreover, if we multiply $\mathsf{A}$ with coefficient vectors of $\beta^{\Lambda_s}_{\Lambda_{s+i}}$ or $\beta^{\Lambda_s}_{\Lambda_{s-i}}$, we always obtain a vector with exactly one $1$ and one $-1$ while all other entries are $0$. One may check that
\begin{align*}
\left(\frac{i+2}{2}\delta-\xi_{s,i+2}\right)-\left(\frac{i}{2}\delta-\xi_{s,i}\right)&=\alpha_0+2\ssum_{j=1}^{s+i} \alpha_j + \alpha_{s+i+1}\in Q_+, \\
\left(\frac{i+2}{2}\delta-\xi_{s,-i-2}\right)-\left(\frac{i}{2}\delta-\xi_{s,-i}\right)&=\alpha_{s-i-1}+2\ssum_{j=s-i}^{\ell-1} \alpha_j + \alpha_\ell\in Q_+.
\end{align*}
Hence, there are arrows $\xymatrix@C=0.6cm{\Lambda_{s+i}\ar[r]&\Lambda_{s+i+2}}$ and $\xymatrix@C=0.6cm{\Lambda_{s-i}\ar[r]&\Lambda_{s-i-2}}$ in $\vec C(\Lambda_s)$. We conclude that the quiver $\vec C(\Lambda_s)$ is displayed as
\begin{equation}\label{equ::quiver-level-one-even}
\vcenter{\xymatrix@C=1.2cm@R=0.4cm{
&*+[F]{\Lambda_{s-2}}\ar[r]&\cdots \ar[r]&*+[F]{\Lambda_{2}}\ar[r]
\ar[r]&*+[F]{\Lambda_{0}}\\
*+[F]{\Lambda_s}\ar[ur]\ar[dr]&&&&\\
&*+[F]{\Lambda_{s+2}}\ar[r]&*+[F]{\Lambda_{s+4}}\ar[r]&\cdots
\ar[r]&*+[F]{\Lambda_{2\left \lfloor \ell/2\right \rfloor}}
}}
\end{equation}
if $s$ is even, and 
\begin{equation}\label{equ::quiver-level-one-odd}
\vcenter{\xymatrix@C=1.2cm@R=0.4cm{
&*+[F]{\Lambda_{s-2}}\ar[r]&\cdots \ar[r]&*+[F]{\Lambda_{3}}\ar[r]
\ar[r]&*+[F]{\Lambda_{1}}\\
*+[F]{\Lambda_s}\ar[ur]\ar[dr]&&&&\\
&*+[F]{\Lambda_{s+2}}\ar[r]&*+[F]{\Lambda_{s+4}}\ar[r]&\cdots
\ar[r]&*+[F]{\Lambda_{2\left \lfloor (\ell-1)/2\right \rfloor+1}}
}}
\end{equation}
if $s$ is odd.

\section{Proof strategy for the main theorem}
In this section, we review some well-known features in the representation theory of $R^{\Lambda}(\beta)$ in type $C^{(1)}_\ell$. 
We recall the results from \cite{AP-rep-type-C-level-1} and \cite{CH-type-c-level-1} for level one cases. We then focus on the case $k\ge 2$ and prove our main theorem given in the introduction: we prove (1) of MAIN THEOREM in Section 5; we give the proofs for (2)(a) and (2)(b) of MAIN THEOREM in Section 6 and Section 7 respectively; we prove (2)(c) of MAIN THEOREM in the remaining sections. We also introduce some reduction lemmas to reduce the problem on $R^{\Lambda}(\beta)$ to cases with small levels of $\Lambda$ and small heights of $\beta$, similar to the strategy in \cite{ASW-rep-type} for type $A^{(1)}_\ell$. These reduction methods play a crucial role in the proof process.

Let us start with the fact that $R^{\Lambda}(\beta)$ is a symmetric algebra (see the Appendix in \cite{SVV-center}). It gives that the representation type of $R^{\Lambda}(\beta)$ is preserved under derived equivalence, see \cite{Rickard-derived-equi} and \cite{Krause-rep-type-stable-equi}. Then, the problem we consider relies on figuring out when $R^{\Lambda}(\beta)$ and $R^{\Lambda}(\beta')$ are derived equivalent. By Chuang and Rouquier's result \cite{CR-categorification}, we know that $R^{\Lambda}(\beta)$ is derived equivalent to $R^{\Lambda}(\beta')$ if $\Lambda-\beta$ and $\Lambda-\beta'$ lie in the same $W$-orbit of $P(\Lambda)$. Furthermore, by \eqref{equ::weight-set} and Proposition \ref{theorem-inverse-of-phi}, the representatives of $W$-orbits of $P(\Lambda)$ with $\Lambda\in \pcl$ are given by $\{\Lambda-\beta_{\Lambda'}-m\delta\mid \Lambda'\in \pcl(\Lambda), m\in \Z_{\ge 0}\}$, where $\pcl(\Lambda)$ is defined at the beginning of Section 3. All in all, it suffices to consider the representation type of $R^\Lambda(\gamma)$ for $\gamma\in O(\Lambda)$, where 
\begin{equation}
O(\Lambda):=\{\beta_{\Lambda'}+m\delta\mid \Lambda'\in \pcl(\Lambda), m\in \Z_{\ge 0}\}.
\end{equation}

\begin{remark}
If $\Lambda'=\Lambda$, i.e., $\beta_{\Lambda'}=0$, then $R^\Lambda(\beta_{\Lambda})\cong \k$ is a simple algebra.
\end{remark}

\subsection{Results in level one cases}
We have given the quiver $\vec C(\Lambda_s)$ for $0\le s\le \ell$ in the previous section, see \eqref{equ::quiver-level-one-0}, \eqref{equ::quiver-level-one-even}, \eqref{equ::quiver-level-one-odd}. Then, the main results of \cite{AP-rep-type-C-level-1} and \cite{CH-type-c-level-1} can be summarized as follows. 

\begin{theorem}\label{theo::result-level-one}
Set $\Lambda_s\in P_{cl,1}^+$ with $0\le s\le \ell$ and $\Lambda'\in P_{cl,1}^+(\Lambda_s)$. Then, the cyclotomic KLR algebra $R^{\Lambda_s}(\beta_{\Lambda'}+m\delta)$ is representation-finite if $m=0$ and $\Lambda'\in \{\Lambda_s, \Lambda_{s-2}, \Lambda_{s+2}\}$, tame if $m=1$, $\ell=2$ and $\Lambda'=\Lambda_s$, wild otherwise.
\end{theorem}

It implies that $R^{\Lambda_s}(\beta_{\Lambda'}+m\delta)$ is wild for all $m\ge 1$ if $\beta_{\Lambda'}\ne 0$, and for all $m\ge 2$ if $\beta_{\Lambda'}=0$. Then, the representation type of $R^{\Lambda_s}(\beta_{\Lambda'})$ and $R^{\Lambda_s}(\delta)$ are determined as in Theorem \ref{theo::result-level-one}.

\subsection{Reduction methods}
In \cite[Section 5]{ASW-rep-type}, level lowering argument and the quiver $\vec C(\Lambda)$ are used to show the wildness of $R^\Lambda(\beta^\Lambda_{\Lambda'}+m\delta)$ in type $A^{(1)}_\ell$, for $m\ge 1+\delta_{\Lambda,\Lambda'}$, where $\delta_{\Lambda,\Lambda'}$ is the Kronecker delta. Similarly, we have 

\begin{lemma}\label{lem::level-lowering-argument}
Suppose $\Lambda=\bar\Lambda+\tilde \Lambda$ for some $\Lambda\in \pcl$, $\bar\Lambda\in P_{cl,k'}^+$ and $\tilde \Lambda\in P_{cl,k-k'}^+$. 
Then, the representation-infiniteness (resp., wildness) of $R^{\bar\Lambda}(\gamma)$ implies the representation-infiniteness (resp., wildness) of $R^\Lambda(\gamma)$. 
\end{lemma}
\begin{proof}
This is similar to the proof of \cite[Lemma~4.1]{ASW-rep-type}
\end{proof}

\begin{lemma}\label{cor-arrow}
Suppose $\xymatrix@C=0.7cm{\Lambda'\ar[r]&\Lambda''}$ in $\vec C(\Lambda)$. Then, the representation-infiniteness (resp., wildness) of $R^\Lambda(\beta_{\Lambda'}+m\delta)$ implies the representation-infiniteness (resp., wildness) of $R^\Lambda(\beta_{\Lambda''}+m\delta)$, for any $m\in\Z_{\ge 0}$.
\end{lemma}
\begin{proof}
This is similar to the proof of \cite[Lemma~4.2]{ASW-rep-type}, by using Theorem~\ref{theo::arrow-weights}, \cite[Proposition~2.3]{EN-rep-type-Hecke} and \cite[Theorem~5.2]{KK-categorification}.
\end{proof}

\begin{corollary}\label{cor::reduction-path}
If $R^{\Lambda}(\beta_{\Lambda'}+m\delta)$ for $\Lambda'\in\vec C(\Lambda)$ and $m\in\Z_{\ge 0}$ is representation-infinite (resp., wild) and there is a directed path from $\Lambda'$ to $\Lambda''$ in $\vec{C}(\Lambda)$, then $R^{\Lambda}(\beta_{\Lambda''}+m\delta)$ is also representation-infinite (resp., wild).
\end{corollary}

\section{Proof of the first part of the main theorem}
We are able to show the following result.
\begin{theorem}
Suppose $\Lambda\in \pcl$ with $k\ge 2$. 
Then, $R^{\Lambda}(\beta_{\Lambda'}+m\delta)$ is wild for any $m\ge 1$ and $\Lambda'\in \pcl(\Lambda)$.
\end{theorem}
\begin{proof}
Set $\Lambda=\Lambda_s+\tilde \Lambda$ with $0\le s\le \ell$. If $m\ge 2$, then $R^{\Lambda_s}(m\delta)$ is wild by Theorem \ref{theo::result-level-one}, and so is $R^{\Lambda}(m\delta)$ by Lemma \ref{lem::level-lowering-argument}. Since there exists a directed path from $\Lambda$ to any $\Lambda\ne \Lambda'\in \pcl(\Lambda)$, we deduce that $R^{\Lambda}(\beta_{\Lambda'}+m\delta)$ is wild for any $m\ge 2$ and $\Lambda'\in \pcl(\Lambda)$, by Corollary \ref{cor::reduction-path}. If $m=1$ and $\ell\ge 3$, then $R^{\Lambda_s}(\delta)$ is wild following Theorem \ref{theo::result-level-one}, which implies that $R^{\Lambda}(\beta_{\Lambda'}+\delta)$ is wild for any $\Lambda'\in \pcl(\Lambda)$.

Suppose $m=1$ and $\ell=2$. 
Then, $\delta=\alpha_0+2\alpha_1+\alpha_2$. 
We have to consider the cases $\Lambda\in \{2\Lambda_0,2\Lambda_1,2\Lambda_2, \Lambda_0+\Lambda_1, \Lambda_1+\Lambda_2, \Lambda_0+\Lambda_2\}$. 
\begin{enumerate}
\item Set $A:=e R^{2\Lambda_0}(\delta)e$ with $e=e(0121)$. Then, $\dim_q A=1+2q^2+2q^4+2q^6+q^8$. 
We show that $A$ has a basis $\{ x_2^ae, x_2^ax_4e\mid 0\le a\le 3\}$. 
First, we have $x_1^2e=x_1^2e'=0$, where $e':=e(\nu')=e(0112)$.
Since $e(s_1\nu)=e(s_1\nu')=e(s_2\nu)=0$, we have $\psi_1e=\psi_2e=\psi_1 e'=0$ and hence $\psi_1^2e=\psi_2^2e=\psi_1^2e'=0$. 
This implies $x_1e=x_2^2e$, $x_3e=x_2^2e$, so that we may replace $x_1e$ and $x_3e$ with $x_2^2e$, and $x_1e'=x_2^2e'$. 
Let $f=x_1-x_2^2$ and $\partial_2 f=\frac{s_2f-f}{x_2-x_3}$.
Then Lemma \ref{KK-lemma} implies $(\partial_2 f)e'=0$
since $\nu'_2=\nu'_3$ and $fe'=0$. 
Hence, $x_3e'=-x_2e'$. 
This implies that 
\begin{center}
$x_4\psi_3\psi_2\psi_3e=x_4\psi_3e'\psi_2\psi_3=\psi_3x_3e'\psi_2\psi_3=-x_2\psi_3\psi_2\psi_3e$.
\end{center}
On the other hand, we have $\psi_3\psi_2\psi_3 e=(\psi_3\psi_2\psi_3-\psi_2\psi_3\psi_2)e=(x_2+x_4)e$.
Hence, 
$$ x_4(x_2+x_4)e=-x_2(x_2+x_4)e,$$ and we may replace $x_4^2e$ with $-(x_2^2+2x_2x_4)e$.
Moreover, if $e\psi_we\ne 0$, then we can choose $\psi_w=1$ or $\psi_w=\psi_2\psi_3\psi_2$. 
The latter one can not happen since $\psi_2e=0$. Therefore, we obtain the required basis following the graded dimension.
Further, we have a surjective algebra homomorphism from $A$ to $B:=\k[X,Y]/(X^3,Y^2,X^2Y)$ sending $x_2$ and $x_2+x_4$ to $X$ and $Y$, respectively. Since $B$ is a wild local algebra by Proposition \ref{prop::tame-local-alg}, $A$ is also wild.

\item Set $A:=(e_1+e_2)R^{2\Lambda_1}(\delta)(e_1+e_2)$ with $e_1=e(1210)$ and $e_2=e(1201)$. We have
\begin{center}
$
\begin{aligned}
\dim_q e_1Ae_1=\dim_q e_2Ae_2&=1+2q^2+2q^4+2q^6+q^8,\\
\dim_q e_1Ae_2=\dim_q e_2Ae_1&=q^2+2q^4+q^6.
\end{aligned}
$
\end{center}
Then, $A$ is wild by Lemma \ref{lem::wild-two-point-alg}.

\item Set $A:=(e_1+e_2)R^{\Lambda_0+\Lambda_1}(\delta)(e_1+e_2)$ with $e_1=e(0121)$ and $e_2=e(1201)$. Then,
\begin{center}
$
\begin{aligned}
\dim_q e_1Ae_1&=1+2q^2+3q^4+2q^6+q^8, \\
\dim_q e_2Ae_2&=1+q^2+2q^4+q^6+q^8,\\
\dim_q e_1Ae_2&=\dim_q e_2Ae_1=q^2+q^4+q^6.
\end{aligned}
$
\end{center}
Then, $A$ is wild by Lemma \ref{lem::wild-two-point-alg}.

\item Set $A:=eR^{\Lambda_0+\Lambda_2}(\delta)e$ with $e=e(2101)$. We obtain
\begin{center}
$
\dim_q eAe=1+3q^2+4q^4+3q^6+q^8.
$
\end{center}
Then, $A$ is wild by Lemma \ref{lem::wild-three-loops-part1}.
\end{enumerate}
In the above 4 cases, $R^\Lambda(\delta)$ is wild since we find an idempotent truncation of $R^\Lambda(\delta)$ being wild.
Using Proposition \ref{prop::iso-sigma}, we conclude that all the remaining cases are wild.
\end{proof}

Combining with the bijection between $\pcl(\Lambda)$ and $\max^+(\Lambda)$ as we mentioned in Proposition \ref{theorem-inverse-of-phi}, we conclude that $R^\Lambda(\beta)$ is wild if $\Lambda-\beta$ is not a maximal dominant weight. This gives a proof of (1) of MAIN THEOREM.
Now, in the case of $k\ge 2$, we only need to determine the representation type of $R^{\Lambda}(\beta_{\Lambda'})$ for $\Lambda'\in \pcl(\Lambda)$. This will be accomplished in the following sections.

\section{Proof of the second part--finite representation type}
In the case (f1), $R^\Lambda(\beta)\cong \k[X]/(X^{m_a})$. 
For the first case in (f2), we have $e_1=e(01)=1$ by $e_2=e(10)=x_1^{\langle \alpha_1,m_0\Lambda_0\rangle}e(10)=0$, and $\psi=\psi e_2=e_1\psi=0$, $(x_2^2-x_1)e_{1}=\psi^2e_{1}=0$, so that $R^\Lambda(\beta)\cong \k[X]/(X^{2m_0})$. 
For the second case in (f2), we have $x_1=0$ and that $P_1=\langle e_1, \psi e_1, x_2e_1,\psi^2 e_1\rangle$, $P_2=\langle e_2, \psi e_2, \psi^2e_2\rangle$ are indecomposable projective $R^\La(\beta)$-modules. Then, we see that $R^\Lambda(\beta)$ is a Brauer tree algebra whose Brauer tree is given as
$$
\entrymodifiers={+[Fo]}
\xymatrix@C=1.2cm{
2 \ar@{-}[r]
& \ \  \ar@{-}[r]
& \ \ 
}\ ,
$$
which is of finite representation type.  
By symmetry, we have the results for the case (f3).  
The case (f4) is treated in \cite[Proposition 6.8]{ASW-rep-type} and it is also a Brauer tree algebra. If $R^\La(\beta)$ is derived equivalent to this algebra, we recall that $R^\La(\beta)$ is a cellular algebra when $\ch \k\ne 2$ by \cite[Theorem A]{EM-cellular-symmetrictypeA} because we choose a special value for the parameter $t$ here and Morita invariance of the cellularity holds when $\ch \k\ne2$.
Thus, the Brauer tree is the straight line with $b-a+2$ vertices without an exceptional vertex. Hence, $R^\La(\beta)$ is Morita equivalent to this algebra when $\ch \k\ne2$ or $R^\Lambda(\beta)$ is a basic algebra. 

The remaining two cases follow from \cite[Lemma 3.3(1)]{AP-rep-type-C-level-1} and 
\cite[Proposition 4.1, Theorem 4.4]{CH-type-c-level-1}: 
In the case (f5), $R^\Lambda(\beta)\cong R^{\Lambda_a}(\beta_{\Lambda_{a+2}})$. It is the Brauer tree algebra whose Brauer tree is the straight line with $a+2$ vertices without an exceptional vertex, and in the case (f6), $R^\Lambda(\beta)\cong R^{\Lambda_b}(\beta_{\Lambda_{b-2}})$, which is the Brauer tree algebra whose Brauer tree is the straight line with $\ell-b+2$ vertices without exceptional vertex.

\section{Proof of the second part--tame representation type}
Before starting the proof for the tame cases, we consider $A=R^{t\Lambda_{\ell-1}+\Lambda_\ell}(\alpha_{\ell-1}+\alpha_\ell)$, for $t\ge2$. Define
\[
e_1=e(\ell-1,\ell), \quad e_2=e(\ell,\ell-1). 
\]
The graded dimensions are given as follows.
$$
\dim_q e_1Ae_1=1+q^2+2\ssum_{i=2}^{t-1} q^{2i} + q^{2t}+q^{2t+2},
$$
$$
\dim_q e_2Ae_2=\ssum_{i=0}^{t+1} q^{2i}, \quad \dim_q e_1Ae_2 =\dim_q e_2Ae_1=\ssum_{i=1}^t q^{2i}.
$$

In particular, $\dim A=5t+2$. Then, $A$ is generated by $e_1,e_2,\psi, x_1,x_2$ subject to
\begin{gather*}
e_1+e_2=1, \quad e_ie_j=\delta_{ij}e_i\;(i,j=1,2), \\
x_1^te_1=0, \quad x_1e_2=0 ,\\
\psi^2e_1=(x_1^2-x_2)e_1, \quad \psi^2e_2=(x_2^2-x_1)e_2=x_2^2e_2, \\
\psi e_1=e_2\psi, \quad e_1\psi=\psi e_2, \\
x_1x_2=x_2x_1, \quad x_ie_j=e_jx_i\;(i,j=1,2), \\
\psi x_1=x_2\psi, \quad x_1\psi=\psi x_2.
\end{gather*}
Note that $x_2^2e_1=x_2(x_1^2e_1-\psi^2e_1)=x_1^2x_2e_1-\psi x_1e_2\psi=x_1^2x_2e_1$. Then, since $e_1Ae_1$ is spanned by $x_1^ax_2^be_1$, for $a,b\ge0$, and $x_1^te_1=0$, we have 
\[
e_1Ae_1=\langle x_1^ax_2^be_1 | 0\le a\le t-1,\; 0\le b\le 1\rangle.
\]
Then, using $\deg x_1e_1=2$, $\deg x_2e_1=4$ and the formula for $\dim_q e_1Ae_1$, we know that they form a basis of $e_1Ae_1$. 

Nextly, since $e_2Ae_2$ is spanned by $x_2^be_2$, for $b\ge0$, because $x_1e_2=0$, and
\[
x_2^{t+2}e_2=x_2^t\psi^2e_2=\psi x_1^te_1\psi=0,
\]
we obtain $e_2Ae_2=\langle x_2^be_2 | 0\le b\le t+1\rangle$. 
By $\deg x_2e_2=2$ and the formula for $\dim_q e_2Ae_2$, they form a basis of $e_2Ae_2$. 

By $\psi x_2^te_2=x_1^te_1\psi e_2=0$ and the formula for $\dim_q e_1Ae_2$, we have a basis for $e_1Ae_2$ as $e_1Ae_2=\langle \psi x_2^be_2 | 0\le b\le t-1 \rangle$. Similarly, $e_2Ae_1=\langle \psi x_1^ae_1 | 0\le a\le t-1\rangle$.
If we set 
\[
\alpha=x_1e_1, \;\; \mu=e_1\psi e_2, \;\; \nu=e_2\psi e_1, \;\; \beta=x_2e_2.
\]
Then
\begin{gather*}
\alpha^t=x_1^te_1=0, \;\; \beta^{t+2}=x_2^{t+2}e_2=0, \;\; \beta^2-\nu\mu=x_2^2e_2-\psi^2e_2=0, \\
\alpha\mu-\mu\beta=e_1(x_1\psi-\psi x_2)e_2=0, \;\; \beta\nu-\nu\alpha=e_2(x_2\psi-\psi x_1)e_1=0.
\end{gather*}
Moreover, $\{\alpha, \beta, \mu, \nu\}$ generate $A$ as an algebra.

\begin{lemma}\label{t>2:alpha_0+alpha_1}
Let $A'$ be the algebra with two vertices $1,2$, a loop $\alpha$ at vertex $1$, a loop $\beta$ at vertex $2$, an arrow $\mu$ from vertex $1$ to vertex $2$, 
an arrow $\nu$ from vertex $2$ to vertex $1$, such that they satisfy the following relations
\[
\alpha^t=0, \;\; \beta^{t+2}=0, \;\; \beta^2=\nu\mu, \;\; \alpha\mu=\mu\beta, \;\; \beta\nu=\nu\alpha.
\]
If $t\ge 3$, then $A'$ is isomorphic to $A$. 
\end{lemma}
\begin{proof}
By mapping the generators of the same name, we have a surjective algebra homomorphism $A'\to A$. Hence it suffices to show that $\dim A'=5t+2$. 

First, $\Rad^s(e_2A')/\Rad^{s+1}(e_2A')$ is spanned by $\{\beta^s, \nu\alpha^{s-1}\}$, for $1\le s\le t$. If $s=1$, it is clear. Suppose that the assertion holds for $s$. Then, if we multiply $\nu\alpha^{s-1}$ with $\alpha$ on the right, we obtain $\nu\alpha^s$, and if we multiply $\nu\alpha^{s-1}$ with $\mu$ on the right, 
we obtain
\[
\nu\alpha^{s-1}\mu=\nu\alpha^{s-2}\mu\beta=\cdots=\nu\mu\beta^{s-1}=\beta^{s+1}.
\]
On the other hand, if we multiply $\beta^s$ with $\beta$ on the right, we obtain $\beta^{s+1}$, and if we multiply $\beta^s$ with $\nu$ on the right, we obtain
\[
\beta^s\nu=\beta^{s-1}\nu\alpha=\cdots=\nu\alpha^s.
\]
Hence, $\Rad^{s+1}(e_2A')/\Rad^{s+2}(e_2A')$ is spanned by $\beta^{s+1}$ and $\nu\alpha^s$. Now, $\beta^t\nu=\nu\alpha^t=0$ and $\nu\alpha^{t-1}\mu=\beta^{t+1}$ imply $\dim \Rad^{t+1}(e_2A')=1$, which is spanned by $\beta^{t+1}$, and $\beta^{t+1}\beta=0$, $\beta^{t+1}\nu=0$. Then, $\dim e_2A'=2t+2$ follows. 

Second, it is clear that $\Rad(e_1A')/\Rad^2(e_1A')$ is spanned by $\{\alpha, \mu\}$. We show that if $t\ge 3$, then $\Rad^s(e_1A')/\Rad^{s+1}(e_1A')$ is spanned by $\{\alpha^s, \alpha^{s-1}\mu, \alpha^{s-2}\mu\nu\}$, for $2\le s\le t-1$. 
If $s=2$, then $\{\alpha^2, \alpha\mu=\mu\beta, \mu\nu\}$ spans $\Rad^2(e_1A')/\Rad^3(e_1A')$. 
\begin{itemize}
\item[(i)]
If we multiply $\alpha^s$ with $\alpha$ and $\mu$ on the right, we obtain $\alpha^{s+1}$ and $\alpha^s\mu$.
\item[(ii)]
Observe that $\alpha\mu\nu=\mu\beta\nu=\mu\nu\alpha$. 
If we multiply $\alpha^{s-2}\mu\nu$ with $\alpha$ and $\mu$ on the right, we obtain $\alpha^{s-1}\mu\nu$ and 
\[
\alpha^{s-2}\mu\nu\mu=\alpha^{s-2}\mu\beta^2=\alpha^{s-2}\alpha\mu\beta=\alpha^{s-2}\alpha^2\mu=\alpha^s\mu.
\]
\item[(iii)]
If we multiply $\alpha^{s-1}\mu$ with $\beta$ and $\nu$ on the right, we obtain $\alpha^{s-1}\mu\beta=\alpha^s\mu$ and 
$\alpha^{s-1}\mu\nu$.
\end{itemize}
Hence $\Rad^{s+1}(e_1A')/\Rad^{s+2}(e_1A')$ is spanned by $\{\alpha^{s+1}, \alpha^s\mu, \alpha^{s-1}\mu\nu\}$, as long as $2\le s\le t-2$. 
Now we multiply $\alpha^{t-1}$, $\alpha^{t-2}\mu$, $\alpha^{t-3}\mu\nu$ with $\Rad(A')$ on the right. 
\begin{itemize}
\item[(i)]
If we multiply $\alpha^{t-1}$ with $\alpha$ and $\mu$ on the right, we obtain $\alpha^t=0$ and $\alpha^{t-1}\mu$.
\item[(ii)]
If we multiply $\alpha^{t-3}\mu\nu$ with $\alpha$ and $\mu$ on the right, we obtain $\alpha^{t-2}\mu\nu$ and 
\[
\alpha^{t-3}\mu\nu\mu=\alpha^{t-3}\mu\beta^2=\alpha^{t-3}\alpha\mu\beta=\alpha^{t-3}\alpha^2\mu=\alpha^{t-1}\mu.
\]
\item[(iii)]
If we multiply $\alpha^{t-2}\mu$ with $\beta$ and $\nu$ on the right, we obtain $\alpha^{t-2}\mu\beta=\alpha^{t-1}\mu$ and $\alpha^{t-2}\mu\nu$.
\end{itemize}
Hence $\Rad^t(e_1A')/\Rad^{t+1}(e_1A')$ is spanned by $\{\alpha^{t-1}\mu, \alpha^{t-2}\mu\nu\}$. Now, we multiply $\alpha^{t-1}\mu$ with 
$\beta$ and $\nu$ to obtain $\alpha^{t-1}\mu\beta=\alpha^t\mu=0$ and $\alpha^{t-1}\mu\nu$, we multiply $\alpha^{t-2}\mu\nu$ with 
$\alpha$ and $\mu$ on the right to obtain
\[
\alpha^{t-2}\mu\nu\alpha=\alpha^{t-1}\mu\nu, \;\; \alpha^{t-2}\mu\nu\mu=\alpha^{t-2}\mu\beta^2=\alpha^{t-1}\mu\beta=\alpha^t\mu=0.
\]
Hence, $\dim \Rad^{t+1}(e_1A')=1$, which is spanned by $\alpha^{t-1}\mu\nu$, 
and $\alpha^{t-1}\mu\nu\alpha=0$, $\alpha^{t-1}\mu\nu\mu=\alpha^{t-1}\mu\beta^2=\alpha^t\mu\beta=0$. It follows that $\dim e_1A'=3t$. 
Hence we have proved $\dim A'=\dim e_1A'+\dim e_2A'=(2t+2)+3t=5t+2=\dim A$.
\end{proof}

Recall the wild algebra (31) from \cite[Table W]{H-wild-two-point}, which has the same quiver as $A$ and is bounded by
\[
\beta\nu=\nu\alpha,\;\; \beta^2=\nu\mu=\mu\beta=\alpha\mu=\alpha^3=\nu\alpha^2=0.
\]
It is clear that if $t\ge3$ then the following relations hold in this algebra. 
\[
\alpha^t=0, \;\; \beta^{t+2}=0, \;\; \beta^2=\nu\mu, \;\; \alpha\mu=\mu\beta, \;\; \beta\nu=\nu\alpha.
\]
Hence, $A$ has the wild algebra as a factor algebra, so that $A$ is wild if $t\ge3$.

\begin{lemma}\label{t=2:alpha_0+alpha_1}
Let $A'$ be the algebra with two vertices $1,2$, a loop $\alpha$ on vertex $1$, a loop $\beta$ on vertex $2$, an arrow $\mu$ from vertex $1$ to vertex $2$, 
an arrow $\nu$ from vertex $2$ to vertex $1$, such that they are bounded by the relations
\[
\alpha^2=0, \;\; \beta^2=\nu\mu, \;\; \alpha\mu=\mu\beta, \;\; \beta\nu=\nu\alpha.
\]
If $t=2$, then $A'$ is isomorphic to $A$.
\end{lemma}
\begin{proof}
Recall the defining relations of $A$ when $t=2$. 
\begin{gather*}
e_1+e_2=1, \;\; e_ie_j=\delta_{ij}e_i\;(i,j=1,2) \\
x_1^2e_1=0, \;\; x_1e_2=0 \\
\psi^2e_1=(x_1^2-x_2)e_1=-x_2e_1, \;\; \psi^2e_2=(x_2^2-x_1)e_2=x_2^2e_2 \\
\psi e_1=e_2\psi, \;\; e_1\psi=\psi e_2 \\
x_1x_2=x_2x_1, \;\; x_ie_j=e_jx_i\;(i,j=1,2) \\
\psi x_1=x_2\psi, \;\; x_1\psi=\psi x_2
\end{gather*}
Then, $\alpha=x_1e_1$, $\beta=x_2e_2$, $\mu=e_1\psi e_2$, $\nu=e_2\psi e_1$ satisfy
\[
\alpha^2=0, \;\; \beta^2=\nu\mu, \;\; \alpha\mu=\mu\beta, \;\; \beta\nu=\nu\alpha.
\]
Moreover, they generate $A$, so that we have a surjective algebra homomorphism $A'\to A$ as before. The computation of $\dim e_2A'$ does not change and we obtain $\dim e_2A'=6$. We compute $\dim e_1A'$. It is clear that $\Rad(e_1A')/\Rad^2(e_1A')$ is spanned by $\{ \alpha, \mu\}$. 
\begin{itemize}
\item[(i)]
If we multiply $\alpha$ with $\alpha$ and $\mu$ on the right, we obtain $\alpha^2=0$ and $\alpha\mu$.
\item[(ii)]
If we multiply $\mu$ with $\beta$ and $\nu$ on the right, we obtain $\mu\beta=\alpha\mu$ and $\mu\nu$.
\end{itemize}
Hence $\Rad^2(e_1A')/\Rad^3(e_1A')$ is spanned by $\{ \alpha\mu, \mu\nu\}$. Now,
\begin{itemize}
\item[(i)]
If we multiply $\alpha\mu$ with $\beta$ and $\nu$ on the right, we obtain $\alpha\mu\beta=\alpha^2\mu=0$ and $\alpha\mu\nu$.
\item[(ii)]
If we multiply $\mu\nu$ with $\alpha$ and $\mu$ on the right, we obtain $\mu\nu\alpha=\mu\beta\nu=\alpha\mu\nu$ and $\mu\nu\mu=\mu\beta^2=\alpha^2\mu=0$.
\end{itemize}
Thus, $\dim \Rad^3(e_1A')=1$, which is spanned by $\alpha\mu\nu$, and 
\[
\alpha\mu\nu\alpha=\alpha^2\mu\nu=0, \;\; \alpha\mu\nu\mu=\alpha\mu\beta^2=\alpha^3\mu=0.
\]
Hence $\dim e_1A'=6$, We conclude that $\dim A'=6+6=12=5t+2=\dim A$. 
\end{proof}

Recall the algebra (18) from \cite[Table T]{H-wild-two-point} which has the same quiver as $A$ and is bounded by
\[
\alpha^2=\nu\mu=\mu\beta=\beta\nu=0.
\]
We define a family of $10$ dimensional radical cube zero algebras $A_\xi$, for $\xi\in \k$, by 
\[
\alpha^2=0,\;\; \xi\beta^2=\nu\mu,\;\; \xi\nu\alpha=\beta\nu,\;\; \xi\alpha\mu=\mu\beta, \;\; \Rad^3(A_\xi)=0.
\]
The algebra $A_0$ is a factor algebra of the algebra (18), so that $A_0$ is tame.\footnote{By the shape of the Gabriel quiver, $A_0/\Rad^2(A_0)$ is stably equivalent to the path algebra $\k A^{(1)}_3$ with zigzag orientation, so that the factor algebra is not representation-finite.} If $\xi\ne0$, we change the generators of $A_\xi$ to
\[
\alpha'=\alpha, \;\; \beta'=\xi^{-1}\beta,\;\; \mu'=\xi^{-2}\mu,\;\; \nu'=\xi^{-1}\nu.
\]
Then, the relations with respect to the new generators are ${\alpha'}^2=0$, 
\begin{align*}
{\beta'}^2&=\xi^{-2}\beta^2=\xi^{-3}\nu\mu=\nu'\mu', \\
\alpha'\mu'&=\xi^{-2}\alpha\mu=\xi^{-3}\mu\beta=\mu'\beta', \\
\beta'\nu'&=\xi^{-2}\beta\nu=\xi^{-1}\nu\alpha=\nu'\alpha'.
\end{align*}
and $\Rad^3(A_\xi)=0$. Hence $A_\xi\cong A/\Rad^3(A)$ when $\xi\ne0$. We have shown that $A/\Rad^3(A)$ degenerates to $A_0$. Since $A_0$ is tame, $A/\Rad^3(A)$ is tame. 
Observe that $A$ is a symmetric algebra and $\Rad^3(A)=\Soc(A)$. If an indecomposable $A$-module $M$ has radical length $4$, there is an injective $A$-module homomorphism $P_1\to M$ or $P_2\to M$, which splits because indecomposable projective $A$-modules $P_1$ and $P_2$ are injective $A$-modules. 
Thus, $M\cong P_1$ or $M\cong P_2$. This implies that the representation type of $A$ and $A/\Rad^3(A)$ coincide. We have proved that $A$ is tame if $t=2$. 

\subsection{Proof of the tame cases}
We are ready to prove part (b) in the second part of MAIN THEOREM.
The cases (t1)-(t9) will appear in $R^\Lambda(\beta_{\Lambda'})$, for the first neighbor $\Lambda'$, that is, those $\Lambda'$ for which there is an arrow $\Lambda\to\Lambda'$. As we see below, they are Brauer graph algebra except for (t7) and (t8). All the other cases will appear in $R^\Lambda(\beta_{\Lambda''})$, for the second neighbor $\Lambda''$, namely those $\Lambda''$ for which there is a directed path $\Lambda\to\Lambda'\to\Lambda''$. 

In the cases (t9), (t15)-(t19), we have the isomorphism of algebras $R^\Lambda(\beta)\cong R_A^{\Lambda}(\beta)$. Hence, the results follow from \cite{ASW-rep-type}. For the bound quiver presentation of the cases (t9), (t15)-(t19), see \cite[8.2]{ASW-rep-type}.
Furthermore, it suffices to consider (t2), (t3), (t5), (t7), (t10), (t12), (t13), (t20) in the remaining cases by symmetry.  Cases except for (t2) and (t20) are almost complete already.
\begin{enumerate}
\item[(t3)]
This follows from Lemma \ref{Brauer graph algebra case 1}.
\item[(t5)]
We have $R^\Lambda(\beta)\cong R^{m_0\Lambda_0+\Lambda_a}(\alpha_0+\cdots+\alpha_a)$, for $1\le a\le \ell-1$. If $a=1$ and $m_0\ge2$, it follows from Lemma \ref{Brauer graph algebra case 1}. If $2\le a\le \ell-1$, then it follows from Lemma \ref{Brauer graph algebra case 2}. 
\item[(t7)]
This follows from Lemma \ref{t=2:alpha_0+alpha_1}.
\item[(t10)]
By Lemma \ref{tensor product lemma}, $R^\Lambda(\beta)$ is Morita equivalent to
$$
R^{2\Lambda_0}(\alpha_0)\otimes R^{2\Lambda_i}(\alpha_i)\cong \k[X,Y]/(X^2,Y^2),
$$
which is tame by Proposition \ref{prop::tame-local-alg}.
\item[(t12)]
Since $\ell\ge 4$, we may apply Lemma \ref{tensor product lemma}. Hence, $m_1=m_{\ell-1}=0$ implies that 
$R^\Lambda(\beta)$ is Morita equivalent to
$$ R^{\Lambda_0}(\alpha_0+\alpha_1)\otimes R^{\Lambda_\ell}(\alpha_{\ell-1}+\alpha_\ell)\cong \k[X,Y]/(X^2,Y^2).
$$
Here, we use the proof of (f2) for each of $R^{\Lambda_0}(\alpha_0+\alpha_1)$ and $R^{\Lambda_\ell}(\alpha_{\ell-1}+\alpha_\ell)$ to obtain $\k[X,Y]/(X^2,Y^2)$.
\item[(t13)]
We apply Lemma \ref{tensor product lemma} again. Then $m_1=0$ implies that $R^\Lambda(\beta)$ is Morita equivalent to
$R^{\Lambda_0}(\alpha_0+\alpha_1)\otimes R^{2\Lambda_i}(\alpha_i)$. Then, 
we use the proof of (f2) again to conclude that $R^\Lambda(\beta)$ is Morita equivalent to $\k[X,Y]/(X^2,Y^2)$.
\end{enumerate}

In the next two subsections, we prove the remaining cases (t2) and (t20).

\subsection{The case (t2)}
Set $A:=R^{2\Lambda_{\ell-1}}(2\alpha_{\ell-1}+\alpha_\ell)$ with
$$
e_1=e(\ell-1, \ell, \ell-1),\quad e_2=e(\ell-1, \ell-1, \ell), \quad e'_2=x_2\psi_1e_{2}.
$$
We then have the following graded dimensions.
\begin{align*}
\dim_qe_1Ae_1&=1+2q^2+q^4,\\
\dim_qe_2Ae_2&=(q+q^{-1})^2(1+q^4),\\
\dim_qe_1Ae_2&=\dim_qe_2Ae_1=(q+q^{-1})(q+q^3).
\end{align*}
Let $P_1:=Ae_1$ and $P_2:=Ae_2'\langle 1\rangle$. 
By looking at the graded dimensions, we know that $Ae_2=P_2\langle1\rangle\oplus P_2\langle{-}1\rangle$ and
\begin{center}
$\dim_q\End(P_1)=1+2q^2+q^4, \quad 
\dim_q\End(P_2)=1+q^4$,\\
$\dim_q\Hom(P_1,P_2)=\dim_q\Hom(P_2,P_1)=q+q^3$.
\end{center}
By crystal computation \cite{LV-crystal computation}, we can calculate the number of simple modules which is two. Indeed, we have a one-dimensional irreducible representation $D_1$ given by
\[
x_1,x_2,x_3,\psi_1,\psi_2,e_2\mapsto0,\;e_1\mapsto1
\]
and a two-dimensional irreducible representation $D_2$ given by
\begin{equation*}
\begin{split}
e_1,\psi_2,x_3\mapsto\begin{pmatrix}
0 & 0\\
0 & 0
\end{pmatrix},\;
e_2&\mapsto\begin{pmatrix}
1 & 0\\
0 & 1
\end{pmatrix}  ,\;
\psi_1\mapsto\begin{pmatrix}
0 & 0\\
1 & 0
\end{pmatrix} , \\[5pt]
x_1\mapsto\begin{pmatrix}
0 & {-}1\\
0 & 0
\end{pmatrix}  &,\;
x_2\mapsto\begin{pmatrix}
0 & 1\\
0 & 0
\end{pmatrix} .
\end{split}
\end{equation*}
Thus we can compute $\Ext^1_A(D_1,D_1)=1$ by forcing
\begin{equation*}
\begin{split}
x_1=\begin{pmatrix}
0 & a\\
0 & 0
\end{pmatrix},\;
x_2=&\begin{pmatrix}
0 & b\\
0 & 0
\end{pmatrix},\;
x_3=\begin{pmatrix}
0 & c\\
0 & 0
\end{pmatrix},\\[5pt]
\psi_1=\begin{pmatrix}
0 & s\\
0 & 0
\end{pmatrix},\;
\psi_2=\begin{pmatrix}
0 & t\\
0 & 0
\end{pmatrix}&,\;
e_1=\begin{pmatrix}
1 & \alpha\\
0 & 1
\end{pmatrix},\;
e_2=\begin{pmatrix}
0 & \beta\\
0 & 0
\end{pmatrix}
\end{split}
\end{equation*}
which satisfy the defining relations. Similarly, we compute $\Ext^1_A(D_1,D_2)=1$ by taking
\begin{equation*}
\begin{split}
x_1=\begin{pmatrix}
0 & {-}1&0\\
0 & 0&0\\
0 & 0&0
\end{pmatrix}&,\;
x_2=\begin{pmatrix}
0 & 1&0\\
0 & 0&0\\
0 & 0&0
\end{pmatrix},\;
x_3=\begin{pmatrix}
0 & 0&0\\
0 & 0&0\\
0 & 0&0
\end{pmatrix},\;\\[5pt]
\psi_1=\begin{pmatrix}
0 & 0&0\\
1 & 0&0\\
0 & 0&0
\end{pmatrix},\;
\psi_2=&\begin{pmatrix}
0 & 0&\alpha\\
0 & 0&0\\
0 & 0&0
\end{pmatrix},\;
e_1=\begin{pmatrix}
0 & 0&0\\
0 & 0&0\\
0 & 0&1
\end{pmatrix},\;
e_2=\begin{pmatrix}
1 & 0&0\\
0 & 1&0\\
0 & 0&0
\end{pmatrix}.
\end{split}
\end{equation*}
Hence, we obtain the following radical series:
\[
P_1=\begin{array}{c} D_1 \\ D_1\oplus D_2 \\ D_2\oplus D_1\\ D_1 \end{array}, \quad
P_2=\begin{array}{c} D_2 \\ D_1 \\ D_1\\ D_2 \end{array}.
\]
To obtain its bound quiver presentation and to show that it is a Brauer graph algebra, we follow the argument in the proof of \cite[Theorem 3.7]{AP-rep-type-C-level-1}. For this, we need a uniserial submodule $Q$ of $P_1$ which gives a non-split exact sequence 
$$ 0\rightarrow Q \rightarrow P_1 \rightarrow Q \rightarrow 0. $$ Let us check the existence of such a submodule in our case. 
Recall the restriction functor $e_{i}$ and induction $f_{i}$ of $R^\Lambda(\beta)$-mod, $i=\ell-1,\ell$. 
Let 
$S_1:=e_{\ell-1}D_1$. 
Note also that  $e_{\ell-1}D_2=0$.
Since $\varepsilon_{\ell-1}(D_1)=1$, 
$S_1$ is a simple $R^\Lambda(\alpha_{\ell-1}+\alpha_\ell)$-module  (e.g., \cite[Lemma~3.2]{AP-rep-type-C-level-1}).
Considering the action of the Weyl group, we have 
$r_\ell(2\Lambda_{\ell-1}-\alpha_{\ell-1})=2\Lambda_{\ell-1}-\alpha_\ell-\alpha_{\ell-1}$.
Thus, 
$R^\Lambda(\alpha_{\ell-1}+\alpha_\ell)$ is derived equivalent to the local algebra $R^\Lambda(\alpha_{\ell-1})\cong \k[X]/(X^2)$ and hence  $R^\Lambda(\alpha_{\ell-1}+\alpha_\ell)$ is Morita equivalent to $\k[X]/(X^2)$.
Therefore, $S_1$ is the unique simple module of $R^\Lambda(\alpha_{\ell-1}+\alpha_\ell)$. 
Let $\hat S_1$ be the projective cover of $S_1$.
Then we have 
\begin{equation}\label{exactse}
    0\rightarrow S_1\rightarrow \hat S_1\rightarrow S_1\rightarrow0,
\end{equation}
which is non-split. Moreover, $f_{\ell-1} \hat S_1$ is a projective $A$-module.
We have 
$$ \dim \Hom(f_{\ell-1}\hat S_1,D_i)
=\dim \Hom(\hat S_1, e_{\ell-1}D_i)=
\begin{cases} 
1 \quad &(i=1), \\
0 \quad &(i=2). 
\end{cases}
$$
A similar result holds for $\dim \Hom(D_i, f_{\ell-1}\hat S_1)$. This implies that $f_{\ell-1}\hat S_1\cong P_1$. Now we set $Q:=f_{\ell-1}S_1$ and apply $f_{\ell-1}$ to the non-split sequence \eqref{exactse}. Since $P_1$ is indecomposable, the resulting short exact sequence is non-split. 
Hence, the Gabriel quiver is 
\[
\begin{xy}
(0,0) *[o]+{\circ}="A", (10,0) *[o]+{\circ}="B", 
\ar @(lu,ld)  "A";"A"_{\alpha}
\ar @<1mm> "A";"B"^{\mu}
\ar @<1mm> "B";"A"^{\nu}
\end{xy}
\]
and the relations are $\nu\mu=\alpha^2=0$ and $\alpha\mu\nu=\mu\nu\alpha$.

We see that it is a special biserial algebra 
\footnote{See \cite{Er-tame-block} for the definition of special biserial algebra. It is known that symmetric special biserial algebras are Brauer graph algebras and vice versa. See \cite{Schroll-Brauer-graph}.} Being a symmetric algebra, it is a Brauer graph algebra, whose Brauer graph is as claimed.

\subsection{The case (t20)}
We show that the algebra (t20), namely $A:=R^{2\La_0}(2\alpha_0+2\alpha_1)$ in $\ch \k\ne2$, is tame. 
First of all, crystal computation shows that the number of simple modules is two. Its basic algebra is 
$B=\End(P_1\oplus P_2)^{\rm op}$ where 
\[
P_1=f_1^{(2)}f_0^{(2)}v_\Lambda, \quad P_2=f_0f_1^{(2)}f_0v_\Lambda. 
\]
Let $e_1=e(0011)$ and $e_2=e(0110)$ and $e_3=e(0101)$. 
Graded dimension formula computes
\begin{gather*}
\dim_q e_1Be_1=1+q^2+2q^4+q^6+q^8, \;\; \dim_q e_2Be_2=1+2q^4+q^8, \\
\dim_q e_1Be_2=\dim_q e_2Be_1=q^2+q^6.
\end{gather*}
We set $f_1=x_2\psi_1x_4\psi_3e_1$ and $f_2=x_3\psi_2e_2$. Then, 
$P_1=Af_1\langle 3\rangle$ and $P_2=Af_2\langle 1\rangle$. Thus, 
the graded dimensions of $f_iAf_j$, for $i,j=1,2$, are as follows. 
\begin{align*}
\dim_q f_1Af_2&=\dim_q \Hom_A(Af_1, Af_2)=\dim_q \Hom_A(P_1\langle-3\rangle, P_2\langle -1\rangle) \\
&=\dim_q \Hom_A(P_1,P_2)\langle 2\rangle=q^4+q^8, \\
\dim_q f_2Af_1&=\dim_q \Hom_A(Af_2, Af_1)=\dim_q \Hom_A(P_2\langle-1\rangle, P_1\langle -3\rangle) \\
&=\dim_q \Hom_A(P_2,P_1)\langle -2\rangle=1+q^4, \\
\dim_q f_1Af_1&=\dim_q \Hom_A(Af_1, Af_1)=\dim_q \Hom_A(P_1\langle-3\rangle, P_1\langle -3\rangle) \\
&=\dim_q \Hom_A(P_1,P_1)=1+q^2+2q^4+q^6+q^8, \\
\dim_q f_2Af_2&=\dim_q \Hom_A(Af_2, Af_2)=\dim_q \Hom_A(P_2\langle-1\rangle, P_2\langle -1\rangle) \\
&=\dim_q \Hom_A(P_2,P_2)=1+2q^4+q^8. 
\end{align*}

Let $f=f_1+f_2$. Then $B$ is isomorphic to $fAf$ as ungraded algebras, and 
we are going to prove the tameness of $A$ 
by obtaining the bound quiver presentation of $fAf$. 
The computation is lengthy and not straightforward. 
We start with formulas we will use in the computation.

\begin{lemma}\label{computation_for_t2}
The following formulas hold.
\begin{itemize}
\item[(1)]
$(x_1+x_2)e_1=0$, $(x_2+x_3)e_2=0$, $(x_1+x_3)e_3=0$, $x_1e_2=x_2^2e_2=x_3^2e_2$.
\item[(2)]
$x_3^4e_1=0$, $x_4^2e_2=0$, $(x_3^3+x_3^2x_4+x_3x_4^2+x_4^3)e_1=0$, $(x_3x_4^3+x_3^2x_4^2+x_3^3x_4)e_1=0$.
\item[(3)]
$f_1\psi_1=0$, $f_2\psi_2=0$, $f_1\psi_3=0$.
\item[(4)]
$(x_3+x_4)f_1=f_1(x_3+x_4)$, $x_3x_4f_1=f_1x_3x_4$, $x_1f_2=f_2x_1$ and $x_4f_2=f_2x_4$. 
\item[(5)]
$x_1f_1=0$, $f_1x_3f_1=0$, $f_1x_3^2f_1=-x_3x_4f_1$, $f_1x_3^3f_1=-(x_3+x_4)x_3x_4f_1$. 
\item[(6)]
$f_2x_3f_2=0$. 
\end{itemize}
\end{lemma}
\begin{proof}

\noindent
(1) First, $x_1^2e_1=0$ implies $\partial_1(x_1^2)e_1=0$ by Lemma \ref{KK-lemma}. 
Hence $(x_1+x_2)e_1=0$. 

\noindent
Nextly, $\psi_1e_2=0$ implies $(x_1-x_2^2)e_2=0$. Thus, 
$$ \left(\partial_2(x_1-x_2^2)\right)e_2=-(x_2+x_3)e_2=0. $$

\noindent
Finally, $\psi_1e_3=0$ implies $\psi_2\psi_1\psi_2e_3=\psi_1\psi_2\psi_1e_3+e_3=e_3$. Together with $x_2e_1=-x_1e_1$, we obtain
$$
x_3e_3=(x_3\psi_2)e_1\psi_1\psi_2e_3=\psi_2(x_2e_1)\psi_1\psi_2e_3
=-x_1\psi_2\psi_1\psi_2e_3=-x_1e_3.
$$

\noindent
(2)  Observe that $$ x_4\psi_3^2e_2=x_4(\psi_3e_3)\psi_3=\psi_3(x_3e_3)\psi_3=-\psi_3(x_1e_3)\psi_3=-x_1\psi_3^2e_2. $$
Hence $(x_3^2-x_4)(x_1+x_4)e_2=0$.  Since $x_3^2e_2=x_1e_2$, we obtain
$$ x_4^2e_2=(x_1x_3^2-x_1x_4+x_3^2x_4)e_2=(x_1^2-x_1x_4+x_1x_4)e_2=0. $$
On the other hand, $\psi_1e_3=0$ implies $x_1e_3=x_2^2e_3$. Then
$$ x_3^2\psi_2^2e_1=x_3^2(\psi_2e_3)\psi_2=\psi_2(x_2^2e_3)\psi_2=\psi_2x_1\psi_2e_1=x_1\psi_2^2e_1. $$
Hence $(x_1-x_3^2)(x_2-x_3^2)e_1=0$ and 
$$ x_3^4e_1=(x_1+x_2)x_3^2e_1-x_1x_2e_1=(x_1+x_2)e_1x_3^2+x_1^2e_1=0. $$
Moreover, $\partial_3(x_3^4)e_1=(x_3^3+x_3^2x_4+x_3x_4^2+x_4^3)e_1=0$ by Lemma \ref{KK-lemma}. Multiplying it with $x_3$, we obtain $(x_3^3x_4+x_3^2x_4^2+x_3x_4^3)e_1=0$. 

\noindent
(3) $f_1\psi_1=x_2\psi_1x_4\psi_3e_1\psi_1=x_2x_4\psi_1^2e_1\psi_3=0$. Similarly, we obtain
$f_2\psi_2=0$ and $f_1\psi_3=0$. 

\noindent
(4)  $x_3+x_4$ and $x_3x_4$ commute with $\psi_3$. Thus they commute with $f_1=x_2\psi_1x_4\psi_3e_1$. The proof of 
$x_1f_2=f_2x_1$ and $x_4f_2=f_2x_4$ is straightforward. 

\noindent
(5) $x_1f_1=(x_1x_2)e_1\psi_1x_4\psi_3e_1=-x_1^2e_1\psi_1x_4\psi_3e_1=0$. 
\begin{align*}
f_1x_3f_1&=(x_2\psi_1x_4)\psi_3(x_3x_2\psi_1x_4)\psi_3e_1
=(x_2\psi_1x_4)\psi_3(x_2x_3x_4)\psi_1\psi_3e_1 \\
&=x_2\psi_1(x_2x_3x_4^2)\psi_3\psi_1\psi_3e_1=x_2\psi_1(x_2x_3x_4^2)\psi_1\psi_3^2e_1=0. \\
f_1x_3^2f_1&= f_1(x_3+x_4)x_3f_1-f_1x_3x_4f_1=(x_3+x_4)f_1x_3f_1-x_3x_4f_1=-x_3x_4f_1. \\
f_1x_3^3f_1&=f_1(x_3+x_4)^2x_3f_1-2f_1x_3^2x_4f_1-f_1x_3x_4^2f_1 \\
&=(x_3+x_4)^2f_1x_3f_1-2x_3x_4f_1x_3f_1-x_3x_4f_1x_4f_1 \\
&=-x_3x_4f_1(x_3+x_4)f_1=-(x_3+x_4)x_3x_4f_1.
\end{align*}
\noindent
(6) Using $x_3^2e_2=x_2^2e_2$ and $x_1e_2=x_2^2e_2$, we obtain 
\[
f_2x_3f_2=x_3\psi_2x_3^2\psi_2e_2=x_3\psi_2x_2^2\psi_2e_2
=x_3\psi_2x_1\psi_2e_2=x_1x_3\psi_2^2e_2=0.
\]
We have proved the formulas. 
\end{proof}

\begin{proposition}\label{BGA 4-2-2}
The bases of $f_iAf_j$ $(i,j=1,2)$ are given as follows. 
\[
\begin{aligned}
f_1Af_1&={\rm span}\{f_1,\alpha=(x_3+x_4)f_1, \alpha'=x_3x_4f_1, \alpha^2, \alpha\alpha', \alpha^2\alpha'\}, \\
f_2Af_2&={\rm span}\{f_2, \beta=x_1f_2, \beta'=x_4f_2, \beta\beta'=\beta'\beta \}, \\
f_1Af_2&={\rm span}\{\mu=f_1\psi_2\psi_3f_2, f_1\psi_2\psi_3x_1f_2=\mu\beta\}, \\
f_2Af_1&={\rm span}\{\nu=f_2\psi_3\psi_2\psi_1f_1, f_2x_1\psi_3\psi_2\psi_1f_1=\beta\nu\}.
\end{aligned}
\]
Moreover, $\alpha^3=2\alpha\alpha'$ and ${\alpha'}^2=\alpha^2\alpha'$ hold. 
\end{proposition}
\begin{proof}
(i) We begin by $f_1Af_1$. $f_1\psi_1=0$ and $f_1\psi_3=0$ imply 
$$ f_1Af_1={\rm span}\{ f_1x_1^ax_2^bx_3^cx_4^df_1 \mid a,b,c,d\in \Z_{\ge0} \}. $$
Then, $x_1f_1=0$ and $x_2f_1=-x_1f_1=0$ imply that we may assume $a=b=0$. 
\begin{enumerate}
\item If $\deg f_1x_3^cx_4^df_1=2$, then 
$(c,d)=(1,0), (0,1)$ and $f_1x_3f_1=0$ implies 
that the degree $2$ component of $f_1Af_1$ has the basis $\{\alpha\}$. 
\item If $\deg f_1x_3^cx_4^df_1=4$, then 
$f_1x_3^2f_1=-\alpha'$, $f_1x_3x_4f_1=\alpha'$ and
$$ f_1x_4^2f_1=f_1(x_3+x_4)^2f_1-f_1x_3^2f_1-2f_1x_3x_4f_1=\alpha^2+\alpha'-2\alpha'=\alpha^2-\alpha'. $$
Thus, the degree $4$ component of $f_1Af_1$ has the basis $\{\alpha^2, \alpha'\}$. 
\item If $\deg f_1x_3^cx_4^df_1=6$, then $f_1x_3^3f_1=-\alpha\alpha'$, 
$f_1x_3^2x_4f_1=x_3x_4f_1x_3f_1=0$, $f_1x_3x_4^2f_1=x_3x_4f_1x_4f_1=\alpha'\alpha$, and 
\begin{align*}
f_1x_4^3f_1&=-f_1x_3x_4^2f_1-f_1x_3^2x_4f_1-f_1x_3^3f_1 \\
&=-x_3x_4f_1(x_3+x_4)f_1-x_3x_4f_1x_3f_1+(x_3+x_4)x_3x_4f_1=0.
\end{align*}
Thus, the degree $6$ component of $f_1Af_1$ has the basis $\{\alpha\alpha'\}$. Since
\[
f_1(x_3+x_4)^3f_1=f_1x_3^3f_1+3f_1x_3^2x_4f_1+3f_1x_3x_4^2f_1+f_1x_4^3f_1 \\
\]
We have the relation $\alpha^3=2\alpha\alpha'$ among $\{\alpha^3, \alpha\alpha'\}$. 
\item  If $\deg f_1x_3^cx_4^df_1=8$, then $x_3^4e_1=0$ implies $f_1x_3^4f_1=0$. 
On the other hand, $f_1(x_3+x_4)^3x_4f_1=0$ implies 
$f_1x_3^4f_1=-\alpha^2\alpha'+{\alpha'}^2$, so that we have the relation 
${\alpha'}^2=\alpha^2\alpha'$. Moreover, $(x_3^3+x_3^2x_4+x_3x_4^2+x_4^3)e_1=0$ implies
$$
f_1x_4^4f_1=-f_1(x_3^3x_4+x_3x_4^3)f_1-f_1x_3^2x_4^2f_1=f_1x_3^2x_4^2f_1-f_1x_3^2x_4^2f_1=0.
$$
$\alpha^3=2\alpha\alpha'$ implies $\alpha^4=2\alpha^2\alpha'$. 
We also compute
$$
f_1x_3^3x_4f_1=-\alpha^2\alpha',\;\; f_1x_3^2x_4^2f_1=\alpha^2\alpha', \;\;
f_1x_3x_4^3f_1=\alpha'(\alpha^2-\alpha')=0.
$$
We conclude that the degree $8$ component of $f_1Af_1$ has the basis $\{\alpha^2\alpha'\}$.
\end{enumerate}

\smallskip
\noindent
(ii) We turn to $f_2Af_2$. If $e_2\psi_we_2\ne0$ then 
$$ w\in \{ 1, s_2, s_3s_2s_1s_2s_3, s_2s_3s_2s_1s_2s_3, s_3s_2s_1s_2s_3s_2\}. $$
Since $f_2\psi_2=0$ and $\psi_2\psi_1\psi_2e_3=e_3$ imply
$$ 
e_2\psi_3\psi_2\psi_1\psi_2\psi_3e_2=e_2\psi_3(\psi_2\psi_1\psi_2)e_3\psi_3=\psi_3^2e_2=(x_3^2-x_4)e_2,
$$
$f_2Af_2$ is spanned by $\{ f_2x_1^ax_2^bx_3^cx_4^df_2 \mid a,b,c,d\in \Z_{\ge0}\}$. 
Replacing $x_2e_2$ with $-x_3e_2$, we may assume $b=0$. Replacing $x_3^2e_2$ with $x_1e_2$, we may further assume $c=0,1$. If $c=1$ then
$f_2x_1^ax_3x_4^df_2=x_1^af_2x_3f_2x_4^d=0$. 
Thus, we must have $c=0$. 
Finally, $x_1^2e_2=0$ and $x_4^2e_2=0$ imply that we may assume $a, d\in\{0,1\}$. 
Thus, $f_2Af_2$ has the basis $\{f_2, \beta, \beta', \beta\beta'=\beta'\beta\}$. 

\smallskip
\noindent
(iii) Next we consider $f_1Af_2$. If $e_1\psi_we_2\ne0$ then
$w\in \{ s_1^as_3^bs_2s_3 \mid 0\le a,b\le 1 \}$. However, $f_1\psi_1=0$ and $f_1\psi_3=0$ 
imply $a=b=0$. Then, $x_1^2e_2=0$, $x_1e_2=x_2^2e_2$, $x_3e_2=-x_1e_2$ and $x_4^2e_2=0$ imply
$$ f_1Af_2={\rm span}\{ \psi_2\psi_3x_1^ax_2^bx_4^cf_2 \mid 0\le a,b,c\le 1\}. $$
Furthermore, $x_2e_1=-x_1e_1$ implies 
\begin{gather*}
e_1\psi_2\psi_3x_4e_2=e_1\psi_2(\psi_3x_4)e_2=e_1(\psi_2x_3)e_3\psi_3=e_1x_2\psi_2e_3\psi_3 \\
=-e_1x_1\psi_2\psi_3e_2=-\psi_2\psi_3x_1e_2. 
\end{gather*}
Hence we may assume $c=0$. If $b=1$ then
\begin{gather*}
f_1\psi_2\psi_3x_1^ax_2f_2= f_1\psi_2\psi_3x_2f_2x_1^a= f_1\psi_2\psi_3(x_2x_3)\psi_2e_2x_1^a \\
= f_1\psi_2\psi_3\psi_2(x_2x_3)e_2x_1^a= f_1\psi_3\psi_2\psi_3(x_2x_3)e_2x_1^a=0.
\end{gather*}
We have proved that $f_1Af_2$ has the basis $\{ f_1\psi_2\psi_3f_2, f_1\psi_2\psi_3x_1f_2\}$. 

\smallskip
\noindent
(iv) We consider $f_2Af_1$.  If $e_2\psi_we_1\ne0$ then
$w\in \{s_3s_2 s_1^as_3^b \mid 0\le a,b\le 1 \}$. As before, $x_1^2e_2=0$, $x_1e_2=x_2^2e_2$, $x_3e_2=-x_1e_2$, $x_4^2e_2=0$ and $e_2x_4\psi_3\psi_2e_1=-e_2x_1\psi_3\psi_2e_1$ imply
\[
f_2Af_1={\rm span}\{ f_2x_1^ax_2^b\psi_3\psi_2\psi_1^c\psi_3^df_1 \mid 0\le a,b,c,d\le 1\}.
\]
We shall show that we may assume $d=0$. Suppose to the contrary that $d=1$. 
\begin{enumerate}
\item If $b=0$ then $f_2\psi_2=0$ implies
$$  f_2x_1^a\psi_3\psi_2\psi_1^c\psi_3f_1=x_1^af_2(\psi_3\psi_2\psi_3)e_1\psi_1^cf_1=x_1^af_2\psi_2\psi_3\psi_2\psi_1^cf_1=0. $$
\item If $b=1$ then
\begin{align*}
f_2x_1^ax_2\psi_3\psi_2\psi_1^c\psi_3f_1&=x_1^af_2x_2\psi_3\psi_2\psi_1^c\psi_3f_1=x_1^ax_3\psi_2x_2\psi_2\psi_3\psi_2\psi_1^cf_1 \\
&=x_1^ax_3(\psi_2x_2)e_2\psi_2\psi_3\psi_2\psi_1^cf_1=x_1^ax_3(x_3\psi_2-1)\psi_2\psi_3\psi_2\psi_1^cf_1 \\
&=-x_1^ax_3\psi_2\psi_3\psi_2\psi_1^cf_1=-f_2x_1^a\psi_3\psi_2\psi_1^cf_1.
\end{align*}
\end{enumerate}
We can also show that we may assume $b=0$. Suppose to the contrary that $b=1$. 
Then
\begin{align*}
f_2x_1^ax_2\psi_3\psi_2\psi_1^cf_1&=x_1^af_2\psi_3(x_2\psi_2)e_1\psi_1^cf_1
=x_1^af_2\psi_3\psi_2x_3\psi_1^cf_1\\
&=x_1^af_2\psi_3\psi_2x_3\psi_1^cx_2\psi_1x_4\psi_3e_1
=x_1^af_2\psi_3\psi_2\psi_1^cx_2\psi_1x_3x_4\psi_3e_1 \\
&=x_1^af_2\psi_3\psi_2\psi_1^cx_2\psi_1\psi_3x_3x_4e_1
=x_1^af_2(\psi_3\psi_2\psi_3)e_1\psi_1^cx_2\psi_1x_3x_4 \\
&=x_1^af_2\psi_2\psi_3\psi_2e_1\psi_1^cx_2\psi_1x_3x_4=0.
\end{align*}
Hence $f_2Af_1={\rm span}\{f_2x_1^a\psi_3\psi_2\psi_1^cf_1 \mid 0\le a,c\le 1\}$.
\begin{enumerate}
\item If $a=1$ and $c=0$ then $x_3e_2=-x_2e_2$, $x_3e_3=-x_1e_3$ and $x_2^2e_2=x_1e_2$ imply
\begin{align*}
f_2x_1\psi_3\psi_2f_1&=x_1f_2\psi_3\psi_2f_1=x_1x_3e_2\psi_2\psi_3\psi_2f_1
=-x_1x_2e_2\psi_2\psi_3\psi_2f_1 \\
&=-x_1x_2e_2\psi_2\psi_3\psi_2x_2\psi_1x_4\psi_3e_1 \\
&=-x_1x_2\psi_2(\psi_3x_4)e_3\psi_2x_2\psi_1\psi_3 \\
&=-x_1x_2(\psi_2x_3)e_2\psi_3\psi_2x_2\psi_1\psi_3  \\
&=-x_1x_2(x_2\psi_2+1)e_2\psi_3\psi_2x_2\psi_1\psi_3 \\
&=-x_1^2e_2\psi_2\psi_3\psi_2x_2\psi_1\psi_3-x_1x_2e_2\psi_3(\psi_2x_2)e_1\psi_1\psi_3 \\
&=-x_1x_2e_2\psi_3(x_3\psi_2)e_1\psi_1\psi_3 \\
&=-x_1x_2\psi_3(x_3e_3)\psi_2\psi_1\psi_3 \\
&=x_1x_2\psi_3(x_1e_3)\psi_2\psi_1\psi_3 \\
&=x_1^2x_2e_2\psi_3\psi_2\psi_1\psi_3=0.
\end{align*}
\item If $a=1$ and $c=1$ then $x_2e_2=-x_3e_2$ and $x_1e_1=-x_2e_1$ imply
\begin{align*}
f_2x_1\psi_3\psi_2\psi_1f_1&=x_1f_2\psi_3\psi_2\psi_1f_1=x_1x_3e_2\psi_2\psi_3\psi_2\psi_1f_1 \\
&=-x_1x_2e_2\psi_2\psi_3\psi_2\psi_1f_1 \\
&=-x_1x_2e_2\psi_2\psi_3\psi_2\psi_1x_2\psi_1x_4\psi_3e_1 \\
&=-x_1x_2e_2\psi_2\psi_3\psi_2(x_1\psi_1+1)\psi_1x_4\psi_3e_1 \\
&=-x_1x_2e_2\psi_2\psi_3\psi_2\psi_1x_4\psi_3e_1 \\
&=-x_2e_2\psi_2\psi_3\psi_2x_1\psi_1x_4\psi_3e_1 \\
&=x_3e_2\psi_2\psi_3\psi_2x_1\psi_1x_4\psi_3e_1 \\
&=x_3\psi_2e_2\psi_3\psi_2x_1\psi_1x_4\psi_3e_1 \\
&=f_2\psi_3\psi_2(x_1\psi_1)e_1x_4\psi_3 \\
&=f_2\psi_3\psi_2(x_1e_1)\psi_1x_4\psi_3 \\
&=-f_2\psi_3\psi_2(x_2e_1)\psi_1x_4\psi_3 \\
&=-f_2\psi_3\psi_2f_1. 
\end{align*}
\end{enumerate}
Hence $a=c=1$ and $a=c=0$ give the same basis element up to sign. 
\end{proof}

We find relations among the generators $\alpha, \alpha', \beta, \beta', \mu, \nu$
in order to obtain the bound quiver presentation of $R^{2\Lambda_0}(2\alpha_0+2\alpha_1)$. 
We give detailed computations for $\mu\nu=2\alpha'-\alpha^2$ 
and $\nu\mu=\beta'-\beta$ below. 

\begin{align*}
\mu\nu
&=f_1\psi_2\psi_3f_2\psi_3\psi_2\psi_1f_1
=f_1\psi_2\psi_3x_3\psi_2\psi_3\psi_2\psi_1(x_2\psi_1)x_4\psi_3e_1 \\
&=f_1\psi_2\psi_3x_3\psi_2\psi_3\psi_2\psi_1(\psi_1x_1+1)x_4\psi_3e_1 \\
&=f_1\psi_2\psi_3x_3\psi_2\psi_3\psi_2\psi_1(x_4\psi_3)e_1=f_1\psi_2\psi_3x_3\psi_2\psi_3\psi_2\psi_1(\psi_3x_3+1)e_1 \\
&=f_1\psi_2\psi_3x_3\psi_2\psi_3\psi_2\psi_1\psi_3x_3e_1+f_1\psi_2\psi_3x_3\psi_2\psi_3\psi_2\psi_1e_1 \\
&=f_1\psi_2\psi_3x_3\psi_2(\psi_3\psi_2\psi_3)\psi_1x_3e_1+f_1\psi_2\psi_3x_3\psi_2\psi_3\psi_2\psi_1e_1 \\
&=f_1\psi_2\psi_3x_3(\psi_2^2e_2)\psi_3\psi_2\psi_1x_3e_1
+f_1\psi_2\psi_3x_3\psi_2\psi_3\psi_2\psi_1e_1 \\
&=f_1\psi_2\psi_3x_3\psi_3\psi_2\psi_3\psi_1e_1 
=f_1\psi_2\psi_3(x_3\psi_3)e_3\psi_2\psi_3\psi_1e_1 \\
&=f_1\psi_2(\psi_3^2e_3)x_4\psi_2\psi_3\psi_1e_1
=f_1\psi_2(x_3-x_4^2)x_4\psi_2\psi_3\psi_1e_1 \\
&=f_1\psi_2(x_3\psi_2)x_4\psi_3\psi_1e_1-f_1\psi_2^2x_4^3\psi_3\psi_1e_1 \\
&=f_1(\psi_2^2e_1)x_2x_4\psi_3\psi_1e_1-f_1(\psi_2^2e_1)x_4^3\psi_3\psi_1e_1 \\
&=f_1(x_2-x_3^2)x_2x_4\psi_3\psi_1e_1-f_1(x_2-x_3^2)x_4^3\psi_3\psi_1e_1 \\
&=f_1(x_2^2x_4-x_2x_3^2x_4-x_2x_4^3+x_3^2x_4^3)\psi_3\psi_1e_1 \\
&=f_1(x_2-x_3^2-x_4^2)f_1+f_1x_3^2x_4^3\psi_3\psi_1e_1.
\end{align*}
We use $x_1f_1=0$ to compute the first term as follows. 
\begin{align*}
 f_1(x_2-x_3^2-x_4^2)f_1 &=-f_1(x_1f_1)-f_1(x_3^2+x_4^2)f_1 \\
 &=-(x_3+x_4)^2f_1+2x_3x_4f_1 \\
 &=-\alpha^2+2\alpha'. 
\end{align*}
Then, we see that the second term is zero: 
\begin{align*}
f_1x_3^2x_4^3\psi_3\psi_1e_1&=(x_3x_4)^2f_1x_4\psi_3\psi_1e_1 \\
&=(x_3x_4)^2x_2\psi_1x_4\psi_3x_4\psi_3\psi_1e_1 \\
&=(x_3x_4)^2x_2\psi_1x_4\psi_3x_4\psi_3\psi_1e_1 \\
&=(x_3x_4)^2x_2\psi_1x_4(x_3\psi_3+1)\psi_3\psi_1e_1\\
&=(x_3x_4)^2x_2\psi_1x_4\psi_3\psi_1e_1 \\
&=(x_3x_4)^2x_2x_4\psi_3\psi_1^2e_1 \\
&=0.
\end{align*}
Therefore, $\mu\nu=-\alpha^2+2\alpha'$.

\begin{align*}
\nu\mu
&=f_2\psi_3\psi_2\psi_1f_1\psi_2\psi_3f_2 =f_2\psi_3\psi_2\psi_1x_2\psi_1x_4(\psi_3\psi_2\psi_3)x_3\psi_2e_2 \\
&=f_2\psi_3\psi_2\psi_1x_2\psi_1x_4\psi_2\psi_3\psi_2(x_3\psi_2)e_2 \\
&=f_2\psi_3\psi_2\psi_1x_2\psi_1x_4\psi_2\psi_3\psi_2(\psi_2x_2+1)e_2 \\
&=f_2\psi_3\psi_2\psi_1x_2\psi_1x_4\psi_2\psi_3\psi_2e_2
=f_2\psi_3\psi_2\psi_1x_2\psi_1\psi_2(x_4\psi_3)\psi_2e_2 \\
&=f_2\psi_3\psi_2\psi_1x_2\psi_1\psi_2\psi_3(x_3\psi_2)e_2
=f_2\psi_3\psi_2\psi_1x_2\psi_1\psi_2\psi_3(\psi_2x_2+1)e_2 \\
&=f_2\psi_3\psi_2\psi_1(x_2\psi_1)e_1\psi_2\psi_3\psi_2x_2e_2
+f_2\psi_3\psi_2\psi_1(x_2\psi_1)e_1\psi_2\psi_3e_2 \\
&=f_2\psi_3\psi_2\psi_1(\psi_1x_1+1)e_1\psi_2\psi_3\psi_2x_2e_2
+f_2\psi_3\psi_2\psi_1(\psi_1x_1+1)e_1\psi_2\psi_3e_2 \\
&=f_2\psi_3\psi_2\psi_1(\psi_2\psi_3\psi_2)x_2e_2
+f_2\psi_3(\psi_2\psi_1\psi_2)e_3\psi_3 \\
&=f_2\psi_3\psi_2\psi_1(\psi_2\psi_3\psi_2)x_2e_2
+f_2\psi_3(\psi_2\psi_1\psi_2)e_3\psi_3 \\
&=f_2\psi_3\psi_2\psi_3\psi_1\psi_2\psi_3x_2e_2
+f_2\psi_3(\psi_1\psi_2\psi_1+1)e_3\psi_3 \\
&=f_2(\psi_3\psi_2\psi_3)e_1\psi_1\psi_2\psi_3x_2e_2
+f_2\psi_3^2e_2.
\end{align*}
Then, we see that the first term is zero as follows.
\begin{gather*}
    f_2(\psi_3\psi_2\psi_3)e_1\psi_1\psi_2\psi_3x_2e_2
     =(f_2\psi_2)\psi_3\psi_2\psi_1\psi_2\psi_3x_2e_2=0.
\end{gather*}
Therefore, we obtain
$$
\nu\mu=f_2\psi_3^2e_2=f_2(x_4-x_3^2)e_2
=x_4f_2-f_2x_3^2e_2=x_4f_2-f_2x_1e_2=\beta'-\beta.
$$

Since we assume $\ch \k\ne2$, we may replace $\alpha'$ and $\beta'$ with 
$(\alpha^2-\mu\nu)/2$ and $\beta+\nu\mu$, respectively. In particular,  
$fAf$ is generated by $\alpha, \beta, \mu, \nu$. We may also compute 
$$ \nu\mu\nu=-2\beta\nu, \;\; 2\mu\beta\nu=-\alpha^4, \;\; \mu\nu\mu=-2\mu\beta. $$
We leave the computation to the reader. 

\begin{proposition}
Suppose that $\ch \k\ne2$. Then 
$R^{2\Lambda_0}(2\alpha_0+2\alpha_1)$ is Morita equivalent to 
the following bound quiver algebra. 
\[
\begin{xy}
(0,0) *[o]+{\circ}="A", (10,0) *[o]+{\circ}="B", 
\ar @(lu,ld)  "A";"A"_{\alpha}
\ar @<1mm> "A";"B"^{\mu}
\ar @<1mm> "B";"A"^{\nu}
\ar @(ru,rd)  "B";"B"^{\beta}
\end{xy}
\]
\begin{gather*}
\alpha\mu=\nu\alpha=0, \;\; \beta^2=0, \;\; \alpha^4=(\mu\nu)^2=-2\mu\beta\nu, \\
\beta\nu\mu=\nu\mu\beta, \;\; \nu\mu\nu+2\beta\nu=0, \;\; \mu\nu\mu+2\mu\beta=0
\end{gather*}
\end{proposition}
\begin{proof}
First of all, $\beta^2=x_1^2f_2=0$ is clear and $(x_2+x_3)e_2=0$ implies 
\begin{gather*}
\alpha\mu=f_1(x_3+x_4)\psi_2\psi_3f_2=f_1(x_3\psi_2)e_3\psi_3f_2+f_1\psi_2(x_4\psi_3)f_2 \\
=f_1\psi_2x_2\psi_3f_2+f_1\psi_2\psi_3x_3f_2=f_1\psi_2\psi_3x_2f_2+f_1\psi_2\psi_3x_3f_2=0. 
\end{gather*}
Since $\deg \nu\alpha=2$, degree consideration shows $\nu\alpha=0$. 
Replacing $2\alpha'$ with $\alpha^2+\mu\nu$ in the relation 
${\alpha'}^2=\alpha^2\alpha'$, we obtain $\alpha^4=(\mu\nu)^2$. 
Since $f_2Af_2$ is commutative, $\beta\nu\mu=\nu\mu\beta$ follows. 
We conclude that there is a surjective algebra homomorphism from the bound quiver algebra to $fAf$. By comparing dimensions, we see that the algebra homomorphism is an isomorphism. 
Since $A$ is Morita equivalent to $fAf$, we obtain the result. 
\end{proof}

In the above bound quiver presentation, we set $\gamma=\nu\mu+2\beta$ and replace 
$\beta$ with $(\gamma-\nu\mu)/2$. Then the bound quiver presentation becomes
\[
\begin{xy}
(0,0) *[o]+{\circ}="A", (10,0) *[o]+{\circ}="B", 
\ar @(lu,ld)  "A";"A"_{\alpha}
\ar @<1mm> "A";"B"^{\mu}
\ar @<1mm> "B";"A"^{\nu}
\ar @(ru,rd)  "B";"B"^{\gamma}
\end{xy}
\]
$$ \alpha\mu=\nu\alpha=0, \;\; \gamma\nu=\mu\gamma=0, \;\;
\alpha^4=(\mu\nu)^2, \;\; \gamma^2=-(\nu\mu)^2. $$
We see that the algebra is special biserial. Hence, we have the following corollary. 

\begin{corollary}
If $\ch \k\ne2$ then 
$R^{2\Lambda_0}(2\alpha_0+2\alpha_1)$ is Morita equivalent to 
the Brauer graph algebra whose Brauer graph is 
$$
\entrymodifiers={+[Fo]}
\xymatrix@C=1.2cm{
4 \ar@{-}[r]
& 2 \ar@{-}[r]
& 2 
}\ .
$$
\end{corollary} 

\section{Representation type in level two cases}
The rest of our proof relies on the results when the level is two.  In this section, we are aiming to determine the representation type of $R^{\Lambda}(\beta_{\Lambda'})$ for $\Lambda'\in P_{cl,2}^+(\Lambda)$. There are only two cases to consider: $2\Lambda_a$, for $0\le a \le \ell$,  and $\Lambda_a+\Lambda_b$, for $0\le a<b \le \ell$.

Before proceeding to the study of these two cases, we prove the existence of symmetry on the quiver. Let $Z$ be a set of level two dominant integral weights which is stable under $\sigma:\La_a+\La_b\mapsto \La_{\ell-b}+\La_{\ell-a}$ such as 
$Z=\{ 2\La_a \mid 0\le a\le \ell \}$ or $Z=\{ \La_a+\La_b \mid a\ne b\}$. 
The lemma below implies that, if 
$R^\La(\beta_{\La'})$, for some $\La'=\La_i+\La_j$ has a unique common representation type, for all $\La=\La_a+\La_b\in Z$, then we may conclude that $R^\La(\beta_{\La'})$ and $R^\La(\beta_{\sigma\La'})$ have the same representation type for $\La\in Z$. 

\begin{lemma}
    Let $0\le a\le b\le\ell$ and $0\le i\le j\le \ell$. Then we have an isomorphism of algebras
$$
R^{\La_{\ell-b}+\La_{\ell-a}}(\beta_{\La_{\ell-j}+\La_{\ell-i}})\cong R^{\La_a+\La_b}(\beta_{\La_i+\La_j}).
$$
\end{lemma}
\begin{proof}
Let $P$ be the permutation matrix which swaps $i$ and $\ell-i$, for $0\le i\le \ell$. Then $P\mathsf{A}P=\mathsf{A}$. Hence, if $X$ is the solution of $\mathsf{A}X^t=Y^t$ in the sense of Lemma \ref{system of linear equations}, then $XP$ is the solution of $\mathsf{A}PX^t=PY^t$. It implies $\sigma\beta_{\La_{\ell-j}+\La_{\ell-i}}=\beta_{\La_i+\La_j}$. The result follows from Proposition \ref{prop::iso-sigma}.
\end{proof}

\subsection{The case $2\Lambda_a \ (0\le a \le \ell)$}

Our aim in this subsection is to prove the next theorem.

\begin{theorem}\label{level-2-a=b}
Suppose that $\Lambda=2\Lambda_a$, for $0\le a \le \ell$.
\begin{itemize}
\item[(1)] 
If we have an arrow $\Lambda\to \Lambda'$, the representation type of $R^\Lambda(\beta_{\Lambda'})$ is given as follows. 
\begin{enumerate}
  \item[(i')] 
   If $\Lambda'=2\Lambda_{a-1}$, for $1\le a\le\ell$, then $R^\Lambda(\beta_{\Lambda'})$ is wild if $1\le a\le \ell-2$, tame if $a=\ell-1$, finite if $a=\ell$.
  \item[(i'')] 
   If $\Lambda'=2\Lambda_{a+1}$, for $0\le a\le\ell-1$, then $R^\Lambda(\beta_{\Lambda'})$ is wild if $2\le a\le \ell-1$, tame if $a=1$, finite if $a=0$.
  \item[(ii)] 
   If $\Lambda'=\Lambda_{a-1}+\Lambda_{a+1}$, for $1\le a\le\ell-1$, then  $R^\Lambda(\beta_{\Lambda'})$ is finite.
  \item[(iii')]
   If $\Lambda'=\Lambda_{a-2}+\Lambda_a$, for $2\le a\le\ell$, then 
   $R^\Lambda(\beta_{\Lambda'})$ is wild if $2\le a\le\ell-1$, finite if $a=\ell$.
  \item[(iii'')]
   If $\Lambda'=\Lambda_a+\Lambda_{a+2}$, for $0\le a\le\ell-2$, then 
   $R^\Lambda(\beta_{\Lambda'})$ is wild if $1\le a\le\ell-2$, finite if $a=0$.   
\end{enumerate}
\item[(2)]
If $\La'=\La_{a-2}+\La_{a+2}$, for $2\le a\le \ell-2$, then $R^\Lambda(\beta_{\Lambda'})$ is tame if $\ch \k\ne 2$, wild if $\ch \k=2$.

\item[(3)]
\begin{enumerate}
    \item[(i')]
If $\Lambda=2\Lambda_0$ and $\Lambda'=2\Lambda_2$, then $R^\Lambda(\beta_{\Lambda'})$ is tame if $\ch \k\ne2$, wild otherwise.
    \item[(i'')]
    If $\Lambda=2\Lambda_\ell$ and $\Lambda'=2\Lambda_{\ell-2}$, then $R^\Lambda(\beta_{\Lambda'})$ is tame if $\ch \k\ne2$, wild otherwise.
\end{enumerate}
\item[(4)]
Other $R^\Lambda(\beta_{\Lambda'})$ are all wild.
\end{itemize}
Moreover, if $R^\La(\beta_{\La'})$ is finite or tame, then it is an algebra listed in MAIN THEOREM. 
\end{theorem}

We first give the connected quiver $\vec C(2\Lambda_a)$. Once $a$ is fixed, it is easy to verify whether an arrow (or a vertex) exists or not by Definition \ref{Def:quiver-of-maximal-dominant}.
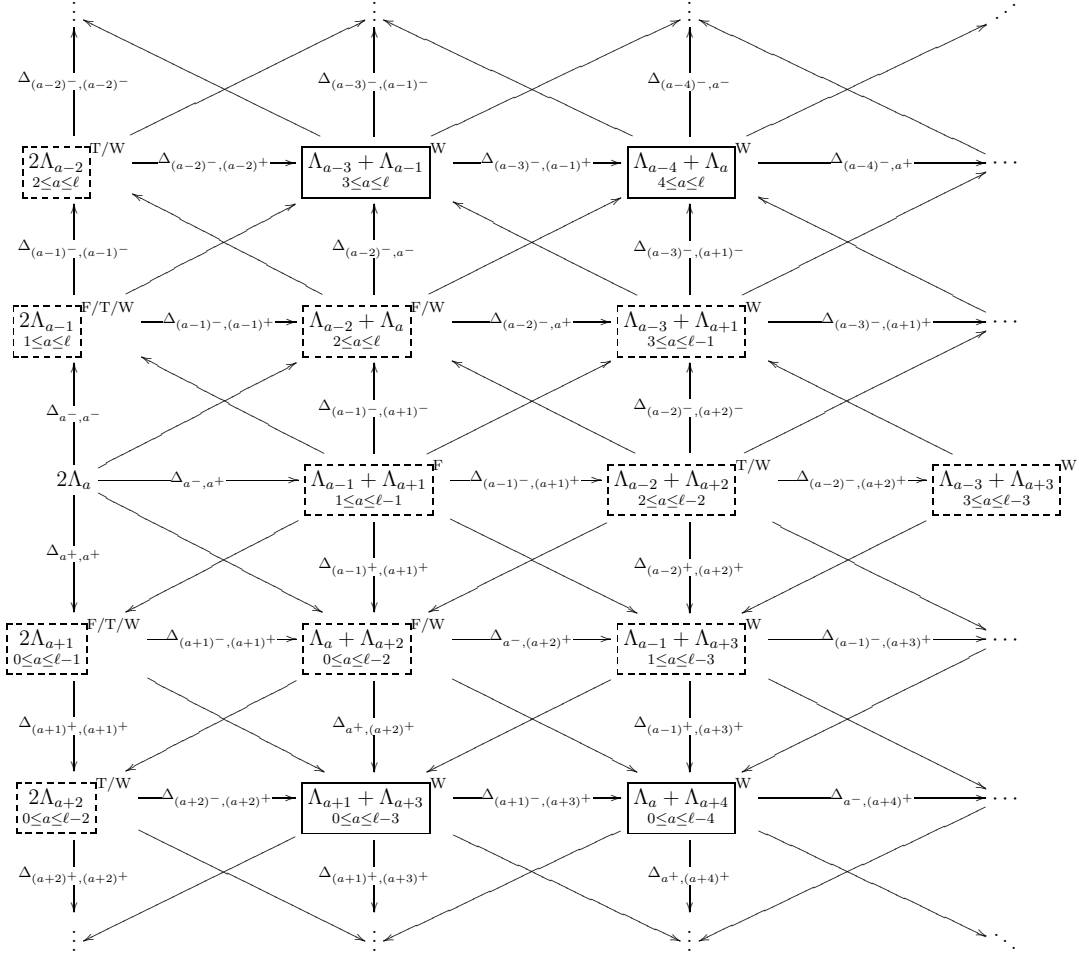
$$
\scalebox{0.7}{
\xymatrix@C=2.8cm@R=1.7cm{
\vdots
&\vdots
&\vdots
&\iddots
\\
\dboxed{\underset{2\le a\le \ell}{2\Lambda_{a-2}}}^{\text{T}/\text{W}}
\ar|-{\Delta_{{(a-2)}^-,{(a-2)}^+}}[r]\ar|-{\Delta_{{(a-2)}^-,{(a-2)}^-}}[u]\ar[ur]
&\boxed{\underset{3\le a\le \ell}{\Lambda_{a-3}+\Lambda_{a-1}}}^{\text{W}}
\ar[ur]\ar[lu]\ar|-{\Delta_{{(a-3)}^-,{(a-1)}^+}}[r]\ar|-{\Delta_{{(a-3)}^-,{(a-1)}^-}}[u]
&\boxed{\underset{4\le a\le \ell}{\Lambda_{a-4}+\Lambda_{a}}}^{\text{W}}
\ar[ur]\ar[lu]\ar|-{\Delta_{{(a-4)}^-,{a}^+}}[r]\ar|-{\Delta_{{(a-4)}^-,{a}^-}}[u]
&\cdots \ar[lu]
\\
\dboxed{\underset{1\le a\le \ell}{2\Lambda_{a-1}}}^{\text{F/T/W}}
\ar|-{\Delta_{{(a-1)}^-,{(a-1)}^+}}[r]\ar|-{\Delta_{{(a-1)}^-,{(a-1)}^-}}[u]\ar[ur]
&\dboxed{\underset{2\le a\le \ell}{\Lambda_{a-2}+\Lambda_{a}}}^{\text{F/W}}
\ar[ur]\ar[lu]\ar|-{\Delta_{{(a-2)}^-,{a}^+}}[r]\ar|-{\Delta_{{(a-2)}^-,{a}^-}}[u]
&\dboxed{\underset{3\le a\le \ell-1}{\Lambda_{a-3}+\Lambda_{a+1}}}^{\text{W}}
\ar[ur]\ar[lu]\ar|-{\Delta_{{(a-3)}^-,{(a+1)}^+}}[r]\ar|-{\Delta_{{(a-3)}^-,{(a+1)}^-}}[u]
&\cdots \ar[lu]
\\
2\Lambda_a\ar|-{\Delta_{a^-,a^+}}[r]\ar[dr]\ar[ur]\ar|-{\Delta_{{a}^+,{a}^+}}[d]\ar|-{\Delta_{{a}^-,{a}^-}}[u]
&\dboxed{\underset{1\le a\le \ell-1}{\Lambda_{a-1}+\Lambda_{a+1}}}^{\text{F}}
\ar|-{\Delta_{{(a-1)}^-,{(a+1)}^+}}[r]\ar[dr]\ar[ur]\ar|-{\Delta_{{(a-1)}^+,{(a+1)}^+}}[d]\ar|-{\Delta_{{(a-1)}^-,{(a+1)}^-}}[u]\ar[lu]\ar[ld]
&\dboxed{\underset{2\le a\le \ell-2}{\Lambda_{a-2}+\Lambda_{a+2}}}^{\text{T}/\text{W}}
\ar|-{\Delta_{{(a-2)}^-,{(a+2)}^+}}[r]\ar[dr]\ar[ur]\ar|-{\Delta_{{(a-2)}^+,{(a+2)}^+}}[d]\ar|-{\Delta_{{(a-2)}^-,{(a+2)}^-}}[u]\ar[lu]\ar[ld]
&\dboxed{\underset{3\le a\le \ell-3}{\Lambda_{a-3}+\Lambda_{a+3}}}^{\text{W}} \ar[lu]\ar[ld]
\\
\dboxed{\underset{0\le a\le \ell-1}{2\Lambda_{a+1}}}^{\text{F/T/W}}
\ar|-{\Delta_{{(a+1)}^-,{(a+1)}^+}}[r]\ar|-{\Delta_{{(a+1)}^+,{(a+1)}^+}}[d]\ar[dr]
&\dboxed{\underset{0\le a\le \ell-2}{\Lambda_{a}+\Lambda_{a+2}}}^{\text{F/W}}
\ar|-{\Delta_{{a}^-,{(a+2)}^+}}[r]\ar[ld]\ar|-{\Delta_{{a}^+,{(a+2)}^+}}[d]\ar[dr]
&\dboxed{\underset{1\le a\le \ell-3}{\Lambda_{a-1}+\Lambda_{a+3}}}^{\text{W}}
\ar|-{\Delta_{{(a-1)}^-,{(a+3)}^+}}[r]\ar[ld]\ar|-{\Delta_{{(a-1)}^+,{(a+3)}^+}}[d]\ar[dr]
&\cdots \ar[ld]
\\
\dboxed{\underset{0\le a\le \ell-2}{2\Lambda_{a+2}}}^{\text{T}/\text{W}}
\ar|-{\Delta_{{(a+2)}^-,{(a+2)}^+}}[r]\ar|-{\Delta_{{(a+2)}^+,{(a+2)}^+}}[d]\ar[dr]
&\boxed{\underset{0\le a\le \ell-3}{\Lambda_{a+1}+\Lambda_{a+3}}}^{\text{W}}
\ar|-{\Delta_{{(a+1)}^-,{(a+3)}^+}}[r]\ar[ld]\ar|-{\Delta_{{(a+1)}^+,{(a+3)}^+}}[d]\ar[dr]
&\boxed{\underset{0\le a\le \ell-4}{\Lambda_{a}+\Lambda_{a+4}}}^{\text{W}}
\ar|-{\Delta_{{a}^-,{(a+4)}^+}}[r]\ar[ld]\ar|-{\Delta_{{a}^+,{(a+4)}^+}}[d]\ar[dr]
&\cdots \ar[ld]
\\
\vdots
&\vdots
&\vdots
&\ddots
}}
$$
In the quiver, the superscript in the upper right corner of each vertex indicates the representation type of $R^{2\Lambda_a}(\beta_{\Lambda'})$, i.e., the corresponding cyclotomic KLR algebra. In particular, the dashed boxes in the quiver show the cases we have to analyze one by one, and the boxes imply that the corresponding algebra is wild by Lemma \ref{cor-arrow}. Here, F means representation-finite, T means tame and W means wild. Finally, all the other remaining vertices of the quiver are wild by Corollary \ref{cor::reduction-path}. 

The second part of Theorem \ref{level-2-a=b} is (t15) if $\ch \k\ne 2$. 
If $\ch \k=2$, it is wild by \cite[Theorem 4.6]{ASW-rep-type}, which refers to \cite[Theorem B]{Ar-rep-type}. There, applying Dynkin automorphism to $2\La_0$ and $\lambda_2^0=\alpha_\ell+2\alpha_0+\alpha_1$, we obtain that $R^{2\La_a}_A(\alpha_{a-1}+2\alpha_a+\alpha_{a+1})$, for $2\le a\le \ell-2$, is wild when $\ch \k=2$. 

\begin{proposition}\label{level-2-aplusminus3}
Let $\La'=\La_{a-3}+\La_{a+3}$, for $3\le a\le \ell-3$. Then $R^\Lambda(\beta_{\Lambda'})$ is wild. 
\end{proposition}
\begin{proof}
We have $\beta_{\La'}=\alpha_{a-2}+2\alpha_{a-1}+3\alpha_a+2\alpha_{a+1}+\alpha_{a+2}$. Applying Dynkin automorphism to $2\La_0$ and $\lambda_3^0=\alpha_{\ell-1}+2\alpha_\ell+3\alpha_0+2\alpha_1+\alpha_2$ as above, we see that $R^\La(\beta_{\La'})$ is wild by  \cite[Theorem 4.6]{ASW-rep-type}. 
\end{proof}

Proposition \ref{level-2-aplusminus3} has the following corollary by Lemma \ref{cor-arrow}. 

\begin{corollary}
If $\La'$ is one of $\La_{a-1}+\La_{a+3}$, $\La_{a-3}+\La_{a+1}$, $\La_{a-3}+\La_{a-1}$, $\La_{a+1}+\La_{a+3}$, for $3\le a\le \ell-3$, then $R^\La(\beta_{\La'})$ is wild. 
\end{corollary}

Next, we prove the first part of Theorem \ref{level-2-a=b}.
We start with (i'). Then we obtain (i'') by symmetry. Since
$ \beta_{\La'}=2\alpha_a+\cdots+2\alpha_{\ell-1}+\alpha_\ell $, we have the following.  
\begin{enumerate}
    \item If $a=\ell$, then $\beta_{\La'}=\alpha_\ell$ and it is finite by (f1). 
    \item If $a=\ell-1$, then $\beta_{\La'}=2\alpha_{\ell-1}+\alpha_\ell$ and it is tame by (t2). 
\end{enumerate}

\begin{proposition}\label{prop::2-La_a-1-level-two}
Let $\Lambda=2\Lambda_a$ and $\Lambda'=2\Lambda_{a-1}$, for $1\le a\le \ell-2$. 
Then $R^\Lambda(\beta_{\Lambda'})$ is wild. 

\end{proposition}
\begin{proof}
We set $A=R^{2\Lambda_a}(\beta_{\Lambda'})$. 
We consider
\begin{align*}
P_1&=f_\ell f^{(2)}_{\ell-1}\cdots f^{(2)}_av_\Lambda, \\
P_2&=f_{\ell-1}\cdots f_af_\ell\cdots f_av_\Lambda, \\
P_3&=f_{\ell-2}f_{\ell-1}f_{\ell}f_{\ell-1}f_{\ell-3}\cdots f_{a}f_{\ell-2}\cdots f_{a}v_\Lambda.
\end{align*}
Recall that $v_\Lambda$ is the empty bipartition $(\emptyset, \emptyset)$.
In the deformed Fock space for type $C^{(1)}_\ell$, we have
\begin{align*}
f_\ell f^{(2)}_{\ell-1}\cdots f^{(2)}_av_\Lambda&=f_\ell(\fbox{$a,\dots,\ell-1$},\fbox{$a,\dots,\ell-1$})\\
&=(\fbox{$a,\dots,\ell-1$},\fbox{$a,\dots,\ell$})+q^2(\fbox{$a,\dots,\ell$},\fbox{$a,\dots,\ell-1$})\\
f_{\ell-1}\cdots f_af_\ell\cdots f_av_\Lambda&=f_{\ell-1}\cdots f_a(\emptyset,\fbox{$a,\dots,\ell$})+qf_{\ell-1}\cdots f_a(\fbox{$a,\dots,\ell$},\emptyset)\\
&=(\fbox{$a,\dots,\ell-2$},\fbox{$a,\dots,\ell,\ell-1$})+q(\fbox{$a,\dots,\ell-1$},\fbox{$a,\dots,\ell$})\\
&\qquad +q(\fbox{$a,\dots,\ell$},\fbox{$a,\dots,\ell-1$})+q^2(\fbox{$a,\dots,\ell,\ell-1$},\fbox{$a,\dots,\ell-2$})
\end{align*}
and $f_{\ell-2}f_{\ell-1}f_{\ell}f_{\ell-1}f_{\ell-3}\cdots f_{a}f_{\ell-2}\cdots f_{a}v_\Lambda$ is equal to
\begin{gather*}
f_{\ell-2}f_{\ell-1}f_{\ell}f_{\ell-1}f_{\ell-3}\cdots f_{a}(\emptyset,\fbox{$a,\dots,\ell-2$})+qf_{\ell-2}f_{\ell-1}f_{\ell}f_{\ell-1}f_{\ell-3}\cdots f_{a}(\fbox{$a,\dots,\ell-2$},\emptyset)\\
=f_{\ell-2}f_{\ell-1}f_{\ell}f_{\ell-1}(\fbox{$a,\dots,\ell-3$},\fbox{$a,\dots,\ell-2$})+qf_{\ell-2}f_{\ell-1}f_{\ell}f_{\ell-1}(\fbox{$a,\dots,\ell-2$},\fbox{$a,\dots,\ell-3$})\\
=f_{\ell-2}f_{\ell-1}(\fbox{$a,\dots,\ell-3$},\fbox{$a,\dots,\ell$})+qf_{\ell-2}f_{\ell-1}(\fbox{$a,\dots,\ell$},\fbox{$a,\dots,\ell-3$})\\
=(\fbox{$a,\dots,\ell-3$},\fbox{$a,\dots,\ell,\ell-1,\ell-2$})+q(\fbox{$a,\dots,\ell-2$},\fbox{$a,\dots,\ell-1$})\\
\qquad\qquad +q(\fbox{$a,\dots,\ell,\ell-1$},\fbox{$a,\dots,\ell-2$})+q^2(\fbox{$a,\dots,\ell,\ell-1,\ell-2$},\fbox{$a,\dots,\ell-3$}).
\end{gather*}
For $i\le 1\le3$, we define idempotents $e_i=e(\nu_i)$ by
\begin{align*}
\nu_1&=(a,a,a+1,a+1,\dots,\ell-1,\ell-1,\ell), \\
\nu_2&=(a,a+1,\dots,\ell,a,a+1,\dots,\ell-1),\\
\nu_3&=(a,a+1,\dots,\ell-2,a,a+1,\dots,\ell-3,\ell-1,\ell,\ell-1,\ell-2).
\end{align*}
Then, we may compute the $q$-dimensions as follows.
\begin{align*}
\dim_q \End(P_1)&=(q+q^{-1})^{-2\ell+2a}\dim_q e_1Ae_1=1+q^4, \\
\dim_q \Hom(P_1,P_2)&=(q+q^{-1})^{-\ell+a}\dim_1 e_1Ae_2=q+q^3, \\
\dim_q \Hom(P_1,P_3)&=(q+q^{-1})^{-\ell+a}\dim_q e_1Ae_3=0, \\
\dim_q \End(P_2)&=\dim_q e_2Ae_2=1+2q^2+q^4, \\
\dim_q \Hom(P_2,P_3)&=\dim_q e_2Ae_3=q+q^3, \\
\dim_q \End(P_3)&=\dim_q e_3Ae_3=1+2q^2+q^4.
\end{align*}

In particular, the projective modules $P_1$, $P_2$, $P_3$ are indecomposable and pairwise non-isomorphic.
For $i\le 1\le3$, let $D_i$ denote the head of $P_i$. 
Let $P=P_1\oplus P_2\oplus P_3$, which is a direct summand of the left regular module $A$, and let $e\in \End(A)^{\rm op}\cong A$ be the projector to $P$. 
Thus, $eAe\cong \End(P)^{\rm op}$ and our aim is to show that $eAe$ is wild. By abuse of notation, we denote $eP_i$ by $P_i$, for $i\le 1\le3$. 
The algebra $eAe$ is non-negatively graded and the composition factors are given by
\begin{align*}
[P_1]&=2[D_1]+2[D_2], \\
[P_2]&=2[D_1]+4[D_2]+2[D_3], \\
[P_3]&=2[D_2]+4[D_3].
\end{align*}
Note that the existence of $q$ in $\dim_q \Hom(P_1,P_2)$ and $\dim_q \Hom(P_2,P_3)$ implies 
\[
\Ext^1(D_1,D_2)=\Ext(D_2,D_1)\ne0, \;\; \Ext^1(D_2,D_3)=\Ext^1(D_3,D_2)\ne0.
\]

Then, the following hold for indecomposable projective $eAe$-modules.
\begin{itemize}
\item[(a)]
$\Rad(P_1)/\Rad^2(P_1)\supseteq D_2$.
\item[(b)]
$\Rad(P_2)/\Rad^2(P_2)\supseteq D_1\oplus D_2\oplus D_3$.
\item[(c)]
$\Rad(P_3)/\Rad^3(P_3)\supseteq D_2\oplus D_3$.
\end{itemize}
Indeed, there is nothing to prove for (a). 
Suppose that $\Ext^1(D_3,D_3)=0$. Then, the self-duality of $P_3$ implies
\[
P_3=
\begin{smallmatrix}
D_3 \\
D_2 \\
D_3\oplus D_3 \\
D_2 \\
D_3
\end{smallmatrix}
\]
But, the existence of $\begin{smallmatrix} D_2 \\ D_3\oplus D_3\end{smallmatrix}$ implies $\dim \Ext^1(D_2,D_3)=\dim \Ext^1(D_3,D_2)=2$, which contradicts 
$\Rad(P_3)/\Rad^2(P_3)=D_2$, that is $\dim \Ext^1(D_3,D_2)=1$. Hence $\Ext^1(D_3,D_3)\ne0$ and we obtain (c). 

Note that the head and the socle of $\Rad(P_3)/\soc(P_3)$ contain $D_2\oplus D_3$ so that the radical length of $P_3$ is $3$ or $4$. 
Suppose that $\Ext^1(D_2,D_2)=0$. Then, the self-duality of $P_2$ implies
\[
P_2=
\begin{smallmatrix}
D_2 \\
D_1\oplus D_3 \\
D_2\oplus D_2 \\
D_1\oplus D_3 \\
D_2
\end{smallmatrix}
\]
We consider the lift of the map $P_3 \to D_3\subseteq \Rad(P_2)/\Rad^2(P_2)$. Then, its image must have length $4$. However, $\soc(P_3)$ must map to $0$ because 
$\soc(P_2)=D_2$, which implies that the image must have length at most $3$, a contradiction. Hence $\Ext^1(D_2,D_2)\ne0$ and we obtain (b).

In particular, the Gabriel quiver of $eAe$ has three vertices $1$, $2$, $3$ and there exist loops on the vertices $2$ and $3$, arrows $2\rightarrow 3$, 
$2\leftarrow 3$ and $1\rightarrow 2$. By \cite[I.10.8(iv)]{Er-tame-block}, we have that $eAe$ is wild and so is $A$.
\end{proof}

The case (ii) has $\beta_{\La'}=\alpha_a$, so that it is finite by (f1). We consider (iii'). Then (iii'') is obtained by symmetry. Then
$$
\beta_{\La'}=\alpha_{a-1}+2\alpha_a+\cdots+2\alpha_{\ell-1}+\alpha_\ell.
$$
If $a=\ell$, $\beta_{\La'}=\alpha_{\ell-1}+\alpha_\ell$ and it is finite by (f3). 

\begin{proposition}\label{prop::La_a-2-La_a-level-two}
Let $\Lambda=2\Lambda_a$ and $\Lambda'=\Lambda_{a-2}+\Lambda_a$, for $2\le a\le \ell-1$. Then, $R^\Lambda(\beta_{\Lambda'})$ is wild.
\end{proposition}
\begin{proof}
If $2\le a\le \ell-2$, then $R^\Lambda(\beta_{\Lambda'})$ is wild by Proposition \ref{prop::2-La_a-1-level-two} and Corollary \ref{cor::reduction-path} since there is an arrow from $2\Lambda_{a-1}$ to $\Lambda_{a-2}+\Lambda_a$.

If $a=\ell-1$, then $\beta_{\Lambda'}=\alpha_{\ell-2}+2\alpha_{\ell-1}+\alpha_\ell$ and set $e=e(\ell-1, \ell, \ell-1, \ell-2)$. We have
$$
\dim_q eR^{\Lambda}(\beta_{\Lambda'})e=1+3q^2+3q^4+q^6.
$$
Using Lemma \ref{lem::wild-three-loops-part1}, we deduce that $R^\Lambda(\beta_{\Lambda'})$ is wild.
\end{proof}

The third part in the case $\ch \k\ne 2$ is (t20) and (t21). When $\ch \k=2$, 
we use the computation in the proof of Proposition \ref{BGA 4-2-2} to show the wildness as follows.

\begin{lemma}\label{lem::2-La_2-level-two} 
Let $\Lambda=2\Lambda_0$ and $\Lambda'=2\Lambda_2$. Then, $R^\Lambda(\beta_{\Lambda'})$ is wild 
if $\ch \k=2$.
\end{lemma}
\begin{proof}
$\beta_{\La'}=2\alpha_0+2\alpha_1$. Let $f_1=x_2\psi_1x_4\psi_3e(0011)$. Then
Proposition \ref{BGA 4-2-2} implies that
$$ f_1Af_1\cong \k[X,Y]/(X^3-2XY, XY^2, Y^2-X^2Y, Y^3) $$ and 
it admits $\k[X,Y]/(X^3,Y^2,X^2Y)$ as a quotient algebra when $\ch \k=2$. 
It follows that  $R^{2\Lambda_0}(2\alpha_0+2\alpha_1)$ in $\ch \k=2$ is wild, by Proposition \ref{prop::tame-local-alg}. 
\end{proof}  

To prove the fourth part of Theorem \ref{level-2-a=b}, namely 
to prove that all the other $R^\La(\beta_{\La'})$ in level two are wild, it suffices to prove the wildness for 
\begin{enumerate}
    \item $\La'=2\Lambda_{a-2}$, for $2\le a\le \ell$, 
    \item $\La'=2\Lambda_{a+2}$, for $0\le a\le \ell-2$,
    \item $\Lambda_{a-3}+\Lambda_{a+1}$, for $a=\ell-2$ and $a= \ell-1$.
    \item $\Lambda_{a+3}+\Lambda_{a-1}$, for $a=1$ and $a=2$,
    \item $\Lambda_{a+1}+\Lambda_{a+3}$, for $0\le a\le 2$,
    \item $\Lambda_{a-3}+\Lambda_{a-1}$, for $\ell-2\le a\le \ell$.
\end{enumerate}

\begin{proposition}\label{prop::2La_a-2-level-two}
The algebra $R^{2\Lambda_a}(\beta_{\Lambda'})$ is wild,
if $\Lambda'=2\Lambda_{a-2}$, for $2\le a\le \ell-1$.
\end{proposition}
\begin{proof}
It follows from Proposition \ref{prop::La_a-2-La_a-level-two} and Lemma \ref{cor-arrow}.
\end{proof}
By symmetry, $R^{2\Lambda_a}(\beta_{\Lambda'})$ is wild, if 
$\Lambda'=2\Lambda_{a+2}$, for $1\le a\le \ell-2$.

The cases (3) and (4) are covered by Lemma \ref{lem::La_a-3+La_a+1-level-two} below. Then, the lemma covers 
the cases (5) and (6), except for 
the case $a=0$ in (5) and the case $a=\ell$ in (6), respectively. These two exceptions are covered by Lemma \ref{lem:2lambda0}. 

\begin{lemma}\label{lem::La_a-3+La_a+1-level-two}
The algebra $R^{2\Lambda_a}(\beta_{\Lambda'})$ is wild, if $\Lambda'=\Lambda_{a-3}+\Lambda_{a+1}$, for $3\le a\le \ell-1$, or $\Lambda'=\Lambda_{a+3}+\Lambda_{a-1}$, for $1\le a\le \ell-3$.
\end{lemma}
\begin{proof}
Suppose that $\Lambda'=\Lambda_{a-3}+\Lambda_{a+1}$ for $3\le a\le \ell-1$. Then by Proposition~\ref{prop::La_a-2-La_a-level-two} $R^{2\Lambda_a}(\beta_{\Lambda''})$ is wild for $\Lambda''=\Lambda_{a-2}+\Lambda_{a}$.
This implies $R^{2\Lambda_a}(\beta_{\Lambda'})$ is wild since we have an arrow from $\Lambda''$ to $\Lambda'$. The other case holds by symmetry. 
\end{proof}

When $a=0$, there is an arrow 
$\Lambda_1+\Lambda_3\rightarrow 2\Lambda_3$. When $a=\ell$, there is an arrow 
$\Lambda_{\ell-3}+\Lambda_{\ell-1}\rightarrow 2\Lambda_{\ell-3}$. Thus, the wildness of 
$R^{2\Lambda_0}(\beta_{2\Lambda_3})$ and $R^{2\Lambda_\ell}(\beta_{2\Lambda_{\ell-3}})$ follow from that of $R^{2\Lambda_0}(\beta_{\Lambda_1+\Lambda_3})$ and $R^{2\Lambda_\ell}(\beta_{\Lambda_{\ell-3}+\Lambda_{\ell-1}})$.

\begin{lemma}\label{lem:2lambda0}
Let $\Lambda=2\Lambda_0$ and $\Lambda'= \Lambda_1+\Lambda_3$. Then $R^\Lambda(\beta_{\Lambda'})$ is wild.
\end{lemma}
\begin{proof}
We have $\beta_{\Lambda'}= 2\alpha_0+2\alpha_1+\alpha_2$. Let $e_1=e(01201)$ and $e_2=e(01210)$. Then 
$$
\begin{aligned}
\dim_q e_1R^\Lambda(\beta_{\Lambda'})e_1&= 1+2q^2+3q^4+3q^6+2q^8+q^{10}\\
\dim_q e_2R^\Lambda(\beta_{\Lambda'})e_2&= 1+q^2+2q^4+2q^6+q^8+q^{10}\\
\dim_q e_1R^\Lambda(\beta_{\Lambda'})e_2&=\dim_q e_2R^\Lambda(\beta_{\Lambda'})e_1 =q^2+q^4+q^6+q^8.\\
\end{aligned} 
$$
By Lemma \ref{lem::wild-two-point-alg}, $R^\Lambda(\beta_{\Lambda'})$ is wild.
\end{proof}


\subsection{The case $\Lambda_a+\Lambda_b \ (0\le a<b \le \ell)$}
Our aim in this subsection is to prove the next theorem.

\begin{theorem}\label{level-2-a<b}
 Suppose that $\Lambda=\Lambda_a+\Lambda_b$, for $0\le a<b\le\ell$.
 \begin{itemize}
\item[(1)]
If we have an arrow $\La\to\La'$, the representation type of $R^\Lambda(\beta_{\Lambda'})$ is given as follows.
 \begin{enumerate}
   \item[(iv')]
    If $\Lambda'=\Lambda_{a-1}+\Lambda_{b-1}$, for $1\le a<b\le\ell$, then 
    $R^\Lambda(\beta_{\Lambda'})$ is wild if $1\le a<b\le \ell-1$, tame if 
    $1\le a\le\ell-2$, $b=\ell$, finite if $a=\ell-1$, $b=\ell$.
   \item[(iv'')]
    If $\Lambda'=\Lambda_{a+1}+\Lambda_{b+1}$, for $0\le a<b\le \ell-1$, then 
    $R^\Lambda(\beta_{\Lambda'})$ is wild if $1\le a<b\le \ell-1$, tame if 
    $a=0$, $1\le b\le\ell-1$, finite if $a=0$, $b=1$.
   \item[(v)]
    If $\Lambda'=\Lambda_{a-1}+\Lambda_{b+1}$, for $1\le a<b\le \ell-1$, then 
    $R^\Lambda(\beta_{\Lambda'})$ is finite.
   \item[(vi)]
    If $\Lambda'=\Lambda_{a+1}+\Lambda_{b-1}$, for $0\le a<b\le \ell$ and $a\le b-2$, then 
    $R^\Lambda(\beta_{\Lambda'})$ is wild.
   \item[(vii')]
    If $\Lambda'=\Lambda_a+\Lambda_{b-2}$, for $0\le a<b\le \ell, a\le b-2$, then 
    $R^\Lambda(\beta_{\Lambda'})\cong R^{\Lambda_b}(\beta_{\Lambda_{b-2}})$ is finite.
   \item[(vii'')]
    If $\Lambda'=\Lambda_{a+2}+\Lambda_b$, for $0\le a<b\le \ell, a\le b-2$, then 
    $R^\Lambda(\beta_{\Lambda'})\cong R^{\Lambda_a}(\beta_{\Lambda_{a+2}})$ is finite.
   \item[(viii')]
    If $\Lambda'=\Lambda_a+\Lambda_{b+2}$, for $0\le a<b\le \ell-2$, then 
    $R^\Lambda(\beta_{\Lambda'})$ is wild.
   \item[(viii'')]
    If $\Lambda'=\Lambda_{a-2}+\Lambda_b$, for $2\le a<b\le \ell$, then 
    $R^\Lambda(\beta_{\Lambda'})$ is wild.  \end{enumerate}
    \item[(2)] If $\Lambda'=\Lambda_{a+2}+\Lambda_{b-2}$ for $0\le a\le b-4\le \ell$, then $R^\Lambda(\beta_{\Lambda'})$ is tame if $a=0$ and $b=\ell$. Otherwise, it is wild.
    \item[(3)]
    All the other $R^\Lambda(\beta_{\Lambda'})$ in level two are wild.
 \end{itemize}
 Moreover, if $R^\La(\beta_{\La'})$ is finite or tame, then it is an algebra listed in MAIN THEOREM. 
\end{theorem}

Set $\Lambda=\Lambda_a+\Lambda_b$ with $0\le a<b\le \ell$. We observe that each element in $P_{cl,2}^+(\Lambda)$ can be written in the form $\Lambda_i+\Lambda_j$ with $0\le i\le j\le \ell$ and $i+j\equiv_2 a+b$. We define
$$
C_s(\Lambda):=\{\Lambda_i+\Lambda_j\mid 0\le i\le j\le \ell, j-i=s, i+j\equiv_2 a+b\}\subseteq P_{cl,2}^+(\Lambda).
$$
Then, $P_{cl,2}^+(\Lambda)=\sqcup_{s\ge 0}C_s(\Lambda)$. We draw $\vec C(\Lambda)$ on the plane by putting elements of $C_s(\Lambda)$ in the same column and arranging $C_s(\Lambda)$'s as columns in increasing order from left to right. In this way, the leftmost column of $\vec C(\Lambda)$ is $C_0(\Lambda)$ if $b-a\equiv_20$ and $C_1(\Lambda)$ if $b-a\equiv_21$. Once $a, b$ are fixed, it is easy to verify whether an arrow (or a vertex) exists or not by Definition \ref{Def:quiver-of-maximal-dominant}. Similar to the case of $2\Lambda_a$, the representation type of $R^{\Lambda_a+\Lambda_b}(\beta_{\Lambda'})$ is mentioned by the superscript in the upper right corner of each vertex. Also, all other remaining cases are wild by Corollary \ref{cor::reduction-path}.
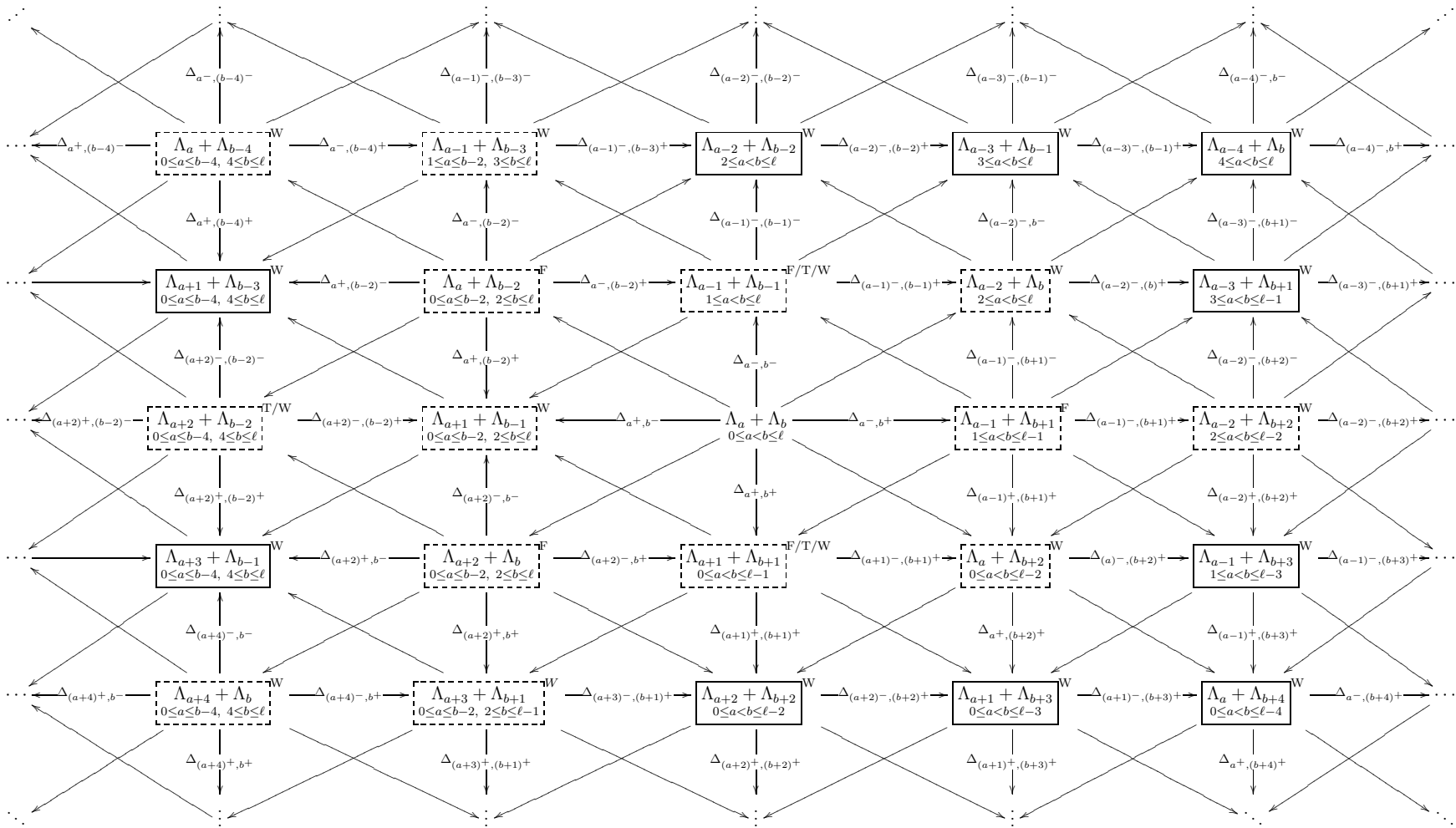
\begin{center}
\begin{adjustbox}{angle=-90}
$
\scalebox{0.7}{
\xymatrix@C=2.5cm@R=1.8cm{
\iddots &\vdots \ar[ld]&
\vdots
&\vdots
&\vdots
&\vdots
&\iddots
\\
\cdots&
\dboxed{\underset{0\le a\le b-4,\ 4\le b\le\ell}{\Lambda_{a}+\Lambda_{b-4}}}^{\text{W}}
\ar[lu]\ar[ru]\ar[ld]
\ar|-{\Delta_{{a}^-,{(b-4)}^+}}[r]\ar|-{\Delta_{{a}^+,{(b-4)}^-}}[l]\ar|-{\Delta_{{a}^+,{(b-4)}^+}}[d]\ar|-{\Delta_{{a}^-,{(b-4)}^-}}[u]&
\dboxed{\underset{1\le a\le b-2,\ 3\le b\le\ell}{\Lambda_{a-1}+\Lambda_{b-3}}}^{\text{W}}
\ar[ru]\ar[ld]\ar[lu]\ar|-{\Delta_{{(a-1)}^-,{(b-3)}^-}}[u]\ar|-{\Delta_{{(a-1)}^-,{(b-3)}^+}}[r]
&\boxed{\underset{2\le a< b\le\ell}{\Lambda_{a-2}+\Lambda_{b-2}}}^{\text{W}}
\ar[lu]\ar[ru]\ar|-{\Delta_{{(a-2)}^-,{(b-2)}^+}}[r]\ar|-{\Delta_{{(a-2)}^-,{(b-2)}^-}}[u]
&\boxed{\underset{3\le a< b\le\ell}{\Lambda_{a-3}+\Lambda_{b-1}}}^{\text{W}}
\ar[lu]\ar[ru]\ar|-{\Delta_{{(a-3)}^-,{(b-1)}^-}}[u]\ar|-{\Delta_{{(a-3)}^-,{(b-1)}^+}}[r]
&\boxed{\underset{4\le a< b\le\ell}{\Lambda_{a-4}+\Lambda_{b}}}^{\text{W}}
\ar[lu]\ar[ru]\ar|-{\Delta_{{(a-4)}^-,{b}^-}}[u]\ar|-{\Delta_{{(a-4)}^-,{b}^+}}[r]&\cdots\ar[lu]
\\
\cdots\ar[r]&
\boxed{\underset{0\le a\le b-4,\ 4\le b\le\ell}{\Lambda_{a+1}+\Lambda_{b-3}}}^{\text{W}}
\ar[lu]\ar[dl]&
\dboxed{\underset{0\le a\le b-2,\ 2\le b\le\ell}{\Lambda_{a}+\Lambda_{b-2}}}^{\text{F}}
\ar[lu]\ar[ru]\ar[ld]
\ar|-{\Delta_{{a}^-,{(b-2)}^+}}[r]\ar|-{\Delta_{{a}^+,{(b-2)}^-}}[l]\ar|-{\Delta_{{a}^+,{(b-2)}^+}}[d]\ar|-{\Delta_{{a}^-,{(b-2)}^-}}[u]
&\dboxed{\underset{1\le a<b\le\ell}{\Lambda_{a-1}+\Lambda_{b-1}}}^{\text{F/T/W}}
\ar[lu]\ar[ru]\ar[ld]\ar|-{\Delta_{{(a-1)}^-,{(b-1)}^+}}[r]\ar|-{\Delta_{{(a-1)}^-,{(b-1)}^-}}[u]
&\dboxed{\underset{2\le a<b\le\ell}{\Lambda_{a-2}+\Lambda_{b}}}^{\text{W}}
\ar[lu]\ar[ru]\ar|-{\Delta_{{(a-2)}^-,{(b)}^+}}[r]\ar|-{\Delta_{{(a-2)}^-,{b}^-}}[u]
&\boxed{\underset{3\le a< b\le\ell-1}{\Lambda_{a-3}+\Lambda_{b+1}}}^{\text{W}}
\ar[lu]\ar[ru]\ar|-{\Delta_{{(a-3)}^-,{(b+1)}^-}}[u]\ar|-{\Delta_{{(a-3)}^-,{(b+1)}^+}}[r]
&\cdots\ar[lu]
\\
\cdots &
\dboxed{\underset{0\le a\le b-4,\ 4\le b\le\ell}{\Lambda_{a+2}+\Lambda_{b-2}}}^{\text{T/W}}
\ar[lu]\ar[ld]\ar|-{\Delta_{{(a+2)}^-,{(b-2)}^+}}[r]\ar|-{\Delta_{{(a+2)}^+,{(b-2)}^-}}[l]\ar|-{\Delta_{{(a+2)}^-,{(b-2)}^-}}[u]\ar|-{\Delta_{{(a+2)}^+,{(b-2)}^+}}[d]
&
\dboxed{\underset{0\le a\le b-2,\ 2\le b\le\ell}{\Lambda_{a+1}+\Lambda_{b-1}}}^{\text{W}}
\ar[lu]\ar[ld]
& \underset{0\le a<b\le\ell}{\Lambda_a+\Lambda_b}\ar[lu]\ar[ru]\ar[ld]\ar[rd]
\ar|-{\Delta_{{a}^-,{b}^+}}[r]\ar|-{\Delta_{{a}^+,{b}^-}}[l]\ar|-{\Delta_{{a}^+,{b}^+}}[d]\ar|-{\Delta_{{a}^-,{b}^-}}[u]
&\dboxed{\underset{1\le a<b\le \ell-1}{\Lambda_{a-1}+\Lambda_{b+1}}}^{\text{F}}
\ar[lu]\ar[ru]\ar[ld]\ar[rd]\ar|-{\Delta_{{(a-1)}^-,{(b+1)}^+}}[r]\ar|-{\Delta_{{(a-1)}^-,{(b+1)}^-}}[u]\ar|-{\Delta_{{(a-1)}^+,{(b+1)}^+}}[d]
&\dboxed{\underset{2\le a<b\le \ell-2}{\Lambda_{a-2}+\Lambda_{b+2}}}^{\text{W}}
\ar[lu]\ar[ru]\ar[ld]\ar[rd]\ar|-{\Delta_{{(a-2)}^-,{(b+2)}^-}}[u]\ar|-{\Delta_{{(a-2)}^+,{(b+2)}^+}}[d]\ar|-{\Delta_{{(a-2)}^-,{(b+2)}^+}}[r]
&\cdots\ar[lu]\ar[ld]
\\
\cdots \ar[r]&
\boxed{\underset{0\le a\le b-4,\ 4\le b\le\ell}{\Lambda_{a+3}+\Lambda_{b-1}}}^{\text{W}}
\ar[dl]\ar[ul]&
\dboxed{\underset{0\le a\le b-2,\ 2\le b\le\ell}{\Lambda_{a+2}+\Lambda_{b}}}^{\text{F}}
\ar[lu]\ar[ld]\ar[rd]
\ar|-{\Delta_{{(a+2)}^-,{b}^+}}[r]\ar|-{\Delta_{{(a+2)}^+,{b}^-}}[l]\ar|-{\Delta_{{(a+2)}^+,{b}^+}}[d]\ar|-{\Delta_{{(a+2)}^-,{b}^-}}[u]
&\dboxed{\underset{0\le a<b\le\ell-1}{\Lambda_{a+1}+\Lambda_{b+1}}}^{\text{F/T/W}}
\ar[ld]\ar[rd]\ar[lu]\ar|-{\Delta_{{(a+1)}^-,{(b+1)}^+}}[r]\ar|-{\Delta_{{(a+1)}^+,{(b+1)}^+}}[d]
&\dboxed{\underset{0\le a<b\le\ell-2}{\Lambda_{a}+\Lambda_{b+2}}}^{\text{W}}
\ar[ld]\ar[rd]\ar|-{\Delta_{{(a)}^-,{(b+2)}^+}}[r]\ar|-{\Delta_{{a}^+,{(b+2)}^+}}[d]
&\boxed{\underset{1\le a< b\le\ell-3}{\Lambda_{a-1}+\Lambda_{b+3}}}^{\text{W}} \ar[ld]\ar[rd]\ar|-{\Delta_{{(a-1)}^+,{(b+3)}^+}}[d]\ar|-{\Delta_{{(a-1)}^-,{(b+3)}^+}}[r]
&\cdots\ar[ld]
\\
\cdots &
\dboxed{\underset{0\le a\le b-4,\ 4\le b\le\ell}{\Lambda_{a+4}+\Lambda_{b}}}^{\text{W}}
\ar[lu]\ar[ld]\ar[rd]
\ar|-{\Delta_{{(a+4)}^-,{b}^+}}[r]\ar|-{\Delta_{{(a+4)}^+,{b}^-}}[l]\ar|-{\Delta_{{(a+4)}^+,{b}^+}}[d]\ar|-{\Delta_{{(a+4)}^-,{b}^-}}[u]&
\dboxed{\underset{0\le a\le b-2,\ 2\le b\le\ell-1}{\Lambda_{a+3}+\Lambda_{b+1}}}^{W}
\ar[lu]\ar[ld]\ar[rd]\ar|-{\Delta_{{(a+3)}^+,{(b+1)}^+}}[d]\ar|-{\Delta_{{(a+3)}^-,{(b+1)}^+}}[r]
&\boxed{\underset{0\le a<b\le\ell-2}{\Lambda_{a+2}+\Lambda_{b+2}}}^{\text{W}}
\ar[ld]\ar[rd]\ar|-{\Delta_{{(a+2)}^-,{(b+2)}^+}}[r]\ar|-{\Delta_{{(a+2)}^+,{(b+2)}^+}}[d]
&\boxed{\underset{0\le a<b\le\ell-3}{\Lambda_{a+1}+\Lambda_{b+3}}}^{\text{W}}
\ar[ld]\ar[rd]\ar|-{\Delta_{{(a+1)}^-,{(b+3)}^+}}[r]\ar|-{\Delta_{{(a+1)}^+,{(b+3)}^+}}[d]
&\boxed{\underset{0\le a<b\le\ell-4}{\Lambda_{a}+\Lambda_{b+4}}}^{\text{W}}
\ar[ld]\ar[rd]\ar|-{\Delta_{{a}^+,{(b+4)}^+}}[d]\ar|-{\Delta_{{a}^-,{(b+4)}^+}}[r]
&\cdots\ar[ld]
\\
\ddots&
\vdots \ar[lu]&
\vdots
&\vdots
&\vdots
&\ddots
&\ddots
}}
$
\end{adjustbox}
\end{center}

We start with (iv') in the first part of Theorem \ref{level-2-a<b}. Then
$$
\beta_{\La'}=\alpha_a+\cdots+\alpha_{b-1}+2\alpha_b+\cdots+2\alpha_{\ell-1}+\alpha_\ell.
$$
If $a=\ell-1$ and $b=\ell$, it is (f3). 
If $1\le a\le \ell-2$ and $b=\ell$, it is (t6). 

\begin{proposition}\label{prop::La_a-1-La_b-1-level-two}
Let $\Lambda=\Lambda_a+\Lambda_b$ and $\Lambda'=\Lambda_{a-1}+\Lambda_{b-1}$, for $1\le a <b\le \ell-1$. Then, $R^{\Lambda}(\beta_{\Lambda'})$ is wild.
\end{proposition}
\begin{proof}
Suppose $1\le a< b\le \ell-1$, we choose a suitable $A:=(e_1+e_2)R^{\Lambda_a+\Lambda_b}(\beta_{\Lambda'})(e_1+e_2)$ that is wild. Recall that $\nu_b=(b,b+1,\ldots,\ell-1,\ell, \ell-1, \ldots,b+1,b, b-1)$.
\begin{itemize}
\item If $1\le a=b-1$, $b\le \ell-1$, we have $\ell\ge 3$ and       
$$
\beta_{\Lambda'}=\alpha_{b-1}+2(\alpha_b+\cdots+\alpha_{\ell-1})+\alpha_{\ell}.
$$
We set $e_1:=e(\nu_b)$ and $e_2:=e(b-1,b,b+1,\ldots,\ell-1,\ell,\ell-1,\ldots, b+1,b)$.

\item If $1\le a\le b-2$, $b\le \ell-1$, we have $\ell\ge 4$ and       
$$
\beta_{\Lambda'}=\alpha_a+\alpha_{a+1}+\cdots+\alpha_{b-1}+2(\alpha_b+\cdots+\alpha_{\ell-1})+\alpha_{\ell}.
$$
We set $e_1:=e(a,a+1,\ldots b-3,b-2,\nu_b)$ and $e_2:=(a,a+1,\ldots, \ell-2, \ell-1, \ell, \ell-1,\ldots,b+1,b)$.
\end{itemize}
In both cases, we have 
\begin{align*}
\dim_q e_1Ae_1&=1+q^2+q^4,\\
\dim_q e_2Ae_2&=1+2q^2+q^4,\\
\dim_q e_1Ae_2&=\dim_qe_2Ae_1=q^2.
\end{align*}
It gives that $A$ is wild by Lemma \ref{lem::wild-two-point-alg}.
\end{proof}

The case (iv'') is obtained by symmetry. 
The case (v) is $\beta_{\La'}=\alpha_a+\cdots+\alpha_b$, for $1\le a<b\le\ell-1$. This is (f4). 
Now we show that (vi) is wild. If $a>0$ and $b<\ell$, then $R^\Lambda(\beta_{\Lambda'})$ is wild by Proposition \ref{prop::La_a-1-La_b-1-level-two} since there is an arrow from $\Lambda_{a-1}+\Lambda_{b-1}$ to $\Lambda_{a+1}+\Lambda_{b-1}$. Thus, we may assume $a=0$ or $b=\ell$. 

\begin{proposition}\label{level-two-delta-plusminus}
Let $\Lambda=\Lambda_a+\Lambda_b$ and $\Lambda'=\Lambda_{a+1}+\Lambda_{b-1}$ with 
$a=0$ or $b=\ell$. Then, $R^{\Lambda}(\beta_{\Lambda'})$ is wild.
\end{proposition}
\begin{proof}
We have three cases. 
\begin{itemize}
\item $a=0$ and $b=\ell$. 
In this case, $\beta_{\Lambda'}= \alpha_0 +\alpha_1 +\cdots +\alpha_\ell$.
\begin{itemize}
\item Suppose $\ell>2$. Let $e_1:=e(0,1,2,\ldots,\ell-2,\ell-1,\ell)$ and $e_2=e(0,\ell,1,2,\ldots ,\ell-3,\ell-2,\ell-1)$.
Then, we have 
$$
\begin{aligned}
\dim_q e_1R^{\Lambda}(\beta_{\Lambda'})e_1&= 1+q^2+q^4+q^6,\\
\dim_q e_2R^{\Lambda}(\beta_{\Lambda'})e_2&= 1+2q^2+2q^4 +q^6,\\
\dim_q e_1R^{\Lambda}(\beta_{\Lambda'})e_2&=\dim e_2R^{\Lambda}(\beta_{\Lambda'})e_1=q^2+q^4. 
\end{aligned} 
$$      
We deduce that $R^{\Lambda}(\beta_{\Lambda'})$ is wild by Lemma \ref{lem::wild-two-point-alg}.
   
\item Suppose $\ell=2$. Let $e:=e_1+e_1+e_3$ with $e_1:=e(012)$, $e_2:=e(021)$ and $e_3:=(210)$. Then, we have 
$$
\begin{aligned}
\dim_q e_iR^{\Lambda}(\beta_{\Lambda'})e_i&=1+q^2+q^4+q^6,\\
\dim_1 e_i R^{\Lambda}(\beta_{\Lambda'})e_j&=\left\{\begin{array}{ll}
q^2+q^4& \text{ if } |i-j|=1,\\
0& \text{ otherwise.}
\end{array}\right. 
\end{aligned} 
$$
This implies the quiver of $R^{\Lambda}(\beta_{\Lambda'})$ is of the form 
$$
\xymatrix@C=0.8cm{1 \ar@<0.5ex>[r]^{ }\ar@(dl,ul)^{ }&2 \ar@<0.5ex>[r]^{ } \ar@<0.5ex>[l]^{ }\ar@(ur,ul)_{ }&3 \ar@<0.5ex>[l]^{ }\ar@(dr,ur)_{ }} 
$$
and hence, it is wild by \cite[I.10.8 (iv)]{Er-tame-block}.
\end{itemize}
   
\item $a>0$ and $b=\ell$. In this case, $\beta_{\Lambda'}= \alpha_0 +2(\alpha_1+\cdots+\alpha_a)+\alpha_{a+1}+\cdots +\alpha_\ell$. 
If $a\le b-4$, then $R^\Lambda(\beta_{\Lambda'})$ is wild by Proposition \ref{prop::La_a+2-La_b-2-level-two} since there is an arrow from $\Lambda_{a+2}+\Lambda_{b-2}$ to $\Lambda_{a+1}+\Lambda_{b-1}$.
It remains to consider $a=b-2=\ell-2$ or $a=b-3=\ell -3$.

Let $\square_a:=(a,a-1,a-2,\ldots, 2,1)$. If $a=\ell-2$, we set $e_1:=e(\square_a,0, a+1, \square_a, \ell)$ and $e_2:=e(\square_a,0, \ell, a+1, \square_a)$. If $a=\ell-3$, we set $e_1:=e(\square_a,0, a+1,a+2,\square_a,\ell)$ and $e_2:=e(\square_a,0, \ell, a+1,a+2, \square_a)$. In both cases, we have the following graded dimensions such that $R^\Lambda(\beta_{\Lambda'})$ is wild, see Lemma \ref{lem::wild-two-point-alg}.
$$
\begin{aligned}
\dim_q e_1R^{\Lambda}(\beta_{\Lambda'})e_1&= 1+q^2+q^4+q^6,\\
\dim_q e_2R^{\Lambda}(\beta_{\Lambda'})e_2&= 1+2q^2+2q^4 +q^6,\\
\dim_q e_1R^{\Lambda}(\beta_{\Lambda'})e_2&=\dim e_2R^{\Lambda}(\beta_{\Lambda'})e_1=q^2+q^4. 
\end{aligned} 
$$    
 
\item $a=0$ and $b<\ell$. In this case, $\beta_{\Lambda'}= \alpha_0  +\alpha_{1}+\cdots +\alpha_{b-1} +2(\alpha_b+\cdots+\alpha_{\ell-1})+\alpha_\ell$. Using the isomorphism in Proposition \ref{prop::iso-sigma}, we conclude that $R^{\Lambda}(\beta_{\Lambda'})$ is wild.   
\end{itemize}   
We have completed the proof.
\end{proof}

The case (vii') is (f6) because
$$
\beta_{\La'}=\alpha_{b-1}+2\alpha_b+\cdots+2\alpha_{\ell-1}+\alpha_\ell.
$$
The case (vii'') is (f5). It remains to show that (viii'') is wild. The case (viii') is obtained by symmetry. 

\begin{proposition}\label{prop::La_a-2-La_b-level-two}
Let $\Lambda=\Lambda_a+\Lambda_b$ and $\Lambda'=\Lambda_{a-2}+\Lambda_{b}$ with $2\le a<b\le \ell$. Then, $R^{\Lambda}(\beta_{\Lambda'})$ is wild.
\end{proposition}
\begin{proof}
If $b<\ell$, then $R^\Lambda(\beta_{\Lambda'})$ is wild by Proposition \ref{prop::La_a-1-La_b-1-level-two} since there is an arrow from $\Lambda_{a-1}+\Lambda_{b-1}$ to $\Lambda_{a-2}+\Lambda_{b}$.
We assume $b=\ell$ in the following. 
\begin{itemize}
\item $a=\ell-1$ and $b=\ell$. In this case, $\beta_{\Lambda'}=\alpha_{\ell-2}+2\alpha_{\ell-1}+\alpha_\ell$. We set $e_1:=(\ell-1,\ell,\ell-2,\ell-1)$ and $e_2:=(\ell,\ell-1,\ell-1,\ell-2)$. Then, 
$$
P_1=f_{\ell-1} f_{\ell-2}f_{\ell}f_{\ell-1}L(0), \quad P_2= f_{\ell-2}f^{(2)}_{\ell-1}f_{\ell} L(0). 
$$
Then 
      $$\begin{aligned}
      &f_{\ell-1} f_{\ell-2}f_{\ell}f_{\ell-1}L(0)=( \ytableausetup{smalltableaux}\ytableaushort{ 
 \ ,\ }, \ytableausetup{smalltableaux}\ytableaushort{ 
 \ ,\ })+ q( \ytableausetup{smalltableaux}\ytableaushort{ 
 \ ,\ },\ytableaushort{\ \ })+q^2( \ytableausetup{smalltableaux}\ytableaushort{\ \ ,\ \ },\emptyset)+ q^3  (  \ytableausetup{smalltableaux}\ytableaushort{\ \ \ ,\ }, \emptyset)\\
 & f_{\ell-2}f^{(2)}_{\ell-1}f_{\ell} L(0)=(\emptyset, \ytableausetup{smalltableaux}\ytableaushort{ 
 \ \ ,\ ,\ })+q(\emptyset, \ytableausetup{smalltableaux}\ytableaushort{ 
 \ \ \ ,\ })+ q (\ytableausetup{smalltableaux}\ytableaushort{\ } , \ytableausetup{smalltableaux}\ytableaushort{ 
 \ ,\ ,\ })+
 q^2 ( \ytableausetup{smalltableaux}\ytableaushort{ 
 \ ,\ }, \ytableausetup{smalltableaux}\ytableaushort{ 
 \ ,\ })+ q^3( \ytableausetup{smalltableaux}\ytableaushort{ 
 \ ,\ },\ytableaushort{\ \ })+q^2( \ytableausetup{smalltableaux}\ytableaushort{\ }, \ytableausetup{smalltableaux}\ytableaushort{\ \ \ }).
      \end{aligned}$$
We may compute the graded dimensions as follows.
$$
\begin{aligned}
\dim_q \End (P_1)&= 1+q^2+q^4+q^6,\\
\dim_q \End (P_2)&= 1+2q^2+2q^4 +q^6,\\
\dim_q \Hom(P_1, P_2)&=\dim_q \Hom(P_2, P_1)=q^2+q^4. 
\end{aligned} 
$$   
This implies that the algebra $R^\Lambda(\beta_{\Lambda'})$ is wild.    

\item $a<\ell-1$ and $b=\ell$. In this case, $\beta_{\Lambda'}=\alpha_{a-1}+2(\alpha_a+\cdots+\alpha_{\ell-1})+\alpha_\ell$. Set 
$$ 
e:=e(\ell,\ell-1,\ldots,a+2,a+1,a,a-1,a,a+1,a+2,\ldots,\ell-2,\ell-1).
$$
Then, $\dim_q eR^\Lambda(\beta_{\Lambda'})e=1+3q^2+3q^4+q^6$ and $R^\Lambda(\beta_{\Lambda'})$ is wild by Lemma \ref{lem::wild-three-loops-part1}.
\end{itemize} 
The proof is completed.
\end{proof}

Next, we prove the second part of  Theorem \ref{level-2-a<b}. If $a=0$ and $b=\ell$, then it is (t12), and we already know that it is tame. Thus, we may assume $a>0$ or $b<\ell$.
\begin{proposition}\label{prop::La_a+2-La_b-2-level-two}
Let $\Lambda=\Lambda_a+\Lambda_b$ and $\Lambda'=\Lambda_{a+2}+\Lambda_{b-2}$ with $0\le a \le b-4$, $4\le b \le \ell$ such that $a>0$ or $b<\ell$. Then, $R^{\Lambda}(\beta_{\Lambda'})$ is   wild.
\end{proposition}
\begin{proof}

If $a=0$, $b\le \ell-1$, then $\beta_{\Lambda'}=\alpha_0+\alpha_1+\alpha_{b-1}+2(\alpha_b+\cdots+\alpha_{\ell-1})+\alpha_{\ell}$.
We define $e_1:=e(0,1,\nu_b)$ and $e_2:=e(0, 1, \nu_b')$ with
$$
\begin{aligned}
\nu_b &:=(b,b+1,\ldots,\ell-1,\ell, \ell-1, \ldots,b+1,b, b-1),\\
\nu_b'&:=(b,b-1,b+1,b+2,\cdots,\ell-1,\ell, \ell-1, \ldots,b+1,b).
\end{aligned}
$$
Setting $A=eR^{\Lambda}(\beta_{\Lambda'})e$ with $e=e_1+e_2$. 
We obtain 
$$
\dim_q e_iAe_i=1+2q^2+q^4 \text{ for } i=1,2, 
\quad
\dim_q e_1Ae_2=\dim_q e_2Ae_1=q+q^3.
$$

Let $k=2(\ell-b)+4$.
Direct computation as above shows that $x_1e_i=x_2^2e_i=0$, $i=1,2$, and 
\begin{equation}\label{equ:compu}
    x_je_1=0, x_h e_2=0 \text{ for } 3\le j \le \ell-b+3,  3\le h\le \ell-b+4.
\end{equation}
We also show that 
\begin{equation}\label{equ:compu11}
   \text{$x_je_i=x_k^2e_i=0$ for $i=1,2$, $3\le j\le k-1$.}
\end{equation}
Suppose that $b=\ell-1$. Then $k=6$ and $x_6^2e_2=0$ by $\psi_{5}e_2=0$ and \eqref{equ:compu}.
Using  $\psi_{3}e_1=0=\psi_{4}e_1$ shows that $(x_{3}+x_{5})e_1=0$ and hence $x_{5}e_1=0$ by \eqref{equ:compu}.
Moreover, $\psi_{5}^2e_1=(x_{5}-x_6)e_1$ and $x_6\psi_{5}^2e_1=0$
imply that 
$x_6^2 e_1=0$. This completes the proof of \eqref{equ:compu11} when $b=\ell-1$. 
The case $b<\ell-1$ can be checked similarly by using $\psi_{\ell-b+2}e_1=0=\psi_{\ell-b+3}e_1$
and  $\psi_{\ell-b+3}e_2=0=\psi_{\ell-b+4}e_2$. 
Furthermore, $e_i \psi_w e_i\ne 0$ only if $\psi_w=1$.
This together with \eqref{equ:compu11} implies that the basis of $e_iAe_h$ is given as follows.
$$\begin{aligned}
    e_iAe_i&=\k\text{-span}\{x_2^mx_k^ne_i\mid 0\le m, n\le 1\}, i=1,2,\\
    e_1Ae_2&=\k\text{-span}\{x_2^m\psi_{k-1}\psi_{k-2}\ldots \psi_{4}e_2\mid 0\le m\le 1 \},\\
    e_2Ae_1&=\k\text{-span}\{x_2^m \psi_{4}\ldots \psi_{k-2}\psi_{k-1}e_1\mid 0\le m\le 1 \}.\\
\end{aligned}$$
By setting $\alpha=x_2e_1$, $\beta=x_2e_2$, $\mu =\psi_{k-1}\psi_{k-2}\ldots \psi_{4}e_2$ and $\nu= \psi_{4}\ldots \psi_{k-2}\psi_{k-1}e_1$, $A$ is isomorphic to the bound quiver algebra defined by
$$
\xymatrix@C=0.8cm{1 \ar@<0.5ex>[r]^{\mu}\ar@(dl,ul)^{\alpha}&2 \ar@<0.5ex>[l]^{\nu}\ar@(ur,dr)^{\beta}}
\quad \text{and}\quad 
\left<\alpha^2, \beta^2, \mu\nu\mu, \nu\mu\nu, \alpha\mu-\mu\beta, \beta\nu-\nu\alpha\right>.
$$
Then, $A/\left<\nu\alpha\right>$ is a wild algebra by \cite[(32)]{H-wild-two-point}.

If $a\ge 1$, $b=\ell$, then $\beta_{\Lambda'}=\alpha_0+2(\alpha_1+\cdots+\alpha_a)+\alpha_{a+1}+\alpha_{\ell-1}+\alpha_{\ell}$. Similar to the case of $a=0,b\le \ell-1$, one may show that $R^\Lambda(\beta_{\Lambda'})$ is wild.

If $a\ge 1$, $b\le\ell-1$, then we have
$$
\beta_{\Lambda'}=\alpha_0+2(\alpha_1+\cdots+\alpha_a)+\alpha_{a+1}+\alpha_{b-1}+2(\alpha_b+\cdots+\alpha_{\ell-1})+\alpha_{\ell}.
$$
We choose $e_1=e(\nu_a,\nu_b)$ and $e_2=e(\nu_a', \nu_b')$, where
$$
\begin{aligned}
\nu_a&:=(a,a-1,\ldots,1,0,1,\ldots, a-1,a, a+1), \\
\nu_a'&:=(a,a+1,a-1,a-2,\ldots,1,0,1,\ldots,a-1,a).
\end{aligned}
$$
and $\nu_b, \nu_b'$ are defined in the case of $a=0,b\le \ell-1$. Set $A:=R^{\Lambda}(\beta_{\Lambda'})$, we obtain 
$$
\dim_q e_iAe_i=1+2q^2+q^4 \text{ for } i=1,2, \quad
\dim_q e_1Ae_2=\dim_q e_2Ae_1=q^2.
$$
Then, $R^{\Lambda}(\beta_{\Lambda'})$ is wild by Lemma \ref{lem::wild-two-point-alg}.
\end{proof}

In order to show that all the other cyclotomic KLR algebras in level two are wild, we construct a neighborhood of $\La$ whose rim are all wild. For this, it suffices to show the wildness for
$$
\La'\in \left\{ \La_{a-2}+\La_{b+2},\; \La_{a+3}+\La_{b+1},\; \La_{a+4}+\La_b,\; 
\La_a+\La_{b-4},\; \La_{a-1}+\La_{b-3}\right\}.
$$

\begin{proposition}
Let $\Lambda=\Lambda_a+\Lambda_b$ and $\Lambda'=\Lambda_{a-2}+\Lambda_{b+2}$ with $2\le a<b\le \ell-2$. Then, $R^{\Lambda}(\beta_{\Lambda'})$ is wild.
\end{proposition}
\begin{proof}
In this case, we have $\beta_{\Lambda'}=
\alpha_{a-1}+2\alpha_a+\cdots+2\alpha_b+\alpha_{b+1}$. Then,
$$
R^\La(\alpha_{a-1}+2\alpha_a+\cdots+2\alpha_b+\alpha_{b+1})
\cong R^{\La_A}(\alpha_{a-1}+2\alpha_a+\cdots+2\alpha_b+\alpha_{b+1}),
$$
and the result follows from \cite{ASW-rep-type}. 
\end{proof}

We prove the case $\Lambda'=\Lambda_{a-1}+\Lambda_{b-3}$ as follows. 
The case $\Lambda_{a+3}+\Lambda_{b+1}$ is obtained by symmetry.

\begin{proposition}
Let $\Lambda=\Lambda_a+\Lambda_b$ and $\Lambda'=\Lambda_{a-1}+\Lambda_{b-3}$ with $0\le a\le b-2$, $2\le b\le \ell$. Then, $R^{\Lambda}(\beta_{\Lambda'})$ is wild.
\end{proposition}

\begin{proof}
Since $b\le \ell-3$, $\Lambda_{a-1}+\Lambda_{b-1}$ is wild by (iv') of Theorem \ref{level-2-a<b}. Then the result holds since we have an arrow  $\Lambda_{a-1}+\Lambda_{b-1}$ to $\Lambda_{a-1}+\Lambda_{b-3}$.
\end{proof}

Finally, we consider the case $\Lambda'=\Lambda_{a}+\Lambda_{b-4}$. 
The case $\Lambda'=\Lambda_{a+4}+\Lambda_{b}$ is obtained by symmetry.

\begin{proposition}\label{prop::La_a-La_b-4-level-two}
Let $\Lambda=\Lambda_a+\Lambda_b$ and $\Lambda'=\Lambda_{a}+\Lambda_{b-4}$ with $0\le a \le b-4$, $4\le b \le \ell$. Then, $R^{\Lambda}(\beta_{\Lambda'})$ is wild.
\end{proposition}
\begin{proof}
In this case, we have  
$$
\beta_{\Lambda'}=\alpha_{b-3}+2\alpha_{b-2}+3\alpha_{b-1}+4\alpha_b+\cdots+4\alpha_{\ell-1}+2\alpha_{\ell}.
$$
Thus, we have an isomorphism of algebras 
$R^\La(\beta_{\Lambda'})\cong R^{\La_b}(\beta_{\Lambda'})$, and 
$R^\Lambda(\beta_{\Lambda'})$ is wild by Theorem \ref{theo::result-level-one}.
\end{proof}

\section{First neighbors in higher level cases}
We consider higher level $R^\Lambda(\beta_{\Lambda'})$, for the first neighbors $\Lambda'$ of $\Lambda$.
We write $\Lambda=\sum_{i=0}^\ell m_i\Lambda_i$.
As we have completed level two in the previous section, we assume 
that the level is $k\ge3$ hereafter.

\subsection{(i') $\Lambda=2\Lambda_a+\widetilde{\Lambda}$ $(1\le a\le\ell)$ and $\Lambda'=2\Lambda_{a-1}+\widetilde{\Lambda}$} 
In this case, 
\[
\beta_{\Lambda'}=2\alpha_a+\cdots+2\alpha_{\ell-1}+\alpha_\ell.
\]
If $1\le a\le \ell-2$, then $R^\Lambda(\beta_{\Lambda'})$ is wild by Theorem \ref{level-2-a=b}(i'). On the other hand, $R^\Lambda(\beta_{\Lambda'})$ is (f1) if $a=\ell$.

Suppose $a=\ell-1$. Then $\beta=2\alpha_{\ell-1}+\alpha_\ell$ and 
$R^{\Lambda}(2\alpha_{\ell-1}+\alpha_\ell)$ is (t2) if 
$m_{\ell-1}=2$ and $m_\ell=0$. 
We show that $R^{\Lambda}(2\alpha_{\ell-1}+\alpha_\ell)$ is wild if $m_{\ell-1}\ge 3$ or $m_\ell\ge 1$. To see this, it suffices to show that 
$$
R^{3\Lambda_{\ell-1}}(2\alpha_{\ell-1}+\alpha_\ell) \text{   and   } R^{2\Lambda_{\ell-1}+\Lambda_\ell}(2\alpha_{\ell-1}+\alpha_\ell)
$$
are wild. 

\begin{lemma}\label{alpha_0+2alpha_1-wild-part1}
The algebra $ R^{2\Lambda_{\ell-1}+\Lambda_\ell}(2\alpha_{\ell-1}+\alpha_\ell)$ is wild.  
\end{lemma}
\begin{proof}
Let $A=R^{2\Lambda_{\ell-1}+\Lambda_\ell}(2\alpha_{\ell-1}+\alpha_\ell)$ and $e_i=e(\nu_i)$, for
\[
\nu_1=(\ell-1,\ell-1,\ell), \;\; \nu_2=(\ell-1,\ell,\ell-1), \;\; \nu_3=(\ell,\ell-1,\ell-1).
\]
By crystal computation, the number of simples is three. Moreover, computation of
\[
f_\ell f_{\ell-1}^{(2)}(\emptyset, \emptyset,\emptyset), \;\; f_{\ell-1}f_\ell f_{\ell-1}(\emptyset, \emptyset,\emptyset), \;\; f_{\ell-1}^{(2)}f_\ell(\emptyset, \emptyset,\emptyset)
\]
shows that
\begin{align*}
\dim_q \End_A(P_1)&=1+q^4+q^8, \\
\dim_q \Hom_A(P_1,P_2)&=2q^2+q^5+q^7, \\
\dim_q \Hom_A(P_1,P_3)&=0, \\
\dim_q \End_A(P_2)&=1+2q^2+6q^4+2q^6+q^8, \\
\dim_q \Hom_A(P_2,P_3)&=q+2q^3+q^5+2q^6, \\
\dim_q \End_A(P_3)&=1+q^2+2q^4+q^6+q^8.
\end{align*}
Let $e=e_1+e_2$ and consider $B=eAe$. Then, we observe the following. 
\begin{itemize}
\item
There are two degree two homomorphisms in $\Hom_A(P_1,P_2)$ and they cannot be linear combination of
composition of two arrows of degree one.
\item
Next we consider $\End_A(P_2)$. There are two endomorphisms of degree two.
The composition of arrows $P_2\to P_3$ and $P_3\to P_2$ of degree one gives one endomorphism of degree two, but there exists
another endomorphism of degree two which is not linear combination of composition of two arrows of degree one.
\end{itemize}
Hence, the Gabriel quiver of $B$ has a loop on vertex $2$,
and two arrows from vertex $1$ to vertex $2$. Hence, $A=R^{2\Lambda_{\ell-1}+\Lambda_\ell}(2\alpha_{\ell-1}+\alpha_\ell)$ is wild.
\end{proof}

\begin{lemma}\label{alpha_0+2alpha_1-wild-part2}
 The algebra $R^{3\Lambda_{\ell-1}}(2\alpha_{\ell-1}+\alpha_\ell)$ is wild.
\end{lemma}
\begin{proof}
Let $A=R^{3\Lambda_{\ell-1}}(2\alpha_{\ell-1}+\alpha_\ell)$ and $e_i=e(\nu_i)$, for
\[
\nu_1=(\ell-1,\ell,\ell-1), \;\; \nu_2=(\ell-1,\ell-1,\ell).
\]
The crystal computation shows that the number of simples is two. Hence, they are the pullbacks of the one dimensional
$R^{2\Lambda_{\ell-1}}(2\alpha_{\ell-1}+\alpha_\ell)$-module $D_1$ and the two dimensional $R^{2\Lambda_{\ell-1}}(2\alpha_{\ell-1}+\alpha_\ell)$-module $D_2$.
Hence we have the following surjective homomorphisms.
\[
P_1\rightarrow \overline{P}_1=\begin{smallmatrix} D_1 \\ D_1\oplus D_2 \\ D_2\oplus D_1 \\ D_1 \end{smallmatrix}, \quad
P_2\rightarrow \overline{P}_2=\begin{smallmatrix} D_2 \\ D_1 \\ D_1 \\ D_2 \end{smallmatrix}
\]
We can also compute
\begin{align*}
\dim_q \End_A(P_1)&=1+2q^2+3q^4+2q^6+q^8 \\
&=(1+2q^2+q^4)+q^4+(q^4+2q^6+q^8), \\
\dim_q \Hom_A(P_1,P_2)&=q+2q^3+2q^5+q^7 \\
&=(q+q^3)+(q^3+q^5)+(q^5+q^7), \\
\dim_q \End_A(P_2)&=1+q^2+2q^4+q^6+q^8 \\
&=(1+q^4)+(q^2+q^6)+(q^4+q^8).
\end{align*}
Hence $[P_1]=9[D_1]+6[D_2]$ and $[P_2]=6[D_1]+6[D_2]$ and
\begin{align*}
[P_1]-2[\overline{P}_1]&=[P_1]-2(4[D_1]+2[D_2])=[D_1]+2[D_2], \\
[P_2]-2[\overline{P}_2]&=[P_2]-2(2[D_1]+2[D_2])=2[D_1]+2[D_2].
\end{align*}
The self-duality implies that $\overline{P}_1$ and $\overline{P}_2$ are submodules of $P_1$ and $P_2$, respectively.
Hence there is a self-dual module $M_1$ with $[M_1]=[D_1]+2[D_2]$ such that
\[
P_1=\begin{smallmatrix} \overline{P}_1 \\ M_1 \\ \overline{P}_1 \end{smallmatrix}
\]
and  there is a self-dual module $M_2$ with $[M_1]=2[D_1]+2[D_2]$ such that
\[
P_2=\begin{smallmatrix} \overline{P}_2 \\ M_2 \\ \overline{P}_2 \end{smallmatrix}
\]
The self-duality of $M_1$ implies
\[
M_1=\begin{smallmatrix} D_2 \\ D_1 \\ D_2 \end{smallmatrix}
\]
Then, it follows that we have the following factor module of $P_2$
\[
\begin{smallmatrix} D_2 \\ D_1 \\ D_1\oplus D_2 \\ D_2 \oplus D_1 \end{smallmatrix}
\]
and its dual appears as a submodule of $P_2$. Namely, the radical series of $P_2$ is
\[
P_2=\begin{smallmatrix} D_2 \\ D_1 \\ D_1\oplus D_2 \\ D_2 \oplus D_1 \\ D_1\oplus D_2 \\ D_2\oplus D_1 \\ D_1 \\ D_2\end{smallmatrix}
\]
In particular, $\Rad^8(P_2)=0$.

We show that $D_2$ appears in $\Rad^3 P_1/\Rad^4 P_1$. Indeed, if otherwise then $P_1$ would have the radical series
\[
\begin{smallmatrix} D_1 \\ D_1\oplus D_2 \\ D_2 \oplus D_1 \\ D_1 \\ D_2 \\ D_1 \\ D_2 \\ D_1 \\ D_1\oplus D_2 \\ D_2\oplus D_1 \\ D_1\end{smallmatrix}
\]
which contradicts $\Rad^8(P_2)=0$. Define the following factor modules of $P_1$ and $P_2$.
\[
Q_1=\begin{smallmatrix} D_1 \\ D_1\oplus D_2 \\ D_2\oplus D_1 \\ D_2 \end{smallmatrix}, \quad
Q_2=\begin{smallmatrix} D_2 \\ D_1 \\ D_1\oplus D_2 \\ D_2 \end{smallmatrix}
\]
Let $Q=Q_1\oplus Q_2$ and $P=P_1\oplus P_2$. Then, $\End_A(P)^{\rm op}$ is the basic algebra of $A$ and we have a surjective algebra homomorphism
\[
\End_A(P)^{\rm op} \longrightarrow B=\End_A(Q)^{\rm op}.
\]
We consider the two-point algebra defined by the quiver 
\[
\begin{xy}
(0,0) *[o]+{\circ}="A", (10,0) *[o]+{\circ}="B", 
\ar @(lu,ld)  "A";"A"_{\alpha}
\ar @<1mm> "A";"B"^{\mu}
\ar @<1mm> "B";"A"^{\nu}
\end{xy}
\]
and the relations
$\alpha^2=\mu\nu$, $\alpha^2\mu=0$, $\alpha^3=0$.
This is the algebra (23) in \cite[Table W]{H-wild-two-point}. 
Their indecomposable projective modules are
\[
\begin{smallmatrix} e_1 \\ \alpha,\; \nu \\ \nu\alpha, \; \mu\nu \\ \nu\mu\nu \end{smallmatrix}
=\begin{smallmatrix} D_1 \\ D_1\oplus D_2 \\ D_2\oplus D_1 \\ D_2 \end{smallmatrix}\quad
\begin{smallmatrix} e_2 \\ \mu \\ \alpha\mu, \; \nu\mu \\ \nu\alpha\mu \end{smallmatrix}
=\begin{smallmatrix} D_2 \\ D_1 \\ D_1\oplus D_2 \\ D_2 \end{smallmatrix}
\]
and this algebra is isomorphic to $B$. Since $B$ is wild, $A=R^{3\Lambda_{\ell-1}}(2\alpha_{\ell-1}+\alpha_\ell)$ is wild.
\end{proof}

\subsection{(i'') $\Lambda=2\Lambda_a+\widetilde{\Lambda}$ $(0\le a\le\ell-1)$ and $\Lambda'=2\Lambda_{a+1}+\widetilde{\Lambda}$ } In this case,  
\[
\beta_{\Lambda'}=\alpha_0+2(\alpha_1+\cdots+\alpha_a).
\]
By symmetry, we obtain the result for case (i'').

\subsection{(ii)  $\Lambda=2\Lambda_a+\widetilde{\Lambda}$ $(1\le a\le\ell-1)$ and $\Lambda'=\Lambda_{a-1}+\Lambda_{a+1}+\widetilde{\Lambda}$.} In this case, $\beta_{\Lambda'}=\alpha_a$ and $R^\Lambda(\beta_{\Lambda'})$ is (f1). 

\subsection{ (iii'') $\Lambda=2\Lambda_a+\widetilde{\Lambda}$ $(0\le a\le\ell-2)$ and $\Lambda'=\Lambda_a+\Lambda_{a+2}+\widetilde{\Lambda}$.} In this case, 
\[
\beta_{\Lambda'}=\alpha_0+2\alpha_1+\cdots+2\alpha_a+\alpha_{a+1}.
\]
If $1\le a\le\ell-2$ then $R^\Lambda(\beta_{\Lambda'})$ is wild by Theorem \ref{level-2-a=b}(1)(iii''). 
The case $a=0$ follows from the general result for $R^\La(\alpha_0+\alpha_1)$ which we will give now.

Recall that $R^\Lambda(\alpha_0+\alpha_1)$ is (f2) if $m_0\ge1$ and $m_1=0$, or $m_0=m_1=1$, and (t3) or (t7) if $m_0\ge2$ and $m_1=1$, or $m_0=1$ and $m_1=2$. Note that $m_0=0$ cannot happen because $\langle \alpha^\vee_0, \Lambda-\alpha_0-\alpha_1 \rangle=-1<0$. 
We show that $R^\Lambda(\alpha_0+\alpha_1)$ is wild if $m_0\ge 2$ and $m_1\ge 2$ or $m_0=1$ and $m_1\ge 3$. 

\begin{lemma}\label{wild alpha_0+alpha_1} 
The algebra $R^{2\Lambda_0+2\Lambda_1}(\alpha_0+\alpha_1)$ is wild. 
\end{lemma}
\begin{proof}
Set $A=R^{2\Lambda_0+2\Lambda_1}(\alpha_0+\alpha_1)$ and $B=e(10)Ae(10)$.
Then 
$$
\dim_q B=1+q^2 +q^4+q^6+q^8+q^{10}.
$$
We have $x_1^2e(10)=0$ and $x_1^2e(01)=0$, which imply 
$$0=-\psi_1x_1^2e(01)\psi_1=-x_2^2\psi_1^2e(10)=-x_2^2(x_1^2-x_2)e(10)=x_2^3e(10). $$ 
This together with $x_1^2e(10)=0$, the graded dimension shows that $B$ has a basis 
$$\{x_1^ax_2^be(10)\mid 0\le a\le 1, 0\le b\le 2 \}.$$
Further, $B/(x_1x_2^2e(10))\cong \Bbbk[X,Y]/(X^2, Y^3,XY^2)$ by sending $x_1e(10)$ and $x_2e(10)$ to $X$ and $Y$, respectively. This implies $B$ is wild and so is $A$, proving (3).
\end{proof}  

\begin{lemma}
The algebra $R^{\Lambda_0+3\Lambda_1}(\alpha_0+\alpha_1)$ is wild. 
\end{lemma}
\begin{proof}
Recall the algebra $A'$ in Lemma \ref{t>2:alpha_0+alpha_1} which is isomorphic to $R^{\Lambda_0+3\Lambda_1}(\alpha_0+\alpha_1)$. It has 
the algebra (31) in \cite[Table W]{H-wild-two-point} as a quotient algebra. The assertion follows. 
\end{proof}  

\subsection{ (iii') $\Lambda=2\Lambda_a+\widetilde{\Lambda}$ $(2\le a\le\ell)$ and $\Lambda'=\Lambda_{a-2}+\Lambda_a+\widetilde{\Lambda}$.}
In this case, \[
\beta_{\Lambda'}=\alpha_{a-1}+2\alpha_a+\cdots+2\alpha_{\ell-1}+\alpha_\ell.
\] 
By symmetry, we have the result for this case from (iii'').

\subsection{ (iv') $\Lambda=\Lambda_a+\Lambda_b+\widetilde{\Lambda}$ $(1\le a<b\le\ell)$ and $\Lambda'=\Lambda_{a-1}+\Lambda_{b-1}+\widetilde{\Lambda}$. }
In this case, \[
\beta_{\Lambda'}=\alpha_a+\cdots+\alpha_{b-1}+2\alpha_b+\cdots+2\alpha_{\ell-1}+\alpha_\ell.
\]
If $1\le a<b\le \ell-1$ then $R^\Lambda(\beta_{\Lambda'})$ is wild by Theorem \ref{level-2-a<b}(iv').

Suppose $1\le a\le\ell-2$ and $b=\ell$. If $m_i=\delta_{ai}$, for $a\le i\le \ell-1$, then $R^\Lambda(\beta_{\Lambda'})$ is (t6). We show that $R^\Lambda(\beta_{\Lambda'})$ is wild if $m_a\ge 2$ or $m_i\ge 1$, for some $a<i<\ell$.

\begin{lemma}
Suppose that $\Lambda=2\Lambda_a+\Lambda_\ell$ and $\Lambda'=\Lambda_{a-1}+\Lambda_a+\Lambda_{\ell-1}$. Then $R^\Lambda(\beta_{\La'})$ is wild.
\end{lemma}
\begin{proof}
Set $e=e(\ell~\ell-1~\ldots ~a+1~a)$ and $A=eR^\Lambda(\beta_{\Lambda'})e$.
Then $\dim_q A=1+2q^2+2q^4+q^6$.
We have $x_1e=0$ and $\psi_ie=0$ for $1\le i\le \ell-a-1$. This implies that 
$$x_2^2e=0, x_ie=x_2e, \text{ for } 3\le i\le \ell-a.$$
Therefore, the degree $2$ and the degree $4$ components of $A$ have bases 
$$ \{x_2e, x_{\ell-a+1}e\} \quad\text{and}\quad\{x_2x_{\ell-a+1}e, x_{\ell-a+1}^2e\},$$
respectively. We conclude that $A/\Rad^3A\cong \Bbbk [X,Y]/(X^2,Y^3,XY^2)$, which is wild.    
\end{proof}
\begin{lemma}
    Suppose that $\Lambda=\Lambda_a+\Lambda_i+\Lambda_\ell$ and $\Lambda'=\Lambda_{a-1}+\Lambda_i+\Lambda_{\ell-1}$ for some $a<i<\ell$. Then $R^\Lambda(\beta)$ is wild.
\end{lemma}
\begin{proof}
Set $A=eR^\Lambda(\beta)e$, where $e=e_1+e_2$ with 
$e_1=e(\ell~\ell-1~\ldots ~a+1~a)$
    and $e_2=e(i~\ell~\ell-1~\ldots i+1~i-1~i-2~\ldots~a+1~a)$.
 If $i<\ell-1$,    then $\dim_q e_1Ae_1=1+3q^2+3q^4+q^6$.
 If $i=\ell-1$, then
 \begin{gather*}
     \dim_q e_1Ae_1=1+2q^2+2q^4+q^6, \;\; \dim_q e_2Ae_2=1+q^2+q^4+q^6, \\
     \dim_q e_1Ae_2=\dim_q e_2Ae_1=q^2+q^4. 
 \end{gather*}
 In any case, we have that $A$ is wild.
\end{proof}

It remains to consider the case $a=\ell-1$ and $b=\ell$. If $m_\ell\ge2$, it is already considered in (iii'). Thus we assume $m_{\ell-1}\ge1$ and $m_\ell=1$. 
$R^\Lambda(\beta_{\Lambda'})$ is (f3) if $m_{\ell-1}=1$. If $m_{\ell-1}\ge2$, we have an isomorphism of algebras
\[
R^\Lambda(\beta_{\Lambda'})\cong R^{m_{\ell-1}\Lambda_{\ell-1}+\Lambda_\ell}(\alpha_{\ell-1}+\alpha_\ell).
\]
This is the algebra we analyzed at the beginning of Section 7. Thus, it is (t8) if $m_{\ell-1}=2$, wild if $m_{\ell-1}\ge3$. 

\subsection{ (iv'')
$\Lambda=\Lambda_a+\Lambda_b+\widetilde{\Lambda}$ $(0\le a<b\le \ell-1)$ and $\Lambda'=\Lambda_{a+1}+\Lambda_{b+1}+\widetilde{\Lambda}$. } In this case, \[
\beta_{\Lambda'}=\alpha_0+2\alpha_1+\cdots+2\alpha_a+\alpha_{a+1}+\cdots+\alpha_b.
\]
By symmetry, we have the result from case (iv').

\subsection{The cases (v), (vi), (viii'), (viii'').}
\begin{itemize}
\item 
[(v)] $\Lambda=\Lambda_a+\Lambda_b+\widetilde{\Lambda}$ $(1\le a<b\le \ell-1)$ and $\Lambda'=\Lambda_{a-1}+\Lambda_{b+1}+\widetilde{\Lambda}$. 
In this case, 
$$
\beta_{\Lambda'}=\alpha_a+\alpha_{a+1}+\cdots+\alpha_b.
$$
Then the result from \cite{ASW-rep-type} for type $A_\ell^{(1)}$ shows that $R^\Lambda(\beta_{\Lambda'})$ is
\begin{itemize}
\item
finite if $m_i=\delta_{ai}+\delta_{bi}$, for $a\le i\le b$, namely (f4),
\item
tame if $m_a\ge2$ and $m_i=\delta_{bi}$, for $a<i\le b$, or $m_b\ge2$ and $m_i=\delta_{ai}$, for $a\le i<b$, namely (t9),
\item
wild otherwise.
\end{itemize}
\item[(vi)]
If $\Lambda=\Lambda_a+\Lambda_b+\widetilde{\Lambda}$ $(0\le a<b\le \ell)$ and $\Lambda'=\Lambda_{a+1}+\Lambda_{b-1}+\widetilde{\Lambda}$,
where $a\le b-2$, the level two result Theorem \ref{level-2-a<b}(vi) implies that $R^\Lambda(\beta_{\Lambda'})$ is wild for $0\le a<b\le \ell$ with $a\ne b-1$.

\item[(viii')]
If $\Lambda=\Lambda_a+\Lambda_b+\widetilde{\Lambda}$ $(0\le a<b\le \ell-2)$ and $\Lambda'=\Lambda_a+\Lambda_{b+2}+\widetilde{\Lambda}$,
\[
\beta_{\Lambda'}=\alpha_0+2\alpha_1+\cdots+2\alpha_b+\alpha_{b+1}.
\]
Then $R^\Lambda(\beta_{\Lambda'})$ is wild, for $0\le a<b\le \ell-2$, by Theorem \ref{level-2-a<b}(viii').

\item[(viii'')]
If $\Lambda=\Lambda_a+\Lambda_b+\widetilde{\Lambda}$ $(2\le a<b\le \ell)$ and $\Lambda'=\Lambda_{a-2}+\Lambda_b+\widetilde{\Lambda}$, then
\[
\beta_{\Lambda'}=\alpha_{a-1}+2\alpha_a+\cdots+2\alpha_{\ell-1}+\alpha_\ell.
\]
By symmetry, Theorem \ref{level-2-a<b}(viii'') implies that 
$R^\Lambda(\beta_{\Lambda'})$ is wild, for $2\le a<b\le \ell$.
\end{itemize}

\subsection{The remaining cases.}
\begin{itemize}
\item[(vii')]
If $\Lambda=\Lambda_a+\Lambda_b+\widetilde{\Lambda}$ $(0\le a<b\le \ell, b\ge2)$ and $\Lambda'=\Lambda_a+\Lambda_{b-2}+\widetilde{\Lambda}$,
it suffices to assume $a\le b-2$, because if $a=b-1$ then $\Lambda_a+\Lambda_{b-2}=\Lambda_{a-1}+\Lambda_{b-1}$ and
it is already treated in (iv'). We have
\[
\beta_{\Lambda'}=\alpha_{b-1}+2\alpha_b+\cdots+2\alpha_{\ell-1}+\alpha_\ell.
\]
If $m_i=\delta_{bi}$, for $b-1\le i\le\ell$, it is (f6). 
If $m_{b-1}\ge1$, the arrow is
\[
\Lambda=\Lambda_{b-1}+\Lambda_b+\widetilde{\Lambda}'\longrightarrow\Lambda'=\Lambda_{b-2}+\Lambda_{b-1}+\widetilde{\Lambda}',
\]
and it is already treated in (iv'). If $m_b\ge2$, the arrow is of the form
\[
\Lambda=2\Lambda_b+\widetilde{\Lambda}'\longrightarrow\Lambda'=\Lambda_{b-2}+\Lambda_b+\widetilde{\Lambda}',
\]
and it is already treated in (iii'). If $m_i\ge1$, for some $b+1\le i\le \ell$, the arrow is
\[
\Lambda=\Lambda_b+\Lambda_i+\widetilde{\Lambda}'\longrightarrow\Lambda'=\Lambda_{b-2}+\Lambda_i+\widetilde{\Lambda}',
\]
and $R^\Lambda(\beta_{\Lambda'})$ is wild by (viii'').

\item[(vii'')]
If $\Lambda=\Lambda_a+\Lambda_b+\widetilde{\Lambda}$ $(0\le a<b\le \ell, a\le \ell-2)$ and $\Lambda'=\Lambda_{a+2}+\Lambda_b+\widetilde{\Lambda}$,
we may assume $a\le b-2$, because if $a=b-1$ then $\Lambda_{a+2}+\Lambda_b=\Lambda_{a+1}+\Lambda_{b+1}$ and
it is already treated in (iv''). We have
\[
\beta_{\Lambda'}=\alpha_0+2\alpha_1+\cdots+2\alpha_a+\alpha_{a+1}.
\]
Then, by symmetry, we see that no new non-wild algebra appears. 
\end{itemize}

\section{Second neighbors in higher level cases}
By the result on the first neighbors, it suffices to check the representation type of $R^\Lambda(\beta_{\Lambda''})$ for $\Lambda\rightarrow \Lambda'\rightarrow \Lambda''$ in the following cases in the second neighbors.

\begin{itemize}
\item[(1)]
$\Lambda=2\Lambda_\ell+\widetilde{\Lambda}\rightarrow \Lambda'=2\Lambda_{\ell-1}+\widetilde{\Lambda}$ and
$\Lambda=2\Lambda_0+\widetilde{\Lambda}\rightarrow \Lambda'=2\Lambda_1+\widetilde{\Lambda}$.
\item[(2)]
$\Lambda=2\Lambda_{\ell-1}+\widetilde{\Lambda}\rightarrow \Lambda'=2\Lambda_{\ell-2}+\widetilde{\Lambda}$ and
$\Lambda=2\Lambda_1+\widetilde{\Lambda}\rightarrow \Lambda'=2\Lambda_2+\widetilde{\Lambda}$.
\item[(3)]
$\Lambda=2\Lambda_a+\widetilde{\Lambda}\rightarrow \Lambda'=\Lambda_{a-1}+\Lambda_{a+1}+\widetilde{\Lambda}$ $(1\le a\le \ell-1)$.
\item[(4)]
$\Lambda=2\Lambda_\ell+\widetilde{\Lambda}\rightarrow \Lambda'=\Lambda_{\ell-2}+\Lambda_\ell+\widetilde{\Lambda}$ and
$\Lambda=2\Lambda_0+\widetilde{\Lambda}\rightarrow \Lambda'=\Lambda_0+\Lambda_2+\widetilde{\Lambda}$.
\item[(5)]
$\Lambda=\Lambda_a+\Lambda_\ell+\widetilde{\Lambda}\rightarrow \Lambda'=\Lambda_{a-1}+\Lambda_{\ell-1}+\widetilde{\Lambda}$ $(1\le a\le \ell-1)$ and

\noindent
$\Lambda=\Lambda_0+\Lambda_b+\widetilde{\Lambda}\rightarrow \Lambda'=\Lambda_1+\Lambda_{b+1}+\widetilde{\Lambda}$ $(1\le b\le \ell-1)$.
\item[(6)]
$\Lambda=\Lambda_a+\Lambda_b+\widetilde{\Lambda}\rightarrow \Lambda'=\Lambda_{a-1}+\Lambda_{b+1}+\widetilde{\Lambda}$ $(1\le a<b\le \ell-1)$.
\item[(7)]
$\Lambda=\Lambda_a+\Lambda_b+\widetilde{\Lambda}\rightarrow \Lambda'=\Lambda_a+\Lambda_{b-2}+\widetilde{\Lambda}$ $(0\le a<b\le \ell,\; a\le b-2)$ and

\noindent
$\Lambda=\Lambda_a+\Lambda_b+\widetilde{\Lambda}\rightarrow \Lambda'=\Lambda_{a+2}+\Lambda_b+\widetilde{\Lambda}$ $(0\le a<b\le \ell,\; a\le b-2)$.
\end{itemize}
The aim of this section is to show that no new non-wild algebra appears in the above seven cases. 
Our strategy for the proof is that we check the wildness of the algebras case by case. Basically, most algebras $R^\Lambda(\beta_{\Lambda''})$ in each case will belong to the following three patterns. Since we will use similar arguments repeatedly in each pattern, 
we adopt the following style of writing in order to avoid repetition. 
\begin{itemize}
    \item[(I)] $\Lambda''$ is already in the first neighbors and hence already done in the previous section. By the definition of arrows,  it is easy to see that $\Lambda''$ can be reached from $\Lambda$ with one move. We list $\Lambda''$ in this pattern without further proof.
    \item[(II)] $\Lambda''$ is not in the first neighbors but there is an arrow 
    $ \Lambda_{mid}\rightarrow\Lambda''$ such that we may know that $R^\Lambda(\beta_{\Lambda_{mid}})$ is wild, by the results of the first neighbors or level two results. Then $R^\Lambda(\beta_{\Lambda''})$ is wild. In this pattern, we will write the arrow (or just $\Lambda_{mid}$ for each $\Lambda''$) and refer to the previous sections for the wildness of $R^\Lambda(\beta_{\Lambda_{mid}})$.   

A variant of this argument is that $R^\Lambda(\beta_{\Lambda_{mid}})$ is not wild, 
but we know by results in the previous sections that $R^\La(\beta_{\La''})$ is wild for the path $\La_{mid}\rightarrow \La''$. 
    
    \item[(III)] We may use Lemma \ref{tensor product lemma}(tensor product lemma) to show that $R^\Lambda(\beta_{\Lambda''})$ is Morita equivalent to the tensor product of two algebras. Then the wildness of the tensor product is easy to see. For this pattern, we will just write the tensor product of two algebras without referring to Lemma \ref{tensor product lemma} explicitly. 
\end{itemize}
For the new non-wild algebras, we will see that they all belong to the tame cases listed in MAIN THEOREM.

\subsection{Case (1).}
This case studies $\Lambda=2\Lambda_0+\widetilde{\Lambda}\to \Lambda'=2\Lambda_1+\widetilde{\Lambda} \to \Lambda''$ and $\beta_{\Lambda'}=\alpha_0$.
We divide the cases according to the number of changes of fundamental weights in $\Lambda$.

\subsubsection{The case there are $2$ changes}
Since we consider $\Lambda=2\Lambda_0+\widetilde{\Lambda}\to \Lambda'=2\Lambda_1+\widetilde{\Lambda}$, we change $2\Lambda_1$ in $\Lambda'$. 
If we obtain $\Lambda_1+\Lambda_3$ by  $\Delta_{1^+}$, then 
$R^{2\Lambda_0}(\beta_{\Lambda''})$ is wild by Lemma~\ref{lem:2lambda0}.
We consider $\Lambda''=\Lambda_0+\Lambda_2+\tilde \Lambda$ obtained 
by $ \Delta_{1^-,1^+}$ and $\Lambda''=2\Lambda_2+\tilde \Lambda$ obtained by $\Delta_{1^+,1^+}$. 
The former case is already handled in (iii'')(b) of the first neighbors, i.e., it belongs to pattern (I).
 
 We consider the case $\Delta_{1^+,1^+}$. Suppose first that $m_1\ge 2$. Observe that 
 we have a path $\Lambda_0+2\Lambda_1+\Lambda_2 \rightarrow 2\Lambda_1+2\Lambda_2$ where 
 $\beta^{2\Lambda_0+2\Lambda_1}_{\Lambda_0+2\Lambda_1+\Lambda_2}=\alpha_0+\alpha_1$ and 
 $\beta^{2\Lambda_0+2\Lambda_1}_{2\Lambda_1+2\Lambda_2}=2\alpha_0+2\alpha_1$. Thus, 
 Lemma \ref{wild alpha_0+alpha_1} implies that $R^\La(\beta_{\La''})$ is wild. 
 If $m_0=2$ and $m_1=0$, then $R^\La(\beta_{\La''})$ is wild if $\ch \k=2$ by Lemma \ref{lem::2-La_2-level-two} . 
 If $\ch \k\ne 2$ then it is (t20). Since we know the representation type in level two, we show that $R^\La(2\alpha_0+2\alpha_1)$ is wild for higher levels $k\ge 3$. For this, it suffices to prove that $R^\Lambda(2\alpha_0+2\alpha_1)$ is wild for the following two cases.
 \begin{itemize}
     \item $m_0=2$ and $m_1=1$.
     \item $m_0\ge 3$ and $m_1=0$.
 \end{itemize}
 In case $\Delta_{i^+}$ of Subsection 10.1.2 below, we show that $R^{3\Lambda_0}(2\alpha_0+\alpha_1)$ is wild. Since there is a path 
 $2\Lambda_1+\Lambda_2\rightarrow \Lambda_0+2\Lambda_2$ and 
 $\beta^{3\La_0}_{2\La_1+\La_2}=2\alpha_0+\alpha_1$, $\beta^{3\La_0}_{\La_0+2\La_2}=2\alpha_0+2\alpha_1$, 
 we see that $R^{3\Lambda_0}(2\alpha_0+2\alpha_1)$ is wild. Hence $R^\Lambda(2\alpha_0+2\alpha_1)$ is wild when $m_0\ge 3$ and $m_1=0$.
 It remains to consider the case $m_0=2, m_1=1$.
We set $\Lambda=2\Lambda_0+\Lambda_1$ and $\Lambda''=\Lambda_1+2\Lambda_2$. Thus, 
$\beta_{\Lambda''}=2\alpha_0+2\alpha_1$. We choose
$P=f_0f_1^{(2)}f_1v_\Lambda\in V(\Lambda_0)\otimes V(\Lambda_0)\otimes V(\Lambda_1)$. Then
$$
P=f_0f_1^{(2)}\left(((0),(1),(0))+q^2((1),(0),(0))\right)
$$
is obtained by applying $f_0$ to
\begin{gather*}
((0),(1^2),(1))+q((0),(2),(1))+q^2((0),(2,1),(0))\\
\qquad\qquad +q^2((1^2),(0),(1))+q^3(((2),(0),(1))+q^4((2,1),(0),(0)).
\end{gather*}
Each $3$-partition has two addable $0$-nodes and no removable $0$-node. Thus,
$$
\begin{aligned}
\dim_q \End(P)&=(1+q^4)(1+q^2+2q^4+q^6+q^8) \\
&=1+q^2+3q^4+2q^6+3q^8+q^{10}+q^{12},
\end{aligned}
$$
and we apply Lemma \ref{lem::wild-three-loops-part2} to conclude that $R^{2\Lambda_0+\Lambda_1}(2\alpha_0+2\alpha_1)$ is wild.

\subsubsection{The case there are $3$ changes}
We consider
$$
\Lambda=2\Lambda_0+\Lambda_i+\widetilde{\Lambda}\to \Lambda'=(\Lambda_1+\Lambda_i)+\Lambda_1+\widetilde{\Lambda},
$$
and we change $\Lambda_1+\Lambda_i$ in $\Lambda'$. Since the number of changes is $3$, we must change $\Lambda_i$.
First we note that the $\Lambda''$ obtained by the arrows $ \Delta_{i^-}$ for $i=2$, $\Delta_{1^+,i^-}$ for $i=1$, $\Delta_{1^-,i^-}$ and $\Delta_{1^-,i^+}$ belong to pattern (I). 
The pattern (II) cases are as follows with the path listed below.
\begin{itemize}
    \item 
[$(\Delta_{+-})$]:  $\Lambda_1+\Lambda_i \to \Lambda_2+\Lambda_{i-1}$ with
$2\le i\le \ell$: 
$$
\Lambda\to
\Lambda_{mid}=(\Lambda_1+\Lambda_{i-1})+\Lambda_0+\widetilde{\Lambda}\to
\Lambda''=\Lambda_1+\Lambda_2+\Lambda_{i-1}+\widetilde{\Lambda}, \quad  \text{by (vi).}
$$
Here by (vi) means (vi) in the first neighbors implies that $R^\Lambda(\beta_{\Lambda_{mid}})$ is wild. 
\item[$(\Delta_{i^+})$] with $2\le i\le \ell-2$
:
$$
\Lambda\to
\Lambda_{mid}=(\Lambda_0+\Lambda_{i+2})+\Lambda_0+\widetilde{\Lambda}
\to \Lambda''=2\Lambda_1+\Lambda_{i+2}+\widetilde{\Lambda},  \quad \text{ by (viii')}.
$$
\end{itemize}
We consider the remaining cases in the following.
\begin{itemize}
\item[$(\Delta_{i^+})$]
The change is $\Lambda_1+\Lambda_i \to \Lambda_1+\Lambda_{i+2}$. We have subcases.
\begin{itemize}
\item[$(i=0)$]
$\Lambda-\Lambda''=3\Lambda_0-2\Lambda_1-\Lambda_2$ and $\beta_{\Lambda''}=2\alpha_0+\alpha_1$.
We show that $A=R^{3\Lambda_0}(2\alpha_0+\alpha_1)$ is wild. Let $e=e(010)$. Then
$$
\dim_q eAe=1+q^2+2q^4+2q^6+2q^8+2q^{10}+q^{12}+q^{14}.
$$

If $e\psi_w e\ne0$, $w=e$ or $s_1s_2s_1=s_2s_1s_2$. We have $\psi_1\psi_2\psi_1e=0$ by $\psi_1 e=e(101)=0$, and
$\psi_2\psi_1\psi_2e=\psi_1\psi_2\psi_1e+e=e$.
Hence, $eAe$ is generated by $x_2e$ and $x_3e$ because $(x_1-x_2^2)e=\psi_1^2e=0$, and $x_1^3e=0$ implies $x_2^6e=0$.

Since $\deg x_2e=2$ and $\deg x_3e=4$,
$$
\begin{array}{|c|c|}\hline
\deg=0 & e \\\hline
\deg=2 & x_2e \\\hline
\deg=4 & x_2^2e,\; x_3e \\\hline
\deg=6 & x_2^3e,\; x_2x_3e \\\hline
\deg=8 & x_2^4e,\; x_2^2x_3e,\; x_3^2e \\\hline
\deg=10 & x_2^5e,\; x_2^3x_3e,\; x_2x_3^2e \\\hline
\deg=12 & x_2^4x_3e,\; x_2^2x_3^2e \\\hline
\deg=14 & x_2^5x_3e \\\hline
\end{array}
$$
Define $X=x_2e$ and $Y=x_3e$, which generate $eAe$, and 
let $J$ be the ideal of $eAe$ spanned by $\k x_2^3e$ and elements of degree greater than or equal to $8$. Then, 
$eAe/J\cong \k[X,Y]/(X^3,X^2Y,Y^2)$, which is wild \cite{Ringel-local-alg}.
\item[$(i=1)$]
Since $\Lambda=2\Lambda_0+\Lambda_1+\widetilde{\Lambda}$ and
$\Lambda''=2\Lambda_1+\Lambda_3+\widetilde{\Lambda}$,
the number of changes is $2$. We know that it is wild, since $R^{2\La_0}(\beta_{\La_1+\La_3})$ is wild by Theorem \ref{level-2-a=b}(2).

\end{itemize}
\item[$(\Delta_{i^-})$]
The change is $\Lambda_1+\Lambda_i \to \Lambda_1+\Lambda_{i-2}$, where
$2< i\le \ell$. We have subcases.

\medskip
\begin{itemize}
\item[$(3\le i\le \ell-2)$]
We have $\beta_{\Lambda''}=\alpha_0+\alpha_{i-1}+2\alpha_i+\cdots+2\alpha_{\ell-1}+\alpha_\ell$. 
We show that $R^{2\Lambda_0+\Lambda_i}(\beta_{2\Lambda_1+\Lambda_{i-2}})$ is wild. By 
Lemma \ref{tensor product lemma}, 
$R^{2\Lambda_0+\Lambda_i}(\beta_{2\Lambda_1+\Lambda_{i-2}})$ is Morita equivalent to 
$R^{2\Lambda_0}(\alpha_0)\otimes R^{\Lambda_i}(\beta_{\Lambda_{i-2}})$ and 
the proof of \cite[Proposition 4.1]{CH-type-c-level-1} showed that
$R^{\Lambda_i}(\beta_{\Lambda_{i-2}})$ is Morita equivalent to the Brauer line algebra with $\ell-i+1$ simples, so that its Gabriel quiver is
\begin{center}
$\begin{xy}
(0,0) *{\circ}="A1",
(15,0) *{\circ}="A2",
(30,0) *{\circ}="A3",
(35,0) *{}="X",
(45,0) *{}="Y",
(50,0) *{\circ}="A5",
(65,0) *{\circ}="A6",
(80,0) *{\circ}="A7"
\ar @<0.5ex>^{\alpha} "A1";"A2"
\ar @<0.5ex>^{\beta} "A2";"A1"
\ar @<0.5ex>^{\alpha} "A2";"A3"
\ar @<0.5ex>^{\beta} "A3";"A2"
\ar @{--} "X";"Y"
\ar @<0.5ex>^{\alpha} "A5";"A6"
\ar @<0.5ex>^{\beta} "A6";"A5"
\ar @<0.5ex>^{\alpha} "A6";"A7"
\ar @<0.5ex>^{\beta} "A7";"A6"
\end{xy}$
\end{center}
Then $R^{2\Lambda_0}(\alpha_0)\cong \k[X]/(X^2)$ implies that we obtain
the Gabriel quiver of
$R^{2\Lambda_0+\Lambda_i}(\beta_{2\Lambda_1+\Lambda_{i-2}})$ by
adding one loop on each vertex. By considering the separated quiver of the Gabriel quiver,
we know that $R^{2\Lambda_0+\Lambda_i}(\beta_{2\Lambda_1+\Lambda_{i-2}})$ is wild because $\ell-i+1\ge3$.
\item[$(i=\ell-1)$]
By the same argument above,
$\beta_{\Lambda''}=\alpha_{\ell-2}+2\alpha_{\ell-1}+\alpha_\ell$ and the basic algebra
of $R^{\Lambda_{\ell-1}}(2\alpha_{\ell-1}+\alpha_\ell)$ is isomorphic to the path algebra
$$
\hspace{1cm}
\begin{xy}
(0,0) *{\circ}="A1",
(15,0) *{\circ}="A2",

\ar @<0.5ex>^{\mu} "A1";"A2"
\ar @<0.5ex>^{\nu} "A2";"A1"
\end{xy}
$$

bounded by the relations $\mu\nu\mu=\nu\mu\nu=0$. Thus, by adding a loop $\alpha$ on the left vertex and a loop $\beta$ on the right vertex, we get the bound quiver presentation and the newly added relations are
$$
\alpha\mu-\nu\beta=\alpha^2=\beta^2=\beta\nu-\nu\alpha=0.
$$
If we also add the relation $\beta\nu=\nu\alpha=0$, we obtain the algebra (32) from \cite[Table W]{H-wild-two-point}. Hence, $R^\Lambda(\beta_{\Lambda''})$ is wild.
\item[$(i=\ell)$]
We have $\beta_{\Lambda''}=\alpha_0+\alpha_{\ell-1}+\alpha_\ell$, which is (t14) if $m_0=2$, $m_{\ell-1}=0$, $m_\ell=1$. Suppose $m_0\ge3$. Then,
$R^{3\Lambda_0+\Lambda_\ell}(\alpha_0+\alpha_{\ell-1}+\alpha_\ell)$ is Morita equivalent to
$R^{3\Lambda_0}(\alpha_0)\otimes R^{\Lambda_\ell}(\alpha_{\ell-1}+\alpha_\ell)$ by Lemma \ref{tensor product lemma}. 
By \cite[Lemma 3.3(1)]{AP-rep-type-C-level-1}, we have $R^{\Lambda_\ell}(\alpha_{\ell-1}+\alpha_\ell)\cong \k[Y]/(Y^2)$. Thus,
$$
R^{3\Lambda_0}(\alpha_0)\otimes R^{\Lambda_\ell}(\alpha_{\ell-1}+\alpha_\ell)
\cong \k[X,Y]/(X^3,Y^2),
$$
which is wild by \cite{Ringel-local-alg}. Now we consider $m_0=2$ but 
$m_{\ell-1}=0$ and $m_\ell=2$, or $m_{\ell-1}=m_\ell=1$. Then,
$R^\Lambda(\alpha_0+\alpha_{\ell-1}+\alpha_\ell)$ is obtained by tensoring $\k[X]/(X^2)$ with
$R^{2\Lambda_\ell}(\alpha_{\ell-1}+\alpha_\ell)$ or 
$R^{\Lambda_{\ell-1}+\Lambda_\ell}(\alpha_{\ell-1}+\alpha_\ell)$. 
Both algebras are (f2): the former is isomorphic to $\k[Y]/(Y^4)$, and 
the Gabriel quiver of the latter is as follows. 
$$
\hspace{1cm}
\begin{xy}
(0,0) *{\circ}="A1",
(15,0) *{\circ}="A2",

\ar @(lu,ld) "A1";"A1"
\ar @<0.5ex> "A1";"A2"
\ar @<0.5ex> "A2";"A1"
\end{xy}
$$
Hence, $R^\La(\beta_{\La'})$ is wild in both cases. 
\end{itemize}
\item[$(\Delta_{1^+,i^+})$]
The change is $\Lambda_1+\Lambda_i \to \Lambda_2+\Lambda_{i+1}$, where
$0\le i\le \ell-1$. We have subcases.

\medskip
\begin{itemize}
\item[$(i=0,1)$]
If $i=0$ (resp. $i=1$), then $\Lambda''$ is already appeared in case  $(\Delta_+)$ (resp. of two changes) above.

\item[$(2\le i\le\ell-1)$]
We consider $R^{2\Lambda_0+\Lambda_i}(2\alpha_0+2\alpha_1+\alpha_2+\cdots+\alpha_i)$. We have subcases.
\begin{itemize}
\item[$(i=2)$]
We set $P_1=f_2f_1^{(2)}f_0^{(2)}v_\Lambda$ and $P_2=f_1^{(2)}f_2f_0^{(2)}v_\Lambda$. Then
$$
\begin{aligned}
\dim_q \Hom(P_1,P_1)&=1+2q^2+4q^4+4q^6+4q^8+2q^{10}+q^{12}, \\
\dim_q \Hom(P_2,P_2)&=1+q^2+2q^4+2q^6+2q^8+q^{10}+q^{12}, \\
\dim_q \Hom(P_1,P_2)&=q^2+q^4+2q^6+q^8+q^{10}.
\end{aligned}
$$
Hence, $P_1$ and $P_2$ are indecomposable projective modules and
the Gabriel quiver of $R^{2\Lambda_0+\Lambda_i}(2\alpha_0+2\alpha_1+\alpha_2)$ contains two loops on vertex $1$, one loop on vertex $2$, and arrows $1$ to $2$ and $2$ to $1$. Thus, $R^\Lambda(\beta_{\Lambda''})$ is wild.
\item[$(3\le i\le \ell-1)$]
We set $P=f_i\cdots f_2f_1^{(2)}f_0^{(2)}v_\Lambda$. Then
$$
\dim_q \End(P)=1+3q^2+5q^4+6q^6+5q^8+3q^{10}+q^{12}.
$$
Hence, $P$ is an indecomposable projective module and 
Lemma \ref{lem::wild-three-loops-part1} implies that $R^\Lambda(\beta_{\Lambda''})$ is wild.
\end{itemize}
\end{itemize}
\end{itemize}

\subsubsection{The case there are $4$ changes}
We consider
$$
\Lambda=2\Lambda_0+\Lambda_i+\Lambda_j+\widetilde{\Lambda}\to \Lambda'=(\Lambda_i+\Lambda_j)+2\Lambda_1+\widetilde{\Lambda},
$$
and suppose that the change is $\Lambda_i+\Lambda_j\to \Lambda_a+\Lambda_b$.
We first list cases in pattern (II) for  
$\Lambda_{mid}=2\Lambda_0+\Lambda_a+\Lambda_b+\widetilde{\Lambda}$. In the list, 
we use the numbering in Theorem \ref{level-2-a=b}(1) and Theorem \ref{level-2-a<b}(1). 

\begin{itemize}
\item[(i'),]
for $2\le i=j\le \ell-2$, where
$\beta_{\Lambda_{mid}}=2\alpha_i+\cdots+2\alpha_{\ell-1}+\alpha_\ell$.  
Note that $i=j=1$ implies $\Lambda''=\Lambda$ and it does not occur.
\item[(i''),]
for $2\le i=j\le \ell-1$, where
$\beta_{\Lambda_{mid}}=\alpha_0+2\alpha_1+\cdots+2\alpha_i$.
\item[(iii'),]
for $2\le i=j\le \ell-1$, where
$\beta_{\Lambda_{mid}}=\alpha_{i-1}+2\alpha_i+\cdots+2\alpha_{\ell-1}+\alpha_\ell$.
\item[(iii''),]
for $1\le i=j\le \ell-2$, where
$\beta_{\Lambda_{mid}}=\alpha_0+2\alpha_1+\cdots+2\alpha_i+\alpha_{i+1}$.
\item[(iv'),]
for $1\le i<j\le \ell-1$, where
$$
\beta_{\Lambda_{mid}}=(\alpha_i+\cdots+\alpha_{\ell-1})+(\alpha_j+\cdots+\alpha_{\ell-1})+\alpha_\ell.
$$
\item[(iv''),]
for $1\le i<j\le \ell-1$, where
$$
\beta_{\Lambda_{mid}}=\alpha_0+(\alpha_1+\cdots+\alpha_i)+(\alpha_1+\cdots+\alpha_j).
$$
\end{itemize}
\begin{itemize}
\item[(vi),]
for $0\le i<j\le \ell, i\ne j-1$, where
$$
\begin{aligned}
\beta_{\Lambda_{mid}}=(\alpha_0+2\alpha_1+\cdots+2\alpha_i)&+(\alpha_{i+1}+\cdots+\alpha_{j-1})\\
&+(2\alpha_j+\cdots+2\alpha_{\ell-1}+\alpha_\ell).
\end{aligned}
$$
\item[(viii'),]
for $2\le i<j\le \ell-2$, where
$\beta_{\Lambda_{mid}}=\alpha_0+2\alpha_1+\cdots+2\alpha_j+\alpha_{j+1}$.
\item[(viii''),]
for $2\le i<j\le \ell$, where
$\beta_{\Lambda_{mid}}=\alpha_{i-1}+2\alpha_i+\cdots+2\alpha_{\ell-1}+\alpha_\ell$.
\end{itemize}

Hence, the cases we must consider are as follows.
\begin{itemize}
\item[(1)]
$i=j=\ell$ and $a=b=\ell-1$, or $i=j=0$ and $a=b=1$.
\item[(2)]
$i=j=\ell-1$ and $a=b=\ell-2$, or $i=j=1$ and $a=b=2$.
\item[(3)]
$1\le i=j\le \ell-1$ and $(a,b)=(i-1,i+1)$.
\item[(4)]
$i=j=\ell$ and $(a,b)=(\ell-2,\ell)$, or $i=j=0$ and $(a,b)=(0,2)$.
\item[(5)]
$1\le i\le \ell-1$, $j=\ell$ and $(a,b)=(i-1,\ell-1)$, or $i=0$, $1\le j\le \ell-1$ and $(a,b)=(1,j+1)$.
\item[(6)]
$1\le i<j\le \ell-1$ and $(a,b)=(i-1,j+1)$.
\item[(7)]
$0\le i<j\le \ell$, $i\le j-2$ and either $(a,b)=(i,j-2)$, or $(a,b)=i+2,j)$.
\end{itemize}
In the cases (2) for $i=j=1$, (3) for $i=1$, (4), and (5) for $(i,j)=(1,\ell), (0,1)$,  (6) for $i=1$, and (7), we change at most three fundamental weights in $\Lambda$,
so that they have already been examined in 10.1.1 and 10.1.2. 
To convince the reader, we explain below that they actually appeared in 10.1.1 and 10.1.2. Recall that we are working with the case $m_0\ge2$ here. 
\begin{itemize}
\item 
When (2) for $i=j=1$, we have a directed path $2\Lambda_0\to2\Lambda_1\to 2\Lambda_2$, and it is $(\Delta_{++})$ in 10.1.1. 
\item 
When (3) for $i=1$, we have a directed path $2\Lambda_0\to2\Lambda_1\to\Lambda_0+\Lambda_2$, which is $(\Delta_{+-}=\Delta_{-+})$ in 10.1.1. This is (iii'') in the first neighbors.
\item
When (4) and $i=j=\ell$, the algebra is $R^\Lambda(\alpha_0+\alpha_{\ell-1}+\alpha_\ell)$. There is a directed path $2\Lambda_0+\Lambda_\ell\to 2\Lambda_1+\Lambda_\ell\to 2\Lambda_1+\Lambda_{\ell-2}$, and it is $(\Delta_-)$ with $i=\ell$ in 10.1.2. for $m_0\ge2$ and $m_\ell\ge2$. 
\item 
When (4) and $i=j=0$, the algebra is $R^\Lambda(2\alpha_0+\alpha_1)$. There is a directed path $3\Lambda_0\to 2\Lambda_1+\Lambda_0 \to 2\Lambda_1+\Lambda_2$, and 
it is $(\Delta_+)$ with $i=0$ in 10.1.2. for $m_0\ge3$. 
\item 
When (5) and $(i,j)=(1,\ell)$, there is a directed path $2\Lambda_0+\Lambda_\ell\to 2\Lambda_1+\Lambda_\ell\to \Lambda_0+\Lambda_1+\Lambda_{\ell-1}$. This is (vi) in the first neighbors. See $(\Delta_{--})$ in 10.1.2.
\item 
When (5) and $(i,j)=(0,1)$,
there is a directed path $3\Lambda_0\to \Lambda_0+2\Lambda_1\to 2\Lambda_1+\Lambda_2$, and the algebra is $R^\Lambda(2\alpha_0+\alpha_1)$. See $(\Delta_{++})$ with $i=0$ in 10.1.2.
\item 
When (6) and $i=1$, there is a directed path $2\Lambda_0+\Lambda_j\to 2\Lambda_1+\Lambda_j\to \Lambda_0+\Lambda_1+\Lambda_{j+1}$. This is (iv'') in the first neighbors. See $(\Delta_{-+})$ in 10.1.2.
\item 
When (7) and $(a,b)=(i,j-2)$,
there is a directed path $2\Lambda_0+\Lambda_j\to 2\Lambda_1+\Lambda_j\to 2\Lambda_1+\Lambda_{j-2}$, and this is $(\Delta_{-})$ in 10.1.2.
\item 
When (7) and $(a,b)=(i+2,j)$, there is a directed path $2\Lambda_0+\Lambda_j\to 2\Lambda_1+\Lambda_j\to 2\Lambda_1+\Lambda_{j+2}$, and this is $(\Delta_{+})$ in 10.1.2.
\end{itemize}

\begin{itemize}

\bigskip
\item
Suppose (1) for $i=j=\ell$ and $a=b=\ell-1$. Thus, $\Lambda=2\Lambda_0+2\Lambda_\ell+\widetilde{\Lambda}$ and $\Lambda''=2\Lambda_1+2\Lambda_{\ell-1}+\widetilde{\Lambda}$. 
Then $\beta_{\Lambda''}=\alpha_0+\alpha_\ell$ and 
$R^{\Lambda}(\beta_{\Lambda''})$ is Morita equivalent to
$$
R^{m_0\Lambda_0}(\alpha_0)\otimes R^{m_\ell\Lambda_\ell}(\alpha_\ell)\cong \k[X,Y]/(X^{m_0}, Y^{m_\ell}),
$$
where $m_0\ge2$ and $m_\ell\ge2$. If either $m_0\ge3$ or $m_\ell\ge3$, then it has
$\k[X,Y]/(X^3,X^2Y,Y^2)$ or $\k[X,Y]/(X^2,XY^2,Y^3)$ as a factor algebra, so that $R^{\Lambda}(\beta_{\Lambda''})$ is wild. If $m_0=m_\ell=2$, it is (t10) and (t11).

\bigskip
\item
Suppose (1) for $i=j=0$. Thus, $\Lambda=4\Lambda_0+\widetilde{\Lambda}$ and $\Lambda''=4\Lambda_1+\widetilde{\Lambda}$. 
Then $\beta_{\Lambda''}=2\alpha_0$ and $R^\Lambda(2\alpha_0)\cong R^{m\Lambda_0}(2\alpha_0)$ is a cyclotomic nilHecke algebra, where $m:=m_0\ge4$.

Recall that the cyclotomic nilHecke algebra $R^{N\Lambda_0}(n\alpha_0)$, for $N\ge n$, is Morita equivalent to
$Z=\k[e_1,\dots,e_n, h_1,\dots,h_{N-n}]/J$, where the ideal $J$ is generated by the coefficients of the equation
$$
(1+e_1t+\cdots+e_nt^n)(1+h_1t+\cdots+h_{N-n}t^{N-n})=1,
$$
and it has a basis consisting of Schur polynomials $s_\lambda$, where $\lambda=(\lambda_1,\dots,\lambda_n)$ and
$0\le \lambda_n\le\dots\le\lambda_1\le N-n$, i.e. partitions contained in $(N-n)^n$.


We may express $h_i$, for $1\le i\le N-n$, by a homogeneous polynomial in $e_1,\dots,e_i$, and the relations are given by
$\sum_{i=k}^{\min(N-n+k-1,n)} e_ih_{N-n+k-i}=0$, for $1\le k\le n$.

Here, $N=m+2$ and $n=2$.
Since relations appear only after degrees greater than or equal to $N-n+1=m+1\ge5$,
$$
\{1,e_1,e_1^2,e_2,e_1^3, e_1e_2, e_1^4, e_1^2e_2, e_2^2\}
$$
is linearly independent in $Z$, and $Z$ has $\k[e_1,e_2]/(e_1^3,e_1^2e_2,e_2^2)$ as a factor algebra. Hence $R^\Lambda(2\alpha_0)$ is wild.

\bigskip
\item
Suppose (2) for $i=j=\ell-1$. Thus, 
$$
\Lambda=2\Lambda_0+2\Lambda_{\ell-1}+\widetilde{\Lambda}, \quad
\Lambda''=2\Lambda_1+2\Lambda_{\ell-2}+\widetilde{\Lambda}. 
$$
Then, $\beta_{\Lambda''}=\alpha_0+2\alpha_{\ell-1}+\alpha_\ell$ and $R^{2\Lambda_0+2\Lambda_{\ell-1}}(\alpha_0+2\alpha_{\ell-1}+\alpha_\ell)$ is Morita equivalent to 
$R^{2\Lambda_0}(\alpha_0)\otimes R^{2\Lambda_{\ell-1}}(2\alpha_{\ell-1}+\alpha_\ell)$. 
By \cite[Theorem 3.7]{AP-rep-type-C-level-1}, $R^{2\Lambda_{\ell-1}}(2\alpha_{\ell-1}+\alpha_\ell)$ is tame and its Gabriel quiver is 
$$
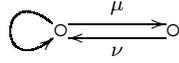
\hspace{1cm}
\begin{xy}
(0,0) *{\circ}="A1",
(15,0) *{\circ}="A2",

\ar @(lu,ld) "A1";"A1"|{}
\ar @<0.5ex>^{\mu} "A1";"A2"
\ar @<0.5ex>^{\nu} "A2";"A1"
\end{xy}
$$
Hence the Gabriel quiver of $R^{2\Lambda_0+2\Lambda_{\ell-1}}(\alpha_0+2\alpha_{\ell-1}+\alpha_\ell)$ is obtained by adding one loop to each of the two vertices, and
we see that it is wild.

\bigskip
\item
Suppose (3). Thus, $\Lambda=2\Lambda_0+2\Lambda_i+\widetilde{\Lambda}$ and
$\Lambda''=2\Lambda_1+\Lambda_{i-1}+\Lambda_{i+1}+\widetilde{\Lambda}$, for $2\le i\le \ell-1$. 
Then, $\beta_{\Lambda''}=\alpha_0+\alpha_i$ and 
$R^{\Lambda}(\alpha_0+\alpha_i)$ is Morita equivalent to
$$
R^{m_0\Lambda_0}(\alpha_0)\otimes R^{m_i\Lambda_i}(\alpha_i)\cong \K[X,Y]/(X^{m_0},Y^{m_i}),
$$
where $m_0\ge2$ and $m_i\ge2$. If if $m_0=m_i=2$, we obtain (t10). Otherwise, 
$R^{\Lambda}(\alpha_0+\alpha_i)$ is wild as in (1).

\bigskip
\item
Suppose (5) for $j=\ell$. Thus, $\Lambda=2\Lambda_0+\Lambda_i+\Lambda_\ell+\widetilde{\Lambda}$, $\Lambda''=2\Lambda_1+\Lambda_{i-1}+\Lambda_{\ell-1}+\widetilde{\Lambda}$, for $2\le i\le \ell-1$. Then, 
$\beta_{\Lambda''}=\alpha_0+\alpha_i+\cdots+\alpha_\ell$ and $R^{2\Lambda_0+\Lambda_i+\Lambda_\ell}(\alpha_0+\alpha_i+\cdots+\alpha_\ell)$ is Morita equivalent to
$$
R^{2\Lambda_0}(\alpha_0)\otimes R^{\Lambda_i+\Lambda_\ell}(\alpha_i+\cdots+\alpha_\ell).
$$
Then, Lemma \ref{Brauer graph algebra case 1} for $i=\ell-1$ and 
Lemma \ref{Brauer graph algebra case 2} for $2\le i\le \ell-2$ tell us that
$R^{\Lambda_i+\Lambda_\ell}(\alpha_i+\cdots+\alpha_\ell)$ is Morita equivalent to the Brauer graph algebra whose Brauer graph is a straight line with $\ell-i+2$ nodes, and the multiplicities of the nodes are $2$ except the first two nodes.
Hence the Gabriel quiver contains
$$
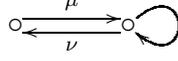
\hspace{1cm}
\begin{xy}
(0,0) *{\circ}="A1",
(15,0) *{\circ}="A2",

\ar @(ru,rd) "A2";"A2"
\ar @<0.5ex>^{\mu} "A1";"A2"
\ar @<0.5ex>^{\nu} "A2";"A1"
\end{xy}
$$
and, by adding one loop on each node, we see that it is wild.

\bigskip
\item
Suppose (5) for $i=0$. Thus, $\Lambda=3\Lambda_0+\Lambda_j+\widetilde{\Lambda}$, $\Lambda''=3\Lambda_1+\Lambda_{j+1}+\widetilde{\Lambda}$, for $2\le j\le \ell-1$. 
Then, $\beta_{\Lambda'}=2\alpha_0+\alpha_1+\cdots+\alpha_j$.
We count the number of simples by Misra-Miwa model for the Kashiwara crystal $B(3\Lambda_0+\Lambda_j)$. The elements in the Misra-Miwa model are 4-partitions
$$
(\lambda^{(1)}, \lambda^{(2)}, \lambda^{(3)}, \lambda^{(4)})\in B(\Lambda_0)^{\otimes 3}\otimes B(\Lambda_j)
$$
whose number of $i$-nodes is $2$ if $i=0$, $1$, for $2\le i\le j$, $0$ otherwise.
Note that the two $0$-nodes can not appear in the same $\lambda^{(i)}$, because otherwise
$$
\ytableausetup{mathmode, boxsize=1em}
\begin{ytableau}
0 & 1 \\
1 & 0
\end{ytableau}
$$
is contained in $\lambda^{(i)}$ and the number of $1$-nodes exceeds $1$. Hence, possible elements are
$((0), (1), (1^k),(1^{j-k+1}))$, for $0\le k\le j$. Hence, the number of simples is $j+1$. Define idempotents
\begin{gather*}
e_0=(\psi_{j+1}x_{j+2})e(j,\dots,1,0,0),\;\; e_1=(\psi_jx_{j+1})e(j,\dots,2,0,0,1) \\
\quad e_2=(\psi_{j-1}x_j)e(j,\dots,3,0,0,1,2), \dots, e_j=(\psi_1x_2)e(0,0,1,\dots,j)
\end{gather*}
and set $P_i=R^{3\Lambda_0+\Lambda_j}(2\alpha_0+\alpha_1+\cdots+\alpha_j)e_i$, for $0\le i\le j$.
They are expressed as follows in the affine type $C$ deformed Fock space.
$$
\begin{aligned}
P_0&=f_0^{(2)}f_1f_2\cdots f_jv_\Lambda=((0),(0),(1),(1^{j+1}))+\cdots, \\
P_1&=f_1f_0^{(2)}f_2\cdots f_jv_\Lambda=((0),(1),(1),(1^j))+\cdots, \\
P_2&=f_2f_1f_0^{(2)}f_3\cdots f_jv_\Lambda=((0),(1),(1^2),(1^{j-1}))+\cdots, \\
   & \qquad \qquad \vdots \\
P_j&=f_jf_{j-1}\cdots f_1f_0^{(2)}v_\Lambda=((0),(1),(1^j),(1))+\cdots.
\end{aligned}
$$
Then, the basic algebra of $R^{3\Lambda_0+\Lambda_j}(2\alpha_0+\alpha_1+\cdots+\alpha_j)$ is
$$
A=\End(P_0\oplus\cdots\oplus P_j)^{\rm op}.
$$
We compute $P_2$ in more detail. Since $\lambda^{(4)}=(1^{j+2-\sum_{i=1}^3|\lambda^{(i)}|})$, we record the first three partitions only. First,
$f_1f_0^{(2)}f_3\cdots f_jv_\Lambda$ is equal to
$$
f_1\left(((0),(1),(1))+q^2((1),(0),(1))+q^4((1),(1),(0))\right).
$$
We compute the action of $f_1$ to obtain
\begin{gather*}
   ((0),(1),(1^2))+q((0),(1),(2))+q^2((0),(1^2),(1))
   +q^3((0),(2),(1)) \\
   +q^2((1),(0),(1^2))+q^3((1),(0),(2))+q^4((1),(1^2),(0)) \\
   +q^5((1),(2),(0))+q^4((1^2),(0),(1)) \\
 +q^5((2),(0),(1))+q^6((1^2),(1),(0))+q^7((2),(1),(0)).
\end{gather*}
We then apply $f_2$ to obtain $P_2$, where, for each of the terms, we either have that
\begin{itemize}
\item[(a)]
the first three partitions do not change, or
\item[(b)]
one $2$-node is added to $(1^2)$, or
\item[(c)]
one node is added to $(2)$.
\end{itemize}
Hence, $B=\End(P_2)$ has the graded dimension
$$
\begin{aligned}
\dim_q B&=(1+q^2)(1+q^2+2q^4+2q^6+2q^8+2q^{10}+q^{12}+q^{14})\\
&=1+2q^2+3q^4+\text{higher terms}.
\end{aligned}
$$
Thus, $B$ is wild by Lemma \ref{local algebra 2+3}, and so is $R^{3\Lambda_0+\Lambda_j}(2\alpha_0+\alpha_1+\cdots+\alpha_j)$.

\begin{rem}
We can chose $e=e(0\;1\;\dots\;j\;0)$ instead and consider
$$
B=eR^{3\Lambda_0+\Lambda_j}(2\alpha_0+\alpha_1+\cdots+\alpha_j)e.
$$
Then the graded dimension is the same. Moreover,
\begin{itemize}
\item[(i)]
$x_1^3e=0$ and $\psi_1e=0$ imply $x_1e=x_2^2e$ and $x_2^6e=0$.
\item[(ii)]
$\psi_2e=\cdots=\psi_{j-1}e=0$ implies $x_2e=\cdots=x_je$.
\item[(iii)]
$x_{j+1}^2e=x_{j+1}x_je=x_{j+1}x_2e$ follows from
$$
\begin{aligned}
x_j\psi_j^2e&=\psi_jx_je(0,1,\dots,j-2,j,j-1,0)\psi_j \\
&=\psi_jx_j\psi_{j-1}^2e(0,1,\dots,j-2,j,j-1,0)\psi_j \\
&=\psi_j\psi_{j-1}x_je(0,1,\dots,j-2,j,j-1,0)\psi_{j-1}\psi_j \\
&= \quad \cdots\cdots \quad \cdots\cdots \quad \cdots\cdots \\
&=\psi_j\cdots \psi_1x_1e(j,0,1,\dots,j-1,0)\psi_1\cdots\psi_j=0.
\end{aligned}
$$
\item[(iv)]
If $e\psi_we\ne0$, then $w=1$ or $w=s_1\cdots s_{j+1}\cdots s_1$. But the latter does not survive because
$\psi_1\cdots\psi_{j+1}\cdots \psi_1e=0$.
\end{itemize}

\medskip
\noindent
We conclude that $B$ is generated by $x_2e$, $x_{j+1}e$ and $x_{j+2}e$.
$\{ x_2e, x_{j+1}e\}$ is a basis of the degree $2$ part, $\{x_2^2e, x_2x_{j+1}e, x_{j+2}e\}$ is a basis of the degree $4$ part and
the higher degree parts are contained in $\Rad^2(B)$. Thus, the Gabriel quiver of $B$ has $3$ loops, and $B$ is wild.
\end{rem}

\bigskip
\item
Suppose (6). Thus, $\Lambda=2\Lambda_0+\Lambda_i+\Lambda_j+\widetilde{\Lambda}$, $\Lambda''=2\Lambda_1+\Lambda_{i-1}+\Lambda_{j+1}+\widetilde{\Lambda}$, for $2\le i<j\le \ell-1$. 
Then, $\beta_{\Lambda''}=\alpha_0+\alpha_i+\cdots+\alpha_j$ and $R^{\Lambda}(\beta_{\Lambda''})$ is Morita equivalent to  
$R^{2\Lambda_0}(\alpha_0)\otimes R^{\Lambda_i+\Lambda_j}(\alpha_i+\cdots+\alpha_j)$. 
It was proved in \cite[Proposition 6.8]{ASW-rep-type} that $R^{\Lambda_i+\Lambda_j}(\alpha_i+\cdots+\alpha_j)$ is Morita equivalent to the Brauer line algebra
whose number of nodes is $j-i+2$. We have subcases.
\begin{itemize}
\item[$(j-i\ge2)$]
The Gabriel quiver of $R^{\Lambda}(\beta_{\Lambda''})$ contains
$$
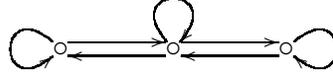
\hspace{1cm}
\begin{xy}
(0,0) *{\circ}="A1",
(15,0) *{\circ}="A2",
(30,0) *{\circ}="A3",

\ar @(lu,ld) "A1";"A1"
\ar @(ul,ur) "A2";"A2"
\ar @(ru,rd) "A3";"A3"
\ar @<0.5ex> "A1";"A2"
\ar @<0.5ex> "A2";"A1"
\ar @<0.5ex> "A2";"A3"
\ar @<0.5ex> "A3";"A2"
\end{xy}
$$
since it has at least $3$ simples. Hence, by considering the separated quiver, we see that it is wild.
\item[$(j=i+1)$]
In this case, $R^{\Lambda}(\beta_{\Lambda''})$ is Morita equivalent to the bound quiver algebra whose quiver is
$$
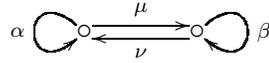
\hspace{1cm}
\begin{xy}
(0,0) *{\circ}="A1",
(15,0) *{\circ}="A2",

\ar @(lu,ld)_{\alpha} "A1";"A1"
\ar @(ru,rd)^{\beta} "A2";"A2"
\ar @<0.5ex>^{\mu} "A1";"A2"
\ar @<0.5ex>^{\nu} "A2";"A1"
\end{xy}
$$
and the relations are
$$
\mu\nu\mu=\nu\mu\nu=\alpha\mu-\mu\beta=\alpha^2=\beta^2=\beta\nu-\nu\alpha=0.
$$
By adding two more relations $\beta\nu=\nu\alpha=0$, we obtain the algebra (32) from \cite[Table W]{H-wild-two-point} as a factor algebra. Hence, $R^\Lambda(\beta_{\Lambda''})$ is wild.
\end{itemize}
\end{itemize}
\subsection{Case (2).}
We consider the path  
\begin{equation}\label{path}
    \Lambda=2\Lambda_1+\tilde \Lambda\rightarrow \Lambda'=2\Lambda_2+\tilde \Lambda\rightarrow \Lambda''.
\end{equation}
We have $\beta_{\Lambda'}=\alpha_0+2\alpha_1$. 
\subsubsection{The case there are two changes.}
We consider the paths of level two:
$\bar\Lambda=2\Lambda_1\rightarrow \bar\Lambda'=2\Lambda_2\rightarrow \bar\Lambda''$. 
Then Theorem \ref{level-2-a=b} tells us  $R^{\bar\Lambda}(\beta_{\Lambda''})$ 
are all wild. Thus, so is $R^{\Lambda}(\beta_{\Lambda''})$.

\subsubsection{The case there are three changes.}
It is enough to consider  the path 
$$\Lambda=2\Lambda_1+\Lambda_i\rightarrow \Lambda'=2\Lambda_2+\Lambda_i\rightarrow \Lambda''$$
such that $\Lambda_i$ is changed in the second step.
First we note that $R^\Lambda(\beta_{\Lambda'})$
is wild if $i=0,1$.  To see this, observe that $i=0$ implies $m_0\ge1$ and $m_1\ge2$, $i=1$ implies $m_1\ge3$. Then, we may apply Lemma \ref{alpha_0+2alpha_1-wild-part1} and Lemma \ref{alpha_0+2alpha_1-wild-part2}, respectively.  So, we may assume $i\ge2$.

Cases in pattern (I) are
 $\Delta_{i^-}$ with $i=3$, $\Delta_{2^-,i^+}$ and $\Delta_{2^-,i^-}$.
 Cases in pattern (II) are
  $\Delta_{i^+}, \Delta_{2^+,i^+}$ and $\Delta_{2^+,i^-}$ with $\Lambda_{mid}=2\Lambda_1+\Lambda_{i+2}, \Lambda_{1}+\Lambda_3+\Lambda_i, \Lambda_1+\Lambda_2+\Lambda_{i-1}$ respectively, and  $R^\Lambda(\beta_{\Lambda_{mid}})$ is wild by (viii'), (iii'')(a), (vi) in the first neighbors, respectively.
It remains to consider the case $\Delta_{i^-}$.
Then $\Lambda''=2\Lambda_2+\Lambda_{i-2}$, $4\le i\le \ell$. Note that $i=2$ can not happen since there is no arrow from $\Lambda'$ to $\Lambda''$. 

\begin{enumerate}
    
    \item Suppose $3<i=\ell$. Then, $\beta_{\Lambda''}=\alpha_0+2\alpha_1+\alpha_{\ell-1}+\alpha_\ell.$ Let
    $e_1=e(1,1,0,\ell,\ell-1)$ and $e_2=e(1,0,1,\ell,\ell-1)$ and $e=e_1+e_2$.  Then
    $eR^{2\Lambda_1+\Lambda_\ell}(\beta_{\Lambda''})e$ is Morita equivalent to
    $R^{2\Lambda_1}(\alpha_0+2\alpha_1)\otimes R^{\Lambda_\ell}(\alpha_{\ell-1}+\alpha_\ell)$. 
    Since $R^{\Lambda_\ell}(\alpha_{\ell-1}+\alpha_\ell)\cong \k[X]/(X^2)$ and 
    the quiver of $ R^{2\Lambda_1}(\alpha_0+2\alpha_1)$ is
    $$\xymatrix@C=0.8cm{1 \ar@<0.5ex>[r]^{\mu}\ar@(dl,ul)^{\alpha}&2 \ar@<0.5ex>[l]^{\nu}}$$
    we see that the quiver of $e R^{2\Lambda_1+\Lambda_\ell}(\beta_{\Lambda''})e$  has two loops at 1,  one loop at 2 and arrows $1\rightarrow 2$, $2\rightarrow 1$. Therefore,  $R^{2\Lambda_1+\Lambda_\ell}(\beta_{\Lambda''})$ is wild.

  \item Suppose that $3<i<\ell$. Then, $\Lambda=2\Lambda_{1}+\Lambda_i$ and 
  we consider the path $$\Lambda\rightarrow \Lambda_0+\Lambda_2+\Lambda_i\rightarrow \Lambda_{mid}=\Lambda_0+\Lambda_2+\Lambda_{i-2}\rightarrow  \Lambda''= 2\Lambda_2+\Lambda_{i-2}, $$
  where we have  $$\beta_{\Lambda_{mid}}=\alpha_1+\alpha_{i-1}+2\alpha_i+\ldots+2\alpha_{\ell-1}+\alpha_\ell.$$ 
      Let $\gamma=\alpha_{i-1}+2\alpha_i+\ldots+2\alpha_{\ell-1}+\alpha_\ell$.
    Then  
    $$ eR^{2\Lambda_1+\Lambda_i}(\beta_{\Lambda_{mid}})e \cong R^{2\Lambda_1}(\alpha_1)\otimes R^{\Lambda_i}(\gamma) , $$
    where $e=\sum_{\nu\in I^{\beta}}e(1*\nu)$.
     Here we note that $R^{\Lambda_i}(\gamma)$ is (f6) and that we may follow the same proof as in case ($\Delta_-$) of Case (1) to show that $R^{2\Lambda_1+\Lambda_i}(\beta_{\hat\Lambda'})$ is wild. Hence,  $R^{2\Lambda_1+\Lambda_i}(\beta_{ \Lambda''})$ is wild. 
\end{enumerate}

\subsubsection{The case of four changes}
It is enough to consider the path   
$$\Lambda=2\Lambda_1+\Lambda_i+\Lambda_j \rightarrow \Lambda'=2\Lambda_2+\Lambda_i+\Lambda_j \rightarrow \Lambda''$$
such that both $\Lambda_i$ and $\Lambda_j$ are changed in the second step.
First, we note that 
 $\beta_{\La'}=\alpha_0+2\alpha_1$ and $R^\Lambda(\beta_{\Lambda'})$
is wild if $i=0,1$ or $j=0,1$ as in 10.2.2.  So, we may assume $2\le i\le j\le \ell$. 

Then cases in pattern (I) are the cases $\Delta_{i^-,j^+}$ and $\Delta_{i^-,j^-}$ with $i=2$, and cases in pattern (II) are the cases $\Delta_{i^+,j^+}$ and $\Delta_{i^+,j^-}$ with $\Lambda_{mid}=2\Lambda_1+\Lambda_{i+1}+\Lambda_{j+1}$, $2\Lambda_1+\Lambda_{i+1}+\Lambda_{j-1}$, respectively.  

Moreover, $R^\Lambda(\beta_{\Lambda_{mid}})$ is wild by (i''), (iv'') and (vi) in the first neighbors, respectively.

\begin{enumerate}

  \item[($\Delta_{i^-,j^+}$)]$\Lambda''=2\Lambda_2+\Lambda_{i-1}+\Lambda_{j+1}$, $3\le i\le j\le\ell-1$.      
      Then $\beta_{\Lambda''}=\beta_1+\beta_2$, where $\beta_1=\alpha_0+2\alpha_1$ and $\beta_2=\alpha_i+\alpha_{i+1}+\ldots+\alpha_j$.
      We have that $R^{2\Lambda_1+\Lambda_i+\Lambda_j}(\beta_{\Lambda''})$ is Morita equivalent to
      $R^{2\Lambda_1}(\beta_1)\otimes R^{\Lambda_i+\Lambda_j}(\beta_2)$. 
      
      The algebra $R^{2\La_1}(\beta_1)$ is (t1). 
      If $i=j$ then
      $R^{\La_i+\La_j}(\beta_2)\cong \k[X]/(X^2)$. If $i<j$ then $R^{\La_i+\La_j}(\beta_2)$ is (f4), and the proof of \cite[Proposition 6.8]{ASW-rep-type} shows that it is the Brauer line algebra with $j-i+2$ vertices. 
      Then we see that  $R^\Lambda(\beta_{\Lambda''})$ is wild. 
   
    \item [($\Delta_{i^-,j^-}$)]$\Lambda''=2\Lambda_2+\Lambda_{i-1}+\Lambda_{j-1}$, $3\le i\le j\le\ell$.
     Then $\beta_{\Lambda''}= \beta_1+\beta_2$, where $\beta_1=\alpha_0+2\alpha_1$ and 
     $\beta_2=\alpha_i+\alpha_{i+1}+\ldots \alpha_{j-1}+2\alpha_j+\ldots+2\alpha_{\ell-1}+\alpha_\ell$. We have that $R^{2\Lambda_1+\Lambda_i+\Lambda_j}(\beta_{\Lambda''})$ is Morita equivalent to 
     $R^{2\Lambda_1}(\beta_1)\otimes R^{\Lambda_i+\Lambda_j}(\beta_2)$. 
     
     The algebra $R^{2\La_1}(\beta_1)$ is (t1). On the other hand, Theorem \ref{level-2-a=b} and Theorem \ref{level-2-a<b} tell us what the algebras $R^{2\Lambda_1}(\beta_1)$ and $R^{\Lambda_i+\Lambda_j}(\beta_2)\cong R^{\La_i+\La_j}(\beta_{\La_{i-1}+\La_{j-1}})$ are. Then we see that $R^\Lambda(\beta_{\Lambda''})$ is wild.
 \end{enumerate}

\subsection{Case (3)}
The case we consider is 
$$
\xymatrix@C=1.3cm@R=1cm{
\Lambda=2\Lambda_a+\tilde\Lambda 
\ar[r]^-{(a^-,a^+)}
&\Lambda'=\Lambda_{a-1}+\Lambda_{a+1}+\tilde\Lambda
\ar[r]^-{}
&\Lambda''
}
$$
with $1\le a\le \ell-1$. 

\subsubsection{The case there are 2 changes}
We have the following graph
$$
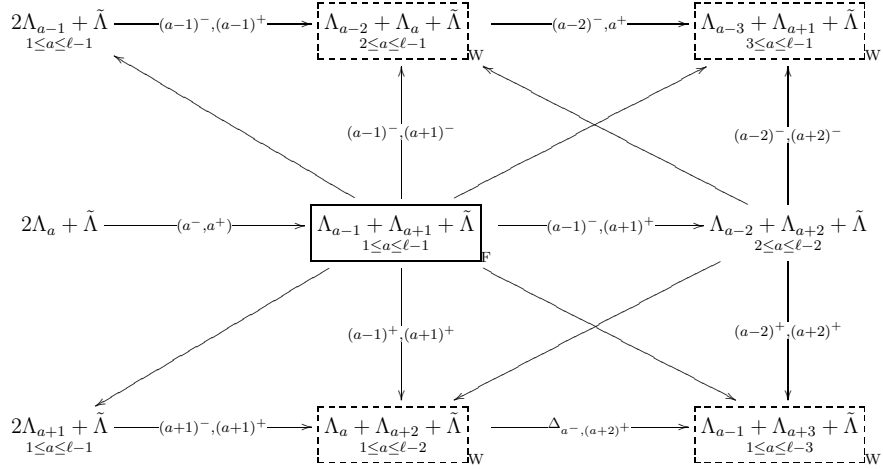
\scalebox{0.7}{
\xymatrix@C=3.6cm@R=2.5cm{
\underset{1\le a\le \ell-1}{2\Lambda_{a-1}+\tilde\Lambda}
\ar|-{{(a-1)}^-,{(a-1)}^+}[r]
&\dboxed{\underset{2\le a\le \ell-1}{\Lambda_{a-2}+\Lambda_{a}+\tilde\Lambda}}_{\text{W}}\ar|-{{(a-2)}^-,{a}^+}[r]
&\dboxed{\underset{3\le a\le \ell-1}{\Lambda_{a-3}+\Lambda_{a+1}+\tilde\Lambda}}_{\text{W}}
\\
2\Lambda_a+\tilde\Lambda
\ar|-{(a^-,a^+)}[r]  
&\boxed{\underset{1\le a\le \ell-1}{\Lambda_{a-1}+\Lambda_{a+1}+\tilde\Lambda}}_{\text{F}}
\ar|-{{(a-1)}^-,{(a+1)}^+}[r]\ar[dr]\ar[ur]\ar|-{{(a-1)}^+,{(a+1)}^+}[d]\ar|-{{(a-1)}^-,{(a+1)}^-}[u]\ar[lu]\ar[ld]
& \underset{2\le a\le \ell-2}{\Lambda_{a-2}+\Lambda_{a+2}+\tilde\Lambda} \ar|-{{(a-2)}^+,{(a+2)}^+}[d]\ar|-{{(a-2)}^-,{(a+2)}^-}[u]\ar[lu]\ar[ld]
\\
\underset{1\le a\le \ell-1}{2\Lambda_{a+1}+\tilde\Lambda}
\ar|-{{(a+1)}^-,{(a+1)}^+}[r]
&\dboxed{\underset{1\le a\le \ell-2}{\Lambda_{a}+\Lambda_{a+2}+\tilde\Lambda}}_{\text{W}}\ar|-{\Delta_{{a}^-,{(a+2)}^+}}[r]
&\dboxed{\underset{1\le a\le \ell-3}{\Lambda_{a-1}+\Lambda_{a+3}+\tilde\Lambda}}_{\text{W}}
}}
$$
Here, the symbol $\dboxed{\Lambda''}_{\text{W}}$ indicates that $R^{\Lambda}(\beta_{\Lambda''})$ is wild,  which follow from Theorem \ref{level-2-a=b} and Theorem \ref{level-2-a<b}.    Thus, we only need to consider three cases.
\begin{itemize}
\item $\Lambda''=\Lambda_{a-2}+\Lambda_{a+2}+\tilde\Lambda$ with $2\le a\le \ell-2$. In this case, 
$\beta_{\Lambda''}=\alpha_{a-1}+2\alpha_a+\alpha_{a+1}$.
Then, $R^{\Lambda}(\beta_{\Lambda''})$ is wild if $m_a\ge 3$ \cite[Lemma 6.9]{ASW-rep-type}, or $m_a=2, m_{a-1}\ge 1$, or $m_a=2, m_{a+1}\ge 1$ \cite[Lemma 6.10]{ASW-rep-type}. If $m_a=2$, $m_{a-1}=m_{a+1}=0$, then $R^{\Lambda}(\beta_{\Lambda''})$ is wild if $\ch \k=2$ and (t15) if $\ch \k\ne 2$ \cite[Proposition 11.4]{Ar-rep-type}.

\item $\Lambda''=2\Lambda_{a-1}+\tilde\Lambda$  with $a=\ell-1,\ell$ or $2\Lambda_{a+1}+\tilde\Lambda$ with $a=0,1$.  These   are in the  pattern (I) cases.
\end{itemize}

\subsubsection{The case there are 3 changes}
We suppose 
$$
\xymatrix@C=1.3cm@R=1cm{
\Lambda=2\Lambda_a+\Lambda_b+\tilde\Lambda 
\ar[r]^-{(a^-,a^+)}
&\Lambda'=\Lambda_{a-1}+\Lambda_{a+1}+\Lambda_b+\tilde\Lambda
\ar[r]^-{}
&\Lambda''
}
$$
with $1\le a\le \ell-1$ and $0\le b\le \ell$. Hence, $m_a\ge 3$ if $a=b$ and $m_a\ge 2, m_b\ge 1$ if $a\ne b$. All possible arrows starting from $\Lambda'$ to obtain $\Lambda''$ are given in the quiver below, in which the conditions for the existence of arrows or vertices are explicitly given.

First, cases in pattern (I) are $\Delta_{(a-1)^+,b^-}$, $\Delta_{(a-1)^+,b^+}$, $\Delta_{b^-}$ $(b=a+2,a+1)$, $\Delta_{(a+1)^-,b^-}$ and $\Delta_{(a+1)^+,b^-}$ $(b=a+1)$.
Second,  the cases $\Delta_{b^-}$ with $b=a\ne \ell-1$, or $b=a=\ell-1$ or $b\le a-1$
and $\Delta_{(a+1)^+,b^-}$ with $(b\ge a+3)$ belong to pattern (II) with  
$\Lambda_{mid}=\Lambda_a+\Lambda_{a+2}+\Lambda_{a-2}+\tilde\Lambda$, $2\Lambda_{a-1}+\Lambda_a+\tilde\Lambda$ and $\Lambda_{a-1}+\Lambda_a+\Lambda_{b-1}+\tilde\Lambda$, and $\Lambda_a+\Lambda_{a+2}+\Lambda_b+\tilde\Lambda $, respectively. We have $R^\Lambda(\beta_{\Lambda_{mid}})$ are wild by the 
level two results, (i'), (iv') and (iii'') in the first neighbors, respectively. 
Similarly, for the case $\Delta_{(a-1)^-,b^-}$, we choose $\Lambda_{mid}$ to be $\Lambda_{a-2}+\Lambda_a+\Lambda_b+\tilde\Lambda$ (for $b\le a-1$) and $\Lambda_{a-1}+\Lambda_a+\Lambda_{b-1}$ and use (iii') and (iv') in the first neighbors, respectively.  By symmetry we also have $R^\Lambda(\beta_{\Lambda''})$ is wild for the case $\Delta_{(a+1)^+,b^+}$. The following are the remaining cases.

$$
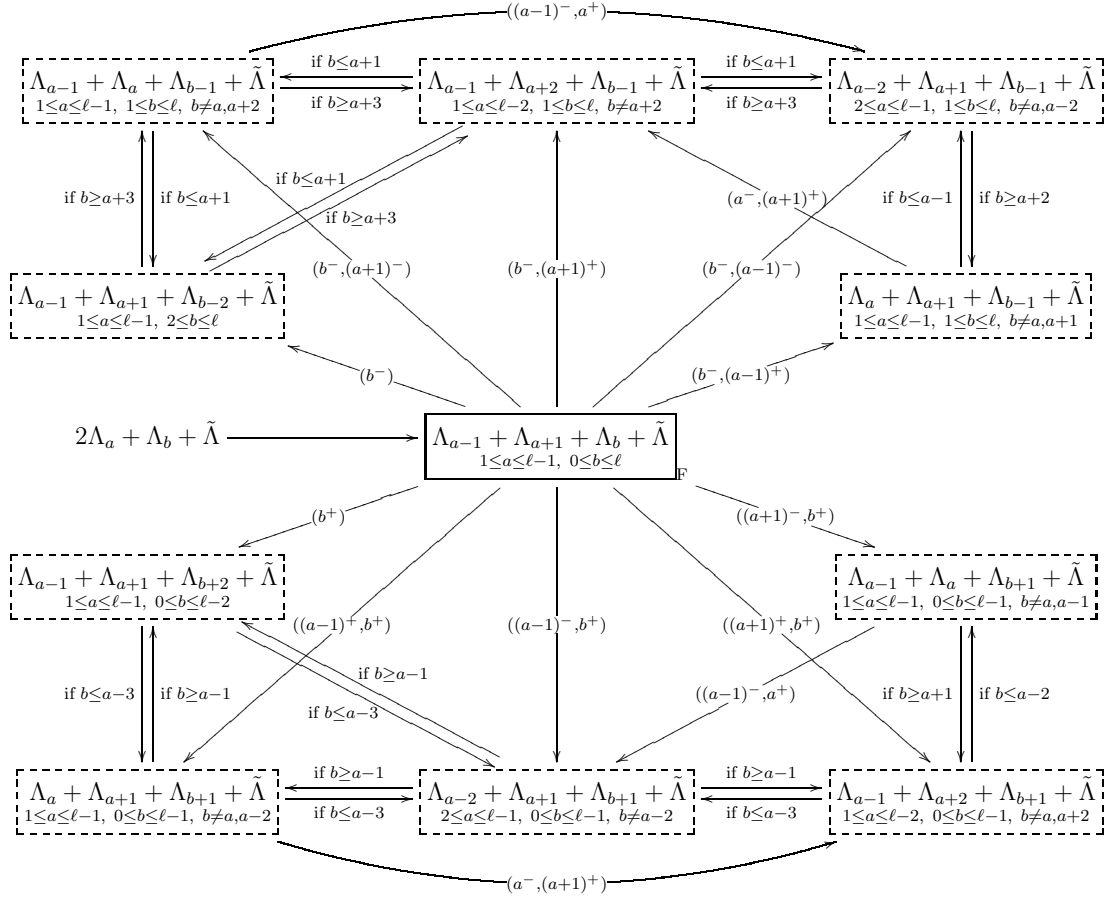
\scalebox{0.8}{
\xymatrix@C=2cm@R=1cm{
\dboxed{\underset{1\le a\le \ell-1,\ 1\le b\le \ell, \ b\ne a, a+2}{\Lambda_{a-1}+\Lambda_{a}+\Lambda_{b-1}+\tilde\Lambda}}
\ar@/^1.2cm/[rr]|{((a-1)^-, a^+)}
\ar@<-0.5ex>[r]_-{\text{if } b\ge a+3}
\ar@<0.5ex>[dd]^-{\text{if } b\le a+1}
&\dboxed{\underset{1\le a\le \ell-2,\ 1\le b\le \ell, \ b\ne a+2}{\Lambda_{a-1}+\Lambda_{a+2}+\Lambda_{b-1}+\tilde\Lambda }}
\ar@<0.5ex>[r]^-{\text{if } b\le a+1}
\ar@<-0.5ex>[l]_-{\text{if } b\le a+1}
\ar@<-0.5ex>[ddl]_-{\text{if } b\le a+1}
&\dboxed{\underset{2\le a\le \ell-1,\ 1\le b\le \ell, \ b\ne a, a-2}{\Lambda_{a-2}+\Lambda_{a+1}+\Lambda_{b-1}+\tilde\Lambda}}
\ar@<0.5ex>[l]^-{\text{if } b\ge a+3}
\ar@<0.5ex>[dd]^-{\text{if } b\ge a+2}
\\
&&
\\
\dboxed{\underset{1\le a\le \ell-1,\ 2\le b\le \ell}{\Lambda_{a-1}+\Lambda_{a+1}+\Lambda_{b-2}+\tilde\Lambda}}
\ar@<-0.5ex>[uur]_-{\text{if } b\ge a+3}
\ar@<0.5ex>[uu]^-{\text{if } b\ge a+3}
&&\dboxed{\underset{1\le a\le \ell-1,\ 1\le b\le \ell, \ b\ne a, a+1}{\Lambda_{a}+\Lambda_{a+1}+\Lambda_{b-1}+\tilde\Lambda}}
\ar@<0.5ex>[uu]^-{\text{if } b\le a-1}
\ar[uul]|-{(a^-, (a+1)^+)}
\\
2\Lambda_a+\Lambda_b+\tilde\Lambda \ar[r] &\boxed{\underset{1\le a\le \ell-1,\ 0\le b\le \ell}{\Lambda_{a-1}+\Lambda_{a+1}+\Lambda_b+\tilde\Lambda}}_{\text{F}}
\ar[uuu]|-{(b^-,(a+1)^+)}\ar[ddd]|-{((a-1)^-,b^+)} 
\ar[ru]|-{(b^-,(a-1)^+)}\ar[rd]|-{((a+1)^-,b^+)}
\ar[ul]|-{(b^-)}\ar[dl]|-{(b^+)}
\ar[uuul]|-{(b^-,(a+1)^-)}\ar[dddl]|-{((a-1)^+,b^+)} 
\ar[uuur]|-{(b^-,(a-1)^-)}\ar[dddr]|-{((a+1)^+,b^+)} 
&
\\
\dboxed{\underset{1\le a\le \ell-1,\ 0\le b\le \ell-2}{\Lambda_{a-1}+\Lambda_{a+1}+\Lambda_{b+2}+\tilde\Lambda}}
\ar@<-0.5ex>[dd]_-{\text{if } b\le a-3}
\ar@<-0.5ex>[ddr]_-{\text{if } b\le a-3}
&&\dboxed{\underset{1\le a\le \ell-1,\ 0\le b\le \ell-1, \ b\ne a, a-1}{\Lambda_{a-1}+\Lambda_{a}+\Lambda_{b+1}+\tilde\Lambda}}
\ar[ddl]|-{((a-1)^-, a^+)}
\ar@<-0.5ex>[dd]_-{\text{if } b\ge a+1}
\\
&&
\\
\dboxed{\underset{1\le a\le \ell-1,\ 0\le b\le \ell-1, \ b\ne a, a-2}{\Lambda_{a}+\Lambda_{a+1}+\Lambda_{b+1}+\tilde\Lambda}}
\ar@<-0.5ex>[r]_-{\text{if } b\le a-3}
\ar@<-0.5ex>[uu]_-{\text{if } b\ge a-1}
\ar@/_1.5cm/[rr]|{(a^-, (a+1)^+)}
&\dboxed{\underset{2\le a\le \ell-1,\ 0\le b\le \ell-1, \ b\ne a-2}{\Lambda_{a-2}+\Lambda_{a+1}+\Lambda_{b+1}+\tilde\Lambda}}
\ar@<0.5ex>[r]^-{\text{if } b\ge a-1}
\ar@<-0.5ex>[l]_-{\text{if } b\ge a-1}
\ar@<-0.5ex>[uul]_-{\text{if } b\ge a-1}
&\dboxed{\underset{1\le a\le \ell-2,\ 0\le b\le \ell-1, \ b\ne a, a+2}{\Lambda_{a-1}+\Lambda_{a+2}+\Lambda_{b+1}+\tilde\Lambda}}
\ar@<0.5ex>[l]^-{\text{if } b\le a-3}
\ar@<-0.5ex>[uu]_-{\text{if } b\le a-2}
}}
$$

\begin{enumerate}
\item [($\Delta_{b^-}$)] $\Lambda''=\Lambda_{a-1}+\Lambda_{a+1}+\Lambda_{b-2}+\tilde\Lambda$ with $1\le a\le \ell-3, b\ge a+3$. It gives
$$
\beta_{\Lambda''}=\alpha_a+\alpha_{b-1}+2(\alpha_b+\alpha_{b+1}+\cdots+\alpha_{\ell-1})+\alpha_\ell.
$$ 
\begin{itemize}
\item[$(b=\ell)$] Then, $\beta_{\Lambda''}=\alpha_a+\alpha_{\ell-1}+\alpha_\ell$ and $R^\Lambda(\beta_{\Lambda''})$ is Morita equivalent to
\begin{center}
$R^{m_a\Lambda_a}(\alpha_a)\otimes R^{m_{\ell-1}\Lambda_{\ell-1}+m_\ell\Lambda_\ell}(\alpha_{\ell-1}+\alpha_\ell)$,
\end{center}
 where $R^{m_a\Lambda_a}(\alpha_a)\cong \k[X]/(X^{m_a})$ with $m_a\ge2$. Suppose $m_{\ell-1}\ge1$. Lemma \ref{Brauer graph algebra case 1} tells us that $R^{\Lambda_{\ell-1}+\Lambda_\ell}(\alpha_{\ell-1}+\alpha_\ell)$ is a Brauer tree algebra whose Brauer tree has three vertices with multiplicities $1$, $m_\ell$ and $2m_\ell$, respectively. Tensoring with $R^{m_a\Lambda_a}(\alpha_a)$, 
$R^{m_a\Lambda_a+\Lambda_{\ell-1}+\Lambda_\ell}(\beta_{\Lambda''})$ is wild, and it follows that $R^\Lambda(\beta_{\Lambda''})$ is wild. 
Suppose $m_{\ell-1}=0$. Then $R^\Lambda(\beta_{\Lambda''})\cong \k[X,Y]/(X^{m_a}, Y^{2m_\ell})$, so that if 
$m_a\ge3$ or $m_\ell\ge2$ then it is wild. 
If $m_a=2$ and $m_\ell=1$, it is (t14).
 
\item[$(b\ne \ell)$] Since 
$R^{2\Lambda_a+\Lambda_b}(\beta_{\Lambda_{a-1}+\Lambda_{a+1}+\Lambda_{b-2}})$ 
is Morita equivalent to
$R^{2\Lambda_a}(\alpha_a)\otimes R^{\Lambda_b}(\beta_{\Lambda_{b-2}})$
and $R^{\Lambda_b}(\beta_{\Lambda_{b-2}})$ is a Brauer line algebra with at least three vertices, we see that $R^{2\Lambda_a+\Lambda_b}(\beta_{\Lambda_{a-1}+\Lambda_{a+1}+\Lambda_{b-2}})$ is wild.
\end{itemize}

\item[$(\Delta_{(a+1)^+,b^-})$] $\Lambda''=\Lambda_{a-1}+\Lambda_{a+2}+\Lambda_{b-1}+\tilde\Lambda$ with $1\le a\le \ell-2,\ 1\le b\le a$. 

\begin{itemize}

\item If $b=a$, then $\beta_{\Lambda''}=2\alpha_a+\alpha_{a+1}$ with $1\le a\le \ell-2$. We have 
$$
R^\Lambda(\beta_{\Lambda''})\cong R^{m_a\Lambda_a+m_{a+1}\Lambda_{a+1}}(2\alpha_a+\alpha_{a+1}).
$$
Since $m_a\ge 3$, $R^\Lambda(\beta_{\Lambda''})$ is wild if $m_a\ge 4$ (\cite[Lemma 6.12]{ASW-rep-type}), or $m_a=3$ and $m_{a+1}\ge 1$ (\cite[Lemma 6.13]{ASW-rep-type}). If $m_a=3, m_{a+1}=0$,  $R^\Lambda(\beta_{\Lambda''})$ is wild if $\ch \k=3$, (t16) if $\ch \k\ne 3$.

\item If $b\le a-1$,  we have $1\le b< a\le \ell-2$. 
Then $\beta_{\Lambda''}=\alpha_b+\ldots+\alpha_{a-1}+2\alpha_a+\alpha_{a+1}$ and we have $R^\Lambda(\beta_{\Lambda''})\cong R^\Lambda_A(\beta_{\Lambda''})$, which is (t15) if $b=a-1$, $m_a=2$, $m_{a\pm1}=0$ and $\ch \k\ne2$, wild otherwise by \cite[Theorem 4.6]{ASW-rep-type}.
\end{itemize}
\end{enumerate}

\subsubsection{The case there are 4 changes}
We suppose 
$$
\xymatrix@C=1.3cm@R=1cm{
\Lambda=2\Lambda_a+\Lambda_b+\Lambda_c+\tilde\Lambda 
\ar[r]^-{(a^-,a^+)}
&\Lambda'=\Lambda_{a-1}+\Lambda_{a+1}+\Lambda_b+\Lambda_c+\tilde\Lambda
\ar[r]^-{?}
&\Lambda''
}
$$
with $1\le a\le \ell-1$, $0\le b, c\le \ell$ and $m_a\ge 2, m_b\ge 1, m_c\ge 1$. All possible arrows $\Lambda'\rightarrow \Lambda''$ are given in the following quiver, in which the conditions for the existence of arrows or vertices are explicitly given.
$$
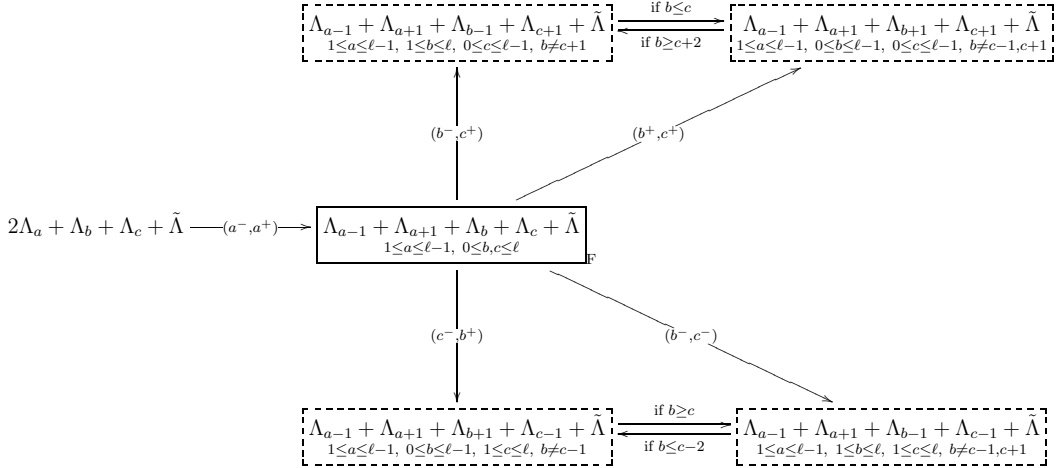
\scalebox{0.7}{
\xymatrix@C=2cm@R=2.5cm{
&\dboxed{\underset{1\le a\le \ell-1,\ 1\le b\le \ell,\ 0\le c\le \ell-1, \ b\ne c+1}{\Lambda_{a-1}+\Lambda_{a+1}+\Lambda_{b-1}+\Lambda_{c+1}+\tilde\Lambda}}
\ar@<0.5ex>[r]^-{\text{if } b\le c}
&\dboxed{\underset{1\le a\le \ell-1,\ 0\le b\le \ell-1,\ 0\le c\le \ell-1, \ b\ne c-1, c+1}{\Lambda_{a-1}+\Lambda_{a+1}+\Lambda_{b+1}+\Lambda_{c+1}+\tilde\Lambda}}
\ar@<0.5ex>[l]^-{\text{if } b\ge c+2}
\\
2\Lambda_a+\Lambda_b+\Lambda_c+\tilde\Lambda
\ar|-{(a^-,a^+)}[r]  
&\boxed{\underset{1\le a\le \ell-1,\ 0\le b, c\le \ell}{\Lambda_{a-1}+\Lambda_{a+1}+\Lambda_b+\Lambda_c+\tilde\Lambda}}_{\text{F}}
\ar|-{(b^-, c^-)}[dr]\ar|-{(b^+, c^+)}[ur]\ar|-{(c^-, b^+)}[d]\ar|-{(b^-, c^+)}[u]
& 
\\
&\dboxed{\underset{1\le a\le \ell-1,\ 0\le b \le \ell-1,\ 1\le c\le \ell,\ b\ne c-1}{\Lambda_{a-1}+\Lambda_{a+1}+\Lambda_{b+1}+\Lambda_{c-1}+\tilde\Lambda}}
\ar@<0.5ex>[r]^-{\text{if } b\ge c}
&\dboxed{\underset{1\le a\le \ell-1,\ 1\le b \le \ell,\ 1\le c\le \ell,\ b\ne c-1, c+1}{\Lambda_{a-1}+\Lambda_{a+1}+\Lambda_{b-1}+\Lambda_{c-1}+\tilde\Lambda}}
\ar@<0.5ex>[l]^-{\text{if } b\le c-2}
}}
$$
First,  $\Delta_{b^-,c^+}$ with $a=b-1$ or $c=a-1$, and 
$\Delta_{b^+,c^+}$ with $a=b+1$ or $a=c+1$, belong to pattern (I). Moreover, 
the following cases belong to pattern (II):
\begin{itemize}
    \item[$(\Delta_{b^-,c^+})$] with $b\ge c+2$: $\Lambda_{mid}= \Lambda_{a-1}+\Lambda_a+\Lambda_{b-1}+\Lambda_c+\tilde\Lambda$ is in the first neighbors if $a\ne b$ (resp. $a=b$) by  (iv') (resp., (i')).
\item[$(\Delta_{b^+,c^+})$] with $b\ge c+2$ and $c\ge 1$: $\Lambda_{mid}=2\Lambda_{a}+\Lambda_{b+1}+\Lambda_{c+1}+\tilde\Lambda $ is wild by Theorem \ref{level-2-a<b}(iv'').
\item[$(\Delta_{b^+,c^+})$] with $2\le b=c$: $\Lambda_{mid}=2\Lambda_{a}+2\Lambda_{c+1}+\tilde\Lambda$ is wild by Theorem \ref{level-2-a=b}(i'').
\end{itemize}
We consider the remaining cases as follows. 
\begin{enumerate}
 
\item[($\Delta_{b^-,c^+}$)] 
with $b\le c$ and $a\ne b-1, c+1$. Then we have $a\le b-2$ or $b\le a\le c$ or  $a\ge c+2$.

\begin{itemize}
    \item If $a+2\le b$ or $a\ge c+2$, we 
    set $$ A:=R^{m_a\Lambda_a+m_b\Lambda_b+m_c\Lambda_c}(\beta_{\Lambda_{a-1}+\Lambda_{a+1}+\Lambda_{b-1}+\Lambda_{c+1}}). $$ Then, $A$ is Morita equivalent to 
$R^{m_a\Lambda_a}(\alpha_a)\otimes R^{m_b\Lambda_b+m_c\Lambda_c}(\beta_{\Lambda_{b-1}+\Lambda_{c+1}})$. 
we have two cases to consider. 
\begin{itemize}
\item[(i)] $b=c$ and $m_a\ge 2$, $m_b\ge 2$. Then, 
$R^{m_a\Lambda_a}(\alpha_a)\otimes R^{m_b\Lambda_b+m_c\Lambda_c}(\beta_{\Lambda_{b-1}+\Lambda_{c+1}})$ is isomorphic to $\k[X,Y]/(X^{m_a}, Y^{m_b})$, 
and it is (t18) if $m_a=m_b=2$ and wild otherwise. 
    
\item[(ii)] $b\le c-1$. Since $R^{\Lambda_b+\Lambda_c}(\beta_{\Lambda_{b-1}+\Lambda_{c+1}})$ is a Brauer tree algebra, $A$ is wild. 
\end{itemize}
  \item If $b\le a \le c$, we have $1\le b\le a \le c\le \ell-1$.
\begin{itemize}
\item[(i)] $b=a=c$ and $m_a\ge 4$. In this case, $\beta_{\Lambda''}=2\alpha_a$ and $R^{\Lambda}(2\alpha_a)$ is wild if $m_a\ge 5$ by \cite[Lemma 6.15]{ASW-rep-type}. For $m_a=4$, $R^\Lambda(2\alpha_a)\cong R^{4\Lambda_a}(2\alpha_a)$ is wild if $\ch \k=2$ and (t19) if $\ch K\ne 2$ (\cite[Lemma 6.16]{ASW-rep-type}).

\item[(ii)] $1\le b=a<c\le \ell-1$ and $m_a\ge 3, m_c\ge 1$. 
In this case, $\beta_{\Lambda''}=2\alpha_a+\alpha_{a+1}+\cdots+\alpha_{c}$ and $R^\Lambda(\beta_{\Lambda''})$ is wild by \cite[Lemma 7.7]{ASW-rep-type}.

\item[(iii)] $1\le b< a=c\le \ell-1$ and $m_a\ge 3, m_b\ge 1$. In this case, $\beta_{\Lambda''}=\alpha_b+\cdots+\alpha_{a-1}+2\alpha_{a}$ and  $R^\Lambda(\beta_{\Lambda''})$ is wild by \cite[Lemma 7.7]{ASW-rep-type}.

\item[(iv)] $1\le b< a< c\le \ell-1$ and $m_a\ge 2, m_b\ge 1, m_c\ge 1$.
Then $R^\Lambda(\beta_{\Lambda''})$ is wild by \cite[Lemma 7.9]{ASW-rep-type}.
\end{itemize}
\end{itemize}

\item[($\Delta_{b^+,c^+}$)] $\Lambda''=\Lambda_{a-1}+\Lambda_{a+1}+\Lambda_{b+1}+\Lambda_{c+1}+\tilde\Lambda$ with $1\le a\le \ell-1$, $0\le b,c\le \ell-1$, such that $a\ne b+1, c+1$. Note that we also assume  $b\ne c\pm 1$. \footnote{We also point out that if $b=c\pm 1$ then we have only three changes.} Hence, it suffices to consider the case $b\ge c+2$, $c=0$ and the case $0\le b\le c$. 

\begin{itemize}
 \item Suppose $b\ge c+2$ and $c=0$. We have $1\le a\le \ell-1$, $2\le b\le \ell-1$ and $m_a\ge 2$, $m_b\ge 1$, $m_0\ge 1$. In this case, $\beta_{\Lambda''}=\alpha_0+\alpha_{1}+\cdots+\alpha_{b}+\alpha_a$. There is a path
$$
\xymatrix@C=1.3cm@R=1cm{
\Lambda 
\ar[r]^-{(0^+,b^+)}
&\La_{mid}=2\Lambda_{a}+\Lambda_{b+1}+\Lambda_{1}+\tilde\Lambda
\ar[r]^-{(a^-,a^+)}
&\Lambda''
}.
$$
\begin{itemize}
\item[(i)] If $a\le b$, we have $m_b\ge 3$ when $a=b$ and 
 $ m_a\ge 2, m_b\ge 1$ when $a\ne b$. Then 
$R^\La(\beta_{\La_{mid}})$ is wild as it belongs to (iv') in the first neighbors in both cases. 
Hence $R^\Lambda(\beta_{\Lambda''})$ is wild.

\item[(ii)] Suppose $a\ge b+2$. If $m_b\ge 2$ or $m_i\ge 1$ for some $0<i<b$, then 
$R^\La(\beta_{\La_{mid}})$ belongs to (iv'') in the first neighbors and wild. 
Suppose $m_b=1$ and $m_i=0$ for all $0<i<b$. 
By Lemma \ref{tensor product lemma}, $R^\Lambda(\beta_{\Lambda''})$ is Morita equivalent to
\begin{center}
$R^{m_0\Lambda_0+\Lambda_b}(\alpha_0+\alpha_{1}+\cdots+\alpha_{b})\otimes R^{m_a\Lambda_a}(\alpha_a)$.
\end{center}
By Lemma~\ref{Brauer graph algebra case 2}, $R^{m_0\Lambda_0+\Lambda_b}(\alpha_0+\alpha_{1}+\cdots+\alpha_{b})$ is a Brauer graph algebra with exactly $\ell-b+1$ simples. Thus, $R^\Lambda(\beta_{\Lambda''})$ is wild. 
\end{itemize}
We conclude that $R^\Lambda(\beta_{\Lambda''})$ is wild if $b\ge c+2$ and $c=0$.
\end{itemize}
\begin{itemize}
\item Suppose $0= b\le c$.
\begin{itemize}
\item[(i)] If $0=b=c$, then $\beta_{\Lambda''}=\alpha_0+\alpha_a$ with $2\le a\le \ell-1$.
Then, $R^\Lambda(\beta_{\Lambda''})$ is Morita equivalent to 
\begin{center}
    $R^{m_0\Lambda_0}(\alpha_0)\otimes R^{m_a\Lambda_a}(\alpha_a)\cong \k[X,Y]/(X^{m_0},Y^{m_a})$. 
\end{center}
This is (t10) if $m_0=m_a=2$, and wild otherwise.

\item[(ii)] Suppose  $0=b<c$ with either $c=1$ and $a\ge3$, or $c\ge2$. We have $\beta_{\Lambda''}=\alpha_a+\alpha_0+(\alpha_{1}+\cdots+\alpha_{c})$. If $c=1$ and $a\ge3$, then   $R^\Lambda(\alpha_0+\alpha_1+\alpha_a)$ is Morita equivalent to $R^{m_0\Lambda_0+m_1\Lambda_1}(\alpha_0+\alpha_1)\otimes R^{m_a\Lambda_a}(\alpha_a)$, which is a wild algebra as mentioned before. If $b=0$ and $c\ge2$, then this case is the same as the case $b\ge c+2$ and $c=0$. Thus, $R^\Lambda(\beta_{\Lambda''})$ is wild. 
\end{itemize}

\item Suppose $1\le b\le c$. 
\begin{itemize}

\item[(i)] If $1=b=c\le a-2$, we have $3\le a \le \ell-1$, $m_a\ge 2$, $m_1\ge 2$. In this case, $\beta_{\Lambda''}=\alpha_0+2\alpha_1+\alpha_a$. For $m_1\ge 3$ or $m_0\ge 1$, $R^\Lambda(\beta_{\Lambda''})$ is wild by (i'') in the first neighbors. For $m_1=2$, $m_0=0$, Lemma \ref{tensor product lemma} implies that 
$R^\Lambda(\beta_{\Lambda''})\cong R^{2\Lambda_1+m_a\Lambda_a}(\beta_{\Lambda''})$ is 
Morita equivalent to
\begin{center}
$R^{2\Lambda_1}(\alpha_0+2\alpha_1)\otimes R^{m_a\Lambda_a}(\alpha_a)$.
\end{center}
Since $R^{2\Lambda_1}(\alpha_0+2\alpha_1)$ is the tame algebra in \cite[Theorem 3.7]{AP-rep-type-C-level-1} and $R^{m_a\Lambda_a}(\alpha_a)\cong \k[X]/(X^2)$, $R^\Lambda(\beta_{\Lambda''})$ is wild.

\item[(ii)] Otherwise, we have either $1\le b\le c-2$ or $b=c\ge 2$.
Recall that if $b=c\ge2$, it is wild by pattern (II) stated above. Suppose $1\le b\le c-2$. Then, there is an arrow
$$
\Lambda'_{b^-,c^+}=\La_{a-1}+\La_{a+1}+\La_{b-1}+\La_{c+1}+\tilde\La \rightarrow \Lambda''.
$$
and $R^\Lambda(\beta_{\Lambda'_{b^-,c^+}})$ is wild by 
$(\Delta_{b^-,c^+})$(ii) above. 

\end{itemize}
\end{itemize}
Finally, the case $\Delta_{b^-,c^-}$ is equivalent to the case $\Delta_{b^+,c^+}$ by symmetry.
\end{enumerate}

\subsection{ Case (4).}
In this subsection, we consider $R^\Lambda(\beta_{\Lambda''})$ for those $\Lambda''$  in  the path  
\begin{equation}\label{path4}
    \Lambda=2\Lambda_0+\tilde \Lambda\rightarrow \Lambda'=\Lambda_0+\Lambda_2+\tilde \Lambda\rightarrow \Lambda''.
\end{equation}
In this case, we have $\beta_{\Lambda'}=\alpha_0+\alpha_1$.
 
\subsubsection{The cases which appear already in level two.}
In this case, we consider the path
$2\Lambda_0+\widetilde{\La}\rightarrow \Lambda_0+\Lambda_2+\widetilde{\La}\rightarrow \Lambda''$. 
Then,
we have that $R^{\Lambda}(\beta_{\Lambda''})$ is wild by Theorem \ref{level-2-a=b} except for $\Lambda''=2\Lambda_2+\tilde \Lambda$, which is already treated in Case (1) in 
Subsection 10.1.1. 

\subsubsection{The cases $\Lambda=2\Lambda_0+\Lambda_i+\tilde \Lambda$ with at most three changes}
We  consider the path 
$$\Lambda=2\Lambda_0+\Lambda_i+\widetilde \Lambda\rightarrow \Lambda'=\Lambda_0+\Lambda_2+\Lambda_i+\widetilde \Lambda\rightarrow \Lambda''$$
such that $\Lambda_i$ is changed in the second step.
First, the cases $\Delta_{2^-,i^-}, \Delta_{0^+,i^+} $, $\Delta_{2^-,i^+}$,  $\Delta_{0^+,i^-}$ and $\Delta_{2^+,i^-}$ ($i=2$) all belong to Case (1) and their representation types have already been determined.
Second, the cases $\Delta_{i^-}$ ($i=3$) and $\Delta_{2^+,i^-}$ ($i=1$) are in pattern (I). 

For the  cases $\Delta_{i^+}$, $\Delta_{i^-}$ ($3<i<\ell$) and $\Delta_{2^+,i^-}$ ($3<i\le \ell$) we consider the arrow $\Lambda_{mid}\rightarrow\Lambda''$ with  
$\Lambda_{mid}= 2\Lambda_1+\Lambda_{i+2}+\tilde\Lambda, 2\Lambda_1+\Lambda_{i-2}+\tilde\Lambda$ and $\Lambda_1+\Lambda_3+\Lambda_i+\tilde\Lambda$, respectively. Note that all these $\Lambda_{mid}$ belong to the second neighbors in Case (1) and $R^\Lambda(\beta_{\Lambda_{mid}})$ are all wild by the results in Case (1) above. Similarly, we have $R^\Lambda(\beta_{\Lambda''})$ is wild for the case $\Delta_{2^+,i^+}$ by choosing 
$\Lambda_{mid}=\Lambda_1+\Lambda_3+\Lambda_i+\tilde\Lambda$. 

It remains to consider ($\Delta_{i^-}$) (this case cannot happen when $i=2$) with $3<i=\ell$. Then 
 $\Lambda''=\Lambda_0+\Lambda_2+\Lambda_{i-2}+\tilde \Lambda$. and $\beta_{\Lambda''}=\alpha_0+\alpha_1+\alpha_{\ell-1}+\alpha_\ell.$ Let
    $e=e(0,1,\ell,\ell-1)$.  Then, by  the proof of (f2), 
    $R^{2\Lambda_0+\Lambda_\ell}(\beta_{\Lambda''})$ is Morita equivalent to 
    $$ R^{2\Lambda_0}(\alpha_0+\alpha_1)\otimes R^{\Lambda_\ell}(\alpha_{\ell-1}+\alpha_\ell)\cong \k[X]/(X^4)\otimes \k[Y]/(Y^2). $$
     Therefore,  $R^{2\Lambda_0+\Lambda_\ell}(\beta_{\Lambda''})$ is wild and so is $R^\Lambda(\beta_{\Lambda''})$.

\subsubsection{The cases $\Lambda=2\Lambda_0+\Lambda_i+\Lambda_j+\tilde\Lambda$ with at most four changes}
We  
consider the path 
$$\Lambda=2\Lambda_0+\Lambda_i+\Lambda_j+\tilde \Lambda\rightarrow \Lambda'=\Lambda_0+\Lambda_2+\Lambda_i+\Lambda_j+\tilde \Lambda\overset{\alpha}{\rightarrow} \Lambda''$$
such that both $\Lambda_i$ and $\Lambda_j$ are changed in the second step.
Here $\alpha$ is the label of the arrow.
For example, if $\Lambda''=\Lambda_0+\Lambda_2+\Lambda_{i+1}+\Lambda_{j+1}+\tilde \Lambda$, then we write $\alpha=(i^+,j^+)$.
Compare the above path with the following path 
$$\Lambda=2\Lambda_0+\Lambda_i+\Lambda_j+\tilde \Lambda\rightarrow \Lambda'_{(1)}=2\Lambda_1+\Lambda_i+\Lambda_j+\tilde \Lambda\overset{\alpha}\rightarrow \Lambda''_{(1)} $$
with the same label $\alpha$ in the second step.
Then we have an arrow 
$\Lambda''_{(1)}\overset{(1^+,1^-)}\rightarrow \Lambda''$.
 Then, $\Lambda''_{(1)}$ belongs to Case (1) and 
$R^\La(\beta_{\Lambda''_{(1)}})$ is wild 
  except in the following two cases. 

\begin{enumerate}
    \item $\Lambda'_{(1)}=2\Lambda_1+2\Lambda_{\ell-1}+\widetilde \Lambda$ with $i=j=\ell$ and $m_0=m_\ell=2$. 
     The last condition means that $\Lambda_0$ and $\Lambda_\ell$ do not appear in $\widetilde \Lambda$.  
We have $\beta_{\Lambda''}=\alpha_0+\alpha_1+\alpha_\ell$.
If $\ell>2$, then we choose $e=e(01\ell)$ and $eR^{\Lambda}(\beta_{\Lambda''})e$ is Morita equivalent to
$$ R^{2\Lambda_0}(\alpha_0+\alpha_1)\otimes R^{2\Lambda_\ell}(\alpha_\ell)\cong \k[X]/(X^4)\otimes \k[Y]/(Y^2),$$
which is wild.
 If $\ell=2$, then $\Lambda''=\Lambda_0+2\Lambda_1+\Lambda_2 +\widetilde\Lambda$ is  (vi) in the first neighbors and 
 $R^\La(\beta_{\La''})$ is wild.  
    \item $\Lambda'_{(1)}=2\Lambda_1+\Lambda_{i-1}+\Lambda_{i+1}+\tilde \Lambda$ with $2\le i=j<\ell$ such that $\Lambda_0$ and $\Lambda_i$ do not appear in $\widetilde \Lambda$. Then $\beta_{\Lambda''}=\alpha_0+\alpha_1+\alpha_i$.
  If $i>2$, then we apply Lemma \ref{tensor product lemma} again and conclude that $R^\La(\beta_{\Lambda''})$ is wild.
  If $i=2$, then  $\Lambda''$ is (iv''') in the first neighbors and $R^\La(\beta_{\Lambda''})$ is wild.
\end{enumerate}

\subsection{ Case (5).}

This case studies $\Lambda=\Lambda_0+\Lambda_b+\widetilde{\Lambda}\to \Lambda'=\Lambda_1+\Lambda_{b+1}+\widetilde{\Lambda} \to \Lambda''$, for $1\le b\le \ell-1$,
and $\beta_{\Lambda'}=\alpha_0+\cdots+\alpha_b$.

\subsubsection{The case of changing $\Lambda_1+\Lambda_{b+1}$}
First,  cases $\Delta_{(b+1)^-}$, $\Delta_{1^+,(b+1)^-}$ and $\Delta_{1^-,(b+1)^+}$  are in pattern (I).
Second, for the remaining  cases $\Delta_{(b+1)^+}$,  $\Delta_{1^+,(b+1)^+}$,   and $\Delta_{1^+}$ are all in pattern (II) with 
$\Lambda_{mid}= \Lambda_0+\Lambda_{b+2}+\tilde \Lambda$, $\Lambda_0+\Lambda_{b+2}+\tilde\Lambda$ and $\Lambda_4+\Lambda_b+\tilde\Lambda $, respectively.
For the first two,  
$R^\Lambda(\beta_{\Lambda_{mid}})$ is wild by (viii')  in the first neighbors. Finally, $R^\Lambda(\beta_{\Lambda_{mid}})$ for the last one is also wild since  Theorem \ref{theo::result-level-one} shows that $R^{\Lambda_0}(\beta_{\Lambda_4})$ is wild.

\subsubsection{The case of changing $\Lambda_1+\Lambda_i$ or $\Lambda_{b+1}+\Lambda_i$}

Here, we consider the path
$$
\Lambda\to \Lambda'=\Lambda_1+\Lambda_{b+1}+\Lambda_i+\widetilde{\Lambda} \to \Lambda''
$$
and we must change $\Lambda_i$.
First, we have cases in pattern (I):
\begin{itemize}
    \item $\Delta_{i^-,(b+1)^-}$, $\Delta_{i^+,(b+1)^-}$, $\Delta_{i^-,1^-}$, $\Delta_{i^+,1^-}$,
    \item $\Delta_{i^-}$ for $2=i\le b-1$, or $i=b+2, b+1$,
    \item $\Delta_{1^+,i^-}$ for $1\le b=i-1$, or $1=b=i$, or $i=1, 2\le b\le \ell-1$ or $i=2, 3\le b\le \ell-1$,
    \item $\Delta_{(b+1)^+,i^-}$ for $1\le b=i-2$ or $1\le b=i-1$.
\end{itemize}
Second, we have the following cases in pattern (II):
\begin{itemize}
    \item[$(\Delta_{i^+})$] with $1\le i\le \ell-2$:  $\Lambda_{mid}=\Lambda_0+\Lambda_{i+2}+\Lambda_b+\tilde\Lambda$, by (viii') in the first neighbors. 
    \item[$(\Delta_{i^-})$] with $2\le i=b$ or $3\le i\le b-1$:
$\Lambda_{mid}=\Lambda_0+\Lambda_{b-2}+\Lambda_b+\tilde\Lambda$ and $\Lambda_0+\Lambda_{i-2}+\Lambda_b+\tilde\Lambda$, respectively,  by 
Theorem \ref{level-2-a=b}(iii') and    
Theorem \ref{level-2-a<b}(viii''), respectively.
    \item[$(\Delta_{1^+,i^+})$] with $i\ne 0,b$ or $2\le i=b$:
$\Lambda_{mid}=\Lambda_0+\Lambda_{b+1}+\Lambda_{i+1}+\tilde\Lambda$ and $\Lambda_0+2\Lambda_{b+1}+\tilde\Lambda$, respectively, by
Theorem \ref{level-2-a<b}(iv''), and  
Theorem  by \ref{level-2-a=b}(ii'') respectively.
    \item[$(\Delta_{(b+1)^+,i^+})$] 
$\Lambda_{mid}=\Lambda_0+\Lambda_{b+2}+\Lambda_i+\tilde\Lambda$, by (viii') in the first neighbors.
    \item[$(\Delta_{1^+,i^-})$] with $1\le b\le i-2$ and $(\Delta_{(b+1)^+,i^-})$ with $1\le b\le i-3$: For both cases, 

  \noindent
$\Lambda_{mid}=\Lambda_1+\Lambda_{b}+\Lambda_{i-1}+\tilde\Lambda$, by (vi) in the first neighbors.
\end{itemize}
Other than patterns (I) and (II), we have the following cases.

\begin{itemize}
\item[$(\Delta_{i^+})$]
We have $\Lambda''=\Lambda_1+\Lambda_{b+1}+\Lambda_{i+2}+\widetilde{\Lambda}$, for $0\le i\le \ell-2$. Here, it remains to consider the following subcases.
\begin{itemize}

\item[$(i=0, 2\le b\le \ell-1)$]
We choose $[P]=f_2f_1^{(2)}f_0^{(2)}f_3\cdots f_bv_\Lambda \in V(\Lambda_0)\otimes V(\Lambda_0)\otimes V(\Lambda_b)$. Then
$[P]=f_2f_1^{(2)}((1),(1),(1^{b-2}))$ is obtained by applying $f_2$ to
\begin{multline*}
((1),(2,1),(1^{b-2}))+q((1^2),(1^2),(1^{b-2}))+q^2((1^2),(2),(1^{b-2}))\\
\quad +q^2((2),(1^2),(1^{b-2}))+q^3((2),(2),(1^{b-2}))+q^4((2,1),(1),(1^{b-2})).
\end{multline*}
Each $3$-partition has three addable $2$-nodes and no removable $2$-node. Hence,
$$
\begin{aligned}
\dim_q \End(P)&=(1+q^2+q^4)(1+q^2+2q^4+q^6+q^8) \\
&=1+2q^2+4q^4+4q^6+4q^8+2q^{10}+q^{12},
\end{aligned}
$$
and $P=f_2f_1^{(2)}f_0^{(2)}f_3\cdots f_bR^\La(0)$ is an indecomposable projective module. 
We apply Lemma \ref{local algebra 2+3} to conclude that $R^{2\Lambda_0+\Lambda_b}(\beta_{\Lambda''})$ is wild.
\item[$(i=0, b=1)$] We have $\Lambda=2\Lambda_0+\Lambda_1+\widetilde{\Lambda}$ and $\Lambda''=\Lambda_1+2\Lambda_2+\widetilde{\Lambda}$,
$\beta_{\Lambda''}=2\alpha_0+2\alpha_1$. We already proved in Subsection 10.1.1 that this algebra is wild.

\end{itemize}

\item[$(\Delta_{i^-})$]
We have $\Lambda''=\Lambda_1+\Lambda_{b+1}+\Lambda_{i-2}+\widetilde{\Lambda}$, for $2\le i\le \ell$. It remains to consider the case $b+3\le i\le \ell$. 
We have $\Lambda''=\Lambda_1+\Lambda_{b+1}+\Lambda_{i-2}+\widetilde{\Lambda}$ and
$$
\beta_{\Lambda''}=\alpha_0+\cdots+\alpha_b+\alpha_{i-1}+2\alpha_i+\cdots+2\alpha_{\ell-1}+\alpha_\ell.
$$
Thus Lemma \ref{tensor product lemma} implies that $R^{\Lambda_0+\Lambda_b+\Lambda_i}(\beta_{\Lambda''})$ is Morita equivalent to
$$
 R^{\Lambda_0+\Lambda_b}(\alpha_0+\cdots+\alpha_b)\otimes
R^{\Lambda_i}(\alpha_{i-1}+2\alpha_i+\cdots+2\alpha_{\ell-1}+\alpha_\ell),
$$
which is $R^{\Lambda_0+\Lambda_b}(\beta_{\Lambda_1+\Lambda_{b+1}})\otimes R^{\Lambda_i}(\beta_{\Lambda_{i-2}})$. 
In \cite[Proposition 4.1]{CH-type-c-level-1}, it was proved that $R^{\Lambda_i}(\beta_{\Lambda_{i-2}})$ is the Brauer line algebra whose number of simple modules is
$\ell-i+1$. Thus, we may choose an idempotent $e$ such that $eR^{\Lambda_i}(\beta_{\Lambda_{i-2}})e\cong \k[X]/(X^2)$.

On the other hand, $R^{\Lambda_0+\Lambda_b}(\beta_{\Lambda_1+\Lambda_{b+1}})$ is (t5) and the number of simples is $b+1\ge2$. Thus, by considering the three leftmost vertices of the Brauer graph, we may obtain an idempotent truncation whose Gabriel quiver is
$$
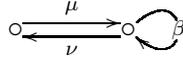
\hspace{1cm}
\begin{xy}
(0,0) *{\circ}="A1",
(15,0) *{\circ}="A2",
\ar @(ru,rd) "A2";"A2"|{\beta}
\ar @<0.5ex>^{\mu} "A1";"A2"
\ar @<0.5ex>^{\nu} "A2";"A1"
\end{xy}
$$
Therefore, an idempotent truncation of $R^{\Lambda_0+\Lambda_b+\Lambda_i}(\beta_{\Lambda''})$ has the Gabriel quiver which is obtained by adding one loop to each vertex. Hence, $R^{\Lambda_0+\Lambda_b+\Lambda_i}(\beta_{\Lambda''})$ is wild, which implies that $R^\La(\beta_{\La''})$ is wild.

\item[ $(\Delta_{1^+,i^+})$] 
We have 
$\Lambda''=\Lambda_2+\Lambda_{i+1}+\Lambda_{b+1}+\widetilde{\Lambda}$. Then, 
the following are the remaining cases. 

\begin{itemize}

\item[$(i=0)$]
$\Lambda=2\Lambda_0+\Lambda_b+\widetilde{\Lambda}$, $\Lambda''=\Lambda_1+\Lambda_2+\Lambda_{b+1}+\widetilde{\Lambda}$, and
$$
\beta_{\Lambda''}=2\alpha_0+2\alpha_1+\alpha_2+\cdots+\alpha_b.
$$
If $b=1$, we already showed that $R^{2\Lambda_0+\Lambda_1}(2\alpha_0+2\alpha_1)$ is wild in $(\Delta_+)$. Thus, we assume $b\ge2$ and
choose
$$
[P]=f_0f_1^{(2)}f_2\cdots f_bf_0v_\Lambda\in V(\Lambda_0)\otimes V(\Lambda_0)\otimes V(\Lambda_b).
$$
We then obtain $[P]$ by applying $f_0$ to
\begin{multline*}
f_1^{(2)}\left((0),(1),(1^{b-1}))+q^2((1),(0),(1^{b-1}))\right) \\
\qquad=((0),(1^2),(1^b))+q((0),(2),(1^b))+q^2((0),(2,1),(1^{b-1})) \\
\qquad +q^2((1^2),(0),(1^b))+q^3((2),(0),(1^b))+q^4((2,1),(0),(1^{b-1})).
\end{multline*}
Each $3$-partition has two addable $0$-nodes and no removable $0$-node. Thus,
$$
\begin{aligned}
\dim_q \End(P)&=(1+q^4)(1+q^2+2q^4+q^6+q^8)\\
&=1+q^2+3q^4+2q^6+3q^8+q^{10}+q^{12}
\end{aligned}
$$
and we apply Lemma \ref{lem::wild-three-loops-part2} to conclude that 
$\End(P)$ and $R^\La(\beta_{\La''})$ are wild.

\item[$(i=b=1)$]
$\Lambda=\Lambda_0+2\Lambda_1+\widetilde{\Lambda}$, $\Lambda''=3\Lambda_2+\widetilde{\Lambda}$ and $\beta_{\Lambda''}=2\alpha_0+3\alpha_1$. We consider
$R^{\Lambda_0+2\Lambda_1}(2\alpha_0+3\alpha_1)$ and choose $[P]=f_1^{(2)}f_0^{(2)}f_1v_\Lambda$. Then
$$
\dim_q \End(P)=1+2q^2+3q^4+3q^6+2q^8+q^{10}
$$
by the similar computation above. Hence, Lemma \ref{local algebra 2+3} applies.
\end{itemize}

\item [$(\Delta_{(b+1)^+,i^-})$]
We have $\Lambda''=\Lambda_1+\Lambda_{b+2}+\Lambda_{i-1}+\widetilde{\Lambda}$. Then, we consider the following remaining cases.

\begin{itemize}

\item[$(2\le b=i)$]
$\Lambda=\Lambda_0+2\Lambda_b+\widetilde{\Lambda}\to \Lambda''=\Lambda_1+\Lambda_{b-1}+\Lambda_{b+2}+\widetilde{\Lambda}$ and
$$
\beta_{\Lambda''}=\alpha_0+\cdots+\alpha_{b-1}+2\alpha_b+\alpha_{b+1}.
$$
We choose 
$[P]=f_bf_{b-1}\cdots f_0f_{b+1}f_bv_\Lambda\in V(\Lambda_0)\otimes V(\Lambda_b)\otimes V(\Lambda_b)$. 
Then $[P]$ is obtained by applying $f_bf_{b-1}$ to
\begin{gather*}
((1^{b-1}),(0),(2))+q((b-1),(0),(2))+q((1^{b-1}),(2),(0)) \\
+q^2((b-1),(2),(0)).
\end{gather*}
Hence, we obtain
$$
\dim_q \End(P)=1+4q^2+6q^4+4q^6+q^8
$$
and $R^\Lambda(\beta_{\Lambda''})$ is wild by Lemma \ref{local algebra 2+3}.
\item[$(1=b=i)$]
This case is similar to the previous case. We choose $[P]=f_2f_1f_0f_1v_\Lambda$ and compute graded dimensions. Then,
$$
\begin{aligned}
\dim_q \End(P)&=(1+q^2)(1+3q^2+2q^4+3q^6+q^8)\\
&=1+4q^2+5q^4+5q^6+4q^8+q^{10}.
\end{aligned}
$$
Hence, $R^\Lambda(\beta_{\Lambda''})$ is wild.
\item[$(1\le i< b\le \ell-1)$]
In this case, we have
$$
\beta_{\Lambda''}=\alpha_0+\cdots+\alpha_{i-1}+2\alpha_i+\cdots+2\alpha_b+\alpha_{b+1}.
$$
We choose $[P]\in  V(\Lambda_0)\otimes V(\Lambda_i)\otimes V(\Lambda_b)$ as
$$
[P]=f_i(f_{i+1}f_i)(f_{i+2}\cdots f_{b+1})(f_{i+1}\cdots f_b)(f_{i-1}\cdots f_0)v_\Lambda.
$$
Then, one can show
$$
\begin{aligned}
\dim_q \End(P)&=(1+q^2)(1+q^2+2q^4+q^6+q^8) \\
&=1+2q^2+3q^4+3q^6+2q^8+q^{10}.
\end{aligned}
$$
Hence, $R^\Lambda(\beta_{\Lambda''})$ is wild by Lemma \ref{local algebra 2+3}.
\end{itemize}

\end{itemize}

\subsubsection{The case of changing $\Lambda_i+\Lambda_j$}

Here, we consider $\Lambda=\Lambda_0+\Lambda_b+\Lambda_i+\Lambda_j+\widetilde{\Lambda}$, for $0\le i\le j\le \ell$, and the path
$$
\Lambda \to \Lambda'=\Lambda_1+\Lambda_{b+1}+\Lambda_i+\Lambda_j+\widetilde{\Lambda} \to \Lambda''.
$$
In the path, we must change $\Lambda_i+\Lambda_j$ in the second step.
Cases in pattern (I) are 
\begin{enumerate}
    \item[$(\Delta_{i^-.j^+})$]
    $i=j=b+1$, or $1\le i<j\le \ell-1$ and $i=1$, or $1\le i<j\le \ell-1$ and $i=b+1$.
    \item[$(\Delta_{i^-,j^-})$]
    $i=j=b=\ell-1$, or $i=j=b+1=\ell$, or $i=b=1$ and $j=\ell$, or $i=1<j\le \ell-1$.
\end{enumerate}
Thus, their representation types have already been determined. 

Next, we consider cases in pattern (II). Let $\Lambda_{mid}$ be the dominant integral weight which is obtained by changing $\La_i+\La_j$ in $\La$. We shall check when $R^\La(\beta_{\La_{mid}})$ is wild, and whether there is an arrow $\La_{mid}\rightarrow \La''$. 

The following is the list of $\La_{mid}$ such that $R^\La(\beta_{\La_{mid}})$ is wild. 
Then, we check whether $\beta_{\Lambda_{mid}}+(\alpha_0+\cdots+\alpha_b)-\delta\not\in Q_+$, in order to know the existence of the arrow.
The numbering in the list follows Theorem \ref{level-2-a=b}(1) and Theorem \ref{level-2-a<b}(1) as before. 

\begin{itemize}
\item[(i')]
$\Lambda-\Lambda_{mid}=2\Lambda_i-2\Lambda_{i-1}$, for $2\le i=j\le \ell-2$. Then,
$$
\beta_{\Lambda_{mid}}=2\alpha_i+\cdots+2\alpha_{\ell-1}+\alpha_\ell.
$$
Hence, we need to treat the cases $i=j=\ell-1$ and $i=j=\ell$ below.
Note that $i=j=1$ implies $\Lambda''=\Lambda$ and it does not occur.
\item[(i'')]
$\Lambda-\Lambda_{mid}=2\Lambda_i-2\Lambda_{i+1}$, for $2\le i=j\le \ell-1$. Then,
$$
\beta_{\Lambda_{mid}}=\alpha_0+2\alpha_1+\cdots+2\alpha_i.
$$
Hence, we need to treat the cases $i=j=0$ and $i=j=1$ below.
\item[(iv')]
$\Lambda-\Lambda_{mid}=\Lambda_i-\Lambda_{i-1}+\Lambda_j-\Lambda_{j-1}$, for $2\le i<j\le \ell-1$. Then,
$$
\beta_{\Lambda_{mid}}=(\alpha_i+\cdots+\alpha_{\ell-1})+(\alpha_j+\cdots+\alpha_{\ell-1})+\alpha_\ell.
$$
Hence, we need to treat the case $j=\ell$ below. Note that 
the arrow $\La'\rightarrow \La''$ does not exist when $i=1$.

\item[(iv'')]
$\Lambda-\Lambda_{mid}=\Lambda_i-\Lambda_{i+1}+\Lambda_j-\Lambda_{j+1}$, for $1\le i<j\le \ell-1$. Then,
$$
\beta_{\Lambda_{mid}}=\alpha_0+(\alpha_1+\cdots+\alpha_i)+(\alpha_1+\cdots+\alpha_j).
$$
Hence, we need to treat the case $i=0<j$ below.

\item[(vi)]
$\Lambda-\Lambda_{mid}=\Lambda_i-\Lambda_{i+1}+\Lambda_j-\Lambda_{j-1}$, 
for $0\le i<j\le \ell$ and $b,i\le j-2$. 
$$
\beta_{\Lambda_{mid}}=(\alpha_0+2\alpha_1+\cdots+2\alpha_i)+(\alpha_{i+1}+\cdots+\alpha_{j-1})+(2\alpha_j+\cdots+2\alpha_{\ell-1}+\alpha_\ell).
$$
\end{itemize}

We do not need to consider (iii'), (iii''), (viii') and (viii''),
because there are only three changes. Below, we handle the cases that $R^\La(\beta_{\La_{mid}})$ is not wild. 

\begin{itemize}
\item[$(\Delta_{--})$]
 \begin{itemize}
  \item[(i)]
Suppose that $i=j=\ell-1$.
Then, $R^\La(\beta_{\La_{mid}})$ is the case (i') with $i=j=\ell-1$, which is not wild.
$$
\Lambda=\Lambda_0+\Lambda_b+2\Lambda_{\ell-1}+\widetilde{\Lambda}, \;\;
\Lambda''=\Lambda_1+\Lambda_{b+1}+
2\Lambda_{\ell-2}+\widetilde{\Lambda}
$$
and $\beta_{\Lambda''}=(\alpha_0+\cdots+\alpha_b)+(2\alpha_{\ell-1}+\alpha_\ell)=\beta_{\La'}+\beta_{\La_{mid}}$.
    \begin{itemize}
     \item[$(1\le b\le \ell-3)$]
Lemma \ref{tensor product lemma} implies that 
$R^{\Lambda_0+\Lambda_b+2\Lambda_{\ell-1}}(\beta_{\Lambda''})$ is Morita equivalent to
$$
R^{\Lambda_0+\Lambda_b}(\alpha_0+\cdots+\alpha_b)\otimes R^{2\Lambda_{\ell-1}}(2\alpha_{\ell-1}+\alpha_\ell).
$$
We know that $R^{\Lambda_0+\Lambda_b}(\alpha_0+\cdots+\alpha_b)$ is the Brauer graph algebra such that the Gabriel quiver of 
an idempotent truncation contains 
$$
\hspace{1cm}
\begin{xy}
(0,0) *{\circ}="A1",
(15,0) *{\circ}="A2",

\ar @(ru,rd) "A2";"A2"
\ar @<0.5ex> "A1";"A2"
\ar @<0.5ex> "A2";"A1"
\end{xy}
$$
and that $R^{2\Lambda_{\ell-1}}(2\alpha_{\ell-1}+\alpha_\ell)$ has an indecomposable projective module $P$ with $\End(P)^{\rm op}\cong \k[X]/(X^2)$. Thus,
$R^{\Lambda_0+\Lambda_b+2\Lambda_{\ell-1}}(\beta_{\Lambda''})$ is wild.
    \item[$(b=\ell-1)$]
We have $\Lambda_{mid}=(\Lambda_1+\Lambda_{\ell-2})+2\Lambda_{\ell-1}+\widetilde{\Lambda}$. If $\ell\ge3$ there is a path
$$
\Lambda=\Lambda_0+3\Lambda_{\ell-1}+\widetilde{\Lambda}\to \Lambda_{mid}\to \Lambda''=\Lambda_1+2\Lambda_{\ell-2}+\Lambda_\ell+\widetilde{\Lambda},
$$
since $\beta_{\Lambda_{mid}}=\alpha_0+\cdots+\alpha_{\ell-2}+2\alpha_{\ell-1}+\alpha_\ell$ and
$\beta_{\Lambda''}=\beta_{\Lambda_{mid}}+\alpha_{\ell-1}$. Thus, it is wild because $R^\Lambda(\beta_{\Lambda_{mid}})$ is wild. If $\ell=2$, we have the arrow
$$
\Lambda=\Lambda_0+3\Lambda_1+\widetilde{\Lambda}\to \Lambda''=2\Lambda_0+\Lambda_1
+\Lambda_2+\widetilde{\Lambda},
$$
which is in the first neighbors and $\beta_{\Lambda''}=\alpha_1$. Hence, it is (f1) if $\ell=2$.
    \end{itemize}

\item[(ii)] 
Next, we consider the case $i=j=\ell$, for $1\le b\le \ell-2$.
Then, $R^\La(\beta_{\La'})$ is from case (i') with $i=j=\ell$, which is not wild. Recall 
$$
\Lambda=\Lambda_0+\Lambda_b+2\Lambda_\ell+\widetilde{\Lambda}, \;\;
\Lambda''=\Lambda_1+\Lambda_{b+1}+2\Lambda_{\ell-1}+\widetilde{\Lambda}
$$
and $\beta_{\Lambda''}=(\alpha_0+\cdots+\alpha_b)+\alpha_\ell$.
Lemma \ref{tensor product lemma} implies that 
$R^{\Lambda_0+\Lambda_b+2\Lambda_\ell}(\beta_{\Lambda''})$ is Morita equivalent to
$R^{\Lambda_0+\Lambda_b}(\alpha_0+\cdots+\alpha_b)\otimes R^{2\Lambda_\ell}(\alpha_\ell)$, which is wild.
 \end{itemize}

\item[$(\Delta_{++})$]
\begin{itemize}
    \item[(i)] Suppose that $i=j=1$. Then, 
$R^\La(\beta_{\La_{mid}})$ is the algebra from case (i'') with $i=j=1$, which is not wild. In this case,
$$
\Lambda=\Lambda_0+\Lambda_b+2\Lambda_1+\widetilde{\Lambda},\;\; \Lambda''=\Lambda_1+\Lambda_{b+1}+2\Lambda_2+\widetilde{\Lambda}
$$
and there is a path
$$
\Lambda_0+\Lambda_1+\Lambda_b\to \Lambda_0+\Lambda_2+\Lambda_{b+1}\to 2\Lambda_2+\Lambda_{b+1}.
$$
If $2\le b\le \ell-1$, $R^{\Lambda_1+\Lambda_b}(\beta_{\Lambda_2+\Lambda_{b+1}})$ is wild. If $b=1$, then
we already computed in Case (5) $(\Delta_{++})(i=b=1)$ that $R^{\Lambda_0+2\Lambda_1}(2\alpha_0+3\alpha_1)$ is wild. To see this, we computed $\dim_q \End(P)$, for $[P]=f_1^{(2)}f_0^{(2)}f_1v_\La$. 
Thus, $R^\Lambda(\beta_{\Lambda''})$ is wild.

   \item[(ii)]
Next, we consider the case $i=j=0$. 
This $R^\La(\beta_{\La'})$ is a non-wild algebra from case (i'') with $i=j=0$. Then, 
$$
\Lambda=3\Lambda_0+\Lambda_b+\widetilde{\Lambda},\;\; \Lambda''=3\Lambda_1+\Lambda_{b+1}+\widetilde{\Lambda}.
$$
and $\beta_{\Lambda''}=2\alpha_0+\alpha_1+\cdots+\alpha_b$.
\begin{itemize}
\item[$(b=1)$]
We consider projective $R^{3\Lambda_0+\Lambda_1}(2\alpha_0+\alpha_1)$-modules
$[P_1]=f_1f_0^{(2)}v_\Lambda$ and $[P_2]=f_0^{(2)}f_1v_\Lambda$ in $V(\Lambda_0)^{\otimes 3}\otimes V(\Lambda_b)$. Then,
$$
\begin{aligned}
\dim_q \End(P_1)&=1+q^2+2q^4+2q^6+3q^8+2q^{10}+2q^{12}+q^{14}+q^{16},\\
\dim_q \End(P_2)&=1+q^4+2q^8+q^{12}+q^{16}, \\
\dim_q \Hom(P_1,P_2)&=q^4+q^8+q^{12}.
\end{aligned}
$$
Since $\dim_q \Hom(P_1,P_2)=\dim_q \Hom(P_2,P_1)$ starts with degree $4$, we have one loop of degree $2$ and one loop of degree $4$ on vertex $1$,
one loop of degree $4$ on vertex $2$. Hence, $R^{3\Lambda_0+\Lambda_1}(2\alpha_0+\alpha_1)$ is wild.
\item[$(2\le b\le \ell-1)$]
Set $[P]=f_b\cdots f_1f_0^{(2)}v_\Lambda\in V(\Lambda_0)^{\otimes 3}\otimes V(\Lambda_b)$. Then
$$
\dim_q \End(P)=1+2q^2+3q^4+4q^6+4q^8+4q^{10}+3q^{12}+2q^{14}+q^{16}.
$$
Thus, Lemma \ref{local algebra 2+3} implies that $R^{3\Lambda_0+\Lambda_b}(2\alpha_0+\alpha_1+\cdots+\alpha_b)$ is wild.
    \end{itemize}
\item[$(\Delta_{+-}=\Delta_{-+})$] 
We consider the case $1\le i=j\le \ell-1$ here. 
We have
$$
\Lambda=\Lambda_0+\Lambda_b+2\Lambda_i+\widetilde{\Lambda},\;\; \Lambda''=\Lambda_1+\Lambda_{b+1}+\Lambda_{i-1}+\Lambda_{i+1}+\widetilde{\Lambda}.
$$
and $\beta_{\Lambda''}=(\alpha_0+\cdots+\alpha_b)+\alpha_i$.
\begin{itemize}
\item[$(b+2\le i\le \ell-1)$]
By Lemma \ref{tensor product lemma},  
$R^{\Lambda_0+\Lambda_b+2\Lambda_i}(\alpha_0+\cdots+\alpha_b+\alpha_i)$ is Morita equivalent to
$$
R^{\Lambda_0+\Lambda_b}(\alpha_0+\cdots+\alpha_b)\otimes R^{2\Lambda_i}(\alpha_i),
$$
which is wild.

\item[$(i=b)$]
In this case, we have $\Lambda-\Lambda''=(\Lambda_0+3\Lambda_b)-(\Lambda_1+\Lambda_{b-1}+2\Lambda_{b+1})$ and
$\beta_{\Lambda''}=\alpha_0+\cdots+\alpha_{b-1}+2\alpha_b$.
We set
$$
[P]=f_{b-1}\cdots f_0f_b^{(2)}v_\Lambda\in V(\Lambda_0)\otimes V(\Lambda_b)^{\otimes 3}.
$$
Then
\begin{multline*}
f_{b-2}\cdots f_0f_b^{(2)}v_\Lambda=((1^{b-1}),(0),(1),(1))+q((b-1),(0),(1),(1))\\
\qquad+q((1^{b-1}),(1),(0),(1))+q^2((b-1),(1),(0),(1))\\
\qquad+q^2((1^{b-1}),(1),(1),(0))+q^3((b-1),(1),(1),(0))
\end{multline*}
and each $4$-partition has $3$ addable $(b-1)$-nodes and no removable $(b-1)$-node. Therefore,
$$
\begin{aligned}
\dim_q \End(P)&=(1+q^2+q^4)(1+2q^2+2q^4+q^6) \\
&=1+3q^2+5q^4+5q^6+3q^8+q^{10}
\end{aligned}
$$
and the Gabriel quiver of $\End(P)$ has three loops. Hence $R^\Lambda(\beta_{\Lambda''})$ is wild.
\item[$(1\le i\le b-1)$]
$\beta_{\Lambda''}=\alpha_0+\cdots+\alpha_{i-1}+2\alpha_i+\alpha_{i+1}+\cdots+\alpha_b$. We set
$$
[P]=f_i^{(2)}f_{i-1}\cdots f_0f_{i+1}\cdots f_bv_\Lambda \in V(\Lambda_0)\otimes V(\Lambda_i)^{\otimes 2}\otimes V(\Lambda_b).
$$
Then $f_{i-1}\cdots f_0f_{i+1}\cdots f_bv_\Lambda$ is equal to
\begin{multline*}
((1^i),(0),(0),(1^{b-i}))+q((1^i),(0),(1^{b-i}),(0))+q^2((1^i),(1^{b-i}),(0),(0))\\
+q((i),(0),(0),(1^{b-i}))+q^2((i),(0),(1^{b-i}),(0))+q^3((i),(1^{b-i}),(0),(0))
\end{multline*}
and each $4$-partition has $4$ addable $i$-nodes and no removable $i$-node. Hence,
$$
\begin{aligned}
\dim_q \End(P)&=(1+q^2+2q^4+q^6+q^8)(1+2q^2+2q^4+q^6)\\
&=1+3q^2+6q^4+8q^6+8q^8+6q^{10}+3q^{12}+q^{14}
\end{aligned}
$$
and it is wild.
\end{itemize}

\item[$(\Delta_{--})$]
We consider the case $2\le i<j=\ell$. 
 These $R^\La(\beta_{\La'})$ are the non-wild algebras from case (iv'). We have
$$
\Lambda=\Lambda_0+\Lambda_b+\Lambda_i+\Lambda_\ell+\widetilde{\Lambda}, \;\; \Lambda''=\Lambda_1+\Lambda_{b+1}+\Lambda_{i-1}+\Lambda_{\ell-1}+\widetilde{\Lambda}.
$$
\begin{itemize}
\item[(i)]
First, we consider the case $1\le b\le i-2$. 
We set $$\Lambda_{mid}=\Lambda_0+\Lambda_{b+1}+\Lambda_{i-1}+\Lambda_\ell+\widetilde{\Lambda}.$$ Then,  
there is a path $\La \rightarrow \La_{mid} \rightarrow \La''$ because 
$$
\begin{aligned}
\beta_{\Lambda_{mid}}&=\alpha_0+2\alpha_1+\cdots+2\alpha_b \\
&\qquad\;\;+\alpha_{b+1}+\cdots+\alpha_{i-1}+2\alpha_i+\cdots+2\alpha_{\ell-1}+\alpha_\ell, \\
\beta_{\Lambda''}&=2\alpha_0+3\alpha_1+\cdots+3\alpha_b+\alpha_{b+1}\\
&\qquad\quad+\cdots+\alpha_{i-1}+2\alpha_i+\cdots+2\alpha_{\ell-1}+\alpha_\ell.
\end{aligned}
$$
Hence, the wildness of $R^\Lambda(\beta_{\Lambda''})$ follows.
\item[(ii)]
Second, we consider the case $b=i$ 
and set $\Lambda_{mid}=\Lambda_1+2\Lambda_b+\Lambda_{\ell-1}+\widetilde{\Lambda}$. Then, we have
$$
\begin{aligned}
\beta_{\Lambda_{mid}}&=\alpha_0+\cdots+\alpha_{\ell-1}, \\ \beta_{\Lambda''}&=\alpha_0+\cdots+\alpha_{b-1}+2\alpha_b+\alpha_{b+1}+\cdots+\alpha_\ell.
\end{aligned}
$$

\item[(iii)]
Third, we consider the case $b=i+1$. 
In this case, we have
$$
\Lambda=\Lambda_0+\Lambda_{b-1}+\Lambda_b+\Lambda_\ell+\widetilde{\Lambda}, \;\; \Lambda''=\Lambda_1+\Lambda_{b-2}+\Lambda_{b+1}+\Lambda_{\ell-1}+\widetilde{\Lambda},
$$
and $\beta_{\Lambda''}=\alpha_0+\cdots+\alpha_{b-2}+2\alpha_{b-1}+2\alpha_b+\alpha_{b+1}+\cdots+\alpha_\ell$.

Define an indecomposable $R^{\Lambda_0+\Lambda_{b-1}+\Lambda_b+\Lambda_\ell}(\beta_{\Lambda''})$-module $P$ by
$$
[P]=f_{b-1}^{(2)}f_b^{(2)}f_{b+1}\cdots f_{\ell}f_{b-2}\cdots f_0v_\Lambda\in V(\Lambda_0)\otimes V(\Lambda_{b-1})\otimes V(\Lambda_b)\otimes V(\Lambda_\ell).
$$
Then, $f_b^{(2)}f_{b+1}\cdots f_{\ell}f_{b-2}\cdots f_0v_\Lambda$ is equal to
\begin{multline*}
((1^{b-1}),(0),(1),(1^{\ell-b+1}))+q^2((b-1),(0),(1),(1^{\ell-b+1}))\\
+q^2((1^{b-1}),(0),(1),(\ell-b+1))+q^4((b-1),(0),(1),(\ell-b+1)).
\end{multline*}
Each $4$-partition has $4$ addable $(b-1)$-nodes and no removable $(b-1)$-node. Thus,
$$
\begin{aligned}
\dim_q \End(P)&=(1+2q^4+q^8)(1+q^2+2q^4+q^6+q^8)\\
&=1+q^2+4q^4+3q^6+6q^8+3q^{10}+4q^{12}+q^{14}+q^{16},
\end{aligned}
$$
and both Lemma \ref{lem::wild-three-loops-part2} and Lemma \ref{local algebra 2+3} implies that it is wild.

\item[(iv)]
Finally, we consider the case $i+2\le b\le\ell-1$. 
$$\Lambda-\Lambda''=(\Lambda_0-\Lambda_1+\Lambda_b-\Lambda_{b+1})+(\Lambda_i-\Lambda_{i-1}+\Lambda_\ell-\Lambda_{\ell-1})$$ and
$\beta_{\Lambda''}=(\alpha_0+\cdots+\alpha_b)+(\alpha_i+\cdots+\alpha_\ell)$. 
Then, Lemma \ref{tensor product lemma} implies that 
$R^{\Lambda_0+\Lambda_i+\Lambda_b+\Lambda_\ell}(\beta_{\Lambda''})$ is Morita equivalent to
$$
R^{\Lambda_0+\Lambda_b}(\alpha_0+\cdots+\alpha_b)\otimes R^{\Lambda_i+\Lambda_\ell}(\alpha_i+\cdots+\alpha_\ell).
$$
Both algebras are Brauer graph algebras we already computed, which implies that $R^{\Lambda_0+\Lambda_i+\Lambda_b+\Lambda_\ell}(\beta_{\Lambda''})$ is wild.
\end{itemize}
\end{itemize}

\item[$(\Delta_{++})$]
\begin{itemize}
    \item[(i)]
We consider the case $1\le i<j=\ell-1$. 
These  $R^\La(\beta_{\La'})$ are the non-wild algebras from case (iv''). We have
$$
\Lambda=\Lambda_0+\Lambda_b+\Lambda_i+\Lambda_j+\widetilde{\Lambda}, \;\; \Lambda''=\Lambda_1+\Lambda_{b+1}+\Lambda_{i+1}+\Lambda_{j+1}+\widetilde{\Lambda}.
$$
We choose $\Lambda_{mid}=\Lambda_0+\Lambda_b+\Lambda_{i+1}+\Lambda_{j+1}+\widetilde{\Lambda}$. Then
$$
\beta_{\Lambda_{mid}}=\alpha_0+(\alpha_1+\cdots+\alpha_i)+(\alpha_1+\cdots+\alpha_j).
$$
Since $\Lambda-\Lambda_{mid}=\Lambda_i-\Lambda_{i+1}+\Lambda_j-\Lambda_{j+1}$ and
$$
\Lambda-\Lambda''=\Lambda-\Lambda_{mid}+\Lambda_0-\Lambda_1+\Lambda_b-\Lambda_{b+1},
$$
we have $\beta_{\Lambda''}=\beta_{\Lambda_{mid}}+(\alpha_0+\cdots+\alpha_b)$. 

\item[(ii)] 
Next we consider the case $i=0<j=\ell-1$. 
These $R^\La(\beta_{\La'})$ are the other non-wild algebras from case (iv''). We have
$$
\Lambda=2\Lambda_0+\Lambda_b+\Lambda_j+\widetilde{\Lambda}, \;\; \Lambda''=2\Lambda_1+\Lambda_{b+1}+\Lambda_{j+1}+\widetilde{\Lambda}.
$$
Then, $\beta_{\Lambda''}=2\alpha_0+(\alpha_1+\cdots+\alpha_b)+(\alpha_1+\cdots+\alpha_j)$. 

We define
$[P_1], [P_2]\in V(\Lambda_0)^{\otimes 2}\otimes V(\Lambda_b)\otimes V(\Lambda_j)$ by
$$
\begin{aligned}
[P_1]&=f_1^{(2)}f_2^{(2)}\cdots f_{\min(b,j)}^{(2)}f_0^{(2)}f_{\min(b,j)+1}\cdots f_{\max(b,j)}v_\Lambda,\\
[P_2]&=f_0^{(2)}f_1^{(2)}\cdots f_{\min(b,j)}^{(2)}f_{\min(b,j)+1}\cdots f_{\max(b,j)}v_\Lambda.
\end{aligned}
$$
Then, we have the following.
\begin{itemize}
    \item $[P_1]=f_1^{(2)}((1),(1),(1^{b-1}),(1^{j-1}))$ and 
$((1),(1),(1^{b-1}),(1^{j-1}))$ has $6$ addable $1$-nodes and no removable $1$-node. 
\item $[P_2]=f_0^{(2)}((0),(0),(1^b),(1^j))$ and 
$((0),(0),(1^b),(1^j))$ has $4$ addable $0$-nodes and no removable $0$-node. 
\end{itemize}
Then, we may find that
$$
\begin{aligned}
\dim_q  \End(P_1)&=1+q^2+2q^4+2q^6+3q^8+2q^{10}+2q^{12}+q^{14}+q^{16}, \\
\dim_q \End(P_2)&=1+q^4+2q^8+q^{12}+q^{16}, \\
\dim_q \Hom(P_1,P_2)&=\dim \Hom(P_2, P_1)=q^8.
\end{aligned}
$$
Hence, there are $2$ loops, one is of degree $2$ and the other is of degree $4$, on vertex $1$, and one loop of degree $4$ on vertex $2$.
Thus, it is wild.
\end{itemize}

\item[$(\Delta_{i^-,j^+}) $]
We consider the case $2\le i<j=\ell-1$. 
These  $R^\La(\beta_{\La'})$ are the non-wild algebras from case (v). We have
$$
\Lambda=\Lambda_0+\Lambda_b+\Lambda_i+\Lambda_j+\widetilde{\Lambda}, \;\; \Lambda''=\Lambda_1+\Lambda_{b+1}+\Lambda_{i-1}+\Lambda_{j+1}+\widetilde{\Lambda}
$$
and $\beta_{\Lambda''}=(\alpha_0+\cdots+\alpha_b)+(\alpha_i+\cdots+\alpha_j)$.
\begin{itemize}
\item[$(1\le b\le i-2)$]
In this case, $R^{\Lambda_0+\Lambda_b+\Lambda_i+\Lambda_j}(\beta_{\Lambda''})$ is Morita equivalent to
$$
R^{\Lambda_0+\Lambda_b}(\alpha_0+\cdots+\alpha_b)\otimes R^{\Lambda_i+\Lambda_j}(\alpha_i+\cdots+\alpha_j).
$$
Both are Brauer graph algebras which we have computed. Then, we see that $R^{\Lambda_0+\Lambda_b+\Lambda_i+\Lambda_j}(\beta_{\Lambda''})$ is wild.

\item[$(i\le b\le\ell-1)$]
In this case, we have
$$
\Lambda=\Lambda_0+\Lambda_i+\Lambda_b+\Lambda_j+\widetilde{\Lambda}, \;\; \Lambda''=\Lambda_1+\Lambda_{i-1}+\Lambda_{b+1}+\Lambda_{j+1}+\widetilde{\Lambda}, 
$$
\begin{align*}
\beta_{\Lambda''}=(\alpha_0 &+\cdots +\alpha_{i-1})+
(2\alpha_i +\cdots+2\alpha_{\min(b,j)})\\
&+(\alpha_{\min(b,j)+1}+\cdots+\alpha_{\max(b,j)}).
\end{align*}
We define $[P]\in V(\Lambda_0)\otimes V(\Lambda_i)\otimes V(\Lambda_b)\otimes V(\Lambda_j)$ by
$$
[P]=f_b^{(2)}f_{b-1}^{(2)}\cdots f_i^{(2)}f_{i-1}\cdots f_0f_{\min(b,j)+1}\cdots f_{\max(b,j)}v_\Lambda.
$$
Then, one can show that
$$
\begin{aligned}
\dim_q \End(P)&=(1+q^2+2q^4+q^6+q^8)(1+q^4)\\
&=1+q^2+3q^4+2q^6+3q^8+q^{10}+q^{12}.
\end{aligned}
$$
Lemma \ref{lem::wild-three-loops-part2} implies that it is wild.
\end{itemize}

\item[$(\Delta_{i^+,j^-})$] 
We consider the case $0\le i<j=\ell$, $i\le j-2$. 
These  $R^\La(\beta_{\La'})$ are the non-wild algebras from case (vi). We have
$$
\Lambda=\Lambda_0+\Lambda_b+\Lambda_i+\Lambda_j+\widetilde{\Lambda}, \;\; \Lambda''=\Lambda_1+\Lambda_{b+1}+\Lambda_{i+1}+\Lambda_{j-1}+\widetilde{\Lambda}.
$$
Recall that the arrow $\La'\rightarrow \La''$ does not exist if $1\le j-1\le b$. 
\begin{itemize}
\item[$(1\le b\le j-2)$]
We choose $\Lambda_{mid}=\Lambda_0+\Lambda_b+\Lambda_{i+1}+\Lambda_{j-1}+\widetilde{\Lambda}$. Then,
\begin{align*}
\beta_{\Lambda_{mid}}=(\alpha_0 &+2\alpha_1+\cdots+2\alpha_i)+(\alpha_{i+1}+\cdots+\alpha_{j-1})\\
&+(2\alpha_j+\cdots+2\alpha_{\ell-1}+\alpha_\ell) \\
\beta_{\Lambda''}&=\beta_{\Lambda_{mid}}+(\alpha_0+\cdots+\alpha_b).
\end{align*}
Then, we see that $R^\La(\beta_{\La_{mid}})$ is wild. 

\end{itemize}
\end{itemize}

\subsection{Case (6). }
In this subsection, we consider the path 
$$ \Lambda= \Lambda_a+\Lambda_b+\tilde \Lambda\rightarrow \Lambda_{a-1}+\Lambda_{b+1}+\tilde \Lambda\rightarrow \Lambda'',$$
for $1\le a<b\le \ell-1$.

\subsubsection{ The cases which appear in level two}
In this case, we consider the path of level two 
$$ \Lambda=\Lambda_a+\Lambda_b\rightarrow \Lambda_{a-1}+\Lambda_{b+1}\rightarrow  \Lambda''.$$
Then, by Theorem \ref{level-2-a<b}, algebras that appear in the next step after $\Lambda_{a-1}+\Lambda_{b+1}$ are all wild. We have that $R^\Lambda(\beta_{\Lambda''})$ is wild when $1\le a<b\le \ell-1$.

\subsubsection{The cases with three changes} We consider the path 
$$\Lambda=\La_a+\La_b+\La_i 
\rightarrow \Lambda'=\Lambda_{a-1}+\Lambda_{b+1}+\Lambda_i\rightarrow \Lambda''  $$
such that $\Lambda_i$ is changed in the second step.
Note that $R^\La(\beta_{\Lambda'})$ is wild if $a< i< b$ since $\Lambda'$ is (v) in the first neighbors.
So, we assume $i\le a$ or $i\ge b$ in the following.
First, we have the following cases in pattern (I).
\begin{itemize}
    \item $\Delta_{(a-1)^+,i^+}$, $\Delta_{(a-1)^+,i^-}$, $\Delta_{(b+1)^-,i^+}$, $\Delta_{(b+1)^-,i^-}$.
    \item  $\Delta_{i^+}$ with $i=a-1, a-2$,
    \item $\Delta_{i^-}$ with $i=b+1, b+2$,
    \item $\Delta_{(a-1)^-,i^+}$ with $i=a-1$.
\end{itemize}

 Second, the following cases are in pattern (II).
\begin{itemize}
    \item[$(\Delta_{i^+})$] with $b\le i\le \ell-2$ or $i=a$:  
$\Lambda_{mid}=\Lambda_a+\Lambda_b+\Lambda_{i+2}$
and $\Lambda_a+\Lambda_{a+2}+\Lambda_b$, respectively.
\item [$(\Delta_{i^-})$] with $2\le i\le a$ or $2\le i=b\le \ell-1$:  
$\Lambda_{mid}=\Lambda_a+\Lambda_b+\Lambda_{i-2}$
and $\Lambda_a+\Lambda_{b-2}+\Lambda_b$, respectively.
\item [$(\Delta_{(a-1)^-,i^+})$] with $i<a-1$ or $i\ge b$:  
$\Lambda_{mid}=\Lambda_{a-2}+\Lambda_b+\Lambda_{i}$
and $\Lambda_a+\Lambda_{b+1}+\Lambda_{i+1}$, respectively.
\item [$(\Delta_{(a-1)^-,i^-})$] with $i\le a$ or $i\ge b$:  
$\Lambda_{mid}=\Lambda_{a-2}+\Lambda_b+\Lambda_{i}$
and $\Lambda_{a-1}+\Lambda_{b}+\Lambda_{i-1}$, respectively.
\item  [$(\Delta_{(b+1)^+,i^+})$]
with $i\le a$ or $i\ge b$:
$\Lambda_{mid}=\Lambda_{a}+\Lambda_{b+2}+\Lambda_{i}$.
\item  [$(\Delta_{(b+1)^+,i^-})$] with $i>b+2$: 
$\Lambda_{mid}=\Lambda_{a}+\Lambda_{b+2}+\Lambda_{i}$.
\end{itemize}
The following are the remaining cases.
 
\begin{enumerate}
    \item [($\Delta_{i^+}$)] $\Lambda''=\Lambda_{a-1}+\Lambda_{b+1}+\Lambda_{i+2} $, for $0\le i<a-2$. Then 
    $$ \beta_{\Lambda''}=\alpha_0+2\alpha_1+\ldots+2\alpha_i+\alpha_{i+1}+\alpha_a+\ldots+\alpha_b.$$
    Let $\beta_1=\alpha_0+2\alpha_1+\ldots+2\alpha_i+\alpha_{i+1}$ and $\beta_2=\alpha_a+\ldots+\alpha_b$.
    By Lemma \ref{tensor product lemma} we conclude that $R^\La(\beta_{\Lambda''})$ is wild since 
    $R^{\Lambda_i}(\beta_1)\otimes R^{\Lambda_a+\Lambda_b}(\beta_2)$ is wild. 

\item[($\Delta_{i^-}$)] $\Lambda''=\Lambda_{a-1}+\Lambda_{b+1}+\Lambda_{i-2}+\tilde \Lambda$, for $i>b+2$. Then $\beta_{\Lambda''}=\beta_1+\beta_2$, where 
  $\beta_1=\alpha_{i-1}+2\alpha_i+\ldots+2\alpha_{\ell-1}+\alpha_\ell$
  and $\beta_2=\alpha_a+\ldots +\alpha_b$.
  Applying Lemma \ref{tensor product lemma} again, we conclude that  $R^\La(\beta_{\Lambda''})$ is wild.

 \item[($\Delta_{(b+1)^+,i^-}$)] $\Lambda''=\Lambda_{a-1}+\Lambda_{b+2}+\Lambda_{i-1}+\tilde\Lambda$, for $i\le b+1$. Then 
$$
\beta_{\Lambda''}=\alpha_a+\alpha_{a+1}+\cdots+\alpha_{i-1}+2\alpha_i+\cdots+2\alpha_b+\alpha_{b+1}
$$
belongs to $\Z_{\ge0}\alpha_1\oplus\cdots\oplus\Z_{\ge0}\alpha_{\ell-1}$. Thus, $R^\La(\beta_{\La''})\cong R_A^\La(\beta_{\La''})$ and it is wild by \cite[Proposition 6.8]{ASW-rep-type}.

\end{enumerate}
\subsubsection{The cases with four changes}
We consider the path 
$$ \Lambda=\Lambda_a+\Lambda_b+\Lambda_i+\Lambda_j+\tilde \Lambda \rightarrow \Lambda'=\Lambda_{a-1}+\Lambda_{b+1}+\Lambda_i+\Lambda_j\rightarrow \Lambda'' $$
such that both $\Lambda_i$ and $\Lambda_j$ are changed in the second step.  Then $R^\La(\beta_{\Lambda'})$ is wild if $a< i<b$ or $a<j<b$ since $\Lambda'$ is (v) in the first neighbors. Thus it suffices to assume that $i\le j\le a$ or $b\le i\le j$ or $i\le a<b\le j$.
First, we have the following cases in pattern (I): 
\begin{itemize}
    \item[$(\Delta_{i^+,j^+})$] with $0\le i\le j=a-1$.
    \item[$(\Delta_{i^-,j^-})$] with $b+1=i\le j\le \ell$.
\end{itemize}
Second,  the following are in pattern (II): 
 \begin{itemize}
     \item[$(\Delta_{i^+,j^-})$] with $i<j-1$: $\Lambda_{mid}=\Lambda_a+\Lambda_b+\Lambda_{i+1}+\Lambda_{j-1}$,
     \item [$(\Delta_{i^+,j^+})$] with $b\le i\le j\le \ell-1$ or $0\le a<b\le j$ or $0\le i\le j=a$. For the first two, we choose $\Lambda_{mid}=\Lambda_a+\Lambda_b+\Lambda_{i+1}+\Lambda_{j+1}$. For the third case, we choose $\Lambda_{mid}=2\Lambda_a+\Lambda_{b+1}+\Lambda_{i+1}$.
     \item [$(\Delta_{i^-,j^-})$] with $1\le i\le j\le a$ or $1\le i\le a<b\le j$ or $b= i\le j$: $\Lambda_{mid}=\Lambda_a+\Lambda_b+\Lambda_{i-1}+\Lambda_{j-1}$.
 \end{itemize} The following are the remaining cases.
 
 \begin{enumerate}
 \item [($\Delta_{i^-,j^+}$)] $\Lambda''=\Lambda_{a-1}+\Lambda_{b+1}+\Lambda_{i-1}+\Lambda_{j+1}+\tilde \Lambda$.
 Then, $\beta_{\Lambda''}\in \Z_{\ge0}\alpha_1\oplus\cdots\oplus\Z_{\ge0}\alpha_{\ell-1}$ and $R^\La(\beta_{\La''})$ is wild by \cite[Proposition 6.8]{ASW-rep-type}.

     \item [($\Delta_{i^+,j^+}$)] $\Lambda''=\Lambda_{a-1}+\Lambda_{b+1}+\Lambda_{i+1}+\Lambda_{j+1}+\tilde\Lambda$ for $j<a-1$. Then $\beta_{\Lambda''}=\beta_1+\beta_2$, where 
        $$\beta_1=\alpha_0+2\alpha_1+\ldots +2\alpha_i +\alpha_{i+1}+\ldots +\alpha_{j}, \quad \beta_2= \alpha_a+\alpha_{a+1}+\ldots+ \alpha_b.$$
        We see that $R^\La(\beta_{\Lambda''})$ is wild since 
        $R^{\Lambda_{i}+\Lambda_{j}}(\beta_1)\otimes R^{\Lambda_a+\Lambda_b}(\beta_2)$ is wild.
   
\item[($\Delta_{i^-,j^-}$)] $\Lambda''=\Lambda_{a-1}+\Lambda_{b+1}+\Lambda_{i-1}+\Lambda_{j-1}+\tilde \Lambda$ for $i>b+1$. Then $\beta_{\Lambda''}=\beta_1+\beta_2$, where $\beta_1=\alpha_a+\alpha_{a+1}+\ldots+\alpha_b$ and
$$\beta_2=\alpha_i+\ldots+\alpha_{j-1}+2(\alpha_{j}+\ldots+\alpha_{\ell-1})+\alpha_\ell.$$
Then $R^\La(\beta_{\Lambda''})$ is wild since 
$R^{\Lambda_a+\Lambda_b}(\beta_1)\otimes R^{\Lambda_i+\Lambda_j}(\beta_2) $ is wild.

 \end{enumerate}
\subsection{Case (7).}
In this subsection, we consider $R^\Lambda(\beta_{\Lambda''})$ for those $\Lambda''$ in the path 
$$\Lambda=\Lambda_a+\Lambda_b+\tilde \Lambda\rightarrow \Lambda'= \Lambda_{a+2}+\Lambda_b+\tilde \Lambda\rightarrow \Lambda''$$
with  $0\le a<b\le \ell$, $a\le b-2$
\subsubsection{The cases which appear in level two}
We consider the path
$$\Lambda=\Lambda_a+\Lambda_b+\tilde \Lambda\rightarrow \Lambda'= \Lambda_{a+2}+\Lambda_b+\tilde \Lambda\rightarrow \Lambda''$$
such that the second step changes  $\Lambda_{a+2}+\Lambda_b$ and fixes $\tilde\Lambda$.
This path comes from the path
$$\bar \Lambda=\La_a+\La_b \rightarrow \bar \Lambda'=\Lambda_{a+2}+\Lambda_b\rightarrow \bar \Lambda'' $$
such that $\Lambda''=\bar \Lambda''+\tilde\Lambda$. Then, Theorem \ref{level-2-a<b} implies 
 that 
$R^{\bar\Lambda}(\beta_{\bar\Lambda''})$ is wild 
(and so is $R^\Lambda(\beta_{\Lambda''})$) unless $\bar \Lambda''=\Lambda_{a+1}+\Lambda_{b+1}$ or $\bar \Lambda''=\Lambda_{a+2}+\Lambda_{b-2}$ with $a=0$ and $b=\ell$.
However, the first exception is in the first neighbors. The second exception will be treated in (1)(b) of Subsection \ref{case(7)}. 

\subsubsection{The cases with at most three changes}\label{case(7)}
We consider the path 
$$\Lambda=\Lambda_a+\Lambda_b+\Lambda_i+\tilde\Lambda\rightarrow \Lambda'=\Lambda_{a+2}+\Lambda_b+\Lambda_i+\tilde\Lambda\rightarrow \Lambda'' $$
such that $\Lambda_i$ is changed in the second step.
If $a\ge1$, then $R^\Lambda(\beta_{\Lambda'})$ is wild since $\La'$ is either (iii'') or (viii') or (iv'') in the first neighbors according to $i=a$ or $0\le i\le a-1$ or $i=a+1$, respectively. 
Hence, it suffices to assume that $i\ge a+2$ if $a\ge 1$. We have the following two cases.
 \begin{enumerate}
     \item $\Lambda_b$ is fixed, i.e., $\Lambda_{a+2}+\Lambda_i$ or $\Lambda_i$ is changed in the second step. Then $\Delta_{(a+2)^-,i^\pm}$ and $\Delta_{i^-}$ with $a=0,i=\ell=3$ are in pattern (I). Otherwise,  we use the path  
     $$\bar \Lambda=\Lambda_a+\Lambda_i\rightarrow \bar\Lambda'=\Lambda_{a+2}+\Lambda_i\rightarrow \bar\Lambda''$$
     where 
     $\Lambda''=\bar \Lambda''+\Lambda_b+\tilde\Lambda$.
     Then Theorem \ref{level-2-a<b} implies that $R^{\bar \Lambda}(\beta_{\bar\Lambda''})$ is wild except in the following two cases.
     
   \begin{itemize}
       \item 
    $\bar \Lambda''=\Lambda_2+\Lambda_{\ell-2}$ with  $a=0$ and $i=\ell\ge 4$. Then 
    $\Lambda''=\Lambda_2+\Lambda_b+\Lambda_{\ell-2} +\tilde \Lambda$ and $\beta_{\Lambda''}=\alpha_0+\alpha_1+\alpha_{\ell-1}+\alpha_{\ell}$ with $\ell\ge 3$. 
    Let $\beta_1=\alpha_0+\alpha_1$ and $\beta_2=\alpha_{\ell-1}+\alpha_{\ell}$. Using Lemma \ref{tensor product lemma}, we see that
    $R^\Lambda(\beta_{\Lambda''})$ is Morita equivalent to
$$ R^{m_0\Lambda_0+m_1\Lambda_1}(\beta_1)\otimes R^{m_{\ell-1}\Lambda_{\ell-1}+m_\ell\Lambda_\ell}(\beta_2).$$ 
   Then $R^\Lambda(\beta_{\Lambda''})$ is wild except for $m_0=m_\ell=1$ and $m_1=m_{\ell-1}=0$, which is (t12).
   \item $\bar\Lambda =2\Lambda_0$ and $\bar \Lambda''=2\Lambda_2$ with $a=0=i$. Then $\Lambda''$ already appeared in Case (1), and there is nothing to prove.
   \end{itemize}
  \item The cases where both $\Lambda_b$ and $\Lambda_i$ are changed in the second step. We have the pattern (I) cases: $\Delta_{b^-,i^+}$ with $i=b-2$ and  $\Delta_{b^\pm,i^-}$ with $i=a+2$. The following are the remaining cases. 
\begin{enumerate}
    \item[($\Delta_{b^+,i^+}$)] $\Lambda''=\Lambda_{a+2}+\Lambda_{b+1}+\Lambda_{i+1}+\tilde\Lambda$. 
    Suppose $i>0$. If $i\ne b$ (resp. $i=b$), this is pattern (II) with $\Lambda_{mid}=\Lambda_a+\Lambda_{b+1}+\Lambda_{i+1} +\tilde \Lambda $
  by (iv'') (resp. (i'')) in the first neighbors.
    Suppose $i=0$. Then  $a=0$, $\Lambda=2\Lambda_0+\Lambda_b+\tilde \Lambda$ and $\Lambda''=\Lambda_1+\Lambda_2+\Lambda_{b+1}+\tilde \Lambda$ belongs to Case (1).

 \item[($\Delta_{b^-,i^+}$)]  $\Lambda''= \Lambda_{a+2}+\Lambda_{b-1}+\Lambda_{i+1}+\tilde \Lambda$ with $a<b-2$. 
 \begin{enumerate}
     \item Suppose that $i\ge b$. Then $R^\La(\beta_{\La''})$ is Morita equivalent to 
      $$ R^{\Lambda_a}(\beta_1)\otimes R^{\Lambda_b+\Lambda_i}(\beta_2) $$
      where $\beta_1=\alpha_0+2\alpha_1+\ldots+2\alpha_a+ \alpha_{a+1}$ and $\beta_2= \alpha_b+\ldots+\alpha_i.$
    We find that $R^\La(\beta_{\Lambda''})$ is wild unless $a=0$ and $b=i$. 
     Suppose that $a=0$ and $b=i$. Note that $a<b-2$ implies $b>2$. Then $\beta_{\Lambda''}=\alpha_0+\alpha_1+\alpha_i$
     and we may conclude that 
     $R^{\Lambda}(\beta_{\Lambda''})$ is (t13) if $m_0=1$, $m_1=0$ and $m_i=2$, and wild otherwise.
       
     \item Suppose that $i\le b-2$ (note that $i=b-1$ can not happen). Then this is pattern (II) with $\Lambda_{mid}=\Lambda_a+\Lambda_{b-1}+\Lambda_{i+1} +\tilde\Lambda $ by  (vi) in the first neighbors.    
 \end{enumerate}
 \item [($\Delta_{b^+,i^-}$)] $\Lambda''=\Lambda_{a+2}+\Lambda_{b+1}+\Lambda_{i-1}+\tilde\Lambda$. We may assume $i\ne b$ in the following because $\Delta_{b^+,i^-}=\Delta_{b^-,i^+}$ if $i=b$. 
\begin{itemize}
    \item Suppose that  $i<b$.
    \begin{enumerate}
        \item If $i>a+2$, $\Lambda''$ is wild since 
    $R^{\Lambda_a}(\beta_1)\otimes R^{\Lambda_i+\Lambda_b}(\beta_2)$ is wild in this case ($i\ne b$), where $\beta_1$ and $\beta_2$ are the same as those in $\Delta_{b^-,i^+}$.
    \item If $1\le i\le a+1$, then this happens only when $a=0$ since we are assuming $i\ge a+2$ if $a\ge1$ in this argument, as was explained at the start of 10.7.2. Hence, we must have  $i=1$. Then $\Lambda''=\Lambda_0+\Lambda_2+\Lambda_{b+1}+\tilde\Lambda$ is (iv'') in the first neighbors and $R^\La(\beta_{\La''})$ is wild.
    \end{enumerate}

    \item Suppose that $i>b+2$ (Note that $i=b+1$ can not happen). Then this is pattern (II) with $\Lambda_{mid}=\Lambda_a+\Lambda_{b+1}+\Lambda_{i-1}+\tilde\Lambda$ by (iv'') in the first neighbors. 
    \end{itemize}
 \item[($\Delta_{b^-,i^-}$)] $\Lambda''=\Lambda_{a+2}+\Lambda_{b-1}+\Lambda_{i-1}+\tilde\Lambda$. In this case, we must have $i\ge 1$.
 \begin{itemize}
 \item Suppose that $b\le \ell-1$. 
 \begin{enumerate}
     \item If $i\ne b$ or $i=b<\ell-1$, then this is pattern (II) with $\Lambda_{mid}=\Lambda_a+\Lambda_{b-1}+\Lambda_{i-1}+\tilde\Lambda$
by  (iv') (resp. (i')) in the first neighbors if $i\ne b$ (resp. $i=b<\ell-1$).
\item If $i=b=\ell-1$, then we must have $\ell\ge3$ since $b\ge2$. Moreover, we have $a\le \ell-3$ since $a\le b-2$. If  $a=\ell-3$, then  $\Lambda''=\Lambda_{\ell-1}+2\Lambda_{\ell-2}+\tilde\Lambda$ is (vi) in the first neighbors, which is wild.

 Suppose that $a<\ell-3$. Then 
$\Lambda''=\Lambda_{a+2}+2\Lambda_{\ell-2}+\tilde \Lambda$ and $\beta_{\Lambda''}=\beta_1+\beta_2$, where $\beta_1$ is the same as $\Delta_{b^-,i^+}$ and $\beta_2=2\alpha_{\ell-1}+\alpha_\ell$.
We see that $R^\La(\beta_{\Lambda''})$ is wild since 
$R^{\Lambda_a}(\beta_1)\otimes R^{2\Lambda_{\ell-1}}(\beta_2)$ is wild.
 \end{enumerate}
 \item Suppose that $b=\ell$.
 \begin{enumerate}
     \item If $a+2<i$, 
     then $\beta_{\Lambda''}=\beta_1+\beta_2$, where $\beta_1=\beta^{\Lambda_a}_{\Lambda_{a+2}}$ and 
     $\beta_2=\beta^{\Lambda_b+\Lambda_i}_{\Lambda_{b-1}+\Lambda_{i-1}}$.
   Hence, $R^\La(\beta_{\Lambda''})$ is wild unless $a=0$ and $b=i=\ell$ by the wildness of $R^{\Lambda_a}(\beta_1)\otimes R^{\Lambda_b+\Lambda_i}(\beta_2)$. 
   
   If $a=0$ and $b=i=\ell$, then
   $R^\Lambda(\beta_{\Lambda''})$ is Morita equivalent to $R^{m_0\Lambda_0+m_1\Lambda_1}(\alpha_0+\alpha_1)\otimes R^{m_\ell\Lambda_\ell}(\alpha_\ell)$ and we may conclude that $R^\La(\beta_{\Lambda''})$ is wild unless $m_0=1$, $m_1=0$ and $m_\ell=2$, which is (t13).
   \item If $i\le a+1$, then this happens only when $a=0$ and hence $i=1$.
   Then $\Lambda''=\Lambda_0+\Lambda_2+\Lambda_{\ell-1}+\tilde\Lambda$ is (vi) in the first neighbors, which is wild. 
 \end{enumerate}    
 \end{itemize}
\end{enumerate}

 \end{enumerate}
\subsubsection{The cases with at most four changes}
We consider the path 
\begin{equation}\label{path711}
    \Lambda=\Lambda_a+\Lambda_b+\Lambda_i+\Lambda_j+\tilde\Lambda \rightarrow \Lambda'=\Lambda_{a+2}+\Lambda_b+\Lambda_i+\Lambda_j+\tilde\Lambda\rightarrow \Lambda''
\end{equation} 
such that  both $\Lambda_i$ and $\Lambda_j$ are changed in the second step.

Suppose first that $a\ge 1$. 
Then as explained at the beginning of 10.7.2,
  $R^\Lambda(\beta_{\Lambda'})$ is wild when   
$0\le i\le a+1$ or $0\le j\le a+1$.
Hence, it suffices to assume that 
     $ a+2\le i\le j$ if $a\ge 1$.
Note that $\Lambda_b$ is fixed in each step of the path \eqref{path711} since the second step changes $\Lambda_i$ and $\Lambda_j$ only. Hence, 
we change $\La_a+\La_i$ to $\La_{a+2}+\La_i$ or $\La_a+\La_j$ to $\La_{a+2}+\La_j$. Then, 
those $\Lambda''$ are already considered in the section of three changes. Suppose next that $a=0$. 
If $i\ge 2$ or $j\ge 2$, they are considered in the case of three changes.
It remains to consider the case 
\begin{equation*}
    a=0\le  i\le j\le 1.
\end{equation*}
We divide into subcases.
\begin{enumerate}
\item If $i=0$, then $\Lambda=2\Lambda_0+\Lambda_b+\Lambda_j+\tilde
\Lambda$, $\Lambda'=\Lambda_0+\Lambda_2+\Lambda_j+\Lambda_b+\tilde \Lambda$ and hence $\Lambda''$ belongs to Case (4).
    \item Suppose that $i=j=1$. Then $\Lambda=\Lambda_0+2\Lambda_1+\Lambda_b+\tilde\Lambda$. We have the following subcases.
    
    \begin{itemize}
        \item[($\Delta_{1^+,1^-}$)]  $\Lambda''=\Lambda_0+2\Lambda_2+\Lambda_b+\tilde \Lambda$ and $\beta_{\Lambda''}=\alpha_0+2\alpha_1$.
        Then $\Lambda''$ is (i'') in the first neighbors and it is wild. 
        \item[($\Delta_{1^+,1^+}$)]  $\Lambda''=3\Lambda_2+\Lambda_b+\tilde\Lambda$. Consider the path 
        $$\Lambda\rightarrow \Lambda_{mid}=\Lambda_0+2\Lambda_2+\Lambda_b+\tilde \Lambda\rightarrow\Lambda''.$$
        Then, $R^\La(\beta_{\Lambda''})$ is wild since $R^\La(\beta_{\Lambda_{mid}})$ is wild. 
    \item[($\Delta_{1^-1^-}$)] $\Lambda''=2\Lambda_0+\Lambda_2+\Lambda_b+\tilde\Lambda$. This cannot happen since there is no arrow from $\Lambda'$ to $\Lambda''$ in this case.    
    \end{itemize}
\end{enumerate}
\bigskip

\section{Third neighbors in higher level cases}
\subsection{New non-wild cases in the second neighbors}
Note that we do not need to consider those non-wild algebras that have already appeared in the first neighbors as we have treated them.
Therefore, we only list the new non-wild cases in the second neighbors (and not in the first neighbors).
By the result of the second neighbors, we see that there are no new non-wild algebras in Cases (2), (4), (5), and (6). So, the non-wild cases we have to consider in this section are those listed in 11.1.1, 11.1.2 and 11.1.3 below.
\subsubsection{New non-wild cases in the second neighbors of Case (7)}

\begin{itemize}
    \item[(7)(i)] $\Lambda=\Lambda_0+\Lambda_\ell+\tilde \Lambda$, $\Lambda'=\Lambda_2+\Lambda_\ell+\tilde \Lambda$, $\Lambda''= \Lambda_2+\Lambda_{\ell-2}+\tilde\Lambda $ with  $m_0=m_\ell=1$, $m_1=m_{\ell-1}=0$ and $\ell\ge 4$. In this case, 
    $$\beta_{\Lambda''}=\alpha_0+\alpha_1+\alpha_{\ell-1}+\alpha_\ell.$$
    \item [(7)(ii)] $\Lambda=\Lambda_0+2\Lambda_i+\tilde \Lambda$, $\Lambda'=\Lambda_2+2\Lambda_i+\tilde \Lambda$, $\Lambda''= \Lambda_2+\Lambda_{i-1}+\Lambda_{i+1}+\tilde\Lambda $ with  $m_0=1$, $m_1=0$, $m_i=2$ and $2<i\le \ell-1$. In this case, 
    $$\beta_{\Lambda''}=\alpha_0+\alpha_1+\alpha_i.$$
     \item [(7)(iii)] $\Lambda=\Lambda_0+2\Lambda_\ell+\tilde \Lambda$, $\Lambda'=\Lambda_2+2\Lambda_\ell+\tilde \Lambda$, $\Lambda''= \Lambda_2+2\Lambda_{\ell-1}+\tilde\Lambda $ with  $m_0=1$, $m_1=0$, $m_\ell=2$ and $  \ell\ge3$.
     In this case, $$\beta_{\Lambda''}=\alpha_0+\alpha_1+\alpha_\ell.$$
     \end{itemize}

\subsubsection{New non-wild cases in the second neighbors of  Case (1)}
The path we consider is $$ \Lambda=2\Lambda_0+\tilde\Lambda\rightarrow \Lambda'=2\Lambda_1+\tilde\Lambda\rightarrow\Lambda''. $$

\begin{itemize}
    \item[(1)(i)]$\Lambda=2\Lambda_0+\Lambda_\ell+\tilde \Lambda\rightarrow \Lambda'=2\Lambda_1+\Lambda_\ell+\tilde\Lambda\rightarrow \Lambda''=2\Lambda_1+\Lambda_{\ell-2}+\tilde \Lambda$ and $m_0=2, m_{\ell-1}=0$, $m_\ell=1$. In this case, 
    $\beta_{\Lambda''}=\alpha_0+\alpha_{\ell-1}+\alpha_{\ell}$.
    This also appears in Case (7).\\
    \item [(1)(ii)]
    $\Lambda=2\Lambda_0+2\Lambda_\ell+\tilde\Lambda\rightarrow \Lambda'=2\Lambda_1+2\Lambda_\ell+\tilde\Lambda\rightarrow \Lambda''=2\Lambda_1+2\Lambda_{\ell-1}$, $m_0=2=m_\ell$.
    In this case, 
    $\beta_{\Lambda''}=\alpha_0+\alpha_\ell$.\\ 
  \item[(1)(iii)]
  $\Lambda=2\Lambda_0+2\Lambda_i+\tilde\Lambda\rightarrow \Lambda'=2\Lambda_1+2\Lambda_i+\tilde \Lambda\rightarrow \Lambda''=2\Lambda_1+\Lambda_{i-1}+\Lambda_{i+1}+\tilde \Lambda$, $2\le i\le \ell-1$, $m_0=m_i=2$.
  In this case, $\beta_{\Lambda''}=\alpha_0+\alpha_i$.\\
\item[(1)(iv)]   $\Lambda=2\Lambda_0+\tilde\Lambda\rightarrow \Lambda'=2\Lambda_1+2\tilde \Lambda\rightarrow \Lambda''=2\Lambda_2+\tilde \Lambda$, $m_0=2,m_1=0$, $\ch \k\ne 2$.
  In this case, $\beta_{\Lambda''}=2\alpha_0+2\alpha_1$.\\
\item[(1)(v)] $\Lambda=2\Lambda_\ell+\tilde\Lambda\rightarrow \Lambda'=2\Lambda_{\ell-1}+2\tilde \Lambda\rightarrow \Lambda''=2\Lambda_{\ell-2}+\tilde \Lambda$, $m_\ell=2,m_{\ell-1}=0$, $\ch \k\ne 2$.
  In this case, $\beta_{\Lambda''}=2\alpha_{\ell-1}+2\alpha_\ell$. 
  Note that by symmetry, this case is equivalent to Case (1)(iv).
\end{itemize}

        \subsubsection{New non-wild cases in the second neighbors of Case (3)}
   \begin{enumerate}
       \item[(3)(i)] $\Lambda=2\Lambda_a+\tilde \Lambda\rightarrow \Lambda'=\Lambda_{a-1}+\Lambda_{a+1}+\tilde \Lambda\rightarrow\Lambda''=\Lambda_{a-2}+\Lambda_{a+2}+\tilde\Lambda$, $2\le a\le \ell-2$, $m_a=2$, $m_{a-1}=m_{a+1}=0$, $\ch \k\ne 2$.
        We have
       $\beta_{\Lambda''}=\alpha_{a-1}+2\alpha_a+\alpha_{a+1}$.\\
           \item[(3)(ii)] $\Lambda=3\Lambda_a+\tilde \Lambda\rightarrow \Lambda'=\Lambda_{a-1}+\Lambda_a+\Lambda_{a+1}+\tilde\Lambda\rightarrow\Lambda''=2\Lambda_{a-1}+\Lambda_{a+2}+\tilde\Lambda$, $1\le a\le \ell-2$, $m_a=3,m_{a+1}=0$, $\ch \k\ne 3$.
    We have  
    $\beta_{\Lambda''}= 2\alpha_a+\alpha_{a+1}$.\\
    \item[(3)(iii)]
    $\Lambda=3\Lambda_a+\tilde\Lambda\rightarrow \Lambda'=\Lambda_{a-1}+\Lambda_a+\Lambda_{a+1}+\tilde\Lambda\rightarrow \Lambda''=\Lambda_{a-2}+2\Lambda_{a+1}+\tilde \Lambda$, $2\le a\le \ell-1$, $m_a=3, m_{a-1}=0$ and $\ch \k\ne 3$.
  We have
    $$\beta_{\Lambda''}=\alpha_{a-1}+2\alpha_{a}.$$
    This case is equivalent to Case (3)(ii) by symmetry. \\
    \item[(3)(iv)] $\Lambda=2\Lambda_a+2\Lambda_b+\tilde\Lambda\rightarrow \Lambda'=\Lambda_{a-1}+\Lambda_{a+1}+2\Lambda_b+\tilde\Lambda\rightarrow \Lambda''=\Lambda_{a-1}+\Lambda_{a+1}+\Lambda_{b-1}+\Lambda_{b+1}+\tilde\Lambda$, $1\le a< b-1$, $b \le \ell-1$, $m_a=m_b=2$.
      We have $\beta_{\Lambda''}=\alpha_a+\alpha_b$.\\
    \item[(3)(v)]
    $\Lambda=4\Lambda_a+\tilde\Lambda\rightarrow \Lambda'=\Lambda_{a-1}+\Lambda_{a+1}+2\Lambda_a+\tilde\Lambda\rightarrow 2\Lambda_{a-1}+2\Lambda_{a+1}+\tilde\Lambda$,
    $1\le a\le \ell-1$, $m_a=4$ and $\ch \k\ne 2$. We have 
    $\beta_{\Lambda''}= 2\alpha_a$.\\
       \item[(3)(vi)]
    $\Lambda=2\Lambda_a+\Lambda_0+\tilde\Lambda\rightarrow \Lambda'=\Lambda_{a-1}+\Lambda_{a+1}+\Lambda_0+\tilde\Lambda\rightarrow \Lambda''=\Lambda_2+\Lambda_{a-1}+\Lambda_{a+1}+\tilde\Lambda$, $3\le a\le \ell-1$, $m_a=2$, $m_0=1$, $m_1=0$. In this case,
    $$\beta_{\Lambda''}=\alpha_0+\alpha_1+\alpha_a.$$
    This case also appears in Case (7). \\
    \item[(3)(vii)] $\Lambda=2\Lambda_a+\Lambda_\ell+\tilde\Lambda\rightarrow \Lambda'=\Lambda_{a-1}+\Lambda_{a+1}+\Lambda_\ell+\tilde\Lambda\rightarrow \Lambda''=\Lambda_{\ell-1}+\Lambda_{a-1}+\Lambda_{a+1}+\tilde\Lambda$, $1\le a\le \ell-3$, $m_a=2$, $m_\ell=1$, $m_{\ell-1}=0$. In this case,
    $$\beta_{\Lambda''}=\alpha_a+\alpha_{\ell-1}+\alpha_\ell.$$
   This case also appears in Case (7). \\
    \item[(3)(viii)] $\Lambda=2\Lambda_0+2\Lambda_a+\tilde\Lambda\rightarrow \Lambda_{a-1}+\Lambda_{a+1}+2\Lambda_0+\tilde\Lambda\rightarrow 2\Lambda_1+\Lambda_{a-1}+\Lambda_{a+1}+\tilde\Lambda$, $2\le a\le \ell-1$, $m_0=m_a=2$.
    In this case, 
    $$\beta_{\Lambda''}=\alpha_0+\alpha_a.$$
This case also appears in Case (1).\\
    \item[(3)(ix)] $\Lambda=2\Lambda_\ell+2\Lambda_a+\tilde\Lambda\rightarrow \Lambda_{a-1}+\Lambda_{a+1}+2\Lambda_\ell+\tilde\Lambda\rightarrow 2\Lambda_{\ell-1}+\Lambda_{a-1}+\Lambda_{a+1}+\tilde\Lambda$, $1\le a\le \ell-2$, $m_\ell=m_a=2$.
    In this case, 
    $$\beta_{\Lambda''}=\alpha_\ell+\alpha_a.$$
This case also appears in Case (1).
   \end{enumerate} 
   \subsection{The third neighbors in Case (3)}
   We start with the third neighbors in cases (3)(i), (3)(ii), (3)(iv) and (3)(v). Then 
   we treat Case (7)(i), (7)(ii), (7)(iii), and finally Case (1)(ii), (1)(iii), (1)(iv). Our aim is to show that the algebra in these cases is   
   wild or belongs to the first or the second neighbors.

\subsubsection{Case (3)(i)} It is enough to consider the path at level $2$:
  $$\Lambda=2\Lambda_a \rightarrow \Lambda'=\Lambda_{a-1}+\Lambda_{a+1} \rightarrow\Lambda''=\Lambda_{a-2}+\Lambda_{a+2}\rightarrow\Lambda''',$$ $2\le a\le \ell-2$, $m_a=2$, $m_{a-1}=m_{a+1}=0$, $\ch \k\ne 2$. Similar to the second neighbors, 
  there are  three patterns for $\Lambda'''$, but the first one is slightly different as follows.
\begin{itemize}
    \item[(I')] $\Lambda'''$ belongs to the first neighbors. 
    \item[(I")] $\Lambda'''$ belongs to the second neighbors and hence has already been handled in the previous section. For the reader's convenience, we will list the path $\Lambda\rightarrow \Lambda_{mid}\rightarrow \Lambda'''$.
  \end{itemize} 
  For the second pattern (II), we will write $\Lambda_{mid}$ in the first and the second neighbors explicitly, but we will not refer to the corresponding result for the wildness of $R^\Lambda(\beta_{\Lambda_{mid}})$. It is because each chosen $\Lambda_{mid}$  does not belong to the finite and tame algebras in MAIN THEOREM and hence it is wild by the results in the previous sections.
\begin{enumerate}
    \item \textbf{The case of two changes.}
    It is enough to consider the path at level $2$:
    $$\Lambda=2\Lambda_a\rightarrow\Lambda_{a-1}+\Lambda_{a+1}\rightarrow \Lambda_{a-2}+\Lambda_{a+2}\rightarrow\Lambda''' .$$
    Then, $R^\Lambda(\beta_{\Lambda'''})$ is wild by Theorem \ref{level-2-a=b}.
    \item \textbf{The case of three changes.}
    It is enough to consider the path at level $3$:
    $$\Lambda=2\Lambda_a+\Lambda_i\rightarrow\Lambda'=\Lambda_{a-1}+\Lambda_{a+1}+\Lambda_i\rightarrow \Lambda''=\Lambda_{a-2}+\Lambda_{a+2}+\Lambda_i\rightarrow\Lambda''' $$
    such that $\Lambda_i$ is changed in the last step, where $i\ne a$ (since $m_a=2$).

     By symmetry, it is enough to consider the cases $\Delta_{i^+}$, $\Delta_{(a-2)^+,i^+}, \Delta_{(a+2)^+,i^+}$, $\Delta_{(a-2)^-,i^+}, \Delta_{(a-2)^+,i^-}$.
    If $\Delta_{i^+}$ with $i=0,a=2$, $\Lambda'''$ belongs to  pattern (I').
    On the other hand, cases in pattern (I'') are as follows.
    \begin{itemize}
        \item[$(\Delta_{i^+})$] with $i=0, a=3$. $\Lambda'''=\Lambda_1+\Lambda_2+\Lambda_5$:
        $$\Lambda=\Lambda_0+2\Lambda_3\rightarrow \Lambda_1+\Lambda_3+\Lambda_4\rightarrow\Lambda'''.$$
        \item[ $(\Delta_{(a-2)^+,i^+})$] $\Lambda'''=\Lambda_{a-1}+\Lambda_{a+2}+\Lambda_{i+1}$: $$\Lambda=\Lambda_i+2\Lambda_a\rightarrow \Lambda_{i+2}+2\Lambda_a\rightarrow\hat \Lambda''=\Lambda_{i+2}+\Lambda_{a-1}+\Lambda_{a+1}\rightarrow \Lambda'''.$$
        \item [ $(\Delta_{(a-2)^+,i^-})$] with $i\ge a+1$: $\Lambda\rightarrow  \Lambda_{i-1}+\Lambda_{a+1}+\Lambda_a \rightarrow \Lambda'''.$
    \end{itemize}
    For each case in pattern (II) we list below, we only give the path. 
\begin{itemize}
    \item [($\Delta_{i^+}$)] with $i>0$:             $\Lambda\rightarrow \Lambda_{i+2}+2\Lambda_a\rightarrow\Lambda_{mid}=\Lambda_{i+2}+\Lambda_{a-1}+\Lambda_{a+1}\rightarrow \Lambda''', $
    \item[$(\Delta_{(a+2)^+,i^+})$]
  $\Lambda\rightarrow \Lambda_{mid}= \Lambda_a+\Lambda_{a+1}+\Lambda_{i+1}\rightarrow \Lambda_{a-1}+\Lambda_{a+2}+\Lambda_{i+1}\rightarrow \Lambda''', $
   \item[ $(\Delta_{(a-2)^-,i^+})$] $\Lambda'''= \Lambda_{a-3}+\Lambda_{a+2}+\Lambda_{i+1}$ with $i<a-3$:
      $$\Lambda\rightarrow\Lambda_{mid}=\Lambda_{i+1}+\Lambda_{a-1}+\Lambda_a \rightarrow \Lambda_{a-2}+\Lambda_{a-1}+\Lambda_{i+1}\rightarrow \Lambda'''.$$
       
        \end{itemize}

     The following are the remaining cases.
    
    \begin{enumerate}
        \item [($\Delta_{i^+}$)] $\Lambda'''=\Lambda_{a-2}+\Lambda_{a+2}+\Lambda_{i+2}$.
   with $i=0$ and $a>3$. Then applying Lemma \ref{tensor product lemma}, we see that $R^\La(\beta_{\Lambda'''})$ is wild since 
    $R^{\Lambda_0}(\alpha_0+\alpha_1)\otimes R^{2\Lambda_a}(\alpha_{a-1}+2\alpha_a+\alpha_{a+1}) $ is wild.

  \item[ $(\Delta_{(a-2)^-,i^+})$] $\Lambda'''= \Lambda_{a-3}+\Lambda_{a+2}+\Lambda_{i+1}$.
 with  $a-2\le i$. Then $R^\La(\beta_{\Lambda'''})\cong R_A^\La(\beta_{\Lambda'''}) $ and it is wild by \cite{ASW-rep-type}.

   \item[ $(\Delta_{(a-2)^+,i^-})$] $\Lambda'''= \Lambda_{a-1}+\Lambda_{a+2}+\Lambda_{i-1}$ with  $i\le a-2$. Then $R^\La(\beta_{\Lambda'''})\cong R_A^\La(\beta_{\Lambda'''})$ and it is wild by \cite{ASW-rep-type}.

    \end{enumerate}
\item \textbf{The case of four changes.} It is enough to consider the path at level $4$:
    $$\Lambda=2\Lambda_a+\Lambda_i+\Lambda_j\rightarrow\Lambda'=\Lambda_{a-1}+\Lambda_{a+1}+\Lambda_i+\Lambda_j\rightarrow \Lambda''=\Lambda_{a-2}+\Lambda_{a+2}+\Lambda_i+\Lambda_j\rightarrow\Lambda''' $$
    such that both  $\Lambda_i$ and $\Lambda_j$ are changed in the last step, where $i\le j$ with  $i,j\ne a$ (since $m_a=2$).
    We first list the paths for pattern (II) below.
    \begin{itemize}
        \item [$(\Delta_{i^+,j^+})$] with  $a<j$:   $$\Lambda \rightarrow \Lambda_{mid}= 2\Lambda_a+\Lambda_{i+1}+\Lambda_{j+1}\rightarrow \Lambda_{a-1}+\Lambda_{a+1}+\Lambda_{i+1}+\Lambda_{j+1}\rightarrow\Lambda'''.$$
      \item [ $(\Delta_{i^+,j^-})$] Then, $\Lambda'''=  \Lambda_{a-2}+\Lambda_{a+2}+\Lambda_{i+1}+\Lambda_{j-1}$, for $i<j-1$:
 $$\Lambda\rightarrow \Lambda_{mid}=2\Lambda_a+\Lambda_{i+1}+\Lambda_{j-1}\rightarrow \Lambda_{a-1}+\Lambda_{a+1}+ \Lambda_{i+1}+\Lambda_{j-1}\rightarrow \Lambda'''.$$
      \end{itemize}
      The following are the remaining cases.
    
 \item[$(\Delta_{i^+,j^+})$] $\Lambda'''= \Lambda_{a-2}+\Lambda_{a+2}+\Lambda_{i+1}+\Lambda_{j+1}$.
 \begin{itemize} 
     \item Suppose that $j=a-1$. Then $\Lambda'''= \Lambda_{a-2}+\Lambda_{a+2}+\Lambda_{i+1}+\Lambda_{a}$
     is in the second neighbors:  $\Lambda \rightarrow \Lambda_{a+1}+\Lambda_a+\Lambda_{a-1}+\Lambda_{i+1}\rightarrow \Lambda'''$.
       \item Suppose that $j=a-2$. Then $\Lambda'''= \Lambda_{a-2}+\Lambda_{a+2}+\Lambda_{i+1}+\Lambda_{a-1}$
     is in the second neighbors: $\Lambda \rightarrow \Lambda_{a+1}+\Lambda_a+\Lambda_{a-2}+\Lambda_{i+1}\rightarrow \Lambda'''$.
     \item Suppose that $j<a-2$. Then Lemma \ref{tensor product lemma} implies
     $R^\La(\beta_{\Lambda'''})$ is wild since $R^{\Lambda_i+\Lambda_j}(\beta_{\Lambda_{i+1}+\Lambda_{j+1}})\otimes R^{2\Lambda_a}(\alpha_{a-1}+2\alpha_a+\alpha_{a+1})$ is wild.
 \end{itemize}
  \item[ $(\Delta_{i^-,j^+})$] Then, $R^\La(\beta_{\Lambda'''})\cong R_A^\La(\beta_{\Lambda'''})$ and it is wild by \cite{ASW-rep-type}.
\end{enumerate}
Finally, we obtain the results for the case $(\Delta_{i^-,j^-})$ by symmetry.

\subsubsection{Case (3)(ii)} We consider  $$\Lambda=3\Lambda_a+\tilde \Lambda\rightarrow \Lambda'=\Lambda_{a-1}+\Lambda_a+\Lambda_{a+1}+\tilde\Lambda\rightarrow\Lambda''=2\Lambda_{a-1}+\Lambda_{a+2}+\tilde\Lambda\rightarrow \Lambda''',$$ 
$1\le a\le \ell-2$, $m_a=3,m_{a+1}=0$, $\ch \k\ne 3$. 

\begin{enumerate}
    \item \textbf{The case of three changes.}
    It is enough to consider the path at level $3$: 
    $$\Lambda=3\Lambda_a\rightarrow \Lambda'=\Lambda_{a-1}+\Lambda_a+\Lambda_{a+1} \rightarrow\Lambda''=2\Lambda_{a-1}+\Lambda_{a+2} \rightarrow \Lambda'''$$
    with $1\le a\le \ell-2$, $m_a=3,m_{a+1}=0$.
    First, the cases $\Delta_{(a+2)^-}$ and $\Delta_{(a-1)^+,(a-1)^+}$ are in pattern (I'). Second,    the following cases are in pattern (I'').
    \begin{itemize}
        \item[$\Delta_{(a-1)^+}$]: $\Lambda\rightarrow 2\Lambda_a+\Lambda_{a+2}\rightarrow \Lambda'''=\Lambda_{a-1}+\Lambda_{a+1}+\Lambda_{a+2}$.
        \item[$\Delta_{(a-1)^+,(a+2)^+}$]: $\Lambda\rightarrow 2\Lambda_a+\Lambda_{a+2}\rightarrow \Lambda'''= \Lambda_{a-1}+\Lambda_{a}+\Lambda_{a+3}$.
        \item [$\Delta_{(a-1)^-,(a+2)^-}$]: $\Lambda\rightarrow 2\Lambda_a+\Lambda_{a-2}\rightarrow \Lambda'''=\Lambda_{a-2}+\Lambda_{a-1}+\Lambda_{a+1}$.
    \end{itemize} 
   Next, we have the cases in pattern (II) as below:
    \begin{itemize}
        \item [($\Delta_{(a+2)^+}$)]:
  $\Lambda\rightarrow \Lambda_{mid}= 2\Lambda_a+\Lambda_{a+2}\rightarrow \Lambda_{a-1}+\Lambda_{a+3}+\Lambda_a\rightarrow \Lambda'''.$
   \item[($\Delta_{(a-1)^-}$)]:
  $\Lambda\rightarrow \Lambda_{mid}= 2\Lambda_a+\Lambda_{a-2}\rightarrow \Lambda_{a-3}+\Lambda_{a+1}+\Lambda_a\rightarrow \Lambda'''.$
     \item[($\Delta_{(a-1)^-,(a-1)^-}$)]:
  $\Lambda\rightarrow \Lambda_{mid}= 2\Lambda_a+\Lambda_{a-2}\rightarrow \Lambda_{a-1}+\Lambda_{a+1}+\Lambda_{a-2}\rightarrow \Lambda'''.$

    \end{itemize}
   For the case
 ($\Delta_{(a-1)^-,(a+2)^+}$), we have  $R^\La(\beta_{\Lambda'''})= \Lambda_{a-2} + \Lambda_{a-1}+ \Lambda_{a+3}$, which appears in type $A^{(1)}_\ell$ and wild by \cite{ASW-rep-type}. The same holds for the case ($\Delta_{(a-1)^-,(a-1)^+}$).
 
\item \textbf{The case of four changes.} 
 It is enough to consider the path at level $4$:
 $$\Lambda=3\Lambda_a+\Lambda_i\rightarrow \Lambda'=\Lambda_{a-1}+\Lambda_a+\Lambda_{a+1} +\Lambda_i\rightarrow\Lambda''=2\Lambda_{a-1}+\Lambda_{a+2}+\Lambda_i \rightarrow \Lambda'''$$
 such that $\Lambda_{i}$ is changed in the last step, where      $1\le a\le \ell-2$,  $i\notin \{a,a+1\}$ (since $m_a=3,m_{a+1}=0$). 
 First the case $\Delta_{i^-}$ with $i=a+2$ is in pattern (I'). 
 Second, cases in pattern (I'') are
 \begin{itemize}
     \item[$\Delta_{i^+}$] with $i=a-2$: $\Lambda \rightarrow \Lambda_{a-1}+\Lambda_a+\Lambda_{a+1}+\Lambda_{a-2}\rightarrow \Lambda'''= 2\Lambda_{a-1}+\Lambda_{a+2}+\Lambda_{a}$.
     \item[$\Delta_{i^-}$] with $i=a+3$: $\Lambda\rightarrow \Lambda_{a-1}+2\Lambda_a+\Lambda_{a+2}\rightarrow \Lambda'''.$
     \item[$\Delta_{(a-1)^+,i^+}$]: $\Lambda\rightarrow \Lambda_{a+1}+2\Lambda_a+\Lambda_{i+1}\rightarrow \Lambda'''.$
     \item[$\Delta_{(a+2)^-,i^-}$]: $ \Lambda\rightarrow \Lambda_{a-1}+2\Lambda_a+\Lambda_{i-1}\rightarrow \Lambda'''. $
     \item[$\Delta_{(a-1)^+,i^-}$] with $i>a+1$: $\Lambda\rightarrow 2\Lambda_a+\Lambda_{i-1}+\Lambda_{a-1}\rightarrow \Lambda'''$
     \item[$\Delta_{(a+2)^-,i^+}$] with $i\le a+2$: 
     $\Lambda\rightarrow 2\Lambda_a+\Lambda_{i+1}+\Lambda_{a-1}\rightarrow \Lambda'''$.
 \end{itemize}
 Next, we list cases in pattern (II):
 \begin{itemize}
     \item[($\Delta_{i^+}$)] with   $i>a+1$:
        $\Lambda \rightarrow \Lambda_{mid}= 3\Lambda_a+\Lambda_{i+2}\rightarrow \Lambda_{a-1}+\Lambda_a+\Lambda_{a+1} +\Lambda_{i+2}\rightarrow \Lambda'''.   $
   \item   [($\Delta_{i^-}$)] with $i<a$:
         $\Lambda \rightarrow \Lambda_{mid}= 3\Lambda_a+\Lambda_{i-2}\rightarrow \Lambda_{a-1}+\Lambda_a+\Lambda_{a+1} +\Lambda_{i-2}\rightarrow \Lambda'''.   $
       \item[($\Delta_{(a+2)^+,i^+}$)] :
         $\Lambda \rightarrow \Lambda_{mid}= 2\Lambda_a+\Lambda_{i+1}+\Lambda_{a+1}\rightarrow \Lambda_{a-1}+\Lambda_{i+1}+\Lambda_{a+2}+ \Lambda_{a } \rightarrow \Lambda'''.   $
    \item[($\Delta_{(a-1)^-,i^-}$)]:
         $\Lambda \rightarrow \Lambda_{mid}= 2\Lambda_a+\Lambda_{i-1}+\Lambda_{a-1}\rightarrow 2\Lambda_{a-1}+\Lambda_{i-1}+\Lambda_{a+1} \rightarrow \Lambda'''   $.
\item  [($\Delta_{(a-1)^-,i^+}$)] with $i<a-1$:
    $\Lambda \rightarrow \Lambda_{mid}= 2\Lambda_a+\Lambda_{i+1}+\Lambda_{a+1}\rightarrow \Lambda_{a-1}+\Lambda_{i+1}+\Lambda_{a+2}+\Lambda_a \rightarrow \Lambda'''   $.
   \item[($\Delta_{(a+2)^+,i^-}$)]
     with $i> a+3$:  
     $\Lambda \rightarrow \Lambda_{mid}= 2\Lambda_a+\Lambda_{i-1}+\Lambda_{a+1}\rightarrow \Lambda_{a-1}+\Lambda_{i-1}+\Lambda_{a+2}+\Lambda_a \rightarrow \Lambda'''.   $            
 \end{itemize}
    The following are the remaining cases. 

\begin{enumerate}
    \item[($\Delta_{i^+}$)] $\Lambda'''= 2\Lambda_{a-1}+\Lambda_{a+2}+\Lambda_{i+2}$. Note that we may further assume that $i\ne a-1$ since if $i=a-1$, it belongs to the previous case ($\Lambda_{a-1}$ is not changed in the above path and hence there are only three changes). 
    It remains to consider $i<a-2$. Then  Lemma \ref{tensor product lemma} implies that $R^\La(\beta_{\Lambda'''})$
        is wild since $R^{\Lambda_i}(\beta_{\Lambda_{i+2}})\otimes R^{3\Lambda_a}(2\alpha_a+\alpha_{a+1})$ is wild. 
     
     \item[($\Delta_{i^-}$)]  with $i>a+3$. Then we deduce that $R^\Lambda(\beta_{\Lambda'''})$ is wild by applying Lemma \ref{tensor product lemma} as in the previous cases.
\end{enumerate} 
In the next four cases, we have $R^\La(\beta_{\La'''})\cong R_A^\La(\beta_{\La'''})$ and 
they are wild by \cite{ASW-rep-type}.
\begin{enumerate}
    \item[($\Delta_{(a-1)^+,i^-}$)] $\Lambda'''= \Lambda_{a-1}+\Lambda_{a}+\Lambda_{a+2}+\Lambda_{i-1}$ with  $i\le a+1$. 
    \item[($\Delta_{(a-1)^-,i^+}$)] $\Lambda'''= \Lambda_{a-2}+\Lambda_{a-1}+\Lambda_{a+2}+\Lambda_{i+1}$.
    with $a-1\le i$.  
    
    \item[($\Delta_{(a+2)^-,i^+}$)]
    $\Lambda'''= \Lambda_{a+1}+2\Lambda_{a-1}+\Lambda_{i+1}$ with 
    $a+2\le i$ ($i=a-1$ can not occur). 
    \item[($\Delta_{(a+2)^+,i^-}$)]
    $\Lambda'''= \Lambda_{a+3}+2\Lambda_{a-1}+\Lambda_{i-1}$ with 
    $i\le a+2 $. 
\end{enumerate}
 
\item \textbf{The case of five changes.}

It is enough to consider the path at level $5$: $$\Lambda=3\Lambda_a+\Lambda_i+\Lambda_j\rightarrow \Lambda'=\Lambda_{a-1}+\Lambda_a+\Lambda_{a+1} +\Lambda_i+\Lambda_j\rightarrow\Lambda''=2\Lambda_{a-1}+\Lambda_{a+2}+\Lambda_i+\Lambda_j \rightarrow \Lambda'''$$
 such that both  $\Lambda_{i}$ and $\Lambda_j$ are changed in the last step, where      $1\le a\le \ell-2$, $i\le j$ and  $i,j\notin \{a,a+1\}$ (since $m_a=3,m_{a+1}=0$). Furthermore, we may assume $i,j\ne a-1$: otherwise, there are only four changes and it has already been treated above.
 We then list cases in pattern (II) below.
 \begin{itemize}
     \item [($\Delta_{i^+,j^+}$)] with   $j\ge a$:
     $\Lambda \rightarrow  3\Lambda_a+\Lambda_{i+1}+\Lambda_{j+1}\rightarrow \Lambda_{a-1}+\Lambda_a+ \Lambda_{a+1}+\Lambda_{i+1}+\Lambda_{j+1}\rightarrow \Lambda''' .  $
 \item  [($\Delta_{i^-,j^-}$)] with $i< a$:
     $\Lambda \rightarrow  3\Lambda_a+\Lambda_{i-1}+\Lambda_{j-1}\rightarrow \Lambda_{a-1}+\Lambda_a+ \Lambda_{a+1}+\Lambda_{i-1}+\Lambda_{j-1}\rightarrow \Lambda'''.   $
    \item [($\Delta_{i^+,j^-}$)]:
     $\Lambda \rightarrow  3\Lambda_a+\Lambda_{i+1}+\Lambda_{j-1}\rightarrow \Lambda_{a-1}+\Lambda_a+ \Lambda_{a+1}+\Lambda_{i+1}+\Lambda_{j-1}\rightarrow \Lambda'''.   $
                
 \end{itemize}
 The following are the remaining cases.
 
\begin{enumerate}
    \item [($\Delta_{i^+,j^+}$)] $ \Lambda'''=2\Lambda_{a-1}+\Lambda_{a+2}+\Lambda_{i+1}+\Lambda_{j+1}$.
    with    $j\le a-2$. Then we see that $R^\La(\beta_{\Lambda'''})$ is wild since   $R^{\Lambda_i+\Lambda_j}(\beta_{\Lambda_{i+1}+\Lambda_{j+1}})\otimes R^{3\Lambda_a}(2\alpha_a+\alpha_{a+1})$ is wild. 
     
    \item [($\Delta_{i^-,j^-}$)]$\Lambda'''= 2\Lambda_{a-1}+\Lambda_{a+2}+\Lambda_{i-1}+\Lambda_{j-1}$.
    \begin{itemize}
        
      \item   Suppose that $i=a+2$.  Then $\Lambda'''$ is in the second neighbors: $$ \Lambda\rightarrow 2\Lambda_a+\Lambda_{a-1}+\Lambda_{a+2}+\Lambda_{j-1}\rightarrow\Lambda'''=2\Lambda_{a-1}+\Lambda_{a+2}+\Lambda_{a+1}+\Lambda_{j-1}. $$
        \item Suppose that $i>a+2$. Then we see that $R^\La(\beta_{\Lambda'''})$ is wild since   $R^{\Lambda_i+\Lambda_j}(\beta_{\Lambda_{i-1}+\Lambda_{j-1}})\otimes R^{3\Lambda_a}(2\alpha_a+\alpha_{a+1})$ is wild.
        \end{itemize}
    \item [($\Delta_{i^-,j^+}$)] In this case  $R^\La(\beta_{\Lambda'''})\cong R_A^\La(\beta_{\Lambda'''})$ and it is wild by \cite{ASW-rep-type}.

\end{enumerate}

\end{enumerate}

\subsubsection{Case (3)(iv)} We consider 
$$\Lambda=2\Lambda_a+2\Lambda_b+\tilde\Lambda\rightarrow \Lambda'=\Lambda_{a-1}+\Lambda_{a+1}+2\Lambda_b+\tilde\Lambda\rightarrow \Lambda''=\Lambda_{a-1}+\Lambda_{a+1}+\Lambda_{b-1}+\Lambda_{b+1}+\tilde\Lambda\rightarrow\Lambda''',$$
$1\le a< b-1$, $b \le \ell-1$, $m_a=m_b=2$.
\begin{enumerate}
    \item \textbf{The case of four changes.}
    It is enough to consider the path at level $4$:
    $$\Lambda=2\Lambda_a+2\Lambda_b\rightarrow \Lambda'=\Lambda_{a-1}+\Lambda_{a+1}+2\Lambda_b\rightarrow \Lambda''=\Lambda_{a-1}+\Lambda_{a+1}+\Lambda_{b-1}+\Lambda_{b+1}\rightarrow\Lambda''' $$
with $1\le a< b-1$, $b \le \ell-1$.
The cases $\Delta_{(a-1)^+,(b+1)^-}$ and $\Delta_{(a-1)^+,(b-1)^+}$ belong to the first neighbors. The cases $\Delta_{(a-1)^+, (b-1)^-}$, $\Delta_{(a+1)^-, (b+1)^-}$, $\Delta_{(a-1)^+}$, $ \Delta_{(b-1)^+}$, $\Delta_{(a-1)^+, (b+1)^+}$, $\Delta_{(a+1)^+, (b-1)^+}$, $\Delta_{(b-1)^+, (b+1)^+}$ all belong to the second neighbors in Case (3) above. Since $a-1\le a+1\le b-1\le b+1$, algebras in the cases $\Delta_{(a-1)^-,(a+1)^+}$, $\Delta_{(a-1)^-,(b-1)^+}$, $\Delta_{(a-1)^-,(b+1)^+}$, $\Delta_{(a+1)^-,(b-1)^+}$, $\Delta_{(a+1)^-,(b+1)^+}$, $\Delta_{(b-1)^-,(b+1)^+}$ are cyclotomic quiver Hecke algebras in 
type $A^{(1)}_\ell$ and $R^\La(\beta_{\La'''})$ are all wild by \cite{ASW-rep-type}. 

The following cases are in pattern (II).
\begin{enumerate}
\item [($\Delta_{(a+1)^+,(b-1)^-}$)] where $a+1<b-1$: 
$$\Lambda\rightarrow  \Lambda_a+\Lambda_{a+1}+\Lambda_{b-1}+\Lambda_b \rightarrow \Lambda_{mid}= \Lambda_{a-1}+\Lambda_{a+2}+\Lambda_{b-1}+\Lambda_b\rightarrow\Lambda'''.$$
\item  [($\Delta_{(a+1)^+}$)]: 
$\Lambda\rightarrow  \Lambda_{a-1}+\Lambda_{a+1}+2\Lambda_b \rightarrow \Lambda_{mid}= \Lambda_{a-1}+\Lambda_{a+3}+2\Lambda_b\rightarrow\Lambda'''.$
 The same holds for the case $\Delta_{(b+1)^+}$.
\item [($\Delta_{(a-1)^+,(a+1)^+}$)]: 
$\Lambda\rightarrow  \Lambda_{mid}=\Lambda_{a}+\Lambda_{a+2}+2\Lambda_b   \rightarrow\Lambda'''.$

\item [($\Delta_{(a+1)^+,(b+1)^+}$)]:
$\Lambda\rightarrow  \Lambda_a+\Lambda_{a+1}+\Lambda_{b+1}+\Lambda_b \rightarrow  \Lambda_{mid}=\Lambda_{a-1}+\Lambda_{a+2}+\Lambda_{b+1}+\Lambda_b\rightarrow\Lambda'''.$

\end{enumerate}

Finally, we obtain results for the cases for $\Delta_-$ and $\Delta_{--}$ by symmetry, namely by applying the Dynkin automorphism to the cases $\Delta_+$ and $\Delta_{++}$ above.

\item \textbf{The case of five changes.}
     It is enough to consider the path 
    $$\Lambda=2\Lambda_a+2\Lambda_b+\Lambda_i\rightarrow \Lambda'=\Lambda_{a-1}+\Lambda_{a+1}+2\Lambda_b+\Lambda_i\rightarrow \Lambda''=\Lambda_{a-1}+\Lambda_{a+1}+\Lambda_{b-1}+\Lambda_{b+1}+\Lambda_i\rightarrow\Lambda''' $$
    such that $\Lambda_i$ is changed in the last step, 
where  $1\le a< b-1$, $b \le \ell-1$ and $i\notin\{a,b\}$ (since $m_a=m_b=2$).
 Using symmetry, it suffices to consider the following cases. 
 
 First, we have the following cases in pattern (II):
 \begin{itemize}
     \item [($\Delta_{i^+}$)] with $i>0$ or $i=0$ and  $a=1,2$: $\Lambda\rightarrow 2\Lambda_a+\Lambda_{i+2}+2\Lambda_b \rightarrow \Lambda_{mid}= \Lambda_{a-1}+\Lambda_{a+1}+\Lambda_{i+2}+2\Lambda_b\rightarrow\Lambda'''.$
    \item[($\Delta_{(a+1)^+,i^+}$)]: $\Lambda\rightarrow   \Lambda_{a-1}+\Lambda_{a+1}+\Lambda_{i}+2\Lambda_b\rightarrow\Lambda_{mid}=\Lambda_{a-1}+\Lambda_{a+2}+\Lambda_{i+1}+2\Lambda_b \rightarrow\Lambda'''.$ 
    The same conclusion holds for the case $\Delta_{(b+1)^+,i^+}$.
    \item [($\Delta_{(a+1)^+,i^-}$)] with $i>a+1$:   $\Lambda\rightarrow   \Lambda_{a-1}+\Lambda_{a+1}+\Lambda_{i}+2\Lambda_b\rightarrow\Lambda_{mid}=\Lambda_{a-1}+\Lambda_{a+2}+\Lambda_{i-1}+2\Lambda_b \rightarrow\Lambda'''.$ 
       Similarly, $R^\La(\beta_{\hat\Lambda''})$ is wild for the case $\Delta_{(a-1)^-,i^+}$.
 \end{itemize}
 Second, the following are the remaining cases.

\begin{enumerate}
  
\item [($\Delta_{i^+}$)]   $i=0$ and 
     $a>2$, then  $R^\Lambda(\beta_{\Lambda'''})$ is wild since it is Morita equivalent to $\k[X]/(X^2)\otimes \k[Y]/(Y^2)\otimes \k[Z]/(Z^2)$.
   
  \item[($\Delta_{(a-1)^+,i^+}$)] $\Lambda'''$ is in the second neighbors: $\Lambda\rightarrow \Lambda_{a+1}+\Lambda_{i+1}+\Lambda_a+2\Lambda_b\rightarrow\Lambda'''.$  The same conclusion  holds for the case $\Delta_{(b-1)^+,i^+}$.
  
  \item [($\Delta_{(a-1)^+,i^-}$)] $\Lambda'''$ is in the second neighbors: $\Lambda\rightarrow \Lambda_{a+1}+\Lambda_{i-1}+\Lambda_a+2\Lambda_b\rightarrow\Lambda'''.$ Similarly,  $\Lambda'''$ is in the second neighbors  for the case $\Delta_{(a+1)^-,i^+}$.
  \item [($\Delta_{(a+1)^+,i^-}$)] $\Lambda'''= \Lambda_{a}+\Lambda_{a+2}+\Lambda_{b-1}+\Lambda_{b+1}+\Lambda_{i-1}$ with $i\le a+1$. Then $R^\La(\beta_{\Lambda'''})\cong R_A^\La(\beta_{\Lambda'''})$, which is wild by \cite{ASW-rep-type}.

\end{enumerate}

\item \textbf{The case of six changes.}
It is enough to consider the path at level $6$:
\begin{multline*}
\Lambda=2\Lambda_a+2\Lambda_b+\Lambda_i+\Lambda_j\rightarrow \Lambda'=\Lambda_{a-1}+\Lambda_{a+1}+2\Lambda_b+\Lambda_i+\Lambda_j \\
\rightarrow \Lambda''=\Lambda_{a-1}+\Lambda_{a+1}+\Lambda_{b-1}+\Lambda_{b+1}+\Lambda_i+\Lambda_j\rightarrow\Lambda''' 
\end{multline*}
  such that both $\Lambda_i$and $\Lambda_j$ are changed in the last step,    
where  $1\le a< b-1$, $b \le \ell-1$, $i\le j$ and $i,j\notin\{a,b\}$ (since $m_a=m_b=2$).
Note that $R^\La(\beta_{\La'''})\cong R_A^\La(\beta_{\La'''})$ in the case $(\Delta_{i^-,j^+})$, and $R_A^\La(\beta_{\La'''})$ is wild by \cite{ASW-rep-type}.
\begin{enumerate}
    \item [($\Delta_{i^+,j^-}$)] $\Lambda'''=\Lambda_{a-1}+\Lambda_{a+1}+\Lambda_{b-1}+\Lambda_{b+1}+\Lambda_{i+1}+\Lambda_{j-1}$ ($i<j-1$). This is in pattern (II) by considering  the path 
    $$\Lambda\rightarrow \Lambda_{mid}= 2\Lambda_a+2\Lambda_b+\Lambda_{i+1}+\Lambda_{j-1}\rightarrow \Lambda_{a-1}+\Lambda_{a+1}+2\Lambda_b+\Lambda_{i+1}+\Lambda_{j-1}\rightarrow\Lambda'''. $$
 \item  [($\Delta_{i^+,j^+}$)] $\Lambda'''=\Lambda_{a-1}+\Lambda_{a+1}+\Lambda_{b-1}+\Lambda_{b+1}+\Lambda_{i+1}+\Lambda_{j+1}$ ($i<j-1$). 
 \begin{itemize}
     \item Suppose $j=a-1$. Then $\Lambda'''$ is in pattern (I''):  $$ \Lambda= 2\Lambda_a+2\Lambda_b+\Lambda_i+\Lambda_{a-1}\rightarrow \Lambda_{b-1}+\Lambda_{b+1}+2\Lambda_a+\Lambda_i+\Lambda_{a-1}\rightarrow\Lambda'''. $$
     \item Suppose that  $j<a-1$. Then we see that $R^\La(\beta_{\Lambda'''})$ is wild since 
     $R^{\Lambda_i+\Lambda_j}(\beta_{\Lambda_{i+1}+\Lambda_{j+1}})\otimes \k[X]/(X^2)\otimes \k[Y]/(Y^2)$ is wild.
     \item Suppose that $j>a$. This is in pattern (II):
     $$\Lambda\rightarrow \Lambda_{mid}=2\Lambda_a+2\Lambda_{b}+\Lambda_{i+1}+\Lambda_{j+1}\rightarrow\Lambda_{a-1}+\Lambda_{a+1}+2\Lambda_{b}+\Lambda_{i+1}+\Lambda_{j+1}\rightarrow\Lambda''' .$$
 \end{itemize}
 Finally,  we obtain the results for the case ($\Delta_{i^-,j^-}$) by symmetry.  
\end{enumerate}
\end{enumerate}

\subsubsection{Case (3)(v)} We consider 
$$\Lambda=4\Lambda_a+\tilde\Lambda\rightarrow \Lambda'=\Lambda_{a-1}+\Lambda_{a+1}+2\Lambda_a+\tilde\Lambda\rightarrow 2\Lambda_{a-1}+2\Lambda_{a+1}+\tilde\Lambda\rightarrow \Lambda''',$$
    $1\le a\le \ell-1$, $m_a=4$ and $\ch \k\ne 2$.
\begin{enumerate}
    \item \textbf{The case there are four changes.}
   It suffices to consider the path at level four:
    $$\Lambda=4\Lambda_a \rightarrow \Lambda'=\Lambda_{a-1}+\Lambda_{a+1}+2\Lambda_a\rightarrow 2\Lambda_{a-1}+2\Lambda_{a+1}\rightarrow \Lambda'''.$$
    First, it is easy to see that  the cases $\Delta_{(a-1)^+,(a-1)^+}$ and $\Delta_{(a+1)^-,(a+1)^-}$ belong to pattern (I') and  $\Delta_{(a-1)^+}$, $\Delta_{(a-1)^-,(a+1)^-}$, and $\Delta_{(a-1)^+,(a+1)^+}$ belong to pattern (I''). The following are the remaining cases.
 \begin{enumerate}
    \item[($\Delta_{(a+1)^+}$)] $\Lambda'''=2\Lambda_{a-1}+\Lambda_{a+1}+\Lambda_{a+3}$.
    This belongs to pattern (II):
    $$\Lambda\rightarrow \hat \Lambda'=3\Lambda_a+\Lambda_{a+2}\rightarrow 2\Lambda_a+\Lambda_{a-1}+\Lambda_{a+3}\rightarrow\Lambda''' .$$
    \item [($\Delta_{(a-1)^-, (a+1)^+}$)] In this case, $R^\Lambda(\beta_{\Lambda'''})\cong R_A^\Lambda(\beta_{\Lambda'''})$ and it is wild by \cite{ASW-rep-type}.  By the same reasoning, $R^\Lambda(\beta_{\Lambda'''})$ is wild in  the cases $\Delta_{(a-1)^-, (a-1)^+}$, $\Delta_{(a+1)^-, (a+1)^+}$.
 \item[($\Delta_{(a+1)^+, (a+1)^+}$)] $\Lambda'''=2\Lambda_{a-1}+2\Lambda_{a+2}$. Then there is an arrow from
 $\Lambda_{a-1}+3\Lambda_{a+1}$ to $\Lambda'''$.  Then, 
 $R^{4\La_a}(\beta_{\Lambda'''})$ is wild since we know that $R^{4\La_a}(\beta_{\Lambda_{a-1}+3\Lambda_{a+1}})$ is wild from  the case $\Delta_{(a-1)^+}$ above. 
 \item[($\Delta_{(a-1)^-,(a-1)^-}$)]  we deduce from $\Delta_{(a+1)^+, (a+1)^+}$ by symmetry that $R^\La(\beta_{\Lambda'''})$ is wild. 
  \end{enumerate}  
 \item \textbf{The case there are five changes.} It is enough to consider the path at level $5$:
 $$\Lambda=4\Lambda_a+\Lambda_i\rightarrow \Lambda_{a-1}+\Lambda_{a+1}+2\Lambda_a+\Lambda_i\rightarrow 2\Lambda_{a-1}+2\Lambda_{a+1}+\Lambda_i \rightarrow \Lambda'''$$
 such that $\Lambda_i$ is changed in the last step, where $i\ne a$ (since $m_a=4$). Furthermore, we may assume that $i\ne a-1,a+1$. Otherwise, $\Lambda_i$ is fixed in the above path and hence the case has already been considered in the previous case. By symmetry, it is enough to consider the following seven cases. Among them, the first three cases belong to pattern (II).
\begin{itemize}
    \item[($\Delta_{i^+}$)] with  $i>a$:
         $ \Lambda\rightarrow \Lambda_{mid}=4\Lambda_{a}+\Lambda_{i+2}\rightarrow \Lambda_{a-1}+\Lambda_{a+1}+2\Lambda_a+\Lambda_{i+2} \rightarrow \Lambda'''$
    \item[($\Delta_{(a+1)^+,i^+}$)]: 
  $\Lambda\rightarrow \Lambda'\rightarrow\Lambda_{mid}=\Lambda_{a-1}+\Lambda_{a+2}+\Lambda_{i+1} +2\Lambda_a\rightarrow\Lambda'''. $
 \item[($\Delta_{(a-1)^-,i^+}$)]  
   with $ i<a-2$: 
    $$\Lambda\rightarrow\Lambda_{mid}=  \Lambda_{i+1}+\Lambda_{a-1}+3\Lambda_a\rightarrow \Lambda_{i+1}+\Lambda_{a-2}+\Lambda_{a+1}+2\Lambda_a \rightarrow\Lambda'''.$$       
\end{itemize} 
 Then, the remaining four cases are as follows.
 \begin{enumerate}
     \item [($\Delta_{i^+}$)]  
    for $i<a-1$. Then applying Lemma \ref{tensor product lemma} we see that $R^\La(\beta_{\Lambda'''})$ is wild since 
         $R^{\Lambda_i}(\beta_{\Lambda_{i+2}})\otimes R^{4\Lambda_a}(2\alpha_a) $ is wild.

  \item[($\Delta_{(a-1)^+,i^+}$)] This is in pattern (I''): 
  $ \Lambda\rightarrow2\Lambda_a +\Lambda_{a-1}+\Lambda_{a+1}+\Lambda_i\rightarrow\Lambda'''. $

\item[($\Delta_{(a-1)^-,i^+}$)]
$\Lambda'''=\Lambda_{a-2}+\Lambda_{a-1}+2\Lambda_{a+1}+\Lambda_{i+1}$ with $ i>a$. Then $R^\Lambda(\beta_{\Lambda'''})\cong R_A^\Lambda(\beta_{\Lambda'''})$ appears in the third neighbors and proved wild in \cite{ASW-rep-type}.
    
\item[($\Delta_{(a-1)^+,i^-}$)] This is in pattern (I''): 
$\Lambda\rightarrow \Lambda_{a-1}+\Lambda_{a+1}+2\Lambda_a+\Lambda_{i}\rightarrow\Lambda'''. $
 \end{enumerate}
\item \textbf{The case there are six changes.} 
It is enough to consider the path 
 $$\Lambda=4\Lambda_a+\Lambda_i+\Lambda_j\rightarrow \Lambda_{a-1}+\Lambda_{a+1}+2\Lambda_a+\Lambda_i+\Lambda_j\rightarrow 2\Lambda_{a-1}+2\Lambda_{a+1}+\Lambda_i+\Lambda_j \rightarrow \Lambda'''$$
 such that both $\Lambda_i$ and $\Lambda_j$
 are changed in the last step with $i\le j$. Moreover, for the same reasoning as the previous cases, we may assume that $i,j\notin\{a-1,a,a+1\}$.
\begin{enumerate}
    \item[($\Delta_{i^-,j^+}$)] In this case, 
    $R^\La(\beta_{\Lambda'''})\cong R_A^\La(\beta_{\Lambda'''})$, which is wild by [ASW23]. 
    \item [($\Delta_{i^+,j^-}$)] $i\le j-2$. This is in pattern (II): 
    $$\Lambda\rightarrow\Lambda_{mid}=4\Lambda_a+\Lambda_{i+1}+\Lambda_{j-1}\rightarrow2\Lambda_a+\Lambda_{a-1}+\Lambda_{a+1}+ \Lambda_{i+1}+\Lambda_{j-1}\rightarrow\Lambda'''.$$
    \item[($\Delta_{i^+,j^+}$)] $\Lambda'''=2\Lambda_{a-1}+2\Lambda_{a+1}+\Lambda_{i+1}+\Lambda_{j+1}$.
    \begin{itemize}
        \item Suppose that $a<j$. Then, this is in pattern (II):  $$\Lambda\rightarrow\Lambda_{mid}=4\Lambda_a+\Lambda_{i+1}+\Lambda_{j+1}\rightarrow2\Lambda_a+\Lambda_{a-1}+\Lambda_{a+1}+ \Lambda_{i+1}+\Lambda_{j+1}\rightarrow\Lambda'''.$$
        \item Suppose that $j<a-1$. Then we see that $R^\La(\beta_{\Lambda'''})$ is wild since 
        $R^{\Lambda_i+\Lambda_j}(\beta_{\Lambda_{i+1}+\Lambda_{j+1}})\otimes R^{4\Lambda_a}(2\alpha_a)$ is wild.
          \end{itemize}
      By symmetry, $R^\La(\beta_{\La'''})$ is wild in the case $(\Delta_{i^-, j^-})$.    
\end{enumerate}
\end{enumerate}
 We have completed the proof for Cases (3)(i)--(3)(v) in the third neighbors.

\subsection{The third neighbors in Case (7)}
 Now we consider those $\Lambda'''$ in the third neighbors that appear in the following path
\begin{equation}\label{path73}
    \Lambda=\Lambda_a+\Lambda_b+\tilde \Lambda\rightarrow \Lambda'=\Lambda_{a+2}+\Lambda_b+\tilde \Lambda\rightarrow \Lambda'' \rightarrow \Lambda''',
\end{equation}
where $\Lambda'$ and $\Lambda''$ belong to (7)(i)-(iii) at the beginning of this section.

\subsubsection{Case (i)} 
There are three subcases as follows.
\begin{enumerate}
    \item[\textbf{Case (i)(a):}] \textbf{two changes in \eqref{path73}}. In this case, it is enough to consider the path 
$$ \Lambda=\Lambda_0+\Lambda_\ell \rightarrow \Lambda'=\Lambda_{2}+\Lambda_\ell \rightarrow \Lambda''=\Lambda_2+\Lambda_{\ell-2} \rightarrow \Lambda'''  $$
 Then Theorem \ref{level-2-a<b} implies that $R^\La(\beta_{\Lambda'''})$ is wild.
\item[\textbf{Case (i)(b):}] \textbf{three changes in \eqref{path73}}. In this case, it is enough to consider the path 
$$ \Lambda=\Lambda_0+\Lambda_\ell +\Lambda_i \rightarrow \Lambda'=\Lambda_{2}+\Lambda_\ell +\Lambda_i \rightarrow \Lambda''=\Lambda_2+\Lambda_{\ell-2}+\Lambda_i \rightarrow \Lambda'''  $$
such that $\Lambda_i$ is changed in the last step. Moreover, we may assume $2\le i\le \ell-2$: otherwise $R^\La(\beta_{\Lambda''})$ is wild by Case (7) in the second neighbors. By symmetry, it is enough to consider the cases $ \Delta_{i^+}, \Delta_{2^\pm, i^\pm}$. Moreover, $\Lambda'''$ in cases $\Delta_{2^-,i^-}$ and $\Delta_{2^-,i^+}$ have already appeared in the second neighbors. The remaining cases are all in pattern (II) as below.
\begin{enumerate}
    \item[($\Delta_{i^+}$)]:  
    $\Lambda\rightarrow \Lambda_0+\Lambda_\ell+\Lambda_{i+2}\rightarrow \Lambda_{mid}=\Lambda_{2}+\Lambda_\ell+\Lambda_{i+2}\rightarrow\Lambda''' .$

\item[($\Delta_{2^+,i^+}$)]:
$\Lambda\rightarrow \Lambda'\rightarrow \Lambda_{mid}=\Lambda_3+\Lambda_{i+1}+\Lambda_\ell\rightarrow \Lambda'''.
$

\item[($\Delta_{2^+,i^-}$)]: 
$\Lambda\rightarrow \Lambda'\rightarrow \Lambda_{mid}=\Lambda_3+\Lambda_{i-1}+\Lambda_\ell\rightarrow \Lambda'''.
$ 
\end{enumerate} 
\item[\textbf{Case (i)(c):}] \textbf{four changes in \eqref{path73}}. In this case, it is enough to consider the path 
$$ \Lambda=\Lambda_0+\Lambda_\ell +\Lambda_i+\Lambda_j \rightarrow \Lambda'=\Lambda_{2}+\Lambda_\ell +\Lambda_i+\Lambda_j \rightarrow \Lambda''=\Lambda_2+\Lambda_{\ell-2}+\Lambda_i+\Lambda_j \rightarrow \Lambda'''  $$
such that both $\Lambda_i$ and $\Lambda_j$ are changed in the last step. Moreover, we may assume $2\le i\le j\le \ell-2$ as in Case (i)(b).
We first list cases in pattern (II).
\begin{itemize}
      \item[($\Delta_{i^+,j^+}$)]:
    $\Lambda\rightarrow \Lambda'\rightarrow \Lambda_{mid}=\Lambda_2+\Lambda_\ell+\Lambda_{i+1}+\Lambda_{j+1}\rightarrow \Lambda''' .$
    By symmetry, the same holds for the case ($\Delta_{i^-,j^-}$).
    \item[($\Delta_{i^+,j^-}$)] $i\le j-2$:
    $\Lambda\rightarrow \Lambda'\rightarrow \Lambda_{mid}=\Lambda_2+\Lambda_\ell+\Lambda_{i+1}+\Lambda_{j-1}\rightarrow \Lambda''' .$
\end{itemize}

It remains to consider the following case.

\begin{itemize}
    \item[$(\Delta_{i^-,j^+})$]:  $\Lambda'''=\Lambda_2+\Lambda_{\ell-2}+\Lambda_{i-1}+\Lambda_{j+1}$ with $i\le j$. 
\end{itemize}
 \begin{itemize}
     \item Suppose that $i=2$. Then $\Lambda'''=\Lambda_1+\Lambda_2+\Lambda_{\ell-2}+\Lambda_{j+1}$ is in the second neighbors of $\Lambda=\Lambda_0+\Lambda_2+\Lambda_j+\Lambda_\ell$ and has already been treated.
     \item Suppose that $j=\ell-2$. Similarly, we have $\Lambda'''$ in the second neighbors of $\Lambda$.
     \item Suppose that $2<i\le j<\ell-2$. 
     If $i=j$, then 
$$\beta_{\Lambda'''}=\alpha_0+\alpha_1+\alpha_i+\alpha_{\ell-1}+\alpha_\ell.$$
By Lemma \ref{tensor product lemma},   
$R^\Lambda(\beta_{\Lambda'''})$ is Morita equivalent to the wild local algebra
$$ \k[X]/(X^2)\otimes \k[Y]/(Y^2)\otimes \k[Z]/(Z^2). $$
If $i<j$, then this belongs to pattern (II) since we have the path
$$\Lambda\rightarrow \Lambda'\rightarrow \Lambda_{mid}=\Lambda_2+\Lambda_{i-1}+\Lambda_{j+1}+\Lambda_\ell \rightarrow\Lambda'''.$$    
 \end{itemize}
\end{enumerate}
\subsubsection{Case (ii)} Recall that $\Lambda=\Lambda_0+2\Lambda_i+\tilde \Lambda$, $\Lambda'=\Lambda_2+2\Lambda_i+\tilde \Lambda$, $\Lambda''= \Lambda_2+\Lambda_{i-1}+\Lambda_{i+1}+\tilde\Lambda $ with  $m_0=1$, $m_1=0$, $m_i=2$ and $2<i\le \ell-1$.
There are three subcases as follows.
\begin{enumerate}
    \item[\textbf{Case (ii)(a):}] \textbf{three changes in \eqref{path73}}. We consider the path 
$$ \Lambda=\Lambda_0+2\Lambda_i \rightarrow \Lambda'=\Lambda_{2}+2\Lambda_i \rightarrow \Lambda''=\Lambda_2+\Lambda_{i-1}+\Lambda_{i+1} \rightarrow \Lambda'''.  $$
We see that   cases $\Delta_{2^-,(i+1)^-}$, $\Delta_{2^-,(i-1)^+}$ and $\Delta_{2^+,(i-1)^-}$ ($i=3$) are in pattern (I'). Cases $\Delta_{(i-1)^+}$, $\Delta_{(i+1)^-}$, $\Delta_{2^+,(i-1)^+}$, $\Delta_{2^+,(i+1)^-}$, $\Delta_{(i-1)^+,(i+1)^+}$, $\Delta_{(i-1)^-,(i+1)^-}$ are in pattern (I'') since they are in the next step of $\Lambda'$. Moreover, the following cases are in pattern (I'') with the paths listed below.
\begin{itemize}
    \item [$(\Delta_{2^-,(i-1)^-})$]: $ \Lambda\rightarrow \Lambda_1+\Lambda_i+\Lambda_{i+1}\rightarrow \Lambda'''.$
    \item [$(\Delta_{2^-,(i+1)^+})$]: $ \Lambda\rightarrow \Lambda_0+\Lambda_{i-1}+\Lambda_{i+1}\rightarrow \Lambda'''.$
    \item [$(\Delta_{(i-1)^-,(i+1)^+})$] with $i=3$:
   $\Lambda\rightarrow \Lambda_0+\Lambda_2+\Lambda_4\rightarrow\Lambda'''$.
\end{itemize}
Next, we list the pattern (II) cases:
\begin{itemize}
    \item[($\Delta_{2^+}$)]:
    $\Lambda\rightarrow \Lambda'\rightarrow \Lambda_{mid}=\Lambda_4+2\Lambda_i\rightarrow \Lambda''' $, 
   by Theorem \ref{theo::result-level-one}. 
     
   \item[($\Delta_{(i+1)^+}$)]: 
   $\Lambda\rightarrow  \Lambda_0+\Lambda_{i-1}+\Lambda_{i+1}\rightarrow \Lambda_{mid}=\Lambda_0+\Lambda_{i-1}+\Lambda_{i+3}\rightarrow \Lambda'''$,  
   by Theorem \ref{level-2-a=b}.
 \item[($\Delta_{(i-1)^-}$)]:
 $\Lambda\rightarrow \Lambda_0+\Lambda_{i-1}+\Lambda_{i+1}\rightarrow\Lambda_{mid}=\Lambda_0+\Lambda_{i-3}+\Lambda_{i+1}\rightarrow\Lambda'''$,
 by Theorem \ref{level-2-a=b}.
\item[($\Delta_{2^+,(i+1)^+}$)] :
$\Lambda\rightarrow \Lambda'\rightarrow \Lambda_{mid}=\Lambda_3+\Lambda_{i-1}+\Lambda_i\rightarrow \Lambda'''. $
\item[($\Delta_{2^+,(i-1)^-}$)] with $i>3$: 
    $ \Lambda\rightarrow \Lambda'\rightarrow \Lambda_{mid}=\Lambda_2+\Lambda_{i-2}+\Lambda_i\rightarrow\Lambda'''$.
\end{itemize}

It remains to consider the case $\Delta_{(i-1)^-,(i+1)^+}$ with  $i>3$. Applying Lemma \ref{tensor product lemma}, we see that $\Lambda'''$ is wild since $R^{\Lambda_0}(\alpha_0+\alpha_1)\otimes R^{2\Lambda_i}(\alpha_{i-1}+2\alpha_i+\alpha_{i+1})$ is wild.

 \item[\textbf{Case (ii)(b):}] \textbf{four changes in \eqref{path73}}. In this case, it is enough to consider the path 
$$ \Lambda=\Lambda_0+2\Lambda_i+\Lambda_a \rightarrow \Lambda'=\Lambda_{2}+2\Lambda_i+\Lambda_a \rightarrow \Lambda''=\Lambda_2+\Lambda_{i-1}+\Lambda_{i+1}+\Lambda_a \rightarrow \Lambda'''  $$
such that $\Lambda_a$ is changed in the last step and $2\le a\le \ell$ with $a\ne i$.
The following cases are in pattern (I''): $\Delta_{a^-}$ ($a=\ell=i+1$), $\Delta_{(i+1)^+,a^-}$ ($a=i+1$), 
$\Delta_{(i-1)^+,a^\pm}$, $\Delta_{(i+1)^-,a^\pm}$, $\Delta_{2^-,a^-}$, since they are in the next step of $\Lambda'$. Moreover, the case $\Delta_{2^-,a^+}$ belongs to pattern (I'') by the path  $\Lambda \rightarrow \Lambda_1+2\Lambda_i+\Lambda_{a+1}\rightarrow \Lambda''' $. 

Nextly, we list cases in pattern (II) as below.
\begin{itemize}
  \item [($\Delta_{a^+}$)]:
    $\Lambda\rightarrow \Lambda_{mid}=\Lambda_0+\Lambda_{a+2}+2\Lambda_i\rightarrow \Lambda_2+\Lambda_{a+2}+2\Lambda_i\rightarrow\Lambda'''. $
  \item [($\Delta_{a^-}$)]:   
     $\Lambda\rightarrow \Lambda' \rightarrow \Lambda_{mid}=\Lambda_2+\Lambda_{a-2}+2\Lambda_i\rightarrow\Lambda'''.$
  \item[($\Delta_{2^+,a^+}$)]:
 $\Lambda\rightarrow \Lambda'\rightarrow \Lambda_{mid}=\Lambda_3+\Lambda_{a+1}+2\Lambda_i\rightarrow \Lambda'''. $
     \item[($\Delta_{(i+1)^+,a^+}$)]: 
 $\Lambda\rightarrow \Lambda_0+\Lambda_{i-1}+\Lambda_{i+1}+\Lambda_a\rightarrow \Lambda_{mid}=\Lambda_0+\Lambda_{i-1}+\Lambda_{i+2}+\Lambda_{a+1}\rightarrow \Lambda'''. $ 
 \item[($\Delta_{(i-1)^-,a^-}$)]:
 $\Lambda\rightarrow \Lambda_0+\Lambda_{i-1}+\Lambda_{i+1}+\Lambda_{a}\rightarrow \Lambda_{mid}=\Lambda_0+\Lambda_{i-2}+\Lambda_{i+1}+\Lambda_{a-1}\rightarrow \Lambda'''.$
   \item[($\Delta_{2^+,a^-}$)] $3< a\le \ell$: $ \Lambda\rightarrow \Lambda'\rightarrow \Lambda_{mid}=\Lambda_3+\Lambda_{a-1}+2\Lambda_i\rightarrow \Lambda'''.$
  \item[($\Delta_{(i-1)^-,a^+}$)] with $a<i-2$: 
$$\Lambda\rightarrow \Lambda_0+\Lambda_{i-1}+\Lambda_{i+1}+\Lambda_a\rightarrow \Lambda_{mid}=\Lambda_0+\Lambda_{i-2}+\Lambda_{i+1}+\Lambda_{a+1}\rightarrow \Lambda'''.$$
 \item[($\Delta_{(i+1)^+,a^-}$)] with $a>i+2$:
    $$\Lambda\rightarrow \Lambda_0+\Lambda_{i-1}+\Lambda_{i+1}+\Lambda_a\rightarrow \Lambda_{mid}=\Lambda_0+\Lambda_{i-1}+\Lambda_{i+2}+\Lambda_{a-1}\rightarrow\Lambda'''.$$
\end{itemize}

The following are the remaining cases.

\begin{enumerate}
   
 \item [($\Delta_{a^-}$)] $\Lambda'''=\Lambda_2+\Lambda_{i-1}+\Lambda_{i+1}+\Lambda_{a-2}$,   (this does not happen if $a=2$)  with $a=\ell$ and $i<\ell-1$.
 Then
$\beta_{\Lambda'''}=\alpha_0+\alpha_1+\alpha_i+\alpha_{\ell-1}+\alpha_\ell$ and $R^\La(\beta_{\Lambda'''})$ is wild by
     Lemma \ref{tensor product lemma} as in Case(i)(c) ($\Delta_{i^-,j^+}$).

\item[($\Delta_{(i-1)^-,a^+}$)] $\Lambda'''=\Lambda_2+\Lambda_{i-2}+\Lambda_{a+1}+\Lambda_{i+1}$ with   $a>i$.
   Then $\beta_{\Lambda'''}=\beta_1+\beta_2$, where
$$\beta_1=\alpha_0+\alpha_1,\quad \beta_2=\alpha_{i-1}+2\alpha_i+\alpha_{i+1}+\ldots +\alpha_a.$$
Recall that $i>2$. Applying Lemma \ref{tensor product lemma}, we see that $R^\La(\beta_{\Lambda'''})$ is wild since $R^{\Lambda_0}(\beta_1)\otimes R^{2\Lambda_i+\Lambda_a}(\beta_2)$ is wild. 

\item[($\Delta_{(i+1)^+,a^-}$)] $\Lambda'''=\Lambda_2+\Lambda_{i-1}+\Lambda_{i+2}+\Lambda_{a-1}$ with $a<i$. Then we have the path  
    $$ \Lambda\rightarrow \Lambda_0+\Lambda_{i-1}+\Lambda_{i+1}+\Lambda_a\rightarrow \Lambda_{mid}=\Lambda_0+\Lambda_{i-1}+\Lambda_{i+2}+\Lambda_{a-1}\rightarrow \Lambda'''. $$ Then, 
    $R^\La(\beta_{\Lambda_{mid}})\cong R_A^\La(\beta_{\Lambda_{mid}})$, which is wild.
\end{enumerate}
\item[\textbf{Case (ii)(c):}] \textbf{five changes in \eqref{path73}}. In this case, it is enough to consider the path 
$$ \Lambda=\Lambda_0+2\Lambda_i+\Lambda_a +\Lambda_b \rightarrow \Lambda'=\Lambda_{2}+2\Lambda_i+\Lambda_a +\Lambda_b\rightarrow \Lambda''=\Lambda_2+\Lambda_{i-1}+\Lambda_{i+1}+\Lambda_a +\Lambda_b\rightarrow \Lambda'''  $$
such that both  $\Lambda_a$ and $\Lambda_b$ are changed in the last step and $2\le a\le b\le \ell$ with $a, b\ne i$.
We first list cases in pattern (II).
\begin{itemize}
   \item [($\Delta_{a^+,b^+}$)]:
    $\Lambda\rightarrow \Lambda_{mid}= \Lambda_0+2\Lambda_i+\Lambda_{a+1} +\Lambda_{b+1}\rightarrow\Lambda_2+2\Lambda_i+\Lambda_{a+1} +\Lambda_{b+1}\rightarrow\Lambda'''. $
  \item[($\Delta_{a^-,b^+}$)]  ($a\ne b-1$) with  $a<b$:
$\Lambda\rightarrow\Lambda'\rightarrow\Lambda_{mid}=\Lambda_2+\Lambda_{a-1}+\Lambda_{b+1}+2\Lambda_i\rightarrow\Lambda'''.$
 \item[($\Delta_{a^+,b^-}$)] with $a<b-1$:
     $$ \Lambda\rightarrow \Lambda_{mid}=\Lambda_0+2\Lambda_i+ \Lambda_{a+1} +\Lambda_{b-1}\rightarrow \Lambda_2+2\Lambda_i+ \Lambda_{a+1} +\Lambda_{b-1}\rightarrow\Lambda'''. $$  
\end{itemize}
The following are the remaining cases.

\begin{enumerate}
    
    \item[($\Delta_{a^-,b^-}$)]  $\Lambda'''=\Lambda_2+\Lambda_{i-1}+\Lambda_{i+1}+\Lambda_{a-1} +\Lambda_{b-1}$. Consider the path 
    $$\Lambda\rightarrow \Lambda'\rightarrow \hat\Lambda''=\Lambda_0+2\Lambda_i+\Lambda_{a-1} +\Lambda_{b-1}\rightarrow\Lambda'''.$$
    Note that $R^\La(\beta_{\hat\Lambda''})$ is not wild only when $a=b=\ell$.
    We assume $a=b=\ell$ in the following.
    \begin{itemize}
        \item Suppose $i=\ell-1$. Then $\Lambda'''$ is in pattern (I''): $\Lambda\rightarrow\Lambda'\rightarrow\Lambda'''$.
        \item Suppose that $i<\ell-1$. Then 
        $$\beta_{\Lambda'''}=\alpha_0+\alpha_1+\alpha_i+\alpha_\ell.$$
        By Lemma \ref{tensor product lemma}, we see that 
        $R^\Lambda(\beta_{\Lambda'''})$ is Morita equivalent to the wild local algebra $\k[X]/(X^2)\otimes \k[Y]/(Y^2)\otimes\k[Z]/(Z^2)$.
    \end{itemize}
 \item[($\Delta_{a^+,b^-}$)] $\Lambda'''=\Lambda_2+\Lambda_{i-1}+\Lambda_{i+1}+\Lambda_{a+1} +\Lambda_{b-1}$ ($a\ne b-1$).
with $a=b$. Recall $a\ne i$. If $a>2$, then $R^\La(\beta_{\Lambda'''})$ is wild since 
     $\k[X]/(X^2)\otimes \k[Y]/(Y^2)\otimes\k[Z]/(Z^2)$ is wild. If $a=2$, then $\Lambda'''$ is in pattern (I''):
     $\Lambda \rightarrow \Lambda_1+\Lambda_2+\Lambda_3+2\Lambda_i\rightarrow \Lambda'''$.
\end{enumerate}

\bigskip
We have completed Case (ii). Note that, in the path
$$\Lambda=\Lambda_0+2\Lambda_i+\tilde \Lambda \rightarrow \Lambda_2+2\Lambda_i +\tilde\Lambda \rightarrow \Lambda_2+\Lambda_{i-1}+\Lambda_{i+1}+\tilde\Lambda \rightarrow \Lambda''', $$
there are only three changes in the first two steps and the third step produces at most two new changes. Hence, there are at most five changes in the first three steps.
\end{enumerate}

\subsubsection{Case (iii)} Recall that $\Lambda=\Lambda_0+2\Lambda_\ell+\tilde \Lambda$, $\Lambda'=\Lambda_2+2\Lambda_\ell+\tilde \Lambda$, $\Lambda''= \Lambda_2+2\Lambda_{\ell-1}+\tilde\Lambda $ with  $m_0=1$, $m_1=0$, $m_\ell=2$, $\ell\ge3$.
There are three subcases as follows.
\begin{enumerate}
    \item[\textbf{Case (iii)(a):}] \textbf{three changes in \eqref{path73}}. It is enough to consider the path 
$$ \Lambda=\Lambda_0+2\Lambda_\ell \rightarrow \Lambda'=\Lambda_{2}+2\Lambda_\ell \rightarrow \Lambda''=\Lambda_2+2\Lambda_{\ell-1} \rightarrow \Lambda'''  .$$
We first see that the case  $\Delta_{2^-,(\ell-1)^+}$ and the case $\Delta_{2^+,(\ell-1)^-}$ with $\ell=3$ are in the first neighbors and in pattern (I'). We also see that the following cases are in  pattern (I'').
\begin{itemize}
    \item [$(\Delta_{2^+,(\ell-1)^+})$]:
 $  \Lambda\rightarrow \Lambda_2+2\Lambda_\ell\rightarrow  \Lambda'''.  $
 \item[$(\Delta_{2^-,(\ell-1)^-})$]: 
 $\Lambda\rightarrow \Lambda_0+\Lambda_{\ell-2}+\Lambda_\ell\rightarrow\Lambda'''$.
\end{itemize}
We list cases in pattern (II).
\begin{itemize}
\item[$(\Delta_{2^+})$]:
    $\Lambda\rightarrow \Lambda'\rightarrow \Lambda_{mid}=\Lambda_4+2\Lambda_\ell\rightarrow \Lambda''' $, 
      by   
    Theorem \ref{theo::result-level-one}. 
\item[$(\Delta_{(\ell-1)^-})$] with $\ell\ge4$:
$\Lambda\rightarrow \Lambda_0+2\Lambda_{\ell-1}\rightarrow \Lambda_{mid}=\Lambda_0+\Lambda_{\ell-1}+\Lambda_{\ell-3}\rightarrow \Lambda''$, by Theorem \ref{level-2-a=b}.
 \item[$(\Delta_{2^+,(\ell-1)^-})$] with $\ell>4$: Note that we must have $\ell\ne 4$. Then, we have the path    
 $$\Lambda\rightarrow\Lambda_{mid}=\Lambda_1+\Lambda_{\ell-1}+\Lambda_\ell\rightarrow \Lambda_3+\Lambda_{\ell-1}+\Lambda_\ell\rightarrow\Lambda'''. $$
\end{itemize}
Finally, we consider the following case. 
\begin{itemize}
\item[$(\Delta_{(\ell-1)^-,(\ell-1)^-})$]: $\Lambda'''=\Lambda_2+2\Lambda_{\ell-2}$ and 
$\beta_{\Lambda'''}=\alpha_0+\alpha_1+2\alpha_{\ell-1}+2\alpha_\ell$. 
If $\ell=3$, then $\Lambda'''$ is in the second neighbors: $\Lambda=\Lambda_0+2\Lambda_3\rightarrow \Lambda_0+2\Lambda_2\rightarrow \Lambda'''$. If $\ell>3$, then $R^\Lambda(\beta_{\Lambda'''})$ is wild by 
Lemma \ref{tensor product lemma}. 
\end{itemize}

 \item[\textbf{Case (iii)(b):}] \textbf{four changes in \eqref{path73}}. In this case, it is enough to consider the path 
$$ \Lambda=\Lambda_0+2\Lambda_\ell+\Lambda_a \rightarrow \Lambda'=\Lambda_{2}+2\Lambda_\ell+\Lambda_a \rightarrow \Lambda''=\Lambda_2+2\Lambda_{\ell-1} +\Lambda_a \rightarrow \Lambda'''  $$
such that $\Lambda_a$ is changed in the last step and $2\le a\le \ell-1$.
First, we list cases in pattern (I'') with the paths listed below.
\begin{itemize}
    \item[$(\Delta_{(\ell-1)^+, a^+})$]: $ \Lambda\rightarrow \Lambda'\rightarrow\Lambda''' $. 
    \item [$(\Delta_{(\ell-1)^+, a^-})$]: $\Lambda\rightarrow \Lambda_0+\Lambda_{a-1}+\Lambda_{\ell-1}+\Lambda_\ell\rightarrow\Lambda'''$.
    \item[$(\Delta_{2^-,a^+})$]:  $\Lambda\rightarrow \Lambda_1+\Lambda_{a+1}+2\Lambda_{\ell}\rightarrow\Lambda'''$.
    \item [$(\Delta_{2^+,a^-})$] with $a=2$: $\Lambda\rightarrow \Lambda_1+\Lambda_3+2\Lambda_\ell\rightarrow\Lambda'''$.
\end{itemize}
The remaining are all in pattern (II) as follows.
\begin{itemize}
    \item [$(\Delta_{a^+})$]: 
    $\Lambda\rightarrow \Lambda_{mid}=\Lambda_0+\Lambda_{a+2}+2\Lambda_{\ell}\rightarrow  \Lambda_2+\Lambda_{a+2}+2\Lambda_\ell\rightarrow\Lambda''.$
    \item[$(\Delta_{a^-})$]: $\Lambda\rightarrow \Lambda_{mid}=\Lambda_0+\Lambda_{a-2}+2\Lambda_\ell\rightarrow \Lambda_2+\Lambda_{a-2}+2\Lambda_\ell\rightarrow \Lambda'''. $
    \item[$(\Delta_{2^+,a^+})$] with $\ell>4$:
   $ \Lambda\rightarrow \Lambda_2+2\Lambda_\ell+\Lambda_a\rightarrow \Lambda_{mid}=\Lambda_3+2\Lambda_\ell+\Lambda_{a+1}\rightarrow \Lambda'''. $
 \item[$(\Delta_{2^-,a^-})$]:
   $\Lambda\rightarrow \Lambda'\rightarrow \Lambda_{mid}=\Lambda_1+\Lambda_{a-1}+2\Lambda_\ell\rightarrow\Lambda'''. $
    \item[$(\Delta_{(\ell-1)^-,a^-})$]:
   $\Lambda\rightarrow\Lambda_0+\Lambda_a+2\Lambda_{\ell-1}\rightarrow \Lambda_{mid}=\Lambda_0+\Lambda_{\ell-2}+\Lambda_{\ell-1}+\Lambda_{a-1}\rightarrow \Lambda'''. $ 
    \item[$(\Delta_{2^+,a^-})$]  with $a\ge 4$: 
     $\Lambda\rightarrow \Lambda'\rightarrow \Lambda_{mid}=\Lambda_3+\Lambda_{a-1}+2\Lambda_\ell\rightarrow\Lambda'''.$
  \item [$(\Delta_{(\ell-1)^-,a^+})$]:
   $\Lambda\rightarrow \Lambda_0+\Lambda_a+2\Lambda_{\ell-1}+\Lambda_a\rightarrow \Lambda_{mid}=\Lambda_0+\Lambda_{a-1}+\Lambda_{\ell-2}+\Lambda_{\ell-1}\rightarrow\Lambda'''.$
\end{itemize}

\item[\textbf{Case (iii)(c):}] \textbf{five changes in \eqref{path73}}. In this case, it is enough to consider the path 
$$ \Lambda=\Lambda_0+2\Lambda_\ell+\Lambda_a +\Lambda_b \rightarrow \Lambda'=\Lambda_{2}+2\Lambda_\ell+\Lambda_a +\Lambda_b\rightarrow \Lambda''=\Lambda_2+2\Lambda_{\ell-1}+\Lambda_a +\Lambda_b\rightarrow \Lambda'''  $$
such that both  $\Lambda_a$ and $\Lambda_b$ are changed in the last step and $2\le a\le b\le \ell-1$.
They are all in pattern (II).
\begin{enumerate}
    \item [$(\Delta_{a^+,b^+})$]
    $\Lambda'''=\Lambda_2+2\Lambda_{\ell-1}+\Lambda_{a+1}+\Lambda_{b+1}$
    (the case $a=b=\ell-1$ can not happen): 
    $$\Lambda\rightarrow\Lambda'\rightarrow\Lambda_{mid}=\Lambda_2+\Lambda_{a+1}+\Lambda_{b+1}+2\Lambda_\ell\rightarrow\Lambda'''.$$
   
    \item[$(\Delta_{a^-,b^-})$]: 
    $\Lambda\rightarrow \Lambda_0+2\Lambda_{\ell-1}+\Lambda_a+\Lambda_b\rightarrow\Lambda_{mid}=\Lambda_0+2\Lambda_{\ell-1}+\Lambda_{a-1}+\Lambda_{b-1}\rightarrow \Lambda''' .$ 
  \item[($\Delta_{a^-,b^+}$)]$a\le b$:
    $\Lambda\rightarrow\Lambda'\rightarrow
   \Lambda_{mid}=\Lambda_2+\Lambda_{a-1}+\Lambda_{b+1}+2\Lambda_\ell\rightarrow\Lambda'''. $ 
  \item[($\Delta_{a^+,b^-}$)] $a\le b-2$:
   $\Lambda\rightarrow \Lambda_{mid}=\Lambda_0+2\Lambda_\ell+\Lambda_{a+1}+\Lambda_{b-1}\rightarrow \Lambda_2+2\Lambda_\ell+\Lambda_{a+1}+\Lambda_{b-1}\rightarrow\Lambda''' .$
  \end{enumerate}
\end{enumerate}

 \subsection{The third neighbors in case (1)}

\subsubsection{Case (1)(ii)}We first consider $\La=2\La_0+2\La_\ell+\tilde\La\rightarrow\La'\rightarrow\La''=2\La_1+2\La_{\ell-1}+\tilde\La$ and $m_0=2=m_\ell$.

\textbf{The case of three changes.}  We consider the path\[
\La=2\La_0+2\La_\ell\rightarrow\La''=2\La_1+2\La_{\ell-1}+\rightarrow\La'''.
\]
Cases in pattern (I') are $\De_{(1^-,(\ell-1)^+)}$,   $\De_{(1^-,1^-)}$,  $\De_{(1^-,(\ell-1)^+)}$,  $\De_{((\ell-1)^+,(\ell-1)^+)}$. Cases in pattern (I'') are $\De_{(1^+,1^-)}$,   $\De_{(1^-,(\ell-1)^-)}$,    $\De_{(1^+,(\ell-1)^+)}$, $\De_{((\ell-1)^+,(\ell-1)^-)}$. The following are the remaining cases.
\begin{enumerate}
    \item $\De_{1^+}$: Then $\La'''=\La_3+\La_1+2\La_{\ell-1}$ and $R^\La(\beta_{\La'''})$ is wild as $R^{2\La_0}(\beta_{\La_3+\La_1})$ is wild.
    \item $\De_{(\ell-1)^-}$: Then $\La'''=\La_{\ell-3}+\La_{\ell-1}+2\La_{1}$ and $R^\La(\beta_{\La'''})$ is wild, as is (1).

    \item $\De_{(1^+,1^+)}$: this belongs to pattern (II):  $$\La\rightarrow 2\La_1+2\La_\ell+\rightarrow\Lambda_{mid}=2\La_2+2\La_\ell+\rightarrow\La'''=2\La_2+2\La_{\ell-1}. $$

    \item $\De_{(1^+,(\ell-1)^-)}$: This is in pattern (II) by
    \begin{gather*}     \La\rightarrow\hat\La'=\La_0+\La_2+2\La_\ell 
    \rightarrow\Lambda_{mid}=\La_0+\La_2+\La_{\ell-2}+\La_\ell \\
    \rightarrow\La'''=\La_1+\La_2+\La_{\ell-2}+\La_{\ell-1}.
\end{gather*}
        Similarly,  $R^\La(\beta_{\La'''})$ is wild for the case $\De_{((\ell-1)^-,(\ell-1)^-)}$.

\end{enumerate}

\textbf{The case of four changes.}
We consider
\[
\La=2\La_0+2\La_\ell+\La_a\rightarrow\La''=2\La_1+\La_a+2\La_{\ell-1}+\rightarrow\La'''
\]
with $1\le a\le \ell-1$ such that $a$ is changed in the last step.
They are all in pattern (II):
\begin{enumerate}
    \item $\De_{a^+}$:  
    $  \La\rightarrow 2\La_1+\La_a+2\La_{\ell}\rightarrow \Lambda_{mid}=2\La_1+\La_{a+2}+2\La_{\ell} 
    \rightarrow\La'''=2\La_1+\La_{a+2}+2\La_{\ell-1}.$   
     
    \item $\De_{a^-}$:  
$\La\rightarrow2\La_1+\La_a+2\La_{\ell}\rightarrow\Lambda_{mid}=2\La_1+\La_{a-2}+2\La_{\ell}
    \rightarrow\La'''=2\La_1+\La_{a-2}+2\La_{\ell-1}.$
     
    \item $\De_{(1^-,a^+)}$:  
   $   
\La\rightarrow2\La_1+\La_a+2\La_{\ell}\rightarrow\Lambda_{mid}=\La_0+\La_1+\La_{a+1}+2\La_{\ell}
\rightarrow\La'''.$
        
    \item $\De_{(1^-,a^-)}$:  $      
\La\rightarrow2\La_1+\La_a+2\La_{\ell}\rightarrow\Lambda_{mid}=\La_0+\La_1+\La_{a-1}+2\La_{\ell}
\rightarrow\La'''.$

    \item $\De_{(1^+,a^-)}$:  
$\La\rightarrow 2\La_1+\La_a+2\La_{\ell}\rightarrow \Lambda_{mid}=\La_1+\La_2+\La_{a-1}+2\La_{\ell} 
\rightarrow\La'''.$

    \item $\De_{(1^+,a^+)}$:   
$\La\rightarrow2\La_1+\La_a+2\La_{\ell}\rightarrow\Lambda_{mid}=\La_1+\La_2+\La_{a+1}+2\La_{\ell}
\rightarrow\La'''.$

     \item The remaining four cases $\De_{((\ell-1)^-,a^-)}$, $\De_{((\ell-1)^-,a^+)}$, $\De_{((\ell-1)^+,a^-)}$, $\De_{((\ell-1)^+,a^+)}$ are dealt with in the similar manner and $R^\La(\beta_{\La'''})$ are all wild.
\end{enumerate}

\textbf{The case of five changes.}
We consider
\[
\La=2\La_0+\La_a+\La_b+2\La_\ell\rightarrow\La''=2\La_1+\La_a+\La_b+2\La_{\ell-1}+\tilde\La\rightarrow\La'''
\]
with $1\le a \le b\le \ell-1$ such that both $a$ and $b$ are changed in the last step. We first consider $1\le a<b\le \ell-1$. They are cases in pattern (II):
\begin{itemize}
   \item $\De_{(a^+,b^-)}$:
    $\Lambda\rightarrow\Lambda_{a+1}+\Lambda_{b-1}+2\Lambda_0+2\Lambda_\ell\rightarrow \Lambda_{mid}=2\Lambda_0+2\Lambda_{\ell-1}+\Lambda_{a+1}+\Lambda_{b-1}\rightarrow\Lambda'''. $
     \item $\De_{(a^-,b^-)}$: 
$\La\rightarrow 2\La_1+\La_a+\La_b+2\La_\ell\rightarrow \Lambda_{mid}=2\La_1+\La_{a-1}+\La_{b-1}+2\La_\ell
      \rightarrow\La'''.$  
  \item $\De_{(a^-,b^+)}$: 
$\La\rightarrow 2\La_1+\La_a+\La_b+2\La_\ell\rightarrow \Lambda_{mid}=2\La_1+\La_{a+1}+\La_{b-1}+2\La_\ell
\rightarrow\La'''.$
 \item $\De_{(a^+,b^+)}$:  $\La\rightarrow2\La_1+\La_a+\La_b+2\La_\ell\rightarrow\Lambda_{mid}=2\La_1+\La_{a+1}+\La_{b+1}+2\La_\ell
\rightarrow\La'''.$

\end{itemize}

If $a=b$, the following are cases in pattern (II).

\begin{enumerate}
    \item $\De_{(a^+,a^+)}$: $\La\rightarrow 2\La_1+2\La_a+2\La_\ell\rightarrow 2\La_1+2\La_{a+1}+2\La_\ell
     \rightarrow\La'''=2\La_1+2\La_{a+1}+2\La_{\ell-1}.$ 
   \item $\De_{(a^-,a^-)}$: 
$\La\rightarrow 2\La_1+2\La_a+2\La_\ell\rightarrow 2\La_1+2\La_{a-1}+2\La_\ell 
        \rightarrow\La'''=2\La_1+2\La_{a-1}+2\La_{\ell-1}.$
  \end{enumerate}   
   Finally, we consider the case $\De_{(a^+,a^-)}$: $\La'''=2\La_1+\La_{a-1}+\La_{a+1}+2\La_{\ell-1}$.
    \begin{itemize}
        \item Suppose that $2\le a\le \ell-1$. Lemma \ref{tensor product lemma} implies that $R^\La(\beta_{\Lambda'''})$ is wild since the algebra is Morita equivalent to 
        $\k[X]/(X^2)\otimes \k[Y]/(Y^2)\otimes \k[Z]/(Z^2)$.
        \item Suppose that $a=1$. This is in pattern (II) by the path $$\Lambda=2\Lambda_0+2\Lambda_1+2\Lambda_\ell\rightarrow\Lambda_{mid}=\Lambda_0+\Lambda_2+2\Lambda_1+2\Lambda_\ell \rightarrow \Lambda'''.$$ 
        \item Suppose that $a=\ell-1$. Then the case $a=1$ implies by symmetry that $R^\La(\beta_{\Lambda'''})$ is wild.
            \end{itemize}

\subsubsection{The case (1)(iii)}

We consider $\La=2\La_0+2\La_i+\tilde\La\rightarrow\La''=2\La_1+\La_{i-1}+\La_{i+1}+\tilde\La$ where $m_0=2=m_i$ and $2\le i\le \ell-1$.

\bigskip
\textbf{The case of four changes.}  We consider the path
\[
\La=2\La_0+2\La_i\rightarrow\La''=2\La_1+\La_{i-1}+\La_{i+1}\rightarrow\La'''.
\]
Cases in pattern (I') are $\Delta_{1^-,(i+1)^-}$ and $\Delta_{1^-,(i-1)^+}$.  
Cases in pattern (I'') are:
\begin{itemize}
    \item[$\Delta_{(i-1)^+}$]: $\Lambda\rightarrow2\Lambda_1+2\Lambda_i\rightarrow\Lambda'''$.
    \item [$\Delta_{(i+1)^-}$]: $\Lambda\rightarrow2\Lambda_1+2\Lambda_{i-1}\rightarrow\Lambda'''$.
    \item[$\Delta_{1^+,1^-}$]:    $\La\rightarrow\La_0+\La_2+2\La_{i}\rightarrow\La'''.$
    \item[$\Delta_{1^+,(i-1)^+}$]:    $\La\rightarrow\La_0+\La_2+2\La_{i}\rightarrow\La'''.$
    \item[$\Delta_{1^+,(i+1)^-}$]:    $\La\rightarrow\La_0+\La_2+2\La_{i}\rightarrow\La'''.$
    \item[$\Delta_{1^-,(i-1)^-}$]:    $\La\rightarrow2\La_0+\Lambda_{i-2}+\La_{i}\rightarrow\La'''.$
    \item[$\Delta_{{i-1}^+,(i+1)^+}$]:    $\La\rightarrow2\Lambda_1+2\La_{i}\rightarrow\La'''.$ 
    \item[$\Delta_{{i-1}^-,(i+1)^-}$]:    $\La\rightarrow2\Lambda_1+2\La_{i}\rightarrow\La'''.$ 
    \item[$\Delta_{{1}^-,(i+1)^+}$]:    $\La\rightarrow2\Lambda_0+\La_{i-1}+\Lambda_{i+1}\rightarrow\La'''.$
\end{itemize}
The following are cases in pattern (II):
\begin{itemize}
    \item $\De_{1^+}$: $\Lambda\rightarrow 2\Lambda_1+2\Lambda_i\rightarrow\Lambda_{mid}=\Lambda_1+\Lambda_3+2\Lambda_i\rightarrow\Lambda'''$. 
    \item $\Delta_{(i+1)^+}$ (only if $i\le\ell-3$): 
        $$\La =2\Lambda_0+2\Lambda_i\rightarrow2\La_0+\La_{i-1}+\La_{i+1}\rightarrow\Lambda_{mid}=2\Lambda_0+\Lambda_{i-1}+\Lambda_{i+3}\rightarrow\Lambda'''.$$
  \item $\Delta_{(i-1)^-}$: We consider the path
        $$\La =2\Lambda_0+2\Lambda_i\rightarrow 2\La_0+\La_{i-1}+\La_{i+1}\rightarrow\Lambda_{mid}=2\Lambda_0+\Lambda_{i-3}+\Lambda_{i+1}\rightarrow\Lambda'''.$$  
    \item $\De_{(1^+,(i+1)^+)}$: $\La\rightarrow\La_0+\La_1+\La_{i}+\La_{i+1}\rightarrow\Lambda_{mid}=\La_0+\La_2+\La_i+\La_{i+2}\rightarrow\La'''.$
   
    \item $\De_{(1^+,(i-1)^-)}$: $\La\rightarrow \La_0+\La_1+\La_{i}+\La_{i+1}\rightarrow\Lambda_{mid}=\La_0+\La_1+\La_{i-2}+\La_{i+1}\rightarrow\La'''.$ 
\end{itemize}
It remains to consider the following cases. 
\begin{itemize}
    \item $\Delta_{1^+,1^+}$: 
$\Lambda'''=2\Lambda_2+\Lambda_{i-1}+\Lambda_{i+1}$ and $\beta_{\Lambda'''}=2\alpha_0+2\alpha_1+\alpha_i$. If $i=2$, then $\Lambda'''$ is in the second neighbors: $\Lambda=2\Lambda_0+2\Lambda_2\rightarrow 2\Lambda_1+2\Lambda_2\rightarrow \Lambda'''$. If $i>2$, then $R^\Lambda(\beta_{\Lambda'''})$ is wild by the tensor product lemma.
   \item $\De_{(i-1)^-, (i+1)^+}$ (only if $2\le i\le\ell-2$):
$$\Lambda'''=2\Lambda_1+\Lambda_{i-2}+\Lambda_{i+2} 
\text{ and } \beta_{\Lambda'''}=\alpha_0+\alpha_{i-1}+2\alpha_i+\alpha_{i+1}. $$
If $i=2$, then $\Lambda'''$ is in pattern (I'') since $\Lambda\rightarrow 2\Lambda_0+\Lambda_1+\Lambda_3\rightarrow\Lambda'''$. If $i>2$, then $R^\Lambda(\beta_{\Lambda'''})$ is wild by Lemma \ref{tensor product lemma}.
\end{itemize}

\bigskip
\textbf{The case of five changes.} We consider the path
\[
\La=2\La_0+2\La_i+\La_a\rightarrow\La''=2\La_1+\La_{i-1}+\La_{i+1}+\La_a\rightarrow\La'''.
\]
with $a\ne 0,i$ and $1\le a \le \ell$ such that $a$ is changed in the last step.
The following are cases in pattern (I''): 
\begin{itemize}
    \item $\De_{(1^-,a^+)}$:  $    \La\rightarrow\La_0+\La_1+2\La_i+\La_{a+1}\rightarrow\La'''.$
    The case $\De_{(1^-,a^-)}$  is similar.
    \item $\De_{((i-1)^+,a^+)}$: $\La\rightarrow\hat\La'=2\La_1+2\La_i+\La_{a}\rightarrow\La'''.$
   The case $\De_{((i-1)^+,a^-)}$  is similar.
\end{itemize}

Next, we list cases in pattern (II). 
\begin{enumerate}
    \item $\De_{(1^+,a^+)}$:  $\La\rightarrow2\La_1+2\La_i+\La_a\rightarrow\Lambda_{mid}=\La_1+\La_2+2\La_i+\La_{a+1}\rightarrow\La'''.$ 
   The case $\De_{(1^+,a^-)}$  is similar to this case.
   
    \item $\De_{((i-1)^-,a^+)}$:
    \begin{equation*}
\La\rightarrow\hat\La'=2\La_1+2\La_i+\La_{a}\rightarrow\hat\La''=2\La_1+\La_i+\La_{i+1}+\La_{a+1}\rightarrow\La'''.
    \end{equation*} 
    The case $\De_{((i-1)^-,a^-)}$  is similar.
    
\end{enumerate}
The remaining four cases $\De_{((i+1)^+,a^+)},\De_{((i+1)^-,a^+)},\De_{((i+1)^+,a^-)},\De_{((i+1)^-,a^-)}$ are dealt with in the similar manner.

\bigskip
\textbf{The case of six changes.}
We consider the path
\[
\La=2\La_0+2\La_i+\La_a+\La_b\rightarrow\La''=2\La_1+\La_{i-1}+\La_{i+1}+\La_a+\La_b\rightarrow\La'''.
\]
with $a,b\ne 0,i$ and $1\le a \le b\le \ell$ such that both $a$ and $b$ are changed in the last step.
We first consider $a\ne b$. Then they are cases in pattern (II).

\begin{itemize}
    \item $\De_{(a^+,b^+)}$: $\La\rightarrow2\La_0+2\La_{i}+\La_{a+1}+\La_{b+1}\rightarrow2\La_1+2\La_{i}+\La_{a+1}+\La_{b+1}\rightarrow\La'''.$ 
   
    \item $\De_{(a^+,b^-)}$: $\La\rightarrow2\La_0+2\La_{i}+\La_{a+1}+\La_{b-1}\rightarrow2\La_1+2\La_{i}+\La_{a+1}+\La_{b-1}\rightarrow\La'''.$ 
    \item $\De_{(a^-,b^+)}$: 
$\La\rightarrow2\La_1+2\La_{i}+\La_{a}+\La_{b}
\rightarrow2\La_1+2\La_{i}+\La_{a-1}+\La_{b+1}\rightarrow\La'''.$
    \item $\De_{(a^-,b^-)}$
    \begin{enumerate}
        \item $b=\ell$: $\La\rightarrow2\La_1+2\La_{i}+\La_{a}+\La_{\ell}\rightarrow2\La_1+2\La_{i}+\La_{a-1}+\La_{\ell-1}\rightarrow\La'''$
        \item $b<\ell$: 
    $\Lambda\rightarrow \Lambda_{a-1}+\Lambda_{b-1}+2\Lambda_0+2\Lambda_i\rightarrow 2\Lambda_0+\Lambda_{i-1}+\Lambda_{i+1}+\Lambda_{a-1}+\Lambda_{b-1}\rightarrow\Lambda'''. $
    \end{enumerate}
\end{itemize}

Suppose $a=b$.
\begin{itemize}
    \item $\De_{(a^+,a^+)}$:  We consider the path
\begin{gather*}
    \La\rightarrow 2\La_1+2\La_i+2\La_a 
    \rightarrow 2\La_1+2\La_i+2\La_{a+1} \\
\rightarrow\La'''=2\La_1+\La_{i-1}+\La_{i+1}+2\La_{a+1}.
\end{gather*}
    \item $\De_{(a^-,a^-)}$: We consider the path 
        \begin{gather*}    \La\rightarrow\hat\La'=2\La_1+2\La_i+2\La_a\rightarrow\hat\La''=2\La_1+2\La_i+2\La_{a-1}\\
            \rightarrow\La'''=2\La_1+\La_{i-1}+\La_{i+1}+2\La_{a-1}.
        \end{gather*} 
        Then it is wild in the second neighbors unless $a=\ell$, 
        but $a=\ell$ belongs to Case (1)(ii)(iii) of the second neighbors.
   
\end{itemize}
 Finally, we consider the case $\De_{(a^+,a^-)}$: $\La'''=2\La_1+\La_{i-1}+\La_{i+1}+\La_{a-1}+\La_{a+1}$.  Then $\Lambda'''$ also belongs to the third neighbors in the Case (3)(iv) and is already treated there.

\subsubsection{The case (1)(iv)}
 
Since Case (1)(v) is equivalent to Case (1)(iv),  it remains to consider the case $\Lambda=2\Lambda_0+\tilde\Lambda\rightarrow \Lambda'=2\Lambda_1+2\tilde \Lambda\rightarrow \Lambda''=2\Lambda_2+\tilde \Lambda\rightarrow \Lambda'''$  such that $m_0=2,m_1=0$, $\ch \k\ne 2$.

{\bf Two changes}: We consider the path
$\Lambda=2\Lambda_0\rightarrow 2\Lambda_1\rightarrow 2\Lambda_2\rightarrow \Lambda'''$. 

Then Theorem \ref{level-2-a=b} implies that $R^\La(\beta_{\La'''})$ is wild.

{\bf Three changes:} 
We consider the path 
$$\Lambda\rightarrow \Lambda'=2\Lambda_1+\Lambda_i\rightarrow \Lambda''=2\Lambda_2+\Lambda_i\rightarrow \Lambda''' $$
such that $\Lambda_i$ is changed in the last step, where  $2\le i\le \ell$.
\begin{itemize}
    \item[$(\Delta_{i^+})$]
    where $2\le i\le \ell-2$. Then $\Lambda'''=2\Lambda_2+\Lambda_{i+2}$. We define
    $$ P=f_0f_1^{(3)}f_0^{(2)}(f_2f_1)(f_3f_2)\cdots(f_{i+1}f_i)v_\La\in 
    V(\La_0)\otimes V(\La_0)\otimes V(\La_i). $$ We see that $f_0^{(2)}(f_2f_1)(f_3f_2)\cdots(f_{i+1}f_i)v_\La$ is equal to
    $$ ((0),(1),(2^i1))+q^2((1),(0),(2^i1))+q^4((1),(1),(2^i)) $$ and each bipartition has four addable $1$-nodes and no removable $3$-node. After applying $f_1^{(3)}$, each bipartition has two addable $2$-nodes and no removable $2$-node.  Hence
    \begin{align*} \dim_q \End(P)&=(1+q^4)(1+q^2+q^4+q^6)(1+q^4+q^8) \\
    &=1+q^2+3q^4+3q^6+4q^8+4q^{10}+3q^{12}+3q^{14}+q^{16}+q^{18}.
    \end{align*}
    Thus, Lemma \ref{lem::wild-three-loops-part2} implies that $R^\La(\beta_{\La'''})$ is wild.

    \item[$(\Delta_{i^-})$] where $3\le i\le \ell$. Note that, if $i=2$ then we do not have the path $\La''\rightarrow \La'''$. 
    If $i=3$ then we define 
    $$ P=f_2f_1^{(2)}f_3f_4\cdots f_\ell\cdots f_4f_3f_0^{(2)}v_\La\in V(\La_0)\otimes V(\La_0)\otimes V(\La_3). $$
    Then we obtain 
    \begin{align*}
    \dim_q \End(P)&=(1+q^2+2q^4+q^6+q^8)(1+q^2+q^4+q^6) \\
    &=1+2q^2+4q^4+5q^6+5q^8+4q^{10}+2q^{12}+q^{14}. 
    \end{align*}
    Thus, Lemma \ref{local algebra 2+3} implies that $R^\La(\beta_{\La'''})$ is wild.
    Suppose $4\le i\le \ell$. Then 
    $$ \beta_{\La'''}=2\alpha_0+2\alpha_1+\alpha_{i-1}+2\alpha_i+\cdots+2\alpha_{\ell-1}+\alpha_\ell. $$
    $R^\La(\beta_{\La'''})$ is Morita equivalent to
    $$ R^{2\La_0}(2\alpha_0+2\alpha_1)\otimes R^{\Lambda_i}(\alpha_{i-1}+2\alpha_i+\cdots+2\alpha_{\ell-1}+\alpha_\ell). $$
    Furthermore, we have 
    $$ r_ir_{i+1}\cdots r_\ell\cdots r_{i+1}r_i(\alpha_{i-1}+2\alpha_i+\cdots+2\alpha_{\ell-1}+\alpha_\ell)=\alpha_i, $$
    which implies that $R^{\Lambda_i}(\alpha_{i-1}+2\alpha_i+\cdots+2\alpha_{\ell-1}+\alpha_\ell)$ is derived equivalent to $R^{\La_i}(\alpha_i)\cong\k[X]/(X^2)$, 
    which induces Morita equivalence since $\k[X]/(X^2)$ is local. It follows that 
    $R^\La(\beta_{\La'''})$ is wild. 
    
    \item[$(\Delta_{2^+i^+})$] Then $\Lambda'''=\Lambda_2+\Lambda_3+\Lambda_{i+1}$ and 
$\beta_{\La'''}=3\alpha_0+4\alpha_1+2\alpha_2+\alpha_3+\cdots+\alpha_i$. 
    We define
    $$ P=f_2^{(2)}f_0f_1^{(4)}f_0^{(2)}f_i\cdots f_3v_\La\in  V(\La_0)\otimes V(\La_0)\otimes V(\La_3), $$
    where if $i=2$ then we understand $P=f_2^{(2)}f_0f_1^{(4)}f_0^{(2)}v_\La$. Then
    $$\dim_q \End(P)=1+q^2+3q^4+3q^6+4q^8+3q^{10}+3q^{12}+q^{14}+q^{16}. $$
    Hence Lemma \ref{lem::wild-three-loops-part2} implies that
    $R^\La(\beta_{\La'''})$ is wild.
    
    \item[$(\Delta_{2^-i^-})$] 
    If $i=2$ then $\Lambda'''=2\Lambda_1+\Lambda_2$ is in the first neighbors. If $3\le i\le \ell$, then
    $$ \beta_{\La'''}=2\alpha_0+2\alpha_1+\alpha_2+\cdots+\alpha_{i-1}+2\alpha_i+\cdots+2\alpha_{\ell-1}+\alpha_\ell, $$
    and we define
    $$ P=f_2f_1^{(2)}f_3\cdots f_\ell\cdots f_if_0^{(2)}v_\La\in  V(\La_0)\otimes V(\La_0)\otimes V(\La_i). $$
    The graded dimension of $\End(P)$ is
    \begin{align*}
    \dim_q \End(P)&=(1+q^2+q^4)(1+q^2+2q^4+q^6+q^8) \\
    &=1+2q^2+4q^4+4q^6+4q^8+2q^{10}+q^{12}. 
    \end{align*}
    Hence, Lemma \ref{local algebra 2+3} implies that $R^\La(\beta_{\La'''})$ is wild.

       \item[$(\Delta_{2^+i^-})$] 
    If $i=2$ then $\Lambda'''=\Lambda_1+\Lambda_2+\Lambda_3$ and 
    $\beta_{\La'''}=2\alpha_0+2\alpha_1+\alpha_2$.
    We define
    $$ P=f_2f_1^{(2)}f_0^{(2)}v_\La\in  V(\La_0)\otimes V(\La_0)\otimes V(\La_2). $$
    The graded dimension of $\End(P)$ is
    \begin{align*}
    \dim_q \End(P)&=(1+q^2+q^4)(1+q^2+2q^4+q^6+q^8) \\
    &=1+2q^2+4q^4+4q^6+4q^8+2q^{10}+q^{12}. 
    \end{align*}
    Hence, Lemma \ref{local algebra 2+3} implies that $R^\La(\beta_{\La'''})$ is wild. $i=3$ does not happen. Suppose $4\le i\le \ell-1$. Then
    $$ \beta_{\La'''}=3\alpha_0+4\alpha_1+2\alpha_2+\alpha_3+\cdots+\alpha_{i-1}+2\alpha_i+\cdots+2\alpha_{\ell-1}+\alpha_\ell, $$
    and we define
    $$ P=f_0f_2^{(2)}f_1^{(4)}f_2\cdots f_\ell\cdots f_if_0^{(2)}v_\La\in  V(\La_0)\otimes V(\La_0)\otimes V(\La_i). $$
    Then, the graded dimension of $\End(P)$ is
    \begin{align*}
    \dim_q \End(P)&=(1+2q^2+2q^4+2q^6+2q^8+q^{10})(1+q^4)(1+q^2+q^4) \\
    &=1+3q^2+6q^4+9q^6+11q^8+11q^{10}+9q^{12}+6q^{14}+3q^{16}+q^{18}. 
    \end{align*}
    Hence, Lemma \ref{local algebra 2+3} implies that $R^\La(\beta_{\La'''})$ is wild. 
    \item[$(\Delta_{2^-i^+})$] $(2\le i\le \ell-1)$. Then
    $\beta_{\La'''}=2\alpha_0+2\alpha_1+\alpha_2+\cdots+\alpha_i$. 
    We choose $$ P=f_0f_1^{(2)}f_0f_2\cdots f_iv_\La\in V(\La_0)\otimes V(\La_0)\otimes V(\La_i). $$
    Then, the graded dimension of $\End(P)$ is
    \begin{align*}
    \dim_q \End(P)&=(1+q^4)(1+q^2+2q^4+q^6+q^8) \\
&=1+q^2+3q^4+2q^6+3q^8+q^{10}+q^{12}. 
    \end{align*}
    Hence, Lemma \ref{lem::wild-three-loops-part2} implies that
    $R^\La(\beta_{\La'''})$ is wild.
\end{itemize}

{\bf Four changes:}  
We consider the path 
$$ \Lambda= 2\La_0+\La_i+\La_j\rightarrow \Lambda'= 2\La_1+\La_i+\La_j\rightarrow \Lambda''=2\La_2+\La_i+\La_j\rightarrow \La'''$$
such that  both $\Lambda_i$ and $\La_j$ are changed in the last step, where 
 $2\le i\le j\le\ell$.
 Then we have the following cases in pattern (II):
\begin{itemize}
    \item[($\Delta_{i^+,j^+}$)] 
    \begin{itemize}
        \item[(i)] Suppose $i=j$. Then, 
        by Theorem \ref{level-2-a=b}(i''),   $$ \La=2\La_i+2\La_0\rightarrow \La_{mid}=2\La_{i+1}+2\La_0\rightarrow 
    2\La_{i+1}+2\La_1 \rightarrow \La'''. $$  
  \item[(ii)] Suppose $i<j$. Then, by Theorem \ref{level-2-a<b}(iv''),  
  $$ \La\rightarrow \La_{mid}=\La_{i+1}+\La_{j+1}+2\La_0 \rightarrow \La_{i+1}+\La_{j+1}+2\La_1 \rightarrow \La'''.$$    \end{itemize}
\item[($\Delta_{i^-,j^-}$)]
\begin{itemize}
    \item[(i)] Suppose $2\le i=j\le \ell-2$. Then, 
    by Theorem \ref{level-2-a=b}(i''),
    $$ \La\rightarrow \La_{mid}=\La_{i+1}+\La_{j+1}+2\La_0 \rightarrow \La_{i+1}+\La_{j+1}+2\La_1 \rightarrow \La'''.$$
    \item[(ii)] Suppose $2\le i<j\le\ell-1$. Then, 
    by Theorem \ref{level-2-a<b}(iv'), 
    $$ \La \rightarrow \La_{mid}=2\La_0+\La_{i-1}+\La_{j-1}\rightarrow 2\La_1+\La_{i-1}+\La_{j-1} \rightarrow \La'''. $$
    \item[(iii)] Suppose $i=2$ and $j=\ell$. Then 
    $$ 2\La_0+\Lambda_2+\La_\ell \rightarrow \La_{mid}=\La_0+\La_1+\Lambda_2+\La_{\ell-1} \rightarrow \Lambda'''=\La_1+2\La_2+\La_{\ell-1}. $$
    Then Proposition \ref{level-two-delta-plusminus} implies that 
    $R^{\La_0+\La_\ell}(\beta_{\La_1+\La_{\ell-1}})$ is wild. 
    Indeed, since $\La$ is in the second neighbors, we know that it is wild. 
    We may also appeal to the case of three changes.         \end{itemize}
  \item[($\Delta_{i^+,j^-}$)] where $2\le i\le j-2$. Then, 
  by Theorem \ref{level-2-a<b}(vi), 
  $$\La \rightarrow \La_{mid}=2\La_0+\La_{i+1}+\La_{j-1}\rightarrow 2\La_1+\La_{i+1}+\La_{j-1} \rightarrow \La'''.$$   
\end{itemize}

The following are the remaining cases.
\begin{itemize}

    \item[$(\Delta_{i^-,j^-})$]
       
    \noindent
    Suppose $2\le i=j=\ell-1$. Then $R^\La(\beta_{\La'''})$ is Morita equivalent to
    $$ R^{2\La_0}(2\alpha_0+2\alpha_1)\otimes R^{2\La_{\ell-1}}(2\alpha_{\ell-1}+\alpha_\ell).$$
    $R^{2\La_0}(2\alpha_0+2\alpha_1)$ is (t20) and 
    $R^{2\La_{\ell-1}}(2\alpha_{\ell-1}+\alpha_\ell)$ is (t2). 
    Hence, the tensor product is a wild algebra.
    Suppose $2\le i=j=\ell$. Then $R^\La(\beta_{\La'''})$ is Morita equivalent to
    $$ R^{2\La_0}(2\alpha_0+2\alpha_1)\otimes R^{2\La_\ell}(\alpha_\ell).$$
    Hence, it is wild. 
    Suppose $3\le i<j=\ell$. Then $R^\La(\beta_{\La'''})$ is Morita equivalent to
    $$ R^{2\La_0}(2\alpha_0+2\alpha_1)\otimes R^{\La_i+\La_\ell}(\alpha_i+\cdots+\alpha_\ell). $$
    $R^{\La_i+\La_\ell}(\alpha_i+\cdots+\alpha_\ell)$ is (t6) if $3\le i\le \ell-2$, (f3) if $i=\ell-1$. In both cases, the tensor product is a wild algebra.

    \item[$(\Delta_{i^-j^+})$]
    In this case, we consider $\La'''=2\La_2+\La_{i-1}+\La_{j+1}$ with $2\le i<j\le \ell-1$. Then $\beta_{\La'''}=\alpha_0+\alpha_i+\cdots+\alpha_j$ and $R^\La(\beta_{\La'''})$ is Morita equivalent to 
    $$ R^{2\La_0}(\alpha_0)\otimes R^{\La_i+\La_j}(\alpha_i+\cdots+\alpha_j), $$ where $R^{\La_i+\La_j}(\alpha_i+\cdots+\alpha_j)$ is (f4). This algebra is a Brauer tree algebra without an exceptional vertex, such that the Brauer graph is a straight line, and the number of vertices is $j-i+2\ge 3$. Therefore, $R^\La(\beta_{\La'''})$ is wild. 
\end{itemize}


\section*{Acknowledgements}
The first author is supported in part by JSPS Kakenhi (Grant No. 21K03163). 
The second author thanks Liron Speyer for his helpful comments and conversations.  
The third author is supported partially by NSFC (Grant No. 12071346) and Natural Science Foundation of Shanghai (Grant No. 25ZR1401352). 
The fourth author is supported partially by NSFC (Grant No. 12401048) and Fundamental Research Funds for the Central Universities (Grant No. DUT25RC(3)132).



\begin{thebibliography}{AAAA}
\setlength{\baselineskip}{14pt}

\bibitem[AIR14]{AIR}
{\sc T. Adachi, O. Iyama and I. Reiten},
$\tau$-tilting theory.
{\it Compos. Math.} {\bf 150} (2014), no. 3, 415--452.

\bibitem[Ai13]{Aihara-symmetric-alg}
{\sc T. Aihara},
Tilting-connected symmetric algebras.
{\it Algebr. Represent. Theory} {\bf 16} (2013), no. 3, 873-–894.

\bibitem[AI12]{AI-silting}
{\sc T. Aihara and O. Iyama},
Silting mutation in triangulated categories.
{\it J. Lond. Math. Soc. (2)} {\bf 85} (2012), no. 3, 633--668.

\bibitem[AI15]{AI-silting-discrete}
{\sc T. Aihara},
On silting-discrete triangulated categories.
{\it Proceedings of the 47th Symposium on Ring Theory and Representation Theory} (2015), 7--13.

\bibitem[AZ22]{AZ-brauer-graph}
{\sc M. Antipov and A. Zvonareva},
Brauer graph algebras are closed under derived equivalence.
{\it Math. Z.} {\bf 301} (2022), no. 2, 1963--1981.

\bibitem[Ar17]{Ar-rep-type}
{\sc S. Ariki},
Representation type for block algebras of Hecke algebras of classical type.
{\it Adv. Math.} {\bf 317} (2017), 823--845.

\bibitem[Ar21]{Ar-tame-block}
{\sc S. Ariki},
Tame block algebras of Hecke algebras of classical type.
{\it J. Aust. Math. Soc.} {\bf 111} (2021), 179--201.

\bibitem[AHSW24]{AHSW24}
{\sc S. Ariki, B. Hudak, L. Song and Q. Wang},
Representation type of higher level cyclotomic quiver Hecke algebras in affine type C.
Preprint (2024), arXiv: 2402.09940.

\bibitem[AIP15]{AIP-rep-type-A-level-1}
{\sc S. Ariki, K. Iijima and E. Park},
Representation type of finite quiver Hecke algebras of type $A^{(1)}_\ell$ for arbitrary parameters.
{\it Int. Math. Res. Not.} (2015), 6070--6135.

\bibitem[AK94]{AK-algebra}
{\sc S. Ariki and K. Koike},
A Hecke algebra of $\Z/r\Z\wr \mathfrak S_n$ and construction of its irreducible representations.
{\it Adv. Math.} {\bf 106} (1994), 216--243.

\bibitem[AKMW20]{AKMW-cellular-tamepolygrowth}
{\sc S. Ariki, R. Kase, K. Miyamoto and K. Wada}
Self-injective cellular algebras whose representation type are tame of polynomial growth,
{\it Algebr. Represent. Theory} {\bf 23} (2020), 833-871. 
corrigendum {\bf 23} (2020), 833--871.

\bibitem[AP16]{AP-rep-type-C-level-1}
{\sc S. Ariki and E. Park},
Representation type of finite quiver Hecke algebras of type $C^{(1)}_{\ell}$.
{\it Osaka J. Math.} {\bf 53} (2016), 463--488.

\bibitem[APS19]{APS-type-C} 
{\sc S. Ariki, E. Park and L. Speyer},
Specht modules for quiver Hecke algebras of type $C$. 
{\it Publ. Res. Inst. Math. Sci.} {\bf 55} (2019), no. 3, 565--626.

\bibitem[ASW23]{ASW-rep-type}
{\sc S. Ariki, L. Song and Q. Wang}, 
Representation type of cyclotomic quiver Hecke algebra of type $A^{(1)}_\ell$. 
{\it Adv. Math.} {\bf 434} (2023), 109329.

\bibitem[ARS95]{ARS}
{\sc M. Auslander, I. Reiten and O. Smal\o,}
Representation Theory of Artin algebras.
{\it Cambridge Studies in Advanced Mathematics} \textbf{36}, {\it Cambridge University Press}, 1995.

\bibitem[Au20]{Au-silting}
{\sc J. August},
On the finiteness of the derived equivalence classes of some stable endomorphism rings.
{\it Math. Z.}, {\bf 296} (2020), no. 3-4, 1157--1183.

\bibitem[BK09(1)]{bk-graded-decomp-numbers}
{\sc J. Brundan and A. Kleshchev},
Graded decomposition numbers for cyclotomic Hecke algebras.
{\it Adv. Math.} {\bf 222} (2009), 1883--1942.

\bibitem[BK09(2)]{BK-block}
{\sc J. Brundan and A. Kleshchev},
Blocks of cyclotomic Hecke algebras and Khovanov-Lauda algebras.
{\it Invent. Math.} {\bf 178} (2009), 451--484.

\bibitem[BM93]{BM-Hecke-alg}
{\sc M. Brou\'{e} and G. Malle},
Zyklotomische Heckealgebren.
{\it Ast\'{e}risque} {\bf 212} (1993), 119--189.

\bibitem[CR08]{CR-categorification}
{\sc J. Chuang and R. Rouquier},
Derived equivalences for symmetric groups and $sl_2$-categorification.
{\it Ann. of Math. (2)} {\bf 167 (1)} (2008), 245--298.

\bibitem[CH23]{CH-type-c-level-1}
{\sc C. Chung and B. Hudak},
Representation type of level 1 KLR algebras $R^{\Lambda_k}(\beta)$ in type $C$.
{\it Osaka J. Math.} {\bf 61} (2024), no. 4, 509--528.

\bibitem[CMS]{CMS-type-c-decomposition-matrix}
{\sc C. Chung, A. Mathas and L. Speyer}
Graded decomposition matrices of cyclotomic quiver Hecke algebras in type C for $n\le 12$.
in preparation.

\bibitem[DIJ19]{DIJ-tau-tilting-finite}
{\sc L. Demonet, O. Iyama and G. Jasso},
$\tau$-tilting finite algebras, bricks, and $g$-vectors.
{\it Int. Math. Res. Not.} (2019), no. 3, 852--892.

\bibitem[DJM98]{DJM-cyclotomic-q-schur-alg}
{\sc R. Dipper, G. D. James and A. Mathas}, 
Cyclotomic $q$-Schur algebras. 
{\it Math. Z.} {\bf 229} (1998), 385--416.

\bibitem[D80]{Dr-tame-wild}
{\sc Yu. A. Drozd},
Tame and wild matrix problems, in: Representation Theory II, Lecture Notes in Math., vol. 832.
{\it Springer Verlag}, (1980), pp. 242--258.

\bibitem[E90]{Er-tame-block}
{\sc K. Erdmann},
Blocks of tame representations type and related algebras. Lecture Notes in Mathematics, Vol. 1428.
{\it Springer-Verlag}, Berlin, 1990.

\bibitem[EN02]{EN-rep-type-Hecke}
{\sc K. Erdmann and D. K. Nakano},
Representation type of Hecke algebras of type $A$.
{\it Trans. Amer. Math. Soc.} {\bf 354} (2002), 275--285.

\bibitem[EM22]{EM-cellular-symmetrictypeA}
{\sc A. Evseev and A. Mathas}, 
Content systems and deformations of cyclotomic KLR algebras of type A and C. 
{\it Ann. Represent. Theory}  {\bf 1} (2024), no. 2, 193--297.

\bibitem[F06]{F-ariki-koike-alg}
{\sc M. Fayers},
Weights of multipartitions and representations of Ariki-Koike algebras.
{\it Adv. Math.} {\bf 206} (2006), 112--144.

\bibitem[H02]{H-wild-two-point}
{\sc Y. Han},
Wild two-point algebras.
{\it J. Algebra} {\bf 247} (2002), no.1, 57--77.

\bibitem[HHK07]{HHK-handbook}
{\sc L. A. Hügel, D. Happel and H. Krause},
Handbook of tilting theory.
{\it London Math. Soc. Lecture Note Ser.} {\bf 332}, Cambridge University Press, Cambridge, 2007.

\bibitem[Ka90]{K-Lie-alg}
{\sc Victor G. Kac},
Infinite-dimensional Lie algebras, 3rd ed.
{\it Cambridge University Press}, Cambridge, 1990.

\bibitem[KK12]{KK-categorification}
{\sc S.-J. Kang and M. Kashiwara},
Categorification of highest weight modules via Khovanov-Lauda-Rouquier algebras.
{\it Invent. Math.} {\bf 190 (3)} (2012), 699--742.

\bibitem[KL09]{kl-diagrammatic}
{\sc M. Khovanov and A. D. Lauda}, 
A diagrammatic approach to categorification of quantum groups, I.
{\it Represent. Theory} {\bf 13} (2009), 309--347.

\bibitem[KOO20]{KOO}
{\sc Young-Hun Kim, se-jin Oh and Young-Tak Oh},
Cyclic sieving phenomenon on dominant maximal weights over affine Kac-Moody algebras.
{\it Adv. Math.} {\bf 374} (2020), 107336.

\bibitem[KZ09]{KZ-der-Equiv-book}
{\sc S. Koenig and A. Zimmermann},
Derived Equivalences for Group Rings.
{\it Springer}, Berlin, 2009.

\bibitem[Kr98]{Krause-rep-type-stable-equi}
{\sc  H. Krause}, 
Representation type and stable equivalence of Morita type for finite dimensional algebras.
{\it Math. Z.} {\bf 229} (1998), 601--606.

\bibitem[LM07]{LM-cyclotomic-hecke}
{\sc S. Lyle and A. Mathas},
Blocks of cyclotomic Hecke algebras.
{\it Adv. Math.} {\bf 216} (2007), 854--878.

\bibitem[LV11]{LV-crystal computation}
{\sc A. Lauda and M. Vazirani},
Crystals from categorified quantum groups.
{\it Adv. Math.} {\bf 228} (2011), 803--861.

\bibitem[MT21]{MT-cellularity-type-C}
{\sc A. Mathas and D. Tubbenhauer}, 
Cellularity and subdivision of KLR and weighted KLRW algebras. 
{\it Math. Ann.} {\bf 389} (2024), no. 3, 3043--3122.

\bibitem[MT23]{MT-decnumbers-type-C}
{\sc A. Mathas and D. Tubbenhauer}, 
Cellularity of KLR and weighted KLRW algebras via crystals. 
Preprint (2023), arXiv: 2309.13867.

\bibitem[M24]{Murata-tensor product lemma}
{\sc H. Murata}, 
Private communication.

\bibitem[OZ22]{OZ-brauer-graph}
{\sc S. Opper and A. Zvonareva},
Derived equivalence classification of Brauer graph algebras.
{\it Adv. Math.} {\bf 402} (2022), 108341.

\bibitem[Ric91]{Rickard-derived-equi}
{\sc J. Rickard},
Derived equivalences as derived functors.
{\it J. London Math. Soc.} {\bf 43} (1991), 37--48.

\bibitem[Ric89]{Rickard-tilting-complex}
{\sc J. Rickard},
Morita theory for derived categories.
{\it J. Lond. Math. Soc} {\bf 39 (2)} (1989), 436--456.

\bibitem[Rin75]{Ringel-local-alg}
{\sc C. M. Ringel},
The representation type of local algebras. Lecture Notes in Mathematics, Vol. 488.
{\it Springer} Berlin, 1975.

\bibitem[Ro08]{Ro-2kac}
{\sc R. Rouquier}, 
2-Kac-Moody algebras. 
Preprint (2008), arXiv: 0812.5023.

\bibitem[Sc18]{Schroll-Brauer-graph}
{\sc S. Schroll},
Brauer graph algebras: a survey on Brauer graph algebras, associated gentle algebras and their connections to cluster theory. Homological methods, representation theory, and cluster algebras. (2018), 177--223.

\bibitem[Sk06]{Sk06}
{\sc A. Skowro$\acute{\text{n}}$ski},
Selfinjective algebras: finite and tame type,
Trends in Representation Theory of Algebras and Related Topics, 169--238, Contemp. Math.
{\bf 406}, {\it Amer. Math. Soc.}, 2006.

\bibitem[SVV17]{SVV-center}
{\sc P. Shan, M. Varagnolo, and E. Vasserot},
On the center of quiver {H}ecke algebras.
{\it Duke Math. J.} {\bf 166} (2017), no.6, 1005--1101.

\bibitem[W22]{W-two-point-II}
{\sc Q. Wang},
$\tau$-tilting finiteness of two-point algebras II. 
{\it J. Algebra Appl.} {\bf 24} (2025), no. 2, Paper No. 2550054, 33 pp.

\bibitem[Y99]{Y-derived-local}
{\sc A. Yekutieli},
Dualizing complexes, Morita equivalence and derived Picard group of a ring. 
{\it J. London Math. Soc.} {\bf 60} (1999), 723--746.
\end{thebibliography}
\end{document}